\documentclass[times,final]{elsarticle}

\usepackage{jcomp}

\usepackage{framed,multirow}

\usepackage{amssymb,amsmath,amsthm,bm}
\usepackage{latexsym}
\usepackage{mathtools}

\usepackage{cleveref}
\crefname{section}{Section}{Sections}
\usepackage{subcaption}
\usepackage{graphicx}

\usepackage{booktabs}
\usepackage{makecell}
\usepackage{enumitem} 

\usepackage{url}
\usepackage{xcolor}
\usepackage{tikz} 

\definecolor{newcolor}{rgb}{.8,.349,.1}

\makeatletter

\newcommand{\Rmnum}[1]{\uppercase\expandafter{\romannumeral #1}}
\makeatother

\newtheorem{theorem}{Theorem}[section]
\newtheorem{assumption}{Assumption}[section]

\newtheorem{problem}[theorem]{Problem}

\newtheorem{corollary}{Corollary}[section]

\theoremstyle{definition}
\newtheorem{definition}[theorem]{Definition}
\newtheorem{expl}{Example}[section]

\theoremstyle{remark}
\newtheorem{remark}{Remark}[]

\newcommand{\dt}{\Delta t}

\newcommand{\xbm}{\bm x}

\def\jump#1{\llbracket #1 \rrbracket }

\def\norm #1{\Vert #1 \Vert }

\allowdisplaybreaks

\journal{Journal of Computational Physics}
\begin{document}

\verso{Shengrong Ding, Shumo Cui, Kailiang Wu}

\begin{frontmatter}

\title{ {\bf Robust DG Schemes on Unstructured Triangular Meshes: Oscillation Elimination and Bound Preservation via Optimal Convex Decomposition}  \tnoteref{tnote1}
		}

\tnotetext[tnote1]{This work is partially supported by Shenzhen Science and Technology Program (Grant No. RCJC20221008092757098) and 
	National Natural Science Foundation of China (Grant No.~12171227).}

\author[1]{Shengrong {Ding}}
\ead{dingsr@sustech.edu.cn}
\author[1]{Shumo {Cui}}
\ead{cuism@sustech.edu.cn}
\author[2,1]{Kailiang {Wu}\corref{cor1}}
\cortext[cor1]{Corresponding author.}
\ead{wukl@sustech.edu.cn}

\address[1]{Shenzhen International Center for Mathematics, Southern University of Science and Technology, Shenzhen 518055, China}
\address[2]{Department of Mathematics, Southern University of Science and Technology, Shenzhen 518055, China}


\begin{abstract}

Discontinuous Galerkin (DG) schemes on unstructured meshes offer the advantages of compactness and the ability to handle complex computational domains. However, their robustness and reliability in solving hyperbolic conservation laws depend on two critical abilities: suppressing spurious oscillations and preserving intrinsic bounds or constraints. This paper introduces two significant advancements in enhancing the robustness and efficiency of DG methods on unstructured meshes for general hyperbolic conservation laws, while maintaining their inherent accuracy, compactness, and high parallel efficiency. First, we investigate the oscillation-eliminating (OE) DG methods on unstructured meshes, as a continuation of [M.~Peng, Z.~Sun, and K.~Wu, \textit{Mathematics of Computation}, {\tt doi.org/10.1090/mcom/3998}]. These methods not only maintain key features such as conservation, scale invariance, and evolution invariance but also achieve rotation invariance through a novel rotation-invariant oscillation-eliminating (RIOE) procedure. These properties are crucial for ensuring consistently good performance across problems with varying scales, wave speeds, and coordinate systems. Second, we propose, for the first time, the optimal convex decomposition approach for designing efficient bound-preserving (BP) DG schemes on unstructured meshes. Convex decomposition is a critical component in constructing high-order BP schemes and directly influences the BP Courant--Friedrichs--Lewy (CFL) number. Finding the optimal convex decomposition that maximizes the BP CFL number is an important yet challenging problem. While this challenge has been systematically addressed for rectangular meshes in [S.~Cui, S.~Ding, and K.~Wu, {\em Journal of Computational Physics}, 476: 111882, 2023; {\em SIAM Journal on Numerical Analysis}, 62: 775--810, 2024], {\em it remains an open problem for triangular meshes}. This paper successfully constructs the optimal convex decomposition for the widely used $\mathbb{P}^1$ and $\mathbb{P}^2$ spaces on triangular cells, significantly improving the efficiency of $\mathbb{P}^1$- and $\mathbb{P}^2$-based BP DG methods. For example, in the 2D Euler system, the maximum BP CFL numbers are increased by 100\%–200\% for $\mathbb{P}^1$ space and 280.38\%–350\% for $\mathbb{P}^2$ space, compared to BP schemes based on classic decomposition. Furthermore, our new RIOE procedure and optimal decomposition technique can be seamlessly integrated into existing DG codes with minimal and localized modifications. These techniques require only edge-neighboring cell information, thereby retaining the compactness and high parallel efficiency characteristic of original DG methods. The proposed BP OEDG schemes are validated through extensive numerical experiments, demonstrating their accuracy, efficiency, robustness, and BP properties, as well as their capability to handle complex computational domains.

\end{abstract}



\end{frontmatter}



\section{Introduction}

Hyperbolic conservation laws serve as fundamental mathematical models for exploring phenomena across various fields, including fluid dynamics, traffic flow, astrophysics, and space physics. This paper focuses on the following two-dimensional (2D) form of these laws:
\begin{equation}\label{eq:hcl}
	{\bm u}_t + \nabla_{\bm x} \cdot \bm{F}({\bm u}) = {\bm 0}, \qquad \xbm:=(x,y) \in \Omega \subset \mathbb{R}^2, ~~~ t>0,
\end{equation}
where $\bm{F} = (\bm{f}_1, \bm{f}_2)$, and the conservative variables 
\(\bm{u}\) and the fluxes \(\bm{f}_i\) for \(i = 1, 2\) take values in \(\mathbb{R}^d\), with \(d\) representing the number of equations. The nonlinear nature of most hyperbolic conservation laws poses notable challenges for both analytical solutions and numerical simulations. This nonlinearity often results in solutions that develop discontinuities, even when the initial conditions are smooth. Consequently, the design of robust and efficient numerical methods for hyperbolic conservation laws is crucial. However, the inherent characteristics of these laws make high-order numerical methods susceptible to spurious oscillations near discontinuities, leading to numerical instability, nonphysical solutions, and potential computational failures. Thus, the development of high-order, robust numerical methods, particularly on unstructured meshes, is an important and ongoing challenge in this field.

Unstructured meshes are particularly flexible for computational domains with complex geometries and are extensively used in the simulation of practical engineering problems. Over the past few decades, numerous high-order numerical methods applicable on unstructured meshes have been developed, including the discontinuous Galerkin (DG) method \cite{rkdg2,rkdg4,rkdg5,ZXSPP2012}, the finite volume method \cite{HuShu1999WENO,ZhuQiu2018SISC}, the spectral volume method \cite{Wang2002SVI,Wang2002SVIV}, and the spectral difference method \cite{Wang2006SDI,Wang2007SDII}. 
The DG methods, in particular, offer notable advantages due to their use of discontinuous piecewise polynomial finite element spaces. These advantages include a compact stencil, which reduces the communication overhead between elements, and high parallel efficiency, making them well-suited for modern high-performance computing architectures. The first DG method, introduced by Reed and Hill \cite{reed1973triangular}, was originally designed for solving steady-state linear hyperbolic equations. Later, Cockburn, Shu, and collaborators \cite{rkdg1,rkdg2,rkdg3,rkdg4,rkdg5} made substantial contributions to the field by combining the DG method with explicit high-order Runge--Kutta (RK) time discretization, enabling the effective solution of nonlinear time-dependent hyperbolic conservation laws.

Controlling spurious oscillations near discontinuities is a critical aspect of employing DG methods for hyperbolic conservation laws \eqref{eq:hcl}, particularly in cases involving strong shocks. To address this challenge, two primary approaches are commonly utilized. 
The first approach is to introduce suitable artificial viscosity terms (e.g., \cite{zingan2013implementation,yu2020study,huang2020adaptive}), typically based on second-order or higher-order spatial derivatives, to dissipate spurious oscillations.  
The second widely used approach is to employ nonlinear limiters, which prevent nonphysical oscillations by adjusting the DG solution in specific troubled cells. Examples of these limiters include the total variation diminishing (TVD) limiter \cite{Leveque2002}, the minmod-type total variation bounded (TVB) limiter \cite{rkdg2,rkdg3,rkdg4,rkdg5}, the moment limiter \cite{Biswas1994}, the monotonicity-preserving limiter \cite{SureshMP1997}, and the weighted essentially nonoscillatory (WENO) type limiters \cite{qiu2005runge,ZhuQSD2008JCP,zhong2013simple}. These limiters act as post-processing steps, applied after each RK stage. However, many limiters, such as the TVB limiter, may be problem-dependent. Additionally, some limiters can degrade the accuracy of high-order DG methods when incorrectly applied to smooth regions of the solution, as noted in \cite{ZhuQSD2008JCP}. 
While certain limiters, such as those described in \cite{SureshMP1997,qiu2005runge}, can maintain high-order accuracy while mitigating spurious oscillations, they typically require a large stencil for limiting or reconstruction. This requirement can compromise the compactness of the original DG method.

Recently, Lu, Liu, and Shu \cite{lu2021oscillation, LiuLuShu_OFDG_system} introduced a novel approach to controlling spurious oscillations by incorporating damping terms into the semi-discrete DG scheme, termed the oscillation-free DG (OFDG) method. In their work \cite{lu2021oscillation}, they rigorously analyzed the optimal error estimates and superconvergence properties of the OFDG method for linear advection equations. 
The OFDG method has been successfully applied to various systems, including the compressible Euler system \cite{LiuLuShu_OFDG_system, LiuLuShu2024}, shallow-water equations \cite{liu2022oscillation}, multi-component chemically reacting systems \cite{du2023oscillation}, and nonlinear degenerate parabolic equations \cite{tao2023oscillation}. However, the damping terms in the OFDG methods become highly stiff near strong shocks, necessitating very small time steps in explicit time-stepping methods to ensure stability during time integration. To address this issue, Liu, Lu, and Shu suggested using (modified) exponential Runge–Kutta (RK) methods \cite{huang2018bound}. 
Motivated by the damping technique of OFDG method, Peng, Sun, and Wu \cite{PengSunWu2023OEDG} systematically developed the oscillation-eliminating DG (OEDG) method for hyperbolic conservation laws. They provided a rigorous proof of the optimal error estimates for the fully discrete OEDG method for linear scalar conservation laws. The OEDG method effectively eliminates spurious oscillations by solving a quasi-linear damping ordinary differential equation (ODE), known as the OE procedure. Notably, this OE procedure is completely independent of the DG space discretization and time integration method, enabling the OEDG method to use general explicit RK time methods without encountering issues related to excessively restricted time steps. 
Moreover, the OEDG method \cite{PengSunWu2023OEDG} bridged, for the first time, the damping and filtering techniques, revealing the essential role of the OE procedure as a filter based on the jump information of the DG solution. 
More recently, the bound/structure-preserving OEDG schemes were developed in \cite{LiuWuOEMHD2024}  for ideal magnetohydrodynamics and in \cite{yan2024uniformly} for two-phase flows. However, the existing OEDG and OFDG methods do not, in general, preserve the rotation-invariant (RI) property inherent in physical systems such as the Euler and magnetohydrodynamics equations.

In addition to exhibiting discontinuous wave structures, the solutions to hyperbolic conservation laws \eqref{eq:hcl} typically adhere to certain invariant regions. When using numerical schemes to solve these equations, it is highly desirable, or even critical, to ensure that the numerical solutions remain within these invariant regions. Numerical methods that preserve these invariant regions are commonly referred to as bound-preserving (BP) or positivity-preserving schemes \cite{wu2015high,wu2018positivity,wu2019provably,Wu2023Geometric,wu2023provably}.  
While most first-order monotone schemes are BP under an appropriate CFL condition, designing high-order BP schemes is more challenging. To address this issue, Zhang and Shu \cite{zhang2010maximum,zhang2010positivity} introduced a widely-used framework for constructing high-order BP DG and finite volume schemes on rectangular meshes. 
A crucial aspect of this framework is the convex decomposition of cell averages, which is essential for both proving the BP property and determining the BP CFL number, which directly affects the maximum allowable time step size. The convex decomposition proposed by Zhang and Shu, known as the classic decomposition, is based on Gauss–Lobatto quadrature. In fact, the feasible decomposition is not unique. To achieve higher efficiency, it is natural to seek the optimal convex decomposition that maximizes the BP CFL number.

Finding the optimal convex decomposition is a highly nontrivial task and has remained an open problem since its proposal in 2010 \cite{zhang2010maximum}. 
Recently, in \cite{CDWOCAD2023,CuiDingWu2024}, we established a generic approach and theory to investigate the optimal convex decomposition on rectangular cells. 
We proved that the classic convex decomposition is already optimal for general one-dimensional $\mathbb{P}^k$ spaces and general 2D $\mathbb{Q}^k$ spaces on rectangular cells. However, for the widely-used 2D $\mathbb{P}^k$ spaces on rectangular cells, we found that the classic convex decomposition is not optimal. 
To address this, we constructed the optimal convex decomposition with 
 analytic formulas for $\mathbb{P}^k$ spaces with $k \leq 7$ on rectangular cells, and proposed a quasi-optimal convex decomposition that is more practical for cases where $k \geq 8$. By allowing larger time steps, these advances significantly improve the efficiency of high-order BP methods on rectangular meshes for a broad class of hyperbolic equations.
High-order BP DG schemes on 2D unstructured triangular meshes were proposed in \cite{ZXSPP2012}, where the convex decomposition on triangular cells was constructed by projecting the classic convex decomposition from rectangular to triangular cells. {\em However, whether this convex decomposition is optimal, and how to find the optimal convex decomposition on triangular cells, remain open questions.} Addressing these problems would provide a new opportunity to greatly improve the efficiency of high-order BP numerical schemes on unstructured meshes for general hyperbolic conservation laws.

This paper focuses on enhancing the robustness and efficiency of DG methods for general hyperbolic conservation laws on unstructured meshes while maintaining their accuracy and compactness. The new contributions are twofold. First, we investigate the OE procedure for unstructured meshes and propose a novel rotation-invariant OE (RIOE) procedure for physical systems, demonstrating its effectiveness in eliminating spurious oscillations. Second, we derive the optimal convex decomposition on triangular cells for the widely-used $\mathbb{P}^1$ and $\mathbb{P}^2$ spaces, which significantly improves the efficiency of BP DG schemes on unstructured meshes by allowing for much larger time steps. 
Compared to its counterpart on rectangular cells, {\em the optimal convex decomposition problem on triangular cells is much more challenging}. The geometric symmetry of rectangular cells notably simplifies the convex decomposition problem; however, this symmetry is absent in general triangular cells. In this paper, we focus on the $\mathbb{P}^1$ and $\mathbb{P}^2$ spaces on triangular cells, leaving the exploration of optimal convex decomposition for $\mathbb{P}^k$ spaces with $k \geq 3$ to future studies. 
With the proposed RIOE procedure and optimal convex decomposition, the resulting BP OEDG methods offer the following remarkable features:

\begin{description}
	\item[(i) Robustness.] The proposed DG schemes are rigorously proven to be BP under a mild CFL condition, ensuring the preservation of crucial physical constraints (e.g., the positivity of density and pressure for hydrodynamic equations), which guarantees robustness when simulating challenging problems. The proposed RIOE procedure effectively suppresses spurious oscillations without the need for problem-dependent parameters across all tested cases. It also exhibits scale invariance, evolution invariance, and rotation invariance, ensuring consistently good performance and reliability for problems of varying scales, wave speeds, and coordinates. 
 \item[(ii) Efficiency.] The proposed RIOE procedure can be applied in a cell-by-cell manner, requiring information only from edge-neighboring cells rather than a large stencil. This feature allows the OEDG methods to retain the compactness and high parallel efficiency characteristic of DG methods. Additionally, our  RIOE procedure is applied directly to the conservative variables, avoiding the need for characteristic decomposition for hyperbolic systems, thus reducing computational costs. Furthermore, the optimal convex decomposition presented in this paper significantly increases the BP CFL numbers for $\mathbb{P}^1$- and $\mathbb{P}^2$-based BP DG schemes. For instance, in the case of the 2D Euler system, {\em the maximum BP CFL numbers are increased by 100\%--200\% and 280.38\%--350\% for $\mathbb{P}^1$ and $\mathbb{P}^2$ spaces}, respectively, compared to BP schemes using the classic decomposition described in \cite{ZXSPP2012}. This increase allows for fewer time steps to reach the final time, resulting in a proportional reduction in CPU time in numerical experiments.
  \item[(iii) Ease of Implementation.] The proposed RIOE procedure can be seamlessly integrated into existing DG codes in a nonintrusive manner. It is applied solely to the numerical results at the end of each RK stage or time step, allowing for implementation as a separate module with minimal or no modifications to the original DG code. To ensure the BP property provided by the optimal convex decomposition, one only needs to replace the points involved in the BP limiter with our optimal decomposition nodes, requiring only slight and local modifications to the implementation code.
\end{description}

This paper is organized as follows. \Cref{sec:scalar} presents the OEDG method  on unstructured meshes for scalar conservation law, including the semi-discrete DG formulation, the fully discrete OEDG schemes, and the OE procedure. The scale-invariant and evolution-invariant properties of the OEDG method are also discussed. 
The OEDG methods on unstructured meshes for systems are detailed in \cref{sec:system}, where the component-wise and RIOE procedures are proposed. \Cref{sec:OCD} introduces the construction of $\mathbb{P}^1$- and $\mathbb{P}^2$-based BP schemes 
via optimal convex decomposition on triangular cells.
\Cref{sec:numexp} provides various numerical experiments to validate the robustness and efficiency of the proposed DG schemes. 
Finally, the concluding remarks are presented in \cref{sec:conclu}.

\section{OEDG Method on Unstructured Triangular Meshes for Scalar Conservation Law}\label{sec:scalar}
In this section, we introduce the OEDG method for 2D scalar conservation law
\begin{equation}\label{eq:scalar}
	\frac{\partial u}{\partial t} + \frac{\partial f_1(u)}{\partial x} + \frac{\partial f_2(u)}{\partial y} = 0, \qquad \xbm \in \Omega \subset \mathbb{R}^2.
\end{equation}
Assume that the unstructured mesh $\mathcal{T}_h$ is a partition of the 2D computational domain $\Omega$. As shown in \Cref{fig:MeshK}, the following notations are introduced:
\begin{itemize}
	\item $K$ denotes any element in $\mathcal{T}_h$.
	\item $e_K^{(i)}$ represents the $i$-th edge of element $K$ for $i=1,2,3$.
	\item $\mathbf{n}_{K}^{(i)}=\big(n_{K}^{(i),1},n_{K}^{(i),2}\big)^\top$ is the outward unit normal vector of the edge $e_K^{(i)}$.
	\item $K^{(i)}$ signifies the neighboring cell that shares the edge $e_K^{(i)}$ with $K$.
	\item ${\bm v}^{(1)}, {\bm v}^{(2)}$, and ${\bm v}^{(3)}$ denote the three vertices of element $K$, arranged counterclockwise, with corresponding coordinates $(x_1,y_1)$, $(x_2,y_2)$, and $(x_3,y_3)$, respectively.	
\end{itemize}
Additionally, $|K|$ represents the area of element $K$, and $l_K^{(i)}$ stands for the length of edge $e_K^{(i)}$. The time interval is discretized into a non-uniform mesh $\{t^n, n \geq 0\}$ with $t^0 = 0$.

\begin{figure}[!htb]
	\begin{center}
		\begin{tikzpicture}[xscale=2.2,yscale=2.2]
			\draw [ultra thick] (0,0) -- (1,-0.6) -- (1.5,0.2) -- (2.3,0.8) -- (1,1.3) -- (-0.5,0.9) -- (0,0);
			\draw [ultra thick, blue] (0,0) -- (1.5,0.2) -- (1,1.3) -- (0,0);
			\draw [blue] (0.75,0.1) node[below] {$e^{(3)}_{K}$};
			\draw [blue] (1.25,0.75) node[right] {$e^{(1)}_{K}$};
			\draw [blue] (0.5,0.65) node[left] {$e^{(2)}_{K}$};
			
			\draw [blue] (0.7,0.6) node[right] {$K$};
			\draw (0.9,-0.5) node[above] {$K^{(3)}$};
			\draw (2.2,0.75) node[left] {$K^{(1)}$};
			\draw (-0.45,0.85) node[right] {$K^{(2)}$};
			
			\filldraw [black] (0,0) circle (1pt);
			\draw [blue] (0,0) node[left] {$(x_1,y_1)$~~$\bm v^{(1)}$};
			
			\filldraw [black] (1.5,0.2) circle (1pt);
			\draw [blue] (1.5,0.2) node[right] {$\bm v^{(2)}$~~$(x_2,y_2)$};
			
			\filldraw [black] (1,1.3) circle (1pt);
			\draw [blue] (1,1.3) node[above] {$\bm  v^{(3)}$~~$(x_3,y_3)$};
		\end{tikzpicture}
	\end{center}
	\vspace{-0.5cm}
	\caption{Illustration and notations of the element $K\in \mathcal{T}_h$ and its adjacent elements.}
	\label{fig:MeshK}
\end{figure}
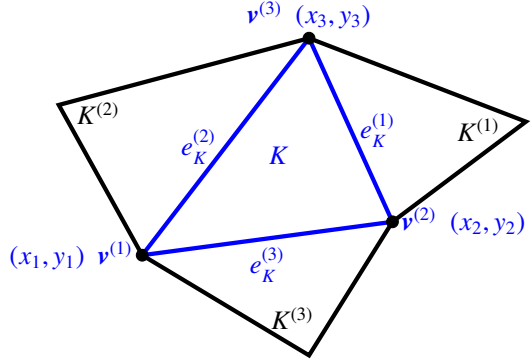


\subsection{Semi-discrete DG Formulation}

Define the finite element space $\mathbb{V}_h^k$ as follows:
\[
\mathbb{V}_h^k = \left\{ v_h \in L^2(\Omega) : v_h|_{K} \in \mathbb{P}^k(K) \quad \forall K \in \mathcal{T}_h \right\},
\]
where $\mathbb{P}^k(K)$ represents polynomials of degree up to $k$ over the element $K$. 
In DG schemes, we seek the numerical solution $u_h(\bm{x},t) \in \mathbb{V}_h^k$ that satisfies
\begin{equation}\label{DGScalar:Int}
	\begin{aligned}
		\int_{K} \frac{\partial u_h(\bm{x},t) }{\partial t} v_h(\bm{x}) \, \mathrm{d} \bm{x} &=
		\int_{K} \bm{F}(u_h(\bm{x},t)) \cdot \nabla v_h(\bm{x}) \, \mathrm{d} \bm{x} \\
		&\quad -  \sum_{i=1}^{3} \int_{e_K^{(i)}} \hat{\bm{F}}(u_h^{\text{int}}(\bm{x},t), u_h^{\text{ext}}(\bm{x},t), \mathbf{n}_{K}^{(i)}) v_h(\bm{x}) \, \mathrm{ds} \qquad \forall v_h(\bm{x}) \in \mathbb{V}_h^k,
	\end{aligned}
\end{equation}
where $u_h^{\text{int}}(\bm{x},t)$ and $u_h^{\text{ext}}(\bm{x},t)$ are the interior and exterior approximations with respect to $K$, respectively, and $\hat{\bm{F}}(\cdot, \cdot, \cdot)$ denotes the numerical flux function. For example, one may use the Lax--Friedrichs (LF) flux:
\begin{equation}\label{fluxLF}
	\hat{\bm{F}}(u^{\text{int}}, u^{\text{ext}}, \mathbf{n}) = \frac{1}{2} \left( \bm{F}(u^{\text{int}}) \cdot \mathbf{n} + \bm{F}(u^{\text{ext}}) \cdot \mathbf{n} - \alpha^{\text{LF}} (u^{\text{ext}} - u^{\text{int}}) \right).
\end{equation}
Here, $\alpha^{\text{LF}}$ is the numerical viscosity parameter. A standard choice is $\alpha^{\text{LF}} = \max \limits_{K}\{\alpha_K^{(1)}, \alpha_K^{(2)}, \alpha_K^{(3)}\}$, where $\alpha_{K}^{(i)}$ denotes an estimate of the local maximum wave speed in $K$ along the direction of $\mathbf{n}_{K}^{(i)}$, for example, we take 
\begin{equation*}
	\alpha_{K}^{(i)} 
	:= \left| \bm{F}^{\prime}(\overline{u}_K) \cdot \mathbf{n}_{K}^{(i)} \right|
	= \left| n_{K}^{(i),1} \frac{\partial f_1(\bar{u}_K)}{\partial u} + n_{K}^{(i),2} \frac{\partial f_2(\bar{u}_K)}{\partial u} \right|,
\end{equation*}
where the cell average $\bar{u}_K$ over the element $K$ is defined as:
\[
\bar{u}_K := \frac{1}{|K|} \int_{K} u_h(\bm{x},t) \, \mathrm{d} \bm{x}.
\]
To achieve a provable BP property, $\alpha^{\text{LF}}$ may alternatively be taken as
\begin{equation}\label{eq:358}
	\alpha^{\text{LF}}
	:=
	\max_{\bm{x} \in e_K^{(i)}, \, i \in \{1,2,3\}, \, K \in \mathcal{T}_h} 
	\left| \bm{F}^{\prime}\left(\bm{u}_h(\bm{x}) \right) \cdot \mathbf{n}_{K}^{(i)} \right|.
\end{equation}

Let $\{\Psi_{K}^{(\ell)}(\bm x)\}_{\ell=0}^{L_k}$ with $L_k={k(k+3)}/{2}$ be a orthogonal basis of the polynomial space $\mathbb{P}^{k}(K)$. 
For example, we take the following orthogonal basis functions of $\mathbb{P}^4(K)$: 
\begin{equation}\label{eq:basisfunc}
	\left\{
	\begin{aligned}
		\Psi^{(0)}_{K}(\bm x)&=1, \\ 
		\Psi^{(1)}_{K}(\bm x) &= 4 \xi + 2 \eta - 2, \\ 
		\Psi^{(2)}_{K}(\bm x) &= 3\eta - 1, \\
		\Psi^{(3)}_{K}(\bm x)&= 24\xi^2 + 24\xi \eta + 4 \eta^2 - 24\xi - 8\eta + 4, \\
		\Psi^{(4)}_{K}(\bm x)&= 20 \xi \eta + 10 \eta^2 - 4 \xi - 12 \eta + 2, \\
		\Psi^{(5)}_{K}(\bm x)&=10\eta^2 - 8\eta + 1, \\
		\Psi^{(6)}_{K}(\bm x)&=160 \xi^3 + 240 \xi^2 \eta + 96 \xi \eta^2 + 8 \eta^3 - 240 \xi^2 - 192 \xi \eta  - 24 \eta^2 + 96 \xi + 24 \eta - 8, \\
		\Psi^{(7)}_{K}(\bm x)&= 168 \xi^2 \eta + 168 \xi \eta^2 + 28 \eta^3 - 24 \xi^2 - 192 \xi \eta - 60 \eta^2 + 24 \xi + 36 \eta - 4, \\
		\Psi^{(8)}_{K}(\bm x)&= 84 \xi \eta^2 + 42 \eta^3 - 48 \xi  \eta - 66 \eta^2 + 4 \xi + 26 \eta - 2, \\
		\Psi^{(9)}_{K}(\bm x)&= 35 \eta^3 - 45 \eta^2 + 15 \eta - 1, \\
		\Psi^{(10)}_{K}(\bm x)&= 1120 \xi^4 + 2240 \xi^3 (\eta - 1) + 1440 \xi^2 (\eta - 1)^2 + 320 \xi (\eta - 1)^3 + 16(\eta - 1)^4,\\
		\Psi^{(11)}_{K}(\bm x)&= (80 \xi^2 + 80 \xi (\eta - 1) + 8(\eta - 1)^2)(\eta + 2 \xi - 1)(9 \eta - 1),\\
		\Psi^{(12)}_{K}(\bm x)&= (24 \xi^2 + 24 \xi (\eta - 1) + 4 (\eta - 1)^2)(36 \eta^2 - 16 \eta + 1),\\
		\Psi^{(13)}_{K}(\bm x)&= - 4 \xi - 44 \eta + 84 \eta \xi +
		210 \eta^2 - 336 \xi \eta^2 - 336 \eta^3 + 336 \xi \eta^3 + 168 \eta^4 + 2,\\
		\Psi^{(14)}_{K}(\bm x)&= 126 \eta^4 - 224 \eta^3 + 126 \eta^2 - 24 \eta + 1,
	\end{aligned}
	\right.
\end{equation}
where the variables $\xi$ and $\eta$ are derived by the following local coordinate transformation
\begin{equation*}
	\begin{pmatrix}
		x-x_1 \\
		y-y_1
	\end{pmatrix}
	= J_K
	\begin{pmatrix}
		\xi \\
		\eta
	\end{pmatrix}
	\quad \textrm{with} \quad
	J_K := \begin{pmatrix}
		x_2-x_1 & x_3-x_1 \\
		y_2-y_1 & y_3-y_1
	\end{pmatrix}.
\end{equation*}
Here, $J_K$ is an invertible matrix and $\det(J_K) = 2|K|$. 
Then, the numerical solution $u_h(\bm x,t)$ can be expressed as 
\begin{equation*}
	u_h(\bm x,t) :=  \sum_{\ell=0}^{L_k} u_K^{(\ell)}(t) \Psi_{K}^{(\ell)}(\bm x) \quad \forall \bm x\in K,
\end{equation*}
where $u_K^{(\ell)}$ represents the $\ell$-th modal coefficient.
It is worth noting that $\bar{u}_K = u_K^{(0)}$.

To achieve $(k+1)$-th-order accuracy for the $\mathbb{P}^k$-based DG scheme, we employ an $N$-point quadrature rule with algebraic precision of at least order $2k$ to approximate the first integral on the right-hand side of \eqref{DGScalar:Int}. 
Let $\left\{(x_K^{\mu}, y_K^{\mu})\right\}_{\mu=1}^N$ (refer to \ref{Appx:Triangle}) represent the quadrature nodes on the triangular element $K$. The associated quadrature weights $\{ \omega_{\mu} \}_{\mu=1}^{N}$ satisfy $\sum_{\mu=1}^{N} \omega_{\mu} = 1$. We set $N = 3, 6,$ and $12$ for $k = 1, 2,$ and $3$, respectively; please refer to \ref{Appx:Triangle} for more details. 
Let $\left\{(x^{(i),\nu}_K, y^{(i),\nu}_K)\right\}_{\nu=1}^Q$ (see equation \eqref{eq:AppxEdge} in \ref{Appx:Edge}) be the $Q$-point Gauss quadrature nodes on the edge $e_K^{(i)}$, with the quadrature weights $\{ \omega_{\nu}^{\tt G} \}_{\nu=1}^{Q}$ satisfying $\sum_{\nu=1}^{Q} \omega_{\nu}^{\tt G} = 1$. Here, we use $Q = k+1$ to achieve $(k+1)$-th-order accuracy. 
Then, for $\ell=0, \ldots, L_k$, we obtain 
\begin{equation}\label{DGScalar:Dis}
	\begin{aligned}
		a_K^{(\ell)} \frac{\mathrm{d} u_K^{(\ell)}(t)}{\mathrm{d} t} = 
		&|K| \sum_{\mu=1}^{N} \omega_{\mu} \, \bm{F}\big(u_h(x_K^{\mu}, y_K^{\mu},t)\big) \cdot \nabla \Psi_{K}^{(\ell)}(x_K^{\mu}, y_K^{\mu}) \\
		& - \sum_{i=1}^{3} l_K^{(i)} \sum_{\nu=1}^{Q} \omega_{\nu}^{\tt G} \, \hat{\bm{F}}\left(u_h^{\text{int}}(x^{(i),\nu}_K, y^{(i),\nu}_K,t), u_h^{\text{ext}}(x^{(i),\nu}_K, y^{(i),\nu}_K,t), \mathbf{n}_{K}^{(i)} \right) \, \Psi_{K}^{(\ell)}(x^{(i),\nu}_K, y^{(i),\nu}_K)
	\end{aligned}
\end{equation}
with
\[
a_K^{(\ell)} := \int_{K} \big(\Psi_{K}^{(\ell)}(\bm{x})\big)^2 \, \mathrm{d} \bm{x}.
\]

\subsection{Fully Discrete OEDG Schemes}

The ordinary differential system \eqref{DGScalar:Dis} for the  modal coefficients  $\big\{u_{K}^{(\ell)}(t)\big\}_{\ell=0,K \in \mathcal{T}_h}^{L_k}$ can be denoted as $\frac{\mathrm{d} u_h}{\mathrm{d} t} = \mathcal{L}(u_h)$. To solve this system in time, one can employ a high-order strong-stability-preserving (SSP) time integration method \cite{gottlieb2001strong}. For example, consider the explicit third-order SSP Runge–Kutta (RK) method:
\begin{equation}\label{RKTime}
	\left\{
	\begin{aligned}
		{u}_h^{[1]} &= {u}_h^{n} + \Delta t \, \mathcal{L}({u}_h^{n}), \\
		{u}_h^{[2]} &= \frac{3}{4}{u}_h^{n} + \frac{1}{4} \left({u}_h^{[1]} + \Delta t \, \mathcal{L}({u}_h^{[1]}) \right), \\
		{u}_h^{n+1} &= \frac{1}{3}{u}_h^{n} + \frac{2}{3} \left({u}_h^{[2]} + \Delta t \, \mathcal{L}({u}_h^{[2]}) \right),
	\end{aligned}
	\right.
\end{equation}
where $u_h^n = u_h(\bm{x}, t^n)$ is the numerical solution at time $t^n$, and $\Delta t := t^{n+1} - t^n$ denotes the time step size.

It is well known that the conventional RKDG method \eqref{RKTime} may generate spurious oscillations near strong discontinuities. Failure to control these oscillations can lead to severe numerical instability and simulation breakdowns. Therefore, we propose an OE procedure $\mathcal{F}_{\tau}$ (refer to \cref{sec:oep}) after each RK stage to eliminate potential spurious oscillations in the numerical solutions. Specifically, when the third-order SSPRK method \eqref{RKTime} is used, the resulting OEDG method reads:
\begin{equation*}
	\left\{
	\begin{aligned}
		{u}_h^{[1]} &= {u}_{\sigma}^{n} + \Delta t \, \mathcal{L}({u}_{\sigma}^{n}), & 
		{u}_{\sigma}^{[1]} &= \mathcal{F}_{\tau}({u}_h^{[1]}), \\
		{u}_h^{[2]} &= \frac{3}{4}{u}_{\sigma}^{n} + \frac{1}{4} \left({u}_{\sigma}^{[1]} + \Delta t \, \mathcal{L}({u}_{\sigma}^{[1]}) \right), &
		{u}_{\sigma}^{[2]} &= \mathcal{F}_{\tau}({u}_h^{[2]}), \\
		{u}_h^{n+1} &= \frac{1}{3}{u}_{\sigma}^{n} + \frac{2}{3} \left({u}_{\sigma}^{[2]} + \Delta t \, \mathcal{L}({u}_{\sigma}^{[2]}) \right), &
		{u}_{\sigma}^{n+1} &= \mathcal{F}_{\tau}({u}_h^{n+1}),
	\end{aligned}
	\right.
\end{equation*}
where 
 the OE procedure applied to ${u}_h^n$ is denoted as ${u}_{\sigma}^{n} = \mathcal{F}_{\tau}{u}_h^n$.

\subsection{OE Procedure for Scalar Equation} \label{sec:oep}

We use the numerical solution $u_h^n$ as an example to illustrate the OE operator $\mathcal{F}_{\tau}: \mathbb{V}_h^k \rightarrow \mathbb{V}_h^k$, which is defined by
\begin{equation}\label{oe:scalar}
	\begin{aligned}
		u_{\sigma}^n(\bm{x}) &= (\mathcal{F}_{\tau} u_h^n)(\bm{x}) = u_{\sigma}(\bm{x},\tau) \Big|_{\tau=\Delta t} \\
		&= u_K^{(0)}(t^n) \Psi_{K}^{(0)}(\bm{x}) +
		\sum_{m=1}^{k} \exp\left\{-\Delta t \sum_{j=0}^{{m}} \sigma_K^{j}(u_h^n)\right\} \sum_{\ell=L_{m-1}+1}^{L_m} u_K^{(\ell)}(t^n) \Psi_{K}^{(\ell)}(\bm{x}) \qquad \forall \bm{x} \in K,
	\end{aligned}
\end{equation}
where the damping coefficient $\sigma_K^{j} (u_h^n)$ plays a crucial role in suppressing spurious oscillations. It should be small in smooth regions of the numerical solution and large near discontinuities \cite{lu2021oscillation,PengSunWu2023OEDG}. Specifically, the coefficient $\sigma_K^{j} (u_h^n)$ is defined as:
\begin{equation}\label{OEScalar:Coef1}
	\sigma_K^j(u_h) = \sum_{i=1}^{3} \frac{\beta_{K}^{(i)}}{h_{K}^{(i)}} \delta_K^{(i),j}(u_h), \quad 0 \leq j \leq k.
\end{equation}
Here, $ h_{K}^{(i)} =  \sup_{\bm{x} \in K} \mathrm{dist} (\bm{x}, e_K^{(i)}) $, and $\beta_{K}^{(i)}$ approximates the local maximum wave speed in the direction of the outward unit normal vector $\mathbf{n}_{K}^{(i)}$ on the edge $ e_K^{(i)} $. For the scalar conservation law \eqref{eq:scalar}, we define $\beta_{K}^{(i)}$ as:
\begin{equation}\label{OEScalar:beta}
	\beta_{K}^{(i)}:= \max_{u \in \{u_c,u_d\}} \left\{
	\left| \mathbf{n}_{K}^{(i)} \cdot \frac{\partial \bm{F}(u^-)}{\partial u} \right|,
	\left| \mathbf{n}_{K}^{(i)} \cdot \frac{\partial \bm{F}(u^+)}{\partial u} \right| \right\}.
\end{equation}
In equation \eqref{OEScalar:beta}, $ u^- $ and $ u^+ $ denote the left and right limits of $ u_h^n $ at the interface $ e_K^{(i)} $, respectively; $ u_c $ and $ u_d $ represent the values of $ u_h^n $ at the two endpoints $\bm{v}_c$ and $\bm{v}_d$ of the edge $ e_K^{(i)} $. 
Furthermore,
\begin{equation}\label{OEScalar:Coef2}
	\delta_K^{(i),j}(u_h) =
	\left\{
	\begin{aligned}
		&0, & &u_h \equiv \bar{u}_h^{\Omega},\\
		&A^{k,j} \frac{(h_{K}^{(i)})^j}{\|u_h-\bar{u}_h^{\Omega}\|_{L^{\infty}(\Omega)}} \sqrt{\frac{1}{l_K^{(i)}} \int_{e_K^{(i)}}  \sum_{|\bm{\alpha}|=j} C_j^{\alpha_1} \jump{ \partial^{\bm{\alpha}} u_h}_{e_K^{(i)}}^2 \, \mathrm{ds}}, & & \text{otherwise},
	\end{aligned}
	\right.
\end{equation}
where $ A^{k,j} = \frac{2j+1}{(2k-1) j!} $ is a constant depending on the degree $ k $ and $ j $,
\[
\bar{u}_h^{\Omega} := \frac{1}{|\Omega|} \int_{\Omega} u_h \, \mathrm{d} \bm{x} = \frac{1}{|\Omega|} \sum_{K \in \mathcal{T}_h} |K| \bar{u}_K
\]
denotes the average of the numerical solution $ u_h $ over the entire computational domain $ \Omega $, and $ |\Omega| $ represents the area of $ \Omega $. 
To mitigate the impact of round-off errors on $\|u_h-\bar{u}_h^{\Omega}\|_{L^{\infty}(\Omega)}$, we introduce a small positive number $\epsilon$, set to $10^{-12}$. Then, the condition $ u_h \equiv \bar{u}_h^{\Omega} $ is replaced with 
\[
\|u_h-\bar{u}_h^{\Omega}\|_{L^{\infty}(\Omega)} \leq \epsilon \max\{1, |\bar{u}_h^{\Omega}|\}
\]
in our simulations. 
In \eqref{OEScalar:Coef2}, $ C_j^{\alpha_1} = \frac{j!}{\alpha_1!(j-\alpha_1)!} $, and $\partial^{\bm{\alpha}} u_h = \frac{\partial^{|\bm{\alpha}|} u_h}{\partial x^{\alpha_1} \partial y^{\alpha_2}}$ with the multi-index $\bm{\alpha} = (\alpha_1, \alpha_2)^\top$ and $\alpha_1 + \alpha_2 = |\bm{\alpha}|$. The notation $\jump{\partial^{\bm{\alpha}} u_h}_{e_K^{(i)}}$ stands for the jump of $\partial^{\bm{\alpha}} u_h$ across the edge $ e_K^{(i)} $. 
Note that the integral in \eqref{OEScalar:Coef2} should be approximated using a quadrature rule. For simplicity, we adopt the trapezoidal rule in our simulations. Thus,
\begin{equation*}
	\frac{1}{l_K^{(i)}} \int_{e_K^{(i)}} \sum_{|\bm{\alpha}|=j} C_j^{\alpha_1} \jump{ \partial^{\bm{\alpha}} u_h}_{e_K^{(i)}}^2 \, \mathrm{ds} \approx \frac{1}{2} \sum_{|\bm{\alpha}|=j} C_j^{\alpha_1} \left(\jump{ \partial^{\bm{\alpha}} u_h}_{\bm{v}_c}^2 + \jump{ \partial^{\bm{\alpha}} u_h}_{\bm{v}_d}^2\right),
\end{equation*}
where $\jump{ \partial^{\bm{\alpha}} u_h}_{\bm{v}_c}$ and $\jump{ \partial^{\bm{\alpha}} u_h}_{\bm{v}_d}$ refer to the jumps of $\partial^{\bm{\alpha}} u_h$ across the edge $ e_K^{(i)} $ at the two endpoints $\bm{v}_c$ and $\bm{v}_d$, respectively.

Equation \eqref{OE:m0} indicates that the OE procedure does not change the cell average of the numerical solution, thus preserving local mass conservation. Since the OE procedure damps the modal coefficients of the DG solution, the OEDG method remains stable under the standard CFL condition of the original DG method. Moreover, it retains many desirable properties of the original DG method, such as optimal convergence rates, as supported by the theoretical justification for linear advection equations in \cite{PengSunWu2023OEDG} and the numerical evidence in section \ref{sec:numexp}. 

The OE procedure $\mathcal{F}_{\tau}$ in \eqref{oe:scalar} is derived from the solution of a damping equation. Specifically, we modify the DG solution $ u_h^n $ to $ u_{\sigma}(\bm{x}, \Delta t) $, where $ u_{\sigma}(\bm{x}, \tau) \in \mathbb{V}_h^k $ is the solution of the following damping equation:
\begin{equation}\label{DGScalar:OE}
	\int_{K} \frac{\mathrm{d} u_{\sigma}(\bm{x}, \tau)}{\mathrm{d} \tau} v_h(\bm{x}) \, \mathrm{d} \bm{x} + \sum_{m=0}^{k} \sigma_K^{m}(u_h) \int_{K} (u_{\sigma} - \mathcal{P}^{m-1} u_{\sigma})(\bm{x}, \tau) v_h(\bm{x}) \, \mathrm{d} \bm{x} = 0 \quad \forall v_h(\bm{x}) \in \mathbb{P}^k(K),
\end{equation}
with the initial condition $ u_{\sigma}(\bm{x},0) = u_h^{n} $. Here, $\tau$ is a pseudo-time distinct from $t$. The operator $\mathcal{P}^m (m \geq 0)$ denotes the standard $L^2$ projection into $\mathbb{V}_h^{m}$, meaning that for any function $w(\bm{x})$, $\mathcal{P}^m w \in \mathbb{V}_h^m$ satisfies
\begin{equation}\label{OEScalar:Opera}
	\int_{K} (\mathcal{P}^m w - w)(\bm{x}) v_h(\bm{x}) \, \mathrm{d} \bm{x} = 0 \quad \forall v_h \in \mathbb{P}^m(K).
\end{equation}
This implies that $\mathcal{P}^0 w(\bm{x}) = \frac{1}{|K|} \int_{K} w(\bm{x}) \, \mathrm{d} \bm{x}$, which represents the cell average of $ w(\bm{x}) $ over $K$. We define $\mathcal{P}^{-1} = \mathcal{P}^{0}$. 
The damping equation \eqref{DGScalar:OE} is linear and can be solved exactly. According to the orthogonal basis presented in \eqref{eq:basisfunc}, $ u_{\sigma}(\bm{x}, \tau) $ can be expressed as
\[
u_{\sigma}(\bm{x}, \tau) = \sum_{\ell=0}^{L_k} u_{\sigma}^{(\ell)}(\tau) \Psi_{K}^{(\ell)}(\bm{x}) \in \mathbb{V}_h^k
\]
with the  modal coefficients $\{u_{\sigma}^{(\ell)}(\tau)\}_{\ell=0}^{L_k}$. Combining this with the definition \eqref{OEScalar:Opera} of the operator $\mathcal{P}$, we have
\begin{equation*}
	\begin{aligned}
		(u_{\sigma} - \mathcal{P}^{m-1} u_{\sigma})(\bm{x},\tau)
		&= \sum_{i=L_{m-1}+1}^{L_k} u_{\sigma}^{(i)}(\tau) \Psi_{K}^{(i)}(\bm{x}) 
		= \sum_{j=m}^{k} \sum_{i=L_{j-1}+1}^{L_j} u_{\sigma}^{(i)}(\tau) \Psi_{K}^{(i)}(\bm{x}) \quad \text{for} \quad 0 \leq m \leq k.
	\end{aligned}
\end{equation*}
Here, we define $ L_{-1} = L_0 $. Taking $ v_h(\bm{x}) = \Psi_{K}^{(\ell)}(\bm{x}) $ for $ 0 \leq \ell \leq L_k $, we obtain
\begin{align*}
	& a_K^{(0)} \frac{\mathrm{d} u_{\sigma}^{(0)}(\tau)}{\mathrm{d} \tau} = 0, \quad \text{for} \quad \ell = 0, \quad m = 0, \\
	& a_K^{(\ell)} \left(\frac{\mathrm{d} u_{\sigma}^{(\ell)}(\tau)}{\mathrm{d} \tau} + \sum_{j=0}^{m} \sigma_K^{j}(u_h^n) u_{\sigma}^{(\ell)}(\tau) \right) = 0, \quad \text{for} \quad L_{m-1}+1 \leq \ell \leq L_m, \quad 1 \leq m \leq k.
\end{align*}
Integrating these equations from $\tau = 0$ to $\Delta t$ yields
\begin{align}
	& u_{\sigma}^{(0)}(\Delta t) = u_{K}^{(0)}(t^n), \quad \text{for} \quad \ell = 0,  \label{OE:m0} \\
	& u_{\sigma}^{(\ell)}(\Delta t) = u_{K}^{(\ell)}(t^n) \exp\Bigg\{-\Delta t \sum_{j=0}^{m} \sigma_K^{j}(u_h^n)\Bigg\}, \quad \text{for} \quad L_{m-1}+1 \leq \ell \leq L_m, \quad 1 \leq m \leq k. \nonumber 
\end{align}
Thus, we obtain the solution $ u_{\sigma}^n(\bm{x}) $ in \eqref{oe:scalar} for the equation \eqref{DGScalar:OE}.

\begin{remark}[Compactness]
	In our simulations, $\|u_h-\bar{u}_h^{\Omega}\|_{L^{\infty}(\Omega)}$ is computed as
	\begin{equation}\label{normOmg}
		\|u_h-\bar{u}_h^{\Omega}\|_{L^{\infty}(\Omega)} = \max_{K \in \mathcal{T}_h} \max_{1 \leq \mu \leq N} \Big|u_h(x_K^\mu,y_K^\mu,t) - \bar{u}_h^{\Omega}\Big|,
	\end{equation}
	where $\left\{(x_K^{\mu}, y_K^{\mu})\right\}_{\mu=1}^N$ (refer to \ref{Appx:Triangle}) denotes the quadrature nodes of an $N$-point quadrature rule on the triangle element $K$. From \eqref{normOmg}, it is evident that $\|u_h-\bar{u}_h^{\Omega}\|_{L^{\infty}(\Omega)}$ depends solely on the known solution $u_h$, allowing its computation beforehand, similar to determining the time step size $\Delta t$. Consequently, $\delta_K^{(i),j}(u_h)$ depends only on the numerical solution within the element $K$ and its adjacent elements. Furthermore, from \eqref{OEScalar:beta}, we observe that $\beta_{K}^{(i)}$ involves only the limiting values at the endpoints of the edge $e_K^{(i)}$. Therefore, the damping coefficient $\sigma_K^j(u_h)$ exhibits compactness and locality. Consequently, the OEDG scheme retains the compactness and efficient parallel implementation of the original DG method.
\end{remark}

\begin{remark}[Simplicity and Efficiency]
	From \eqref{oe:scalar}, it is apparent that the OE procedure only requires a straightforward modification for the  modal coefficients with $\ell \geq 1$. This modification can be readily integrated into existing DG codes through an independent module. Furthermore, the OE procedure does not necessitate characteristic decomposition. Hence, the implementation of the OE procedure is simple and efficient.
\end{remark}

\subsection{Scale Invariance and Evolution Invariance}

To define scale invariance, let $\hat{u}_h = a {u}_h + b$, where $a \neq 0$ and $b$ are constants. Using the coefficients defined in \eqref{OEScalar:Coef1} and \eqref{OEScalar:Coef2}, it is straightforward to verify that $\delta_K^{(i),j}(\hat{u}_h) = \delta_K^{(i),j}({u}_h)$ and $\sigma_K^{j}(\hat{u}_h) = \sigma_K^{j}({u}_h)$. Consequently, the proposed OE procedure $\mathcal{F}_{\tau}$ satisfies the scale-invariant property:
\begin{equation*}\label{oe:sip}
    \mathcal{F}_{\tau} \hat{u}_h = a \mathcal{F}_{\tau} {u}_h + b.
\end{equation*}
This property ensures that the coefficient $\delta_K^{(i),j}({u}_h)$ in \eqref{OEScalar:Coef2} is dimensionless. When combined with $\beta_{K}^{(i)}$ in \eqref{OEScalar:beta}, the argument $\dt \sum_{j=0}^{{m}} \sigma_K^{j}(u_h^n)$ of the exponential in \eqref{oe:scalar} remains dimensionless. This maintains the damping strength and effect, regardless of changes in the scale of $ u_h$.
The dimensionless nature is crucial for ensuring the robustness of OEDG methods across problems with different scales.

Next, consider the following equation:
\begin{equation}\label{eip:eq}
    \frac{\partial u}{\partial t} + \lambda \frac{\partial f_1(u)}{\partial x} + \lambda \frac{\partial f_2(u)}{\partial y} = 0, \qquad \xbm \in \Omega \subset \mathbb{R}^2,
\end{equation}
where $\lambda > 0$, and let the time step size be $\tau_{\lambda} = \frac{\dt}{\lambda}$. Let $\mathcal{F}_{\tau_{\lambda}}^{\lambda}$ denote the OE procedure for equation \eqref{eip:eq}, and $\hat{\beta}_{K}^{(i)}$ represent the approximation of the local maximum wave speed in the outward unit normal vector $\mathbf{n}_{K}^{(i)}$ on the edge $e_K^{(i)}$ for equation \eqref{eip:eq}. 
From \eqref{OEScalar:beta}, we know that $\hat{\beta}_{K}^{(i)} = \lambda {\beta}_{K}^{(i)}$. Since our damping coefficient $\sigma_K^{j}(\cdot)$ in \eqref{OEScalar:Coef1} incorporates the local maximum wave speed information (i.e., ${\beta}_{K}^{(i)}$ for $\mathcal{F}_{\tau}$ and $\hat{\beta}_{K}^{(i)}$ for $\mathcal{F}_{\tau_{\lambda}}^{\lambda}$), the following relation holds:
\begin{equation*}\label{eip}
    \mathcal{F}_{\tau_{\lambda}}^{\lambda} = \mathcal{F}_{\lambda \tau_{\lambda}} = \mathcal{F}_{\tau} ~~~ \forall \lambda>0.
\end{equation*}	
This indicates that our OE procedure $\mathcal{F}_{\tau}$ possesses the evolution-invariant property. Specifically, if $u_h^{\lambda}$ is the numerical solution obtained using the OEDG method for equation \eqref{eip:eq} on a fixed mesh $\mathcal{T}_h$ at time $t=t^n/\lambda$, then under the same initial condition, $u_h^{\lambda} = u_h$. 
The evolution-invariant property ensures that the OE procedure maintains appropriate damping strength for different wave speeds. Moreover, it guarantees that the OEDG scheme produces consistent results at a fixed time, regardless of whether the problem involves slow or fast propagating wave speeds, provided the initial condition is the same.

Therefore, our OE procedure $\mathcal{F}_{\tau}$ demonstrates both scale invariance and evolution invariance. These properties are crucial for effectively eliminating nonphysical oscillations in problems with varying scales and wave speeds (see also  \cite{PengSunWu2023OEDG} for the related discussions on the one-dimensional OEDG schemes and the numerical evidence in Section 5). As a result, these properties contribute to obtaining reliable and robust numerical results.

\section{OEDG Schemes on Unstructured Triangular Meshes for Hyperbolic Systems}\label{sec:system}

In this section, we discuss the OEDG method and RIOE procedure for hyperbolic systems of conservation laws:
\begin{equation}\label{eq:System}
	\frac{\partial \bm{u}}{\partial t} + \frac{\partial \bm{f}_1(\bm{u})}{\partial x} + \frac{\partial \bm{f}_2(\bm{u})}{\partial y} = 0,
\end{equation}
where $\bm{x} \in \Omega \subset \mathbb{R}^2$, $\bm{u} := (u^{(1)}, \ldots, u^{(d)})^\top \in \mathbb{R}^d$ are the conservative variables, and the fluxes $\bm{f}_1$ and $\bm{f}_2$ take values in $\mathbb{R}^d$ with $d \geq 2$.

\subsection{Fully-discrete OEDG Schemes}

For the system \eqref{eq:System}, the semi-discrete DG schemes are formulated as follows: find the numerical solution $\bm{u}_h(\bm{x}, t) \in [\mathbb{V}_h^k]^d$ such that, for any $\bm{v}_h(\bm{x}) \in [\mathbb{V}_h^k]^d$, we have
\begin{equation}\label{DGSystem:Int}
	\begin{aligned}
		\int_{K} \frac{\partial \bm{u}_h(\bm{x},t) }{\partial t} \circ \bm{v}_h(\bm{x}) \, \mathrm{d} \bm{x} =
		\int_{K} \bm{F}(\bm{u}_h(\bm{x},t)) : \nabla \bm{v}_h(\bm{x}) \, \mathrm{d} \bm{x}    
		-  \sum_{i=1}^{3} \int_{e_K^{(i)}} \hat{\bm{F}}(\bm{u}_h^{\mathrm{int}}(\bm{x},t), \bm{u}_h^{\mathrm{ext}}(\bm{x},t), \mathbf{n}_{K}^{(i)}) \circ \bm{v}_h(\bm{x}) \, \mathrm{ds},
	\end{aligned}
\end{equation}
for each triangular cell $K \in \mathcal{T}_h$. Here, “$:$” represents the Frobenius inner product of two matrices, “$\circ$” denotes the Hadamard product of two vectors, and the LF flux from \eqref{fluxLF} is used with the numerical viscosity parameter $\alpha^{\mathrm{LF}} = \max \limits_{K}\{\alpha_K^{(1)}, \alpha_K^{(2)}, \alpha_K^{(3)}\}$. For hyperbolic systems, a choice for the parameter $\alpha_{K}^{(i)} $ is
\begin{equation*}
	\alpha_{K}^{(i)} = R(\overline{\bm{u}}_K, \mathbf{n}_{K}^{(i)}),
\end{equation*}
where $R(\bm{u}, \mathbf{n})$ is the spectral radius of the Jacobian matrix $\bm{F}'(\bm{u}) \cdot \mathbf{n} := n^{1} \frac{\partial \bm{f}_1(\bm{u})}{\partial \bm{u}} + n^{2} \frac{\partial \bm{f}_2(\bm{u})}{\partial \bm{u}}$ in the direction of the outward unit normal vector  $\mathbf{n}=(n^1,n^2)$ of the edge $e$, 
and $\overline{\bm{u}}_K$ is the cell average over element $K$, defined as
\[
\overline{\bm{u}}_K := \frac{1}{|K|} \int_{K} \bm{u}_h(\bm{x},t) \, \mathrm{d} \bm{x}.
\]
Using the orthogonal basis from \eqref{eq:basisfunc}, the numerical solution $\bm{u}_h(\bm{x}, t)$ can be expressed as
\[
\bm{u}_h(\bm{x},t) :=  \sum_{\ell=0}^{L_k} \bm{u}_K^{(\ell)}(t) \Psi_{K}^{(\ell)}(\bm{x}) \quad \forall \bm{x} \in K,
\]
where the vector $\bm{u}_K^{(\ell)}(t)$ is the $\ell$-th  modal coefficient of the numerical solution $\bm{u}_h(\bm{x},t)$. In particular, $\overline{\bm{u}}_K = \bm{u}_K^{(0)}(t)$. Similar to the scalar case, we use the quadrature rules presented in \ref{Appx:Triangle} and \ref{Appx:Edge} to approximate the integrals on the right-hand side of \eqref{DGSystem:Int}. 
For $\ell = 0, \ldots, L_k$, the ODEs for the modal coefficients $\bm{u}_K^{(\ell)}(t)$ are given by
\begin{equation}\label{DGSystem:Dis}
	\begin{aligned}
		a_K^{(\ell)} \frac{\mathrm{d} \bm{u}_K^{(\ell)}(t)}{\mathrm{d} t} = 
		&|K| \sum_{\mu=1}^{N} \omega_{\mu} \bm{F}(\bm{u}_h(x_K^{\mu}, y_K^{\mu},t)) \cdot \nabla \Psi_{K}^{(\ell)}(x_K^{\mu}, y_K^{\mu}) \\
		& - \sum_{i=1}^{3} l_K^{(i)} \sum_{\nu=1}^{Q} \omega_{\nu}^{\tt G} \hat{\bm{F}}\left(\bm{u}_h^{\mathrm{int}}(x^{(i),\nu}_K, y^{(i),\nu}_K,t), \bm{u}_h^{\mathrm{ext}}(x^{(i),\nu}_K, y^{(i),\nu}_K,t), \mathbf{n}_{K}^{(i)} \right) \Psi_{K}^{(\ell)}(x^{(i),\nu}_K, y^{(i),\nu}_K).
	\end{aligned}
\end{equation}
The ODEs in \eqref{DGSystem:Dis} can be formally rewritten as $\frac{\mathrm{d} \bm{u}_h}{\mathrm{d} t} = \mathcal{L}(\bm{u}_h)$, and are typically discretized in time using an SSP time integration method. To suppress nonphysical oscillations, the OE procedure $\mathcal{\bm{F}}_{\tau}$ (refer to \cref{sec:oep2,sec:oep3}) is applied after each RK stage. For instance, with the third-order SSP RK method \eqref{RKTime}, the resulting fully-discrete OEDG method for system \eqref{eq:System} is given by
\begin{equation*}
	\begin{aligned}
		\bm{u}_h^{[1]} &= \bm{u}_{\sigma}^{n} + \Delta t \mathcal{L}(\bm{u}_{\sigma}^{n}), & 
		\bm{u}_{\sigma}^{[1]} &= \mathcal{F}_{\tau}\bm{u}_h^{[1]}, \\
		\bm{u}_h^{[2]} &= \frac{3}{4}\bm{u}_{\sigma}^{n} + \frac{1}{4} \left(\bm{u}_{\sigma}^{[1]} + \Delta t \mathcal{L}(\bm{u}_{\sigma}^{[1]}) \right), &
		\bm{u}_{\sigma}^{[2]} &= \mathcal{F}_{\tau}\bm{u}_h^{[2]}, \\
		\bm{u}_h^{n+1} &= \frac{1}{3}\bm{u}_{\sigma}^{n} + \frac{2}{3} \left(\bm{u}_{\sigma}^{[2]} + \Delta t \mathcal{L}(\bm{u}_{\sigma}^{[2]}) \right), & 
		\bm{u}_{\sigma}^{n+1} &= \mathcal{F}_{\tau}\bm{u}_h^{n+1},
	\end{aligned}
\end{equation*}
where the OE procedure applied to $\bm{u}_h^n$ is denoted as $\bm{u}_{\sigma}^{n} = \mathcal{F}_{\tau}\bm{u}_h^n$.

\subsection{Component-wise OE Procedure} \label{sec:oep2}

As in the scalar case, the OE operator $\mathcal{F}_{\tau}: [\mathbb{V}_h^k]^d \rightarrow [\mathbb{V}_h^k]^d$ for hyperbolic systems is defined by
\begin{equation}\label{oe:sys}
	\begin{aligned}
		\bm{u}_{\sigma}(\bm{x}) &= (\mathcal{F}_{\tau} \bm{u}_h^n)(\bm{x}) = \bm{u}_{\sigma}(\bm{x}, \tau) \Big|_{\tau = \Delta t} \\
		&= \bm{u}_K^{(0)}(t^n) \Psi_{K}^{(0)}(\bm{x}) +
		\sum_{m=1}^{k} \exp\Bigg\{-\Delta t \sum_{j=0}^{m} \bm{\sigma}_K^{j}(\bm{u}_h^n)\Bigg\} 
		\circ 
		\sum_{\ell=L_{m-1}+1}^{L_m} \bm{u}_K^{(\ell)}(t^n) \Psi_{K}^{(\ell)}(\bm{x}) \quad \forall \bm{x} \in K,
	\end{aligned}
\end{equation}
where the damping coefficient vector $\bm{\sigma}_K^{j} (\bm{u}_h^n) = \left(\sigma_K^j(u_h^{(1)}), \ldots, \sigma_K^j(u_h^{(d)}) \right)^\top$ is defined as
\begin{equation*}
	\bm{\sigma}_K^j(\bm{u}_h) = \sum_{i=1}^{3} \frac{\beta_{K}^{(i)}}{h_{K}^{(i)}} \bm{\delta}_K^{(i),j}(\bm{u}_h).
\end{equation*}
Here, $h_{K}^{(i)} = \sup_{\bm{x} \in K} \mathrm{dist}(\bm{x}, e_K^{(i)})$, and the coefficients $\bm{\delta}_K^{(i),j}(\bm{u}_h)$ are defined by
\begin{equation}\label{oe:sysdelta1}
	\bm{\delta}_K^{(i),j}(\bm{u}_h) = \left(\delta_K^{(i),j}(u_h^{(1)}), \ldots, \delta_K^{(i),j}(u_h^{(d)}) \right)^\top
\end{equation}
with $\delta_K^{(i),j}(\cdot)$ specified by equation \eqref{OEScalar:Coef2}. The parameter $\beta_{K}^{(i)}$ is chosen as
\begin{equation*}
	\beta_{K}^{(i)} := \max_{\bm{u} \in \{\bm{u}_c, \bm{u}_d\}} \left\{
	R(\bm{u}^{-}, \mathbf{n}_{K}^{(i)}),
	R(\bm{u}^{+}, \mathbf{n}_{K}^{(i)}) \right\},
\end{equation*}
where $\bm{u}^{-}$ and $\bm{u}^{+}$ denote the left and right states of $\bm{u}_h$ at the interface $e_K^{(i)}$, and $\bm{u}_c$ and $\bm{u}_d$ are the values of $\bm{u}_h^n$ at the endpoints $\bm{v}_c$ and $\bm{v}_d$ of the edge $e_K^{(i)}$.

The OE operator \eqref{oe:sys} is derived from the solution $\bm{u}_{\sigma}(\bm{x}, \tau) \in [\mathbb{V}_h^k]^d$ of the following damping equation:
\begin{equation}\label{DGSystem:OE}
	\int_{K} \frac{\mathrm{d} \bm{u}_{\sigma}(\bm{x}, \tau)}{\mathrm{d} \tau} \circ \bm{v}_h(\bm{x}) \, \mathrm{d} \bm{x} 
	+ \sum_{m=0}^{k} \bm{\sigma}_K^{m}(\bm{u}_h^n) \circ \int_{K} (\bm{u}_{\sigma} - \mathcal{P}^{m-1} \bm{u}_{\sigma})(\bm{x}, \tau) \circ \bm{v}_h(\bm{x}) \, \mathrm{d} \bm{x} = 0 \quad \forall \bm{v}_h(\bm{x}) \in [\mathbb{P}^k(K)]^d,
\end{equation}
with the initial condition $\bm{u}_{\sigma}(\bm{x},0) = \bm{u}_h^n$. Here, $\tau$ is a pseudo-time variable distinct from $t$, and $\mathcal{P}^m$ denotes the standard $L^2$ projection onto $[\mathbb{V}_h^m]^d$, with $\mathcal{P}^{-1} = \mathcal{P}^{0}$. 
The solution $\bm{u}_{\sigma}(\bm{x}, \tau)$ can be expanded as
\[
\bm{u}_{\sigma}(\bm{x}, \tau) = \sum_{\ell=0}^{L_k} \bm{u}_{\sigma}^{(\ell)}(\tau) \Psi_{K}^{(\ell)}(\bm{x}) \in \mathbb{V}_h^k,
\]
where $\bm{u}_{\sigma}^{(\ell)}(\tau)$ represents the $\ell$-th modal coefficient of $\bm{u}_{\sigma}(\bm{x}, \tau)$. By solving the linear damping equations \eqref{DGSystem:OE} exactly, we derive the form of $\bm{u}_{\sigma}(\bm{x})$ as presented in \eqref{oe:sys}.

\subsection{Rotation-invariant OE Procedure} \label{sec:oep3}

To introduce the concept of rotational invariance for the OE procedure, we first define the rotational invariance of a hyperbolic system of conservation laws.

\begin{definition}\label{def:eqri}
	A hyperbolic system \eqref{eq:System} is called RI if, for any unit vector $\bm{\xi} \in \mathbb{R}^2$, there exists an orthogonal matrix $T_{\bm{\xi}}$ such that the flux $\bm{f}_1$ satisfies:
	\begin{equation}\label{eq:rip}
		T_{\bm{\xi}}^{-1} \bm{f}_1 (T_{\bm{\xi}} \bm{u}) = \langle \bm{\xi}, \bm{F}(\bm{u}) \rangle,
	\end{equation}	
	where $\langle \cdot, \cdot \rangle$ denotes the inner product between two vectors. 
    Denote $(\cos \phi, \sin \phi)^\top$ as the polar coordinate representation of $\bm{\xi}$, thus we have 
    \begin{equation}\label{eq:rip2}
	\begin{aligned}
	    T_{\bm{\xi}}^{-1} \bm{f}_1 (T_{\bm{\xi}} \bm{u}) &= \cos \phi \, \bm{f}_1(\bm{u}) + \sin \phi \, \bm{f}_2(\bm{u}),\\
            T_{\bm{\xi}}^{-1} \bm{f}_2 (T_{\bm{\xi}} \bm{u}) &= -\sin \phi \, \bm{f}_1(\bm{u}) + \cos \phi \, \bm{f}_2(\bm{u}).
	\end{aligned}
    \end{equation}	
\end{definition}

Next, we present an example of a hyperbolic system that exhibits rotational invariance.

\begin{expl}[2D Compressible Euler System]\label{ex:euler}
	The 2D compressible Euler system is given by:
	\begin{equation}\label{eq:euler}
		\bm{u} =
		\begin{pmatrix}	
			\rho \\
			m_1 \\
			m_2 \\
			E
		\end{pmatrix},
		\quad
		\bm{f}_1(\bm{u}) =
		\begin{pmatrix}	
			m_1 \\
			m_1 v_1 + p \\
			m_2 v_1 \\
			(E + p) v_1
		\end{pmatrix},
		\quad
		\bm{f}_2(\bm{u}) =
		\begin{pmatrix}	
			m_2 \\
			m_1 v_2 \\
			m_2 v_2 + p \\
			(E + p) v_2
		\end{pmatrix},
	\end{equation}
	where $\rho$ is the density, $p$ is the pressure, $\bm{m} = (m_1, m_2)^\top = \rho \bm{v}$ is the momentum vector, $\bm{v} := (v_1, v_2)^\top \in \mathbb{R}^2$ is the velocity field, and $E = \frac{1}{2} \rho (v_1^2 + v_2^2) + \frac{p}{\gamma - 1}$ denotes the total energy with the adiabatic index $\gamma$. The rotational matrix $T_{\bm{\xi}}$ is given by:
	\begin{equation*}
		T_{\bm{\xi}} = \text{diag}\{1, M_{\bm{\xi}}, 1\},
	\end{equation*}
	where
	\begin{equation}\label{mt2}
		M_{\bm{\xi}} := 
		\begin{pmatrix}
			\cos \phi & \sin \phi \\
			-\sin \phi & \cos \phi 
		\end{pmatrix}.
	\end{equation}
	Defining $\hat{\bm{m}} = M_{\bm{\xi}} \bm{m}$ and $\hat{\bm{v}} = M_{\bm{\xi}} \bm{v}$, we have the following relationships:
	\[
		\|\bm{m}\| := \sqrt{m_1^2 + m_2^2} = \sqrt{\hat{m}_1^2 + \hat{m}_2^2} = \|\hat{\bm{m}}\|,
	\]
	\[
		\|\bm{v}\| := \sqrt{v_1^2 + v_2^2} = \sqrt{\hat{v}_1^2 + \hat{v}_2^2} = \|\hat{\bm{v}}\|.
	\]
\end{expl}

We now consider a general hyperbolic system \eqref{eq:System} with rotational invariance, as characterized by the relation \eqref{eq:rip}. The initial condition for the system \eqref{eq:System} is given by:
\[
\bm{u}(\bm{x}, 0) = \bm{u}_0(\bm{x})
\]
over the domain $\Omega$. By rotating the domain $\Omega := \left\{\bm{x} = (x, y)^\top \right\}$ clockwise by an angle $\phi \in [0, 2\pi]$ while keeping the origin fixed, we obtain a new domain $\hat{\Omega} := \left\{\hat{\bm{x}} = (\hat{x}, \hat{y})^\top \right\}$. The relationship between the coordinates $\bm{x} \in \Omega$ and $\hat{\bm{x}} \in \hat{\Omega}$ is given by:
\[
\hat{\bm{x}} = M_{\bm{\xi}} \bm{x},
\]
with $M_{\bm{\xi}}$ defined as in \eqref{mt2}. The element $\hat{K}$ of the partition $\hat{\mathcal{T}}_h$ for $\hat{\Omega}$ is obtained by rotating the element $K \in \mathcal{T}_h$ of $\Omega$. We denote the $i$-th edge of $\hat{K}$ by $e_{\hat{K}}^{(i)}$ $(i = 1, 2, 3)$, and the outward unit normal vector of $e_{\hat{K}}^{(i)}$ by $\hat{\mathbf{n}}_{\hat{K}}^{(i)}$.
Additionally, the initial condition $\bm{u}_0(\bm{x})$ undergoes the same rotational transformation, resulting in the initial condition $\hat{\bm{u}}_0(\hat{\bm{x}})$ on the domain $\hat{\Omega}$:
\begin{equation*}
\hat{\bm{u}}_0(\hat{\bm{x}}) = T_{\bm{\xi}} \bm{u}_0(\bm{x}).
\end{equation*}

Next, we define the RI property of a numerical flux. Both the standard LF flux and the HLL flux are RI.

\begin{definition}
	A numerical flux $\hat{\bm{F}}(\cdot, \cdot, \cdot)$ is called RI if it satisfies:
	\begin{equation*}
		\hat{\bm{F}}(\hat{\bm{u}}^{\mathrm{int}}, \hat{\bm{u}}^{\mathrm{ext}}, \hat{\mathbf{n}}) = T_{\bm{\xi}} \hat{\bm{F}}(\bm{u}^{\mathrm{int}}, \bm{u}^{\mathrm{ext}}, \mathbf{n}),
	\end{equation*}
	where the unit vector $\hat{\mathbf{n}} = M_{\bm{\xi}} \mathbf{n}$ and $\hat{\bm{u}} = T_{\bm{\xi}} \bm{u}$.
\end{definition}

Suppose $\bm{u}_h(\bm{x}, t)$ and $\hat{\bm{u}}_h(\hat{\bm{x}}, t)$ are the numerical solutions obtained by the DG method for the same system \eqref{eq:System}, with initial conditions $\bm{u}(\bm{x}, 0) = \bm{u}_0(\bm{x})$ and $\hat{\bm{u}}(\hat{\bm{x}}, 0) = \hat{\bm{u}}_0(\hat{\bm{x}})$ on the domains $\Omega$ and $\hat{\Omega}$, respectively. We can then establish the following theorem.

\begin{theorem}\label{thm:DG-OE}
    If both the system \eqref{eq:System} and the numerical flux $\hat{\bm{F}}(\cdot, \cdot, \cdot)$ are RI, then the DG method preserves this property:
    \begin{equation}\label{rip:dg}
        \hat{\bm{u}}_h(\hat{\bm{x}}, t) = T_{\bm{\xi}} \bm{u}_h(\bm{x}, t).
    \end{equation}
    However, the component-wise OE procedure $\mathcal{F}_{\tau}$, as described in \cref{sec:oep2}, does not preserve the RI property:
    \begin{equation*}
        (\mathcal{F}_{\tau} \hat{\bm{u}}_h) (\hat{\bm{x}}) \not\equiv T_{\bm{\xi}} (\mathcal{F}_{\tau} \bm{u}_h) (\bm{x}).
    \end{equation*}
\end{theorem}

\begin{proof}
    We will use induction to prove \eqref{rip:dg}.
    Let $\hat{\bm{u}}_{\hat{K}}^{(\ell)}(t)$ be the $\ell$-th  modal coefficient of the numerical solution $\hat{\bm{u}}_h(\hat{\bm{x}},t)$, so $\hat{\bm{u}}_h(\hat{\bm{x}}, t)$ can be expressed as 
    \[
    \hat{\bm{u}}_h(\hat{\bm{x}},t) =  \sum_{\ell=0}^{L_k} \hat{\bm{u}}_{\hat{K}}^{(\ell)}(t) \Psi_{\hat{K}}^{(\ell)}(\hat{\bm{x}}) \quad \forall \hat{\bm{x}} \in \hat{K}.
    \]
    We first prove that $\hat{\bm{u}}_h(\hat{\bm{x}},0)=T_{\bm{\xi}}\bm{u}_h(\bm{x},0)$.
    According to the initial condition \cref{rip:init}, for each triangular cell $\hat{K} \in \hat{\mathcal{T}}_h$, we have 
    \begin{equation}\label{rip:pf1}
        \int_{\hat{K}} \hat{\bm{u}}_h(\hat{\bm{x}},0) \circ \hat{\bm{v}}_h(\hat{\bm{x}}) \, \mathrm{d} \hat{\bm{x}} = T_{\bm{\xi}} \int_{{K}} \bm{u}_h(\bm{x},0) \circ \bm{v}_h(\bm{x}) \, \mathrm{d} \bm{x}
    \end{equation}
    for any $\hat{\bm{v}}_h(\hat{\bm{x}}) = \bm{v}_h(\bm{x}) \in [{\mathbb V}_h^{k}]^d$.
    Thus, for $\ell=0,\cdots, L_k$, we obtain 
    \begin{align*}
        \hat{\bm{u}}_{\hat{K}}^{(\ell)}(0) \int_{K} (\Psi_{K}^{(\ell)}(\bm{x}))^2 \, \mathrm{d} \bm{x} 
        &= \hat{\bm{u}}_{\hat{K}}^{(\ell)}(0) \int_{\hat{K}} (\Psi_{\hat{K}}^{(\ell)}(\hat{\bm{x}}))^2 \, \mathrm{d} \hat{\bm{x}}
        = \sum_{m=0}^{L_k} \hat{\bm{u}}_{\hat{K}}^{(m)}(0) \int_{\hat{K}} \Psi_{\hat{K}}^{(m)}(\hat{\bm{x}}) \Psi_{\hat{K}}^{(\ell)}(\hat{\bm{x}}) \, \mathrm{d} \hat{\bm{x}} \\
        &\overset{\eqref{rip:pf1}}{=} T_{\bm{\xi}} \sum_{m=0}^{L_k} {\bm{u}}_{{K}}^{(m)}(0) \int_{{K}} \Psi_{{K}}^{(m)}({\bm{x}}) \Psi_{{K}}^{(\ell)}({\bm{x}}) \, \mathrm{d} {\bm{x}} 
        = T_{\bm{\xi}} {\bm{u}}_{{K}}^{(\ell)}(0) \int_{{K}} (\Psi_{{K}}^{(\ell)}({\bm{x}}))^2  \, \mathrm{d} \bm{x}.
    \end{align*}
    This implies that $\hat{\bm{u}}_{\hat{K}}^{(\ell)}(0)=T_{\bm{\xi}} {\bm{u}}_{{K}}^{(\ell)}(0)$. Therefore, $\hat{\bm{u}}_h(\hat{\bm{x}},0)=T_{\bm{\xi}}\bm{u}_h(\bm{x},0)$.

Next, assume that $\hat{\bm{u}}_h(\hat{\bm{x}},t^n)=T_{\bm{\xi}}\bm{u}_h(\bm{x},t^n)$ holds at time step $t^n (n\geq 1)$. We aim to show that this relationship also holds at the subsequent time step, i.e., $\hat{\bm{u}}_h(\hat{\bm{x}},t^{n+1})=T_{\bm{\xi}}\bm{u}_h(\bm{x},t^{n+1})$. For simplicity, we consider the Euler forward time discretization method; the argument extends naturally to other time integration schemes, and thus, we omit those details. 
    The time step size is determined by the spectral radius $R(\bm{u},\mathbf{n})$ of the Jacobian matrix $\bm{F}'(\bm{u}) \cdot \mathbf{n} = n^{1} \frac{\partial \bm{f}_1(\bm{u})}{\partial \bm{u}} + n^{2} \frac{\partial \bm{f}_2(\bm{u})}{\partial \bm{u}}$ in the direction of the outward unit normal vector $\mathbf{n}$ on the edge $e$. 
    Applying the rotational invariance \eqref{eq:rip2}  of the system \eqref{eq:System}, we have
    \begin{align*}
        \frac{\partial \bm{f}_1(\hat{\bm{u}})}{\partial \hat{\bm{u}}} \Big|_{\hat{\bm{u}}=\hat{\bm{u}}_h(\hat{\bm{x}},t^{n})} 
        &= T_{\bm{\xi}} \, \Big( \cos \phi \, \frac{\partial \bm{f}_1(\bm{u})}{\partial \bm{u}} + \sin \phi \, \frac{\partial \bm{f}_2(\bm{u})}{\partial \bm{u}} \Big)\Big|_{{\bm{u}}={\bm{u}}_h({\bm{x}},t^{n})} \, T_{\bm{\xi}}^{-1},\\
        \frac{\partial \bm{f}_2(\hat{\bm{u}})}{\partial \hat{\bm{u}}} \Big|_{\hat{\bm{u}}=\hat{\bm{u}}_h(\hat{\bm{x}},t^{n})} 
        &= T_{\bm{\xi}} \, \Big( -\sin \phi \, \frac{\partial \bm{f}_1(\bm{u})}{\partial \bm{u}} + \cos \phi \, \frac{\partial \bm{f}_2(\bm{u})}{\partial \bm{u}} \Big)\Big|_{{\bm{u}}={\bm{u}}_h({\bm{x}},t^{n})} \, T_{\bm{\xi}}^{-1}.
    \end{align*}
    Combined with the following relations:
    \begin{align*}
        \hat{n}_{\hat{K}}^{(i),1} = \cos \phi \, {n}_{K}^{(i),1} + \sin \phi \, {n}_{K}^{(i),2}, \qquad
        \hat{n}_{\hat{K}}^{(i),2} = -\sin \phi \, {n}_{K}^{(i),1} + \cos \phi \, {n}_{K}^{(i),2},
    \end{align*}
    it follows that 
    \begin{align*}
        \Big(\hat{n}_{\hat{K}}^{(i),1} \, \frac{\partial \bm{f}_1(\hat{\bm{u}})}{\partial \hat{\bm{u}}} + \hat{n}_{\hat{K}}^{(i),2} \, \frac{\partial \bm{f}_2(\hat{\bm{u}})}{\partial \hat{\bm{u}}} \Big)\Big|_{\hat{\bm{u}}=\hat{\bm{u}}_h(\hat{\bm{x}},t^{n})}
        =  T_{\bm{\xi}} \, \Big( {n}_{{K}}^{(i),1} \, \frac{\partial \bm{f}_1({\bm{u}})}{\partial {\bm{u}}} + {n}_{{K}}^{(i),2} \, \frac{\partial \bm{f}_2({\bm{u}})}{\partial {\bm{u}}} \Big)\Big|_{{\bm{u}}={\bm{u}}_h({\bm{x}},t^{n})} \, T_{\bm{\xi}}^{-1}.
    \end{align*}
    This indicates that $R(\hat{\bm{u}}(\hat{\bm x}, t^n),\hat{\mathbf{n}}_{\hat{K}}^{(i)})=R(\bm{u}(\bm x, t^n),\mathbf{n}_{K}^{(i)})$. Hence, the time step size obtained using $\hat{\bm{u}}(\hat{\bm{x}},t^n)$ is the same as that obtained using $\bm{u}(\bm{x}, t^n)$ and is denoted by $\dt$ for convenience.
    Furthermore, according to \eqref{DGSystem:Int}, for each triangular cell $\hat{K} \in \hat{\mathcal{T}}_h$, we have
    \begin{equation*}
	\begin{aligned}
		\int_{\hat{K}}  \hat{\bm{u}}_h(\hat{\bm{x}},t^{n+1}) \circ \hat{\bm{v}}_h(\hat{\bm{x}}) \, \mathrm{d} \hat{\bm{x}} 
            =& \int_{\hat{K}}  \hat{\bm{u}}_h(\hat{\bm{x}},t^{n}) \circ \hat{\bm{v}}_h(\hat{\bm{x}}) \, \mathrm{d} \hat{\bm{x}} 
            +\dt \int_{\hat{K}} \bm{F}(\hat{\bm{u}}_h(\hat{\bm{x}},t^n)) : \nabla \hat{\bm{v}}_h(\hat{\bm{x}}) \, \mathrm{d} \hat{\bm{x}} \\   
		&-  \dt \sum_{i=1}^{3} \int_{e_{\hat{K}}^{(i)}} \hat{\bm{F}}(\hat{\bm{u}}_h^{\mathrm{int}}(\hat{\bm{x}},t^n), \hat{\bm{u}}_h^{\mathrm{ext}}(\hat{\bm{x}},t^n), \hat{\mathbf{n}}_{\hat{K}}^{(i)}) \circ \hat{\bm{v}}_h(\hat{\bm{x}}) \, \mathrm{d\hat{s}} \\
            :=& \Pi_1 + \dt \, \Pi_2 - \dt \sum_{i=1}^{3} \Xi_i.
	\end{aligned}
    \end{equation*} 
    Using the assumption $\hat{\bm{u}}_h(\hat{\bm{x}},t^n)=T_{\bm{\xi}}\bm{u}_h(\bm{x},t^n)$ and the rotational invariance of the numerical flux $\hat{\bm{F}}(\cdot, \cdot, \cdot)$, one can obtain that 
    \begin{align}
        \Pi_1 &= T_{\bm{\xi}} \int_{{K}} {\bm{u}}_h({\bm{x}},t^{n}) \circ {\bm{v}}_h({\bm{x}}) \, \mathrm{d} {\bm{x}} \label{rip:pi1}\\
        \Xi_i &= T_{\bm{\xi}}
		\int_{e_{{K}}^{(i)}} \hat{\bm{F}}(\bm{u}_h^{\mathrm{int}}(\bm{x},t^n), \bm{u}_h^{\mathrm{ext}}(\bm{x},t^n), \mathbf{n}_{K}^{(i)}) \circ \bm{v}_h(\bm{x}) \, \mathrm{ds} ~~~ i=1,2,3. \label{rip:xii}
    \end{align}
    Given that $\hat{\bm{x}}=M_{\bm{\xi}}\bm{x}$, the following formulas hold:
    \begin{align*}
        \frac{\partial \hat{\bm{v}}_h(\hat{\bm{x}})}{\partial \hat{x}} 
        =\cos \phi \, \frac{\partial {\bm{v}}_h({\bm{x}})}{\partial {x}} 
        + \sin \phi \, \frac{\partial {\bm{v}}_h({\bm{x}})}{\partial {y}}, \qquad
        \frac{\partial \hat{\bm{v}}_h(\hat{\bm{x}})}{\partial \hat{y}} 
        =-\sin \phi \, \frac{\partial {\bm{v}}_h({\bm{x}})}{\partial {x}} 
        + \cos \phi \, \frac{\partial {\bm{v}}_h({\bm{x}})}{\partial {y}}.
    \end{align*}
    Therefore, we obtain 
    \begin{align*}
        \Pi_2 =& \int_{\hat{K}} \Big[ \bm{f}_1(\hat{\bm{u}}_h(\hat{\bm{x}},t^n)) \circ \frac{\partial \hat{\bm{v}}_h(\hat{\bm{x}})}{\partial \hat{x}} 
        + \bm{f}_2(\hat{\bm{u}}_h(\hat{\bm{x}},t^n)) \circ \frac{\partial \hat{\bm{v}}_h(\hat{\bm{x}})}{\partial \hat{y}} \Big] \, \mathrm{d} \hat{\bm{x}} \\
        =&\int_{{K}} \Big[ \bm{f}_1(T_{\bm{\xi}} {\bm{u}}_h({\bm{x}},t^n)) \circ 
        \Big(\cos \phi \, \frac{\partial {\bm{v}}_h({\bm{x}})}{\partial {x}} 
        + \sin \phi \, \frac{\partial {\bm{v}}_h({\bm{x}})}{\partial {y}} \Big) 
        + \bm{f}_2(T_{\bm{\xi}} {\bm{u}}_h({\bm{x}},t^n)) \circ 
        \Big(-\sin \phi \, \frac{\partial {\bm{v}}_h({\bm{x}})}{\partial {x}} 
        + \cos \phi \, \frac{\partial {\bm{v}}_h({\bm{x}})}{\partial {y}} \Big)\Big] \, \mathrm{d} {\bm{x}} \\
        =&\int_{{K}} \Big[ T_{\bm{\xi}} \Big(\cos \phi \, \bm{f}_1({\bm{u}}_h({\bm{x}},t^n)) + \sin \phi \, \bm{f}_2({\bm{u}}_h({\bm{x}},t^n)) \Big)\circ 
        \Big(\cos \phi \, \frac{\partial {\bm{v}}_h({\bm{x}})}{\partial {x}} 
        + \sin \phi \, \frac{\partial {\bm{v}}_h({\bm{x}})}{\partial {y}} \Big) \\
        &+ T_{\bm{\xi}} \Big(-\sin \phi \, \bm{f}_1({\bm{u}}_h({\bm{x}},t^n)) + \cos \phi \, \bm{f}_2({\bm{u}}_h({\bm{x}},t^n)) \Big) \circ 
        \Big(-\sin \phi \, \frac{\partial {\bm{v}}_h({\bm{x}})}{\partial {x}} 
        + \cos \phi \, \frac{\partial {\bm{v}}_h({\bm{x}})}{\partial {y}} \Big)\Big] \, \mathrm{d} {\bm{x}} \\
        =& T_{\bm{\xi}} \int_{{K}} \Big[ \bm{f}_1({\bm{u}}_h({\bm{x}},t^n)) \circ \frac{\partial {\bm{v}}_h({\bm{x}})}{\partial {x}} 
        + \bm{f}_2({\bm{u}}_h({\bm{x}},t^n)) \circ \frac{\partial {\bm{v}}_h({\bm{x}})}{\partial {y}} \Big] \, \mathrm{d} {\bm{x}} \\
        =& T_{\bm{\xi}} \int_{{K}} \bm{F}(\bm{u}_h(\bm{x},t^n)) : \nabla \bm{v}_h(\bm{x}) \, \mathrm{d} \bm{x}.
    \end{align*}
    This, together with \eqref{rip:pi1} and \eqref{rip:xii}, implies
    \begin{equation*}
	\int_{\hat{K}} \hat{\bm{u}}_h(\hat{\bm{x}},t^{n+1}) \circ \hat{\bm{v}}_h(\hat{\bm{x}}) \, \mathrm{d} \hat{\bm{x}} = T_{\bm{\xi}} \int_{{K}} \bm{u}_h(\bm{x},t^{n+1}) \circ \bm{v}_h(\bm{x}) \, \mathrm{d} \bm{x}.
    \end{equation*}
    This yields $\hat{\bm{u}}_h(\hat{\bm{x}},t^{n+1})=T_{\bm{\xi}}\bm{u}_h(\bm{x},t^{n+1})$.
    Consequently, for any $t\geq0$, the relation \eqref{rip:dg} holds.

    To show that the OE procedure $\mathcal{F}_{\tau}$ is not RI, consider the Euler system \eqref{eq:euler} as an example. From \eqref{oe:sys}, for any $\hat{\bm{x}} \in \hat{K}$, we have 
	\begin{equation*}
		\begin{aligned}
			\hat{\bm{u}}_{\sigma}(\hat{\bm{x}}) 
			= (\mathcal{F}_{\tau} \hat{\bm{u}}_h) (\hat{\bm{x}}) 
			&= \hat{\bm{u}}_{\hat{K}}^{(0)}(t) \, \Psi_{\hat{K}}^{(0)}(\hat{\bm{x}}) +
			\sum_{m=1}^{k} \exp\Bigg\{-\Delta t \sum_{j=0}^{{m}} \bm{\sigma}_{\hat{K}}^{j}(\hat{\bm{u}}_h)\Bigg\} 
			\circ 
			\sum_{\ell=L_{m-1}+1}^{L_m} \hat{\bm{u}}_{\hat{K}}^{(\ell)}(t)  \, \Psi_{\hat{K}}^{(\ell)}(\hat{\bm{x}}) \\
             &=T_{\bm{\xi}} \Bigg( {\bm{u}}_{{K}}^{(0)}(t) \, \Psi_{{K}}^{(0)}({\bm{x}}) +
			\sum_{m=1}^{k} \exp\Bigg\{-\Delta t \sum_{j=0}^{{m}} \bm{\sigma}_{\hat{K}}^{j}(\hat{\bm{u}}_h)\Bigg\}
			\circ 
			\sum_{\ell=L_{m-1}+1}^{L_m} {\bm{u}}_{{K}}^{(\ell)}(t) \, \Psi_{{K}}^{(\ell)}({\bm{x}}) \Bigg)
		\end{aligned}
	\end{equation*}
	Let $\bm{m}_h$ and $\hat{\bm{m}}_h$ be the numerical approximations of momentum for $\bm{u}_h$ and $\hat{\bm{u}}_h$, respectively. Then $\hat{\bm{m}}_h = M_{\bm{\xi}} \bm{m}_h$, implying that
	\[
	\sigma_K^{j}(\hat{m}_{1,h}) \neq \sigma_K^{j}(m_{1,h}), \qquad
	\sigma_K^{j}(\hat{m}_{2,h}) \neq \sigma_K^{j}(m_{2,h}),
	\]
	due to the coefficient $\delta_K^{(i),j}(\cdot)$ in \eqref{OEScalar:Coef2}. Hence, $\hat{\bm{u}}_{\sigma}(\hat{\bm{x}}) \neq T_{\bm{\xi}} \bm{u}_{\sigma}(\bm{x})$, completing the proof.
\end{proof}

\Cref{thm:DG-OE} indicates that the OEDG scheme, with the component-wise OE procedure described in \cref{sec:oep2}, does not retain the RI property. To address this, we introduce a modified OE procedure $\mathcal{F}^{\rm RI}_{\tau}$, called the RIOE procedure. This procedure redefines the coefficients $\bm{\delta}_K^{(i),j}(\bm{u}_h)$ in \eqref{oe:sysdelta1}. For illustration, consider the Euler system \eqref{ex:euler}. The modified coefficients $\hat{\bm{\delta}}_K^{(i),j}(\bm{u}_h)$ are given by:
\begin{equation}\label{oe:sysdelta2}
	\hat{\bm{\delta}}_K^{(i),j}(\bm{u}_h) = 
	\left(
	\delta_K^{(i),j}(\rho_h),  
	\hat{\delta}_K^{(i),j}(\bm{m}_h), 
	\hat{\delta}_K^{(i),j}(\bm{m}_h),
	\delta_K^{(i),j}(E_h)
	\right)^\top,
\end{equation} 
where $\delta_K^{(i),j}(\cdot)$ is defined by equation \eqref{OEScalar:Coef2}, and 
\begin{equation}\label{RIOESys:CoefM}
	\hat{\delta}_K^{(i),j}(\bm{m}_h) = \max \left\{ \hat{\delta}_K^{(i),j}(m_{n,h}^{(i)}), \hat{\delta}_K^{(i),j}(m_{t,h}^{(i)}) \right\},
\end{equation}
where the momentum components $m_{n,h}^{(i)}$ and $m_{t,h}^{(i)}$ are defined on the outward unit normal vector $\mathbf{n}_{K}^{(i)}$ and the corresponding unit tangent vector $\mathbf{n}_{t,K}^{(i)} = (- n_K^{(i),2}, n_K^{(i),1})^\top$, respectively, as
\[
m_{n,h}^{(i)} = \bm{m}_h \cdot \mathbf{n}_{K}^{(i)} = m_{1,h} \, n_K^{(i),1} + m_{2,h} \, n_K^{(i),2}, \qquad 
m_{t,h}^{(i)} = \bm{m}_h \cdot \mathbf{n}_{t,K}^{(i)} = -m_{1,h} \, n_K^{(i),2} + m_{2,h} \, n_K^{(i),1}.
\]
The coefficients $\hat{\delta}_K^{(i),j}(m_{n,h}^{(i)})$ and $\hat{\delta}_K^{(i),j}(m_{t,h}^{(i)})$ are defined as:
\begin{equation}\label{RIOESys:Coef1}
	\hat{\delta}_K^{(i),j}(m_{n,h}^{(i)}) = 
	\left\{
	\begin{aligned}
		&0, & & \bm{m}_h \equiv \bar{\bm{m}}_h^{\Omega}, \\
		&A^{k,j} \frac{(h_{K}^{(i)})^j}{\|\|\bm{m}_h-\bar{\bm{m}}_h^{\Omega}\|\|_{L^{\infty}(\Omega)}} \sqrt{\frac{1}{l_K^{(i)}} \int_{e_K^{(i)}}  \sum_{|{\bm{\alpha}}|=j} C_j^{\alpha_1} \jump{ \partial^{\bm{\alpha}} m_{n,h}^{(i)}}_{e_K^{(i)}}^2 \, \mathrm{ds}}, & & \text{otherwise},
	\end{aligned}
	\right.
\end{equation}
\begin{equation}\label{RIOESys:Coef2}
	\hat{\delta}_K^{(i),j}(m_{t,h}^{(i)}) = 
	\left\{
	\begin{aligned}
		&0, & & \bm{m}_h \equiv \bar{\bm{m}}_h^{\Omega}, \\
		&A^{k,j} \frac{(h_{K}^{(i)})^j}{\|\|\bm{m}_h-\bar{\bm{m}}_h^{\Omega}\|\|_{L^{\infty}(\Omega)}} \sqrt{\frac{1}{l_K^{(i)}} \int_{e_K^{(i)}}  \sum_{|{\bm{\alpha}}|=j} C_j^{\alpha_1} \jump{ \partial^{\bm{\alpha}} m_{t,h}^{(i)}}_{e_K^{(i)}}^2 \, \mathrm{ds}}, & & \text{otherwise},
	\end{aligned}
	\right.
\end{equation}
where 
$\|\bm{m}_h-\bar{\bm{m}}_h^{\Omega}\| = \sqrt{(m_{1,h}-\bar{m}_{1,h}^{\Omega})^2 + (m_{2,h}-\bar{m}_{2,h}^{\Omega})^2}$, 
$\bar{\bm{m}}_h^{\Omega} = (\bar{m}_{1,h}^{\Omega}, \bar{m}_{2,h}^{\Omega})^\top$, and the averages $\bar{m}_{1,h}^{\Omega}$ and $\bar{m}_{2,h}^{\Omega}$ over the computational domain $\Omega$ are defined as
\[
\bar{m}_{1,h}^{\Omega} := \frac{1}{|\Omega|} \int_{\Omega} m_{1,h}(\bm{x},t) \, \mathrm{d} \bm{x}, \qquad
\bar{m}_{2,h}^{\Omega} := \frac{1}{|\Omega|} \int_{\Omega} m_{2,h}(\bm{x},t) \, \mathrm{d} \bm{x}.
\]
In our simulations, we use the trapezoidal rule to approximate the integrals along the edge $e_K^{(i)}$ in \eqref{RIOESys:Coef1} and \eqref{RIOESys:Coef2}, specifically:
\begin{align*}
	\frac{1}{l_K^{(i)}} \int_{e_K^{(i)}}  \sum_{|{\bm{\alpha}}|=j} C_j^{\alpha_1} \jump{ \partial^{\bm{\alpha}} m_{n,h}^{(i)} }_{e_K^{(i)}}^2 \, \mathrm{ds} &\approx \frac{1}{2} \sum_{|{\bm{\alpha}}|=j} C_j^{\alpha_1} \left(\jump{ \partial^{\bm{\alpha}} m_{n,h}^{(i)} }_{{\bm{v}}_c}^2 + \jump{ \partial^{\bm{\alpha}} m_{n,h}^{(i)} }_{{\bm{v}}_d}^2\right), \\
	\frac{1}{l_K^{(i)}} \int_{e_K^{(i)}}  \sum_{|{\bm{\alpha}}|=j} C_j^{\alpha_1} \jump{ \partial^{\bm{\alpha}} m_{t,h}^{(i)} }_{e_K^{(i)}}^2 \, \mathrm{ds} &\approx \frac{1}{2} \sum_{|{\bm{\alpha}}|=j} C_j^{\alpha_1} \left(\jump{ \partial^{\bm{\alpha}} m_{t,h}^{(i)} }_{{\bm{v}}_c}^2 + \jump{ \partial^{\bm{\alpha}} m_{t,h}^{(i)} }_{{\bm{v}}_d}^2\right).
\end{align*}
To mitigate the effects of round-off errors, the condition $\bm{m}_h \equiv \bar{\bm{m}}_h^{\Omega}$ is replaced with 
\[
\|\|\bm{m}_h-\bar{\bm{m}}_h^{\Omega}\|\|_{L^{\infty}(\Omega)} \leq \epsilon \max\{1, \|\bar{\bm{m}}_h^{\Omega}\|\} = \epsilon \max\left\{1, \sqrt{(\bar{m}_{1,h}^{\Omega})^2 + (\bar{m}_{2,h}^{\Omega})^2}\right\}.
\]

The RIOE procedure $\mathcal{F}^{\rm RI}_{\tau}$ ensures rotational invariance by modifying the coefficients $\delta_K^{(i),j}(m_{1,h})$ and $\delta_K^{(i),j}(m_{2,h})$ to $\hat{\delta}_K^{(i),j}(\bm{m}_h)$. This technique can be directly extended to more general systems with vector variables analogous to momentum $\bm{m}$. For each such vector $\tilde{\bm{m}}$, the corresponding coefficient $\hat{\delta}_K^{(i),j}(\tilde{\bm{m}})$ is defined as \eqref{RIOESys:CoefM}.

Clearly, according to the modified coefficients $\hat{\bm{\delta}}_K^{(i),j}(\bm{u}_h)$, we have the following theorem.

\begin{theorem}[Rotational Invariance]
	The RIOE procedure $\mathcal{F}^{\rm RI}_{\tau}$ preserves the rotational invariance property, namely, 
	\begin{equation*}\label{rioe:rip}
		(\mathcal{F}_{\tau}^{\rm RI} \hat{\bm{u}}_h) (\hat{\bm{x}}) = T_{\bm{\xi}} (\mathcal{F}_{\tau}^{\rm RI} \bm{u}_h) (\bm{x}).
	\end{equation*}
	Therefore, the OEDG method with the RIOE procedure is RI, provided both the system \eqref{eq:System} and the numerical flux $\hat{\bm{F}}(\cdot, \cdot, \cdot)$ satisfy the rotational invariance.
\end{theorem}

\begin{proof}
	Based on $\hat{\bm{\delta}}_K^{(i),j}(\cdot)$ presented in \eqref{oe:sysdelta2}, we have 
	\[
	\bm{\sigma}_K^{j}(\hat{\bm{u}}_h) = \bm{\sigma}_K^{j}(\bm{u}_h)
	\]
	for $0 \leq j \leq k$. Thus, the relation \eqref{rioe:rip} holds. This means that the OEDG method using the RIOE procedure is indeed RI, completing the proof.
\end{proof}

\section{Bound Preservation via Optimal Convex Decomposition on Triangular Cells}\label{sec:OCD}

In addition to controlling spurious oscillations, another crucial feature of robust DG schemes for solving hyperbolic conservation laws is their ability to preserve intrinsic bounds or constraints, commonly known as the bound-preserving (BP) property. The convex decomposition is critical in constructing high-order BP schemes and directly determines the BP CFL number, which affects the computational efficiency. Among the various possible convex decompositions, finding the optimal one that yields the largest BP CFL number is a natural and important challenge.

This challenge has been addressed for rectangular meshes in \cite{CDWOCAD2023, CuiDingWu2024}, but it remains an open problem for triangular meshes. This section focuses on constructing the optimal convex decomposition for the widely used $\mathbb{P}^1$ and $\mathbb{P}^2$ polynomial spaces on triangular cells. As discussed in this section and demonstrated by numerical experiments in \Cref{sec:numexp}, the optimal convex decomposition significantly enhances the efficiency of BP DG schemes on triangular meshes for general hyperbolic conservation laws.

We first introduce the definition of a feasible convex decomposition:

\begin{definition}[Feasible convex decomposition on a triangular cell $K$]\label{def:980}
	A convex decomposition of the cell average on a triangular cell $K$,
	\begin{equation}\label{eq:980}
	\begin{aligned}
		\frac{1}{|K|}\iint_K p(x,y) ~ \textrm{d} x \textrm{d} y
		=&
		\sum_{i=1}^3
		\frac{w_i}{l_K^{(i)}}
		\int_{e_K^{(i)}} p(x,y) ~ \textrm{ds}
		+
		\sum_{s = 1}^{S}
		\omega_s p(\xi_s,\eta_s)
	\end{aligned}
	\end{equation}
	is said to be \emph{feasible} for the polynomial space $\mathbb{P}^k$ if it simultaneously satisfies 
	the following three conditions:
	\begin{enumerate}[label=(\roman*)]
		\item  The convex decomposition holds exactly for all $p(x,y) \in \mathbb{P}^k$;
		\item  The edge weights $\{w_i\}_{i=1}^3$ and the internal node weights $\{\omega_s\}_{s=1}^{S}$ are all positive, with their sum equal to one;
		\item  The internal node set $\mathcal{I}_K = \left\{ ( \xi_s, \eta_s ) \right\}_{s=1}^{S} \subset K$.
	\end{enumerate}
\end{definition}

\noindent\textbf{A Classical Convex Decomposition on a Triangular Cell.}
As an example, Zhang, Xia, and Shu proposed in \cite{ZXSPP2012} the following classic convex decomposition, which is feasible for the $\mathbb{P}^k$ space, on a triangular cell $K$:
\[
	\frac{1}{|K|}\iint_K p(x,y) ~ \textrm{d} x \textrm{d} y
	=
	\sum_{i=1}^3
	\frac{2 \, \omega_1^{\tt GL}}{3 l^{(i)}_K}
	\int_{e_K^{(i)}} p(x,y) ~ \textrm{ds}
	+
	\sum_{s = 1}^{S^{\tt ZXS}}
	\omega^{\tt ZXS}_s p(\xi^{\tt ZXS}_s,\eta^{\tt ZXS}_s),
\]
where $\omega_1^{\tt GL}=\frac{1}{L(L-1)}$ with $L=\lceil \frac{k+3}{2} \rceil$, $\left\{(\xi^{\mathrm{ZXS}}_s,\eta^{\mathrm{ZXS}}_s)\right\}$ denote the coordinates of $S^{\mathrm{ZXS}} = 3\lceil\frac{k-1}{2}\rceil(k+1)$ internal nodes. This convex decomposition was constructed by projecting a convex decomposition on rectangular cells (based on the tensor product of the Gauss and Gauss--Lobatto quadratures) onto triangular cells. The projection maps one edge of the rectangular cell to one vertex of the triangular cell and the remaining three edges of the rectangular cell to the three edges of the triangular cell; see \cite{ZXSPP2012} for more details.

\begin{remark}\label{rem:ChenShu}
In \cite{ChenShu2017}, Chen and Shu introduced another series of quadrature rules on triangular cells for $\mathbb{P}^k$ spaces to construct entropy-stable DG methods for hyperbolic conservation laws. These quadrature rules also satisfy \Cref{def:980}, thus qualifying as feasible convex decompositions for BP design on triangular cells for $\mathbb{P}^k$ spaces; see \cite[Appendix C]{ChenShu2017} for further details on these quadrature rules. However, Chen and Shu's quadrature rules \cite{ChenShu2017} are not optimal for BP studies in general, as demonstrated later in Section \ref{sec:BP}.  
\end{remark}

\begin{remark}
    A feasible convex decomposition \eqref{eq:980} can be reformulated as
    \begin{equation}\label{eq:1180}
    \frac{1}{|K|}\iint_K p(x,y) ~ \mathrm{d}x \, \mathrm{d}y
    =
    \sum_{i=1}^3
    \frac{w_i}{l_K^{(i)}}
    \int_{e_K^{(i)}} p(x,y) ~ \mathrm{ds}
    +
    \left( \sum_{s=1}^S \omega_s \right)
    p^*_K
    =
    \sum_{i=1}^3
    \frac{w_i}{l_K^{(i)}}
    \int_{e_K^{(i)}} p(x,y) ~ \mathrm{ds}
    +
    \left( 1-\sum_{i=1}^3 w_i \right)
    p(x^*_K,y^*_K),
    \end{equation}
    where $p^*_K$ denotes the weighted average of $p(x,y)$ over all the internal nodes:
    \begin{equation}\label{eq:1198}
        p^*_K 
        = 
        \frac{\sum_{s=1}^S \omega_s p(\xi_s,\eta_s)}{\sum_{s=1}^S \omega_s}
        \stackrel{\eqref{eq:980}}{=}
        \frac
        {
        \frac{1}{|K|}\iint_K p(x,y) ~ \mathrm{d}x \, \mathrm{d}y
        -
        \sum_{i=1}^3
        \frac{w_i}{l_K^{(i)}}
        \int_{e_K^{(i)}} p(x,y) ~ \mathrm{ds}
        }
        {1-\sum_{i=1}^3 w_i}.
    \end{equation}
According to Condition (iii) of \Cref{def:980} and the mean value theorem, 
there always exists a point ${\bm x}^*_K = (x^*_K,y^*_K) \in K$ such that $p^*_K = p(x^*_K,y^*_K)$. 
    In the special case of 
    \begin{equation}\label{eq:12151}
        \sum_{s=1}^S \omega_s = 1-\sum_{i=1}^3 w_i = 0,
    \end{equation}
    the formula for $p^*_K$ is not valid (it vanishes), but the decomposition \eqref{eq:1180} still holds in the form:
    \begin{equation}\label{eq:1215}
        \frac{1}{|K|}\iint_K p(x,y) ~ \mathrm{d}x \, \mathrm{d}y
    =
    \sum_{i=1}^3
    \frac{w_i}{l_K^{(i)}}
    \int_{e_K^{(i)}} p(x,y) ~ \mathrm{ds}.
    \end{equation}
    This reformulation \eqref{eq:1180} (or \eqref{eq:1215} in the case of \eqref{eq:12151}) only depends on the values of $p(x,y)$ on the cell edges. It is useful for designing a simplified BP limiter that avoids evaluating $p(x,y)$ at internal nodes; see \cite{zhang2011b} for the original proposal of this idea.
\end{remark}

\subsection{BP Conditions via Feasible Convex Decomposition for Scalar Conservation Laws}\label{sec:1039}

While the case of hyperbolic systems of conservation laws will be discussed in \Cref{sec:1228}, we first focus on scalar conservation laws \eqref{eq:scalar}. In this scenario, the solution satisfies the minimum and maximum principles, i.e., $u({\bm x},t) \in \mathcal{G} := [u_{\min},u_{\max}]$ for any ${\bm x} \in \Omega$ and $t > 0$, where $u_{\min} := \min_{\bm x} u(\bm x, 0)$ and $u_{\max} := \max_{\bm x} u(\bm x, 0)$.

We only discuss the forward Euler time discretization, while the discussions are directly extesiable to high-order SSP RK methods which are convex combinations of forward Euler steps. 
For DG methods using the LF numerical flux \eqref{fluxLF}, the evolution of the cell average on a triangular cell $K$ is given by
\begin{equation}\label{eq:1033}
	\bar{u}_{K}^{n+1}  
	= 
	\bar{u}_{K}^{n} 
	- 
	\frac{\Delta t}{2 |K|} 
	\sum_{i=1}^{3}  l_K^{(i)} 
	\sum_{\nu = 1}^{Q} \omega^{\tt G}_\nu 
	\left[ 
		 \bm{F}(u^{\rm int}_{i,\nu})\cdot{\bm n}_K^{(i)}
		+\bm{F}(u^{\rm ext}_{i,\nu})\cdot{\bm n}_K^{(i)}
		+\alpha u^{\rm int}_{i,\nu}
		-\alpha u^{\rm ext}_{i,\nu}
	\right],
\end{equation}
where $u^{\rm int}_{i,\nu}:=u^{n}_h(x^{(i),\nu}_K,y^{(i),\nu}_K)|_K$ and $u^{\rm ext}_{i,\nu}:=u^{n}_h(x^{(i),\nu}_K,y^{(i),\nu}_K)|_{K^{(i)}}$. Given a feasible convex decomposition \eqref{eq:980}, we define
\begin{equation}\label{eq:1252}
    u^*_K := 
    \frac
    {
    \frac{1}{|K|}\iint_K u^n_h(x,y) ~ \mathrm{d}x \, \mathrm{d}y
    -
    \sum_{i=1}^3
    \frac{w_i}{l_K^{(i)}}
    \int_{e_K^{(i)}} u^n_h(x,y) ~ \mathrm{ds}
    }
    {1-\sum_{i=1}^3 w_i}
    =
    \frac
    {
    \bar u_K^n
    -
    \sum_{i=1}^3 {w_i} \sum_{\nu=1}^Q 
    \omega^{\tt G}_\nu \, u^{\rm int}_{i,\nu}
    }
    {1-\sum_{i=1}^3 w_i}.
\end{equation}

\begin{theorem}[BP via Feasible Convex Decomposition]\label{thm:1048}
	Consider the scalar conservation law \eqref{eq:scalar}. If a convex decomposition in the form of \eqref{eq:980} is feasible for the $\mathbb{P}^k$ space on a triangular cell $K$, and the solution at time level $t^n$ satisfies
    \begin{equation}\label{eq:1275}
        \begin{dcases}
            u^{\rm int}_{i,\nu}, u^{\rm ext}_{i,\nu} \in \mathcal{G} ~~ \forall i, \nu, K,
            &
            \text{if the condition \eqref{eq:12151}}~holds,
            \\
            u^{\rm int}_{i,\nu}, u^{\rm ext}_{i,\nu}, u_K^* \in \mathcal{G} ~~ \forall i, \nu, K,
            &
            \text{otherwise},
        \end{dcases}
    \end{equation}
    then the $\mathbb{P}^k$-based high-order scheme \eqref{eq:1033} with the LF flux and viscosity coefficient $\alpha \ge \alpha^{\rm LF}$ is BP under the CFL condition
	\begin{equation}\label{eq:1050}
		\alpha \frac{\Delta t}{|K|} \le \min 
		\left\{ 
		\frac{w_1}{l^{(1)}_K},
		\frac{w_2}{l^{(2)}_K},
		\frac{w_3}{l^{(3)}_K}
		\right\}:=\mathcal{C}_{\rm BP}.
	\end{equation}
\end{theorem}

\begin{proof}
	First, consider the case when condition \eqref{eq:12151} does not hold. Applying the convex decomposition \eqref{eq:1180} on the polynomial $u_h|_K \in \mathbb{P}^k$ yields
	\begin{align*}
		\bar{u}^n_K = 
		\frac{1}{|K|}\iint_K u_h(x,y) ~ \mathrm{d}x \, \mathrm{d}y
		& \stackrel{\eqref{eq:1180}}{=}
		\sum_{i=1}^3
		\frac{w_i}{l_K^{(i)}}
		\int_{e_K^{(i)}} u_h(x,y) ~ \mathrm{ds}
		+
		\left(1-\sum_{i = 1}^{3} w_i \right) u^*_K
		=
		\sum_{i=1}^3
		w_i
		\sum_{\nu=1}^Q
		\omega^{\tt G}_\nu \, u^{\rm int}_{i,\nu}
		+
		\left(1-\sum_{i = 1}^{3} w_i \right) u^*_K.
	\end{align*}
	Then, the high-order scheme \eqref{eq:1033} can be rewritten as
	\[
		\bar{u}^{n+1}_K = 
		\sum_{i=1}^3 l^{(i)}_K
		\sum_{\nu=1}^Q \omega^{\tt G}_\nu
		\left[ 
			\left(
				 \frac{w_i}{l^{(i)}_K}-\frac{\alpha\,\Delta t}{2|K|} 
			\right)
			u^{\rm int}_{i,\nu} 
			-
			\frac{\Delta t}{2|K|}
			{\bm F}(u^{\rm int}_{i,\nu})\cdot{\bm n}^{(i)}_K
			+
			\frac{\alpha\,\Delta t}{2|K|} u^{\rm ext}_{i,\nu}
			-
			\frac{\Delta t}{2|K|}
			{\bm F}(u^{\rm ext}_{i,\nu})\cdot{\bm n}^{(i)}_K
		\right]
		+
		\left(1-\sum_{i = 1}^{3} w_i\right) u^*_K,
	\]
	which expresses $\bar{u}^{n+1}_K$ as a function of $u^{\rm int}_{i,\nu}$, $u^{\rm ext}_{i,\nu}$, and $u^*_K$. This function is monotonically increasing with respect to all arguments under the CFL condition \eqref{eq:1050}, because
	\begin{gather*}
		\frac{\partial \bar{u}^{n+1}_K}{\partial u^{\rm int}_{i,\nu}}
		=
		\frac{w_i}{l^{(i)}_K}
		-
		\frac{\alpha\,\Delta t}{2|K|} 
		-
		\frac{\Delta t}{2|K|}
		{\bm F}'(u^{\rm int}_{i,\nu})\cdot{\bm n}_K^{(i)}
		\stackrel{\eqref{eq:358}}{\ge}
		\frac{w_i}{l^{(i)}_K}
		-
		\frac{\alpha\,\Delta t}{2|K|} 
		-
		\frac{\alpha^{\rm LF}\,\Delta t}{2|K|}
		\ge
		\frac{w_i}{l^{(i)}_K}
		-
		\frac{\alpha\,\Delta t}{|K|}
		\stackrel{\eqref{eq:1050}}{\ge} 
		0, \\
		\frac{\partial \bar{u}^{n+1}_K}{\partial u^{\rm ext}_{i,\nu}}
		=
		\frac{\alpha\,\Delta t}{2|K|} 
		-
		\frac{\Delta t}{2|K|}
		{\bm F}'(u^{\rm int}_{i,\nu})\cdot{\bm n}_K^{(i)}
		\stackrel{\eqref{eq:358}}{\ge}
		\frac{\alpha\,\Delta t}{2|K|} 
		-
		\frac{\alpha^{\rm LF}\,\Delta t}{2|K|} = 0,
		\qquad
		\frac{\partial \bar{u}^{n+1}_K}{\partial u^*_K}
		=
		1-\sum_{i = 1}^{3} w_i
		\ge
		0.
	\end{gather*}
	Noting that $u^{\rm int}_{i,\nu}$, $u^{\rm ext}_{i,\nu}$, and $u^*_K$ belong to $\mathcal{G} = [u_{\min},u_{\max}]$ for all $i$, $\nu$, and $K$, this monotonicity implies
	\[
		u_{\min}
		=
		\bar{u}^{n+1}_K\Big|_{
			u^{\rm int}_{i,\nu}=
			u^{\rm ext}_{i,\nu}=
			u^*_K=u_{\min},\forall i,\nu
		}
		\le
		\bar{u}^{n+1}_K
		\le
		\bar{u}^{n+1}_K\Big|_{
			u^{\rm int}_{i,\nu}=
			u^{\rm ext}_{i,\nu}=
			u^*_K=u_{\max},\forall i,\nu
		}
		=
		u_{\max},
	\]
	which ensures that the scheme \eqref{eq:1033} is BP. The case when condition \eqref{eq:12151} holds can be similarly proven using the convex decomposition \eqref{eq:1215}, completing the proof.
\end{proof}

\subsection{BP Conditions via Feasible Convex Decomposition for Hyperbolic Systems of Conservation Laws}\label{sec:1228}

We now discuss the construction of BP schemes using a feasible convex decomposition for hyperbolic systems of conservation laws \eqref{eq:System}. In this context, the invariant region $\mathcal{G}$ is typically defined by the positivity or non-negativity of several functions of ${\bm u}$:
\begin{equation}\label{eq:1202}
    \mathcal{G} =
    \left\{~
        {\bm u} \in \mathbb{R}^d: g_j({\bm u}) \succ 0, \; \forall j\in \mathbb{J} \cup \hat{\mathbb{J}}
    ~\right\},
\end{equation}
where the symbol ``$\succ$'' denotes ``$>$'' if $j \in \mathbb{J}$ or ``$\geq$'' if $j \in \hat{\mathbb{J}}$.

For DG methods using the LF numerical flux \eqref{fluxLF} and forward Euler time discretization, the evolution of the cell average on a triangular cell $K$ is given by
\begin{equation}\label{eq:1182}
    \bar{{\bm u}}_{K}^{n+1}  
    = 
    \bar{{\bm u}}_{K}^{n} 
    - 
    \frac{\Delta t}{2 |K|} 
    \sum_{i=1}^{3}  l_K^{(i)} 
    \sum_{\nu = 1}^{Q} \omega^{\tt G}_\nu 
    \left[ 
    \bm{F}({\bm u}^{\rm int}_{i,\nu})\cdot{\bm n}_K^{(i)}
    +\bm{F}({\bm u}^{\rm ext}_{i,\nu})\cdot{\bm n}_K^{(i)}
    +\alpha {\bm u}^{\rm int}_{i,\nu}
    -\alpha {\bm u}^{\rm ext}_{i,\nu}
    \right],
\end{equation}
where ${\bm u}^{\rm int}_{i,\nu}:={\bm u}^{n}_h(x^{(i),\nu}_K,y^{(i),\nu}_K)|_K$ and ${\bm u}^{\rm ext}_{i,\nu}:={\bm u}^{n}_h(x^{(i),\nu}_K,y^{(i),\nu}_K)|_{K^{(i)}}$. Similar to the scalar case discussed in Section \ref{sec:1039}, a feasible convex decomposition is critical in establishing the BP property of \eqref{eq:1182} with respect to $\mathcal{G}$. However, for hyperbolic systems, BP analysis is often more challenging, especially when nonlinear constraints are involved.

We adopt the GQL framework from \cite{Wu2023Geometric}, which transforms any nonlinear constraints in \eqref{eq:1202} into equivalent linear constraints, leading to the following GQL representation of the invariant region $\mathcal{G}$:
\begin{equation}\label{eq:1200}
    \mathcal{G}^\star
    :=
    \left\{\,
    {\bm u} \in \mathbb{R}^d
    \, : \, ({\bm u}-{\bm u}^\star_j) \cdot \mathbf{n}^\star_j \succ 0, \; 
    \forall {\bm u}^\star_j \in \mathcal{S}_j,
    \forall j\in \mathbb{J} \cup \hat{\mathbb{J}}
    \,\right\},
\end{equation}
where $\mathbf{n}^\star_j := \nabla g_j({\bm u}^\star)$, $\mathcal{S}_j := \partial \mathcal{G} \cap \partial \mathcal{G}_j$, and $\mathcal{G}_j := \left\{ {\bm u} \in \mathbb{R}^d : g_j({\bm u})\succ 0\right\}$.

Using the GQL representation \eqref{eq:1200}, one can conduct BP analysis of \eqref{eq:1182} analogously to the scalar case, under the following assumption, which is valid for most hyperbolic systems of conservation laws.

\begin{assumption}[Weak LF Property]\label{asp:1273}
    For any unit vector ${\bm n} \in \mathbb{R}^2$ and any ${\bm u} \in \mathcal{G}$, there exist $\tilde\alpha({\bm u},{\bm n}) > 0$ and ${\bm \zeta}({\bm u^\star}) \in \mathbb{R}^2$ such that
	\begin{equation}\label{eq:1222}
		\tilde\alpha({\bm u},{\bm n})({\bm u}-{\bm u}^\star)\cdot \mathbf{n}_j^\star 
		- 
		({\bm F}({\bm u})\cdot {\bm n})\cdot \mathbf{n}_j^\star
		\succ
		{\bm \zeta}({\bm u^\star}) \cdot {\bm n}
		\qquad \forall {\bm u}^\star \in \mathcal{S}_j.
	\end{equation}
\end{assumption}

The weak LF property implies that for any $\alpha \ge \tilde\alpha({\bm u},{\bm n})$,
\begin{equation}\label{eq:1233}
    \alpha({\bm u}-{\bm u}^\star)\cdot \mathbf{n}_j^\star 
    - 
    ({\bm F}({\bm u})\cdot {\bm n})\cdot \mathbf{n}_j^\star
    \succ
    {\bm \zeta}({\bm u^\star}) \cdot {\bm n}
    \qquad \forall {\bm u}^\star \in \mathcal{S}_j.
\end{equation}

\begin{remark}\label{rmk:1243}
    The GQL representation \eqref{eq:1200} and the weak LF property (Assumption \ref{asp:1273}) apply to many hyperbolic conservation laws, including the 2D Euler system \eqref{eq:euler}, the relativistic hydrodynamic system, and the ten-moment Gaussian closure system. For more details and additional examples, see \cite{Wu2023Geometric}. Specifically, the GQL framework applied to the 2D Euler system \eqref{eq:euler} is summarized in the following theorem.
\end{remark}

\begin{theorem}\label{thm:1355}
For the 2D Euler system \eqref{eq:euler}, the admissible state set for positive density and pressure is defined as
\begin{align}\label{G:Euler}
    \mathcal{G} = 
    \left\{
    {\bm u}=(\rho, m_1, m_2, E)^\top \in \mathbb{R}^4: \rho({\bm u}) > 0, ~~ \mathcal{E}({\bm u}):=E - \frac{m_1^2 + m_2^2}{2\rho} > 0  
    \right\},
\end{align}
which is a convex set \cite{zhang2010positivity} because $\mathcal{E}({\bm u})$ is a concave function of ${\bm u}$ on the set $\{{\bm u} \in \mathbb{R}^4: \rho({\bm u}) > 0\}$. 
  The GQL representation of the invariant region $\mathcal{G}$ is given by 
    \begin{equation*}
    \mathcal{G}^\star
    :=
    \left\{\,
    {\bm u} \in \mathbb{R}^4
    \, : {\bm u} \cdot {\bf n}_1 >0, \quad 
     ({\bm u}-{\bm u}^\star) \cdot \mathbf{n}^\star > 0 ~~ 
    \forall {\bm u}^\star \in \mathcal{S}
    \,\right\},
\end{equation*}
    with ${\bf n}_1=(1,0,0,0)$,
    \begin{align*}
        \mathcal{S} &= \left\{
            {\bm u}^\star = \left(
            \rho^\star,\rho^\star v_1^\star,
            \rho^\star v_2^\star,
            \frac{\rho^\star (v_1^\star)^2+\rho^\star (v_2^\star)^2}{2}
            \right)^\top:
            \rho^\star > 0,
            v_1^\star \in \mathbb{R},
            v_2^\star \in \mathbb{R}
        \right\},\\
        \mathbf{n}^\star &= \left(
        \frac{(v_1^\star)^2+(v_2^\star)^2}{2},
        -v_1^\star,
        -v_2^\star,
        1
        \right)^\top,
    \end{align*}
    and the weak LF property \eqref{eq:1222} holds with 
    $\tilde \alpha({\bm u},{\bm n})$ being the spectral radius of the Jacobian matrix ${\bm F}'({\bm u})\cdot {\bm n}$ and ${\bm \zeta}^\star = {\bm 0}$.
\end{theorem}

\begin{proof}
    The above \Cref{thm:1355} can be proven similarly to its one-dimensional counterpart in \cite{Wu2023Geometric}. Therefore, the proof is omitted here.
\end{proof}

Similarly to \eqref{eq:1198} and \eqref{eq:1252}, we define the weighted average of ${\bm u}_h({\bm x})|_K$ over all internal nodes of $K$ by
\begin{equation*}
    {\bm u}^*_K := 
    \frac
    {
    \frac{1}{|K|}\iint_K {\bm u}^n_h(x,y) ~ \mathrm{d} x \mathrm{d} y
    -
    \sum_{i=1}^3
    \frac{w_i}{l_K^{(i)}}
    \int_{e_K^{(i)}} {\bm u}^n_h(x,y) ~ \mathrm{ds}
    }
    {1-\sum_{i=1}^3 w_i}
    =
    \frac
    {
    \bar {\bm u}_K^n
    -
    \sum_{i=1}^3 {w_i} \sum_{\nu=1}^Q \omega^{\tt G}_\nu \, {\bm u}^{\rm int}_{i,\nu}
    }
    {1-\sum_{i=1}^3 w_i}.
\end{equation*}

We are now ready to demonstrate that a provable BP property of the high-order scheme \eqref{eq:1182} can be derived using any feasible convex decomposition \eqref{eq:980}  on triangular cells.

\begin{theorem}[BP via Feasible Convex Decomposition]\label{thm:1008}
    Consider the hyperbolic system of conservation laws \eqref{eq:System}. If a convex decomposition of the form \eqref{eq:980} is feasible for the $\mathbb{P}^k$ space on a triangular cell $K$, and the solution ${\bm u}_h^n$ satisfies
    \begin{equation}\label{eq:1526}
        \begin{dcases}
            {\bm u}^{\rm int}_{i,\nu}, {\bm u}^{\rm ext}_{i,\nu} \in \mathcal{G} ~~ \forall i, \nu, K
            & \textrm{if the condition \eqref{eq:12151} holds}, \\
            {\bm u}^{\rm int}_{i,\nu}, {\bm u}^{\rm ext}_{i,\nu}, {\bm u}_K^* \in \mathcal{G} ~~ \forall i, \nu, K
            & \textrm{otherwise},
        \end{dcases}
    \end{equation}
    then the $\mathbb{P}^k$-based high-order scheme \eqref{eq:1182} with the LF flux and a viscosity coefficient $\alpha$ satisfying
    \begin{equation}\label{eq:1250}
        \alpha 
        \ge 
        \alpha^{\rm wLF} :=
        \max_{{\bm x}\in {e}^{(i)}_K, i \in \{1,2,3\}, K \in \mathcal{T}_h} 
        \tilde\alpha({\bm u}_h({\bm x}),{\bm n}_K^{(i)})
    \end{equation}
    is BP under the CFL condition
    \begin{equation}\label{eq:1009}
        \alpha \frac{\Delta t}{|K|} \le \min 
        \left\{ 
        \frac{w_1}{l^{(1)}_K},
        \frac{w_2}{l^{(2)}_K},
        \frac{w_3}{l^{(3)}_K}
        \right\}
        =\mathcal{C}_{\rm BP}.
    \end{equation}
\end{theorem}

\begin{proof}
    Let us first consider the case when condition \eqref{eq:12151} does not hold. Applying the convex decomposition \eqref{eq:1180} to ${\bm u}_h|_K \in (\mathbb{P}^k)^d$ yields 
	\begin{equation}\label{eq:1269}
		\bar{\bm u}^n_K = 
		\frac{1}{|K|}\iint_K {\bm u}_h(x,y) ~ \mathrm{d} x \mathrm{d} y
		\stackrel{\eqref{eq:1180}}{=}
		\sum_{i=1}^3
		\frac{w_i}{l_K^{(i)}}
		\int_{e_K^{(i)}} {\bm u}_h(x,y) ~ \mathrm{ds}
		+
		\left( 1-\sum_{i=1}^3 w_i\right) {\bm u}^*_K
            =
		\sum_{i=1}^3
		w_i
		\sum_{\nu=1}^Q
		\omega^{\tt G}_\nu \, {\bm u}^{\rm int}_{i,\nu}
		+
		\left( 1-\sum_{i=1}^3 w_i\right) {\bm u}^*_K.
	\end{equation}
	Notice that for any ${\bm u}^\star \in \mathcal{S}$ and $j \in \mathbb{J}\cup\hat{\mathbb{J}}$,
	\begin{align*}
		&(\bar{\bm u}^{n+1}_K-{\bm u}^\star)\cdot \mathbf{n}^\star_j
		\\
		\stackrel{\eqref{eq:1182}}{=} & 
		\bar{\bm u}^{n}_K\cdot \mathbf{n}^\star_j
		-
		{\bm u}^\star\cdot \mathbf{n}^\star_j
		-
		\frac{\Delta t}{2 |K|} 
		\sum_{i=1}^{3}  l_K^{(i)} 
		\sum_{\nu = 1}^{Q} \omega^{\tt G}_\nu 
		\left[ 
		\bm{F}({\bm u}^{\rm int}_{i,\nu})\cdot{\bm n}_K^{(i)}
		+\bm{F}({\bm u}^{\rm ext}_{i,\nu})\cdot{\bm n}_K^{(i)}
		+\alpha {\bm u}^{\rm int}_{i,\nu}
		-\alpha {\bm u}^{\rm ext}_{i,\nu}
		\right]\cdot \mathbf{n}^\star_j \\
		\stackrel{\eqref{eq:1269}}{=} & 
		\sum_{i = 1}^3 l^{(i)}_K
		\sum_{\nu = 1}^Q \omega^{\tt G}_{\nu}
		\left[
			\left(\frac{w_i}{l_K^{(i)} }-\frac{\alpha\dt}{2|K|}\right)
			({\bm u}^{\rm int}_{i,\nu} - {\bm u}^\star) \cdot \mathbf{n}^\star_j
			-
			\frac{\dt}{2|K|}
			\left(\bm{F}({\bm u}^{\rm int}_{i,\nu})\cdot{\bm n}_K^{(i)}\right)\cdot \mathbf{n}^\star_j
		\right]\\
		& +
		\sum_{i = 1}^3 l^{(i)}_K
		\sum_{\nu = 1}^Q \omega^{\tt G}_{\nu}
		\left[
		\frac{\alpha\dt}{2|K|}
		({\bm u}^{\rm ext}_{i,\nu} - {\bm u}^\star) \cdot \mathbf{n}^\star_j
		-
		\frac{\dt}{2|K|}
		\left(\bm{F}({\bm u}^{\rm ext}_{i,\nu})\cdot{\bm n}_K^{(i)}\right)\cdot \mathbf{n}^\star_j
		\right]
		+
            \left( 1-\sum_{i=1}^3 w_i\right) 
		\left({\bm u}^*_K-{\bm u}^\star\right)\cdot \mathbf{n}^\star_j
		\\
		\stackrel{\eqref{eq:1526}, \eqref{eq:1009}}{\ge} & 
		\sum_{i = 1}^3 l^{(i)}_K
		\sum_{\nu = 1}^Q \omega^{\tt G}_{\nu}
		\left[
		\frac{\alpha\dt}{2|K|}
		({\bm u}^{\rm int}_{i,\nu} - {\bm u}^\star) \cdot \mathbf{n}^\star_j
		-
		\frac{\dt}{2|K|}
		\left(\bm{F}({\bm u}^{\rm int}_{i,\nu})\cdot{\bm n}_K^{(i)}\right)\cdot \mathbf{n}^\star_j
		\right]\\
		& +
		\sum_{i = 1}^3 l^{(i)}_K
		\sum_{\nu = 1}^Q \omega^{\tt G}_{\nu}
		\left[
		\frac{\alpha\dt}{2|K|}
		({\bm u}^{\rm ext}_{i,\nu} - {\bm u}^\star) \cdot \mathbf{n}^\star_j
		-
		\frac{\dt}{2|K|}
		\left(\bm{F}({\bm u}^{\rm ext}_{i,\nu})\cdot{\bm n}_K^{(i)}\right)\cdot \mathbf{n}^\star_j
		\right]
		\\
		\stackrel{\eqref{eq:1233},\eqref{eq:1250}}{\succ} & 
		\sum_{i = 1}^3 l^{(i)}_K
		\sum_{\nu = 1}^Q \omega^{\tt G}_{\nu}
		\frac{\dt}{2|K|}
		\left[
		{\bm \zeta}({\bm u}^\star)\cdot {\bm n}_K^{(i)}
		\right]
		+
		\sum_{i = 1}^3 l^{(i)}_K
		\sum_{\nu = 1}^Q \omega^{\tt G}_{\nu}
		\frac{\dt}{2|K|}
		\left[
		{\bm \zeta}({\bm u}^\star)\cdot {\bm n}_K^{(i)}
		\right]
		\\
		= &
		\frac{\dt}{|K|}
		\sum_{\nu = 1}^Q \omega^{\tt G}_{\nu}
		{\bm \zeta}({\bm u}^\star)
		\cdot 
		\left(
			\sum_{i = 1}^3 l^{(i)}_K
			{\bm n}_K^{(i)}
		\right)
		=
		\frac{\dt}{2|K|}
		\sum_{\nu = 1}^Q \omega^{\tt G}_{\nu}
		{\bm \zeta}({\bm u}^\star)
		\cdot 
		{\bm 0}
		= 0.
	\end{align*}
	Therefore, we have shown that $\bar{\bm u}^{n+1}_K \in \mathcal{G}^\star = \mathcal{G}$ based on \eqref{eq:1200}. The remaining case when condition \eqref{eq:12151} holds can be similarly proven using the convex decomposition \eqref{eq:1215}, completing the proof.
\end{proof}

As shown in Theorems \ref{thm:1048} and \ref{thm:1008}, $\mathcal{C}_{\rm BP}$ is the CFL number  for provably ensuring the BP property. Consequently, we refer to $\mathcal{C}_{\rm BP}$ as the BP CFL number. Different convex decompositions may yield different values of $\mathcal{C}_{\rm BP}$. Our objective is to seek the optimal convex decomposition that maximizes $\mathcal{C}_{\rm BP}$, thereby resulting in the most lenient BP CFL condition. This optimization would enhance the efficiency of high-order BP schemes. This goal leads to the following open problem:
\begin{problem}[Optimal Convex Decomposition Problem]\label{prb:1549}
    Given $k \in \mathbb{N}_+$ and a triangular cell $K$, find the optimal feasible convex decomposition in the form of \eqref{eq:980} that maximizes the BP CFL number:
    \[
        \mathcal{C}_{\rm BP} = 
        \min 
        \left\{ 
        \frac{w_1}{l^{(1)}_K},
        \frac{w_2}{l^{(2)}_K},
        \frac{w_3}{l^{(3)}_K}
        \right\}.
    \]
\end{problem}

\subsection{Optimal Convex Decomposition for $\mathbb{P}^1$ and $\mathbb{P}^2$ Spaces on Triangular Cells}

In this section, we address the optimal convex decomposition problem (\Cref{prb:1549}) for the widely-used $\mathbb{P}^1$ and $\mathbb{P}^2$ spaces on triangular cells. Specifically, we will construct feasible convex decompositions for $\mathbb{P}^1$ and $\mathbb{P}^2$ spaces, respectively, and demonstrate that these decompositions are \emph{optimal}. That is, no other convex decomposition can achieve a larger BP CFL number $\mathcal{C}_{\rm BP}$. For notational simplicity, we consider a particular triangular cell $K$ and assume that the indices of its edges $\{e_K^{(i)}\}_{i=1}^3$ and vertices $\{{\bm v}^{(i)}\}_{i=1}^3$ are arranged such that $l_K^{(1)} \ge l_K^{(2)} \ge l_K^{(3)}$.

\subsubsection{Optimal Convex Decomposition for $\mathbb{P}^1$ Space on Triangular Cells}

We consider a feasible convex decomposition for $\mathbb{P}^1$ space, given by
\begin{subequations}\label{eq:1040}
\begin{equation}\label{eq:1032}
	w_i = \frac
	{2l_K^{(i)}}
	{3l_K^{(1)}+3l_K^{(2)}}, \qquad i = 1,2,3,
\end{equation}
with a single internal node, weighted and positioned as follows:
\begin{equation}\label{eq:1048}
	\omega_1 = 
	\frac
	{l_K^{(1)}+l_K^{(2)}-2l_K^{(3)}}
	{3l_K^{(1)}+3l_K^{(2)}},
	\quad
	(\xi_1, \eta_1)^\top
	=
	\frac{
		(l_K^{(1)}-l_K^{(3)})\,{\bm v}^{(1)}
		+
		(l_K^{(2)}-l_K^{(3)})\,{\bm v}^{(2)}
	}
	{l_K^{(1)}+l_K^{(2)}-2 \, l_K^{(3)}}.
\end{equation}
\end{subequations}
Note that if $l_K^{(1)} = l_K^{(2)} = l_K^{(3)}$, then the internal node weight $\omega_1$ becomes zero, and no internal node is required.

We claim that this convex decomposition \eqref{eq:1040} is \emph{optimal} for $\mathbb{P}^1$ space, meaning the BP CFL number $\mathcal{C}_{\rm BP}$ defined in \eqref{eq:1050} and \eqref{eq:1009} cannot be further improved.

\begin{theorem}\label{thm:1062}
	The convex decomposition \eqref{eq:1040} is feasible and optimal for $\mathbb{P}^1$ space.
\end{theorem}

\begin{proof}
	To prove that the convex decomposition \eqref{eq:1040} is feasible for $\mathbb{P}^1$ space, one can verify that it satisfies the necessary conditions. We will demonstrate its optimality by contradiction. Assume there exists another feasible convex decomposition for $\mathbb{P}^1$:
	\begin{equation}\label{eq:1072}
		\frac{1}{|K|}\iint_K p(x,y) ~ \textrm{d} x \textrm{d} y
		=
		\sum_{i=1}^3
		\frac{\tilde w_i}{l_K^{(i)}}
		\int_{e_K^{(i)}} p(x,y) ~ \textrm{ds}
		+
		\sum_{s = 1}^{\tilde S}
		\tilde \omega_s p(\tilde \xi_s,\tilde \eta_s),
	\end{equation}
	that achieves a strictly larger BP CFL number:
	\[
		\min 
		\left\{ 
		\frac{\tilde w_1}{l^{(1)}_K},
		\frac{\tilde w_2}{l^{(2)}_K},
		\frac{\tilde w_3}{l^{(3)}_K}
		\right\}
		>
		\frac{2}{3(l^{(1)}_K+l^{(2)}_K)},
	\]
	which implies
	\begin{equation}\label{eq:1094}
		\tilde w_i > \frac{2 l^{(i)}_K}{3(l^{(1)}_K+l^{(2)}_K)}, \qquad i = 1,2,3.
	\end{equation}
	Consider a critical polynomial $p_*(x,y) \in \mathbb{P}^1$ defined by
	\begin{equation}\label{eq:1773}
		p_*(x_1,y_1) = p_*(x_2,y_2) = 0, \quad
		p_*(x_3,y_3) = 1.
	\end{equation}
	Notice that $p_*(x,y)$ is non-negative on $K$,
	\begin{equation}\label{eq:1103}
		p_*(x,y) \ge 0 \qquad \forall ~ (x,y) \in K,
	\end{equation}
	and $p_*(x,y)$ is linear on $K$; see the following \Cref{fig:1813} for an illustration. It is clear to see that the cell and edge averages of function $p_*(x,y)$ are
	\[
		\frac{1}{|K|}\iint_K p_*(x,y) ~ \textrm{d} x \textrm{d} y = 
  \frac{1}{|K|} \cdot \frac{|K|}{3} = \frac13, \quad
		\frac{1}{l_K^{(1)}} \int_{e_K^{(1)}} p_*(x,y) ~ \textrm{ds} =
		\frac{1}{l_K^{(2)}} \int_{e_K^{(2)}} p_*(x,y) ~ \textrm{ds} = \frac12, \quad
		\frac{1}{l_K^{(3)}} \int_{e_K^{(3)}} p_*(x,y) ~ \textrm{ds} = 0.
	\]
	Using the convex decomposition \eqref{eq:1072} on $p_*$, we get
	\begin{align*}
		\frac13 & = \frac{1}{|K|}\iint_K p_*(x,y) ~ \textrm{d} x \textrm{d} y
		=
		\sum_{i=1}^3
		\frac{\tilde w_i}{l_K^{(i)}}
		\int_{e_K^{(i)}} p_*(x,y) ~ \textrm{ds}
		+
		\sum_{s = 1}^{S}
		\tilde \omega_s p_*(\tilde \xi_s,\tilde \eta_s) \\
		& =
		\frac{\tilde w_1}{2}+\frac{\tilde w_2}{2}+\sum_{s = 1}^{S}
		\tilde \omega_s p_*(\tilde \xi_s,\tilde \eta_s)
		\stackrel{\eqref{eq:1103}}{\ge}
		\frac{\tilde w_1}{2}+\frac{\tilde w_2}{2}
		\stackrel{\eqref{eq:1094}}{>}
		\frac{l^{(1)}_K}{3(l^{(1)}_K+l^{(2)}_K)}+\frac{l^{(2)}_K}{3(l^{(1)}_K+l^{(2)}_K)} = \frac13,
	\end{align*}
	which is a contradiction, thus proving the optimality.
\end{proof}

\begin{figure}
    \centering
    \includegraphics[width=0.99\linewidth]{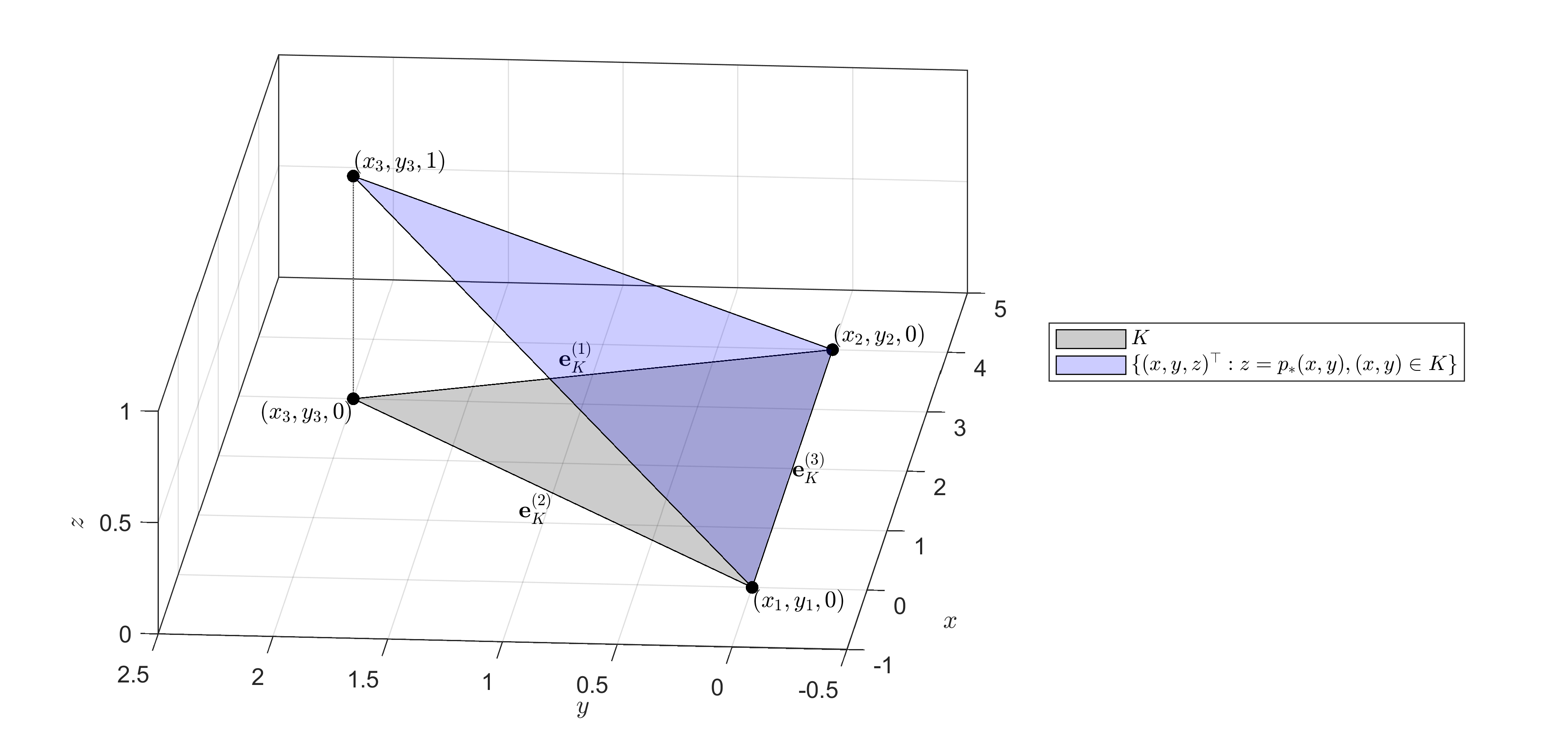}
    \caption{Visualization of the critical polynomial $p_*(x,y)$ defined in \eqref{eq:1773} on an example triangular cell $K$ with vertices $(0,0)$, $(4,0)$, and $(3,2)$.}
    \label{fig:1813}
\end{figure}

\subsubsection{Optimal Convex Decomposition for $\mathbb{P}^2$ Space on Triangular Cells}

We consider a convex decomposition feasible for $\mathbb{P}^2$ space with edge weights
\begin{subequations}\label{eq:1139}
	\begin{equation}
		w_i = \frac{2l_K^{(i)}}{9 \bar l_K + 3 \hat l_K}, \qquad i = 1,2,3,
	\end{equation}
	and two internal nodes with weights and coordinates
	\begin{equation}
		\omega_s = \frac{\bar l_K + \hat l_K}{6 \bar l_K + 2 \hat l_K},
		\quad
		(\xi_s, \eta_s)^\top 
		= 
		\sum_{i=1}^3 \beta_{s,i} \, {\bm v}^{(i)},
		\quad
		\beta_{s,i}
		=
		\frac
		{{\bm l}_K^\top {\bm M}_{s,i} \, {\bm l}_K
			+ 2 \, c_{s,i} \, \hat l_K}
		{18(\bar l_K+\hat l_K)(l_K^{(2)}+\hat l_K)},
		\qquad s = 1,2,
	\end{equation}
	where ${\bm l}_K := (l_K^{(1)}, l_K^{(2)}, l_K^{(3)})^\top$, $\bar l_K := (l_K^{(1)} + l_K^{(2)} + l_K^{(3)})/3$, and
        \begin{equation}\label{eq:1703}
	\begin{aligned}
	\hat l_K := & \sqrt{
		\frac
		{
			\big(l_K^{(1)}\big)^2
			+
			\big(l_K^{(2)}\big)^2
			+
			\big(l_K^{(3)}\big)^2
		}
		{3}
		+
		\frac
		{
			\big(l_K^{(1)}-l_K^{(2)}\big)^2
			+
			\big(l_K^{(2)}-l_K^{(3)}\big)^2
			+
			\big(l_K^{(3)}-l_K^{(1)}\big)^2
		}
		{3}
	}\\
	= &
	\sqrt
	{
		\big(l_K^{(1)}\big)^2
		+
		\big(l_K^{(2)}\big)^2
		+
		\big(l_K^{(3)}\big)^2
		-
		\frac23\Big(
		l_K^{(1)} l_K^{(2)} + l_K^{(2)} l_K^{(3)} + l_K^{(3)} l_K^{(1)}
		\Big)
	}.
	\end{aligned}
        \end{equation}
	The positive coefficients $c_{s,i}$ and the positive definite matrices ${\bm M}_{s,i}$ are given by
	\begin{equation}
		\begin{aligned}
		c_{1,1}
		&=
		3 l_K^{(1)} + 3 l_K^{(2)} + \sqrt{3} l_K^{(2)} - \sqrt{3} l_K^{(3)},
		&
		c_{1,2}
		&=
		6 l_K^{(2)} + \sqrt{3} l_K^{(3)} - \sqrt{3} l_K^{(1)},
		&
		c_{1,3}
		&=
		3 l_K^{(2)} + 3 l_K^{(3)} + \sqrt{3} l_K^{(1)} - \sqrt{3} l_K^{(2)},
		\\
		c_{2,1}
		&=
		3 l_K^{(1)} + 3 l_K^{(2)} + \sqrt{3} l_K^{(3)} - \sqrt{3} l_K^{(2)},
		&
		c_{2,2}
		&=
		6 l_K^{(2)} + \sqrt{3} l_K^{(1)} - \sqrt{3} l_K^{(3)},
		&
		c_{2,3}
		&=
		3 l_K^{(2)} + 3 l_K^{(3)} + \sqrt{3} l_K^{(2)} - \sqrt{3} l_K^{(1)},
		\end{aligned}
	\end{equation}
	\begin{equation}
		\begin{aligned}
			{\bm M}_{1,1}&=
			\begin{pmatrix}
				6 & 1 & -2 \\
				1 & 2\sqrt{3}+6 & -\sqrt{3}-2 \\
				-2 & -\sqrt{3}-2 & 6
			\end{pmatrix},&
			{\bm M}_{2,1}&=
			\begin{pmatrix}
				6 & 1 & -2 \\
				1 & 6-2\sqrt{3} & \sqrt{3}-2 \\
				-2 & \sqrt{3}-2 & 6
			\end{pmatrix},&
			\\
			{\bm M}_{1,2}&=
			\begin{pmatrix}
				6 & -\sqrt{3}-2 & -2 \\
				-\sqrt{3}-2 & 12 & \sqrt{3}-2 \\
				-2 & \sqrt{3}-2 & 6
			\end{pmatrix},&
			{\bm M}_{2,2}&=
			\begin{pmatrix}
				6 & \sqrt{3}-2 & -2 \\
				\sqrt{3}-2 & 12 & -\sqrt{3}-2 \\
				-2 & -\sqrt{3}-2 & 6
			\end{pmatrix},&
			\\
			{\bm M}_{1,3}&=
			\begin{pmatrix}
				6 & \sqrt{3}-2 & -2 \\
				\sqrt{3}-2 & 6-2\sqrt{3} & 1 \\
				-2 & 1 & 6
			\end{pmatrix},&
			{\bm M}_{2,3}&=
			\begin{pmatrix}
				6 & -\sqrt{3}-2 & -2 \\
				-\sqrt{3}-2 & 2\sqrt{3}+6 & 1 \\
				-2 & 1 & 6
			\end{pmatrix}.
		\end{aligned}
	\end{equation}
\end{subequations}

The convex decomposition defined by \eqref{eq:1139} is designed to be feasible for $\mathbb{P}^2$ space and, as we will show, is optimal. This means it achieves the maximum possible BP CFL number $\mathcal{C}_{\rm BP}$, which governs the stability condition of bound-preserving schemes.

\begin{theorem}\label{thm:1239}
	The convex decomposition \eqref{eq:1139} is feasible and optimal for $\mathbb{P}^2$ space.
\end{theorem}

\begin{proof}
	It can be verified that the convex decomposition \eqref{eq:1139} satisfies the conditions (i), (ii), and (iii) of \Cref{def:eqri} for $\mathbb{P}^2$ space. Specifically, condition (i) is satisfied because the decomposition holds exactly for all polynomials in $\mathbb{P}^2$. Condition (ii) is verified by the positivity of the weights $\{w_i\}_{i=1}^3$ and $\{\omega_s\}_{s=1}^2$, ensuring that they sum up to one. Since the barycentric coordinates $\beta_{s,i}$ satisfy $\sum_i \beta_{s,i} = 1$ and $\beta_{s,i} > 0$ for $s=1,2$ and $i=1,2,3$, the internal nodes are strict convex combinations of the vertices of $K$, fulfilling condition (iii).

	To demonstrate the optimality of the convex decomposition \eqref{eq:1139}, we follow an argument similar to that used in the proof of \Cref{thm:1062}. Consider a critical polynomial $p_*(x,y) \in \mathbb{P}^2$ that is nonnegative over $K$. When $l_K^{(1)} > l_K^{(3)}$, we can define $p_*(x,y) = q_*^2(x,y)$, where $q_*(x,y) \in \mathbb{P}^1$ is uniquely determined by the conditions:
	\[
		q_*(x_i,y_i) = 3\bar l_K + \hat l_K - 4 l_K^{(i)}, \quad i = 1,2,3.
	\]
	In the special case of $l_K^{(1)} = l_K^{(2)} = l_K^{(3)}$, we define $p_*(x,y) = q_*^2(x,y)$ with $q_*(x,y) \in \mathbb{P}^1$ such that $q_*(x_1,y_1)=0$, $q_*(x_2,y_2)=1$, and $q_*(x_3,y_3)=-1$. These definitions ensure that $p_*(x,y)$ captures the necessary features to demonstrate the contradiction if a decomposition with a larger BP CFL number exists. Thus, the proof of optimality is completed.
\end{proof}

\subsection{BP High-Order Schemes for Hyperbolic Conservation Laws via Optimal Convex Decomposition}\label{sec:BP}

As specific applications of Theorems \ref{thm:1048}, \ref{thm:1008}, \ref{thm:1062}, and \ref{thm:1239}, this section introduces the construction of $\mathbb{P}^1$- and $\mathbb{P}^2$-based BP schemes using the optimal convex decompositions \eqref{eq:1040} and \eqref{eq:1139}, for scalar conservation laws \eqref{eq:scalar} and hyperbolic systems of conservation laws \eqref{eq:euler} that satisfy the weak LF property (\Cref{asp:1273}).

\begin{corollary}\label{cor:1837}
    For hyperbolic systems of conservation laws \eqref{eq:euler} satisfying the weak LF property (\Cref{asp:1273}) (resp. scalar conservation laws \eqref{eq:scalar}), 
    if the solution at time level $t^n$ satisfies \eqref{eq:1526} (resp. \eqref{eq:1275}), then the $\mathbb{P}^1$-based scheme \eqref{eq:1182} (resp. \eqref{eq:1033}) with the LF flux and viscosity coefficient $\alpha$ satisfying \eqref{eq:1250} (resp. $\alpha \ge \alpha^{\rm LF}$) is BP under the optimal BP CFL condition
	\begin{equation}\label{eq:1692}
		\alpha \frac{\dt}{|K|} \le
		\mathcal{C}_{K,1}^{\tt DCW} \quad \forall K \in \mathcal{T}_h, \quad
		\mathcal{C}_{K,1}^{\tt DCW} := \frac{2}{3(l^{(1)}_K+l^{(2)}_K)}.
	\end{equation}
    This BP CFL condition \eqref{eq:1692} cannot be improved by selecting another feasible convex decomposition.
\end{corollary}

\begin{remark}
	Based on the classical convex decomposition, Zhang, Xia, and Shu \cite{ZXSPP2012} proposed a $\mathbb{P}^1$-based BP scheme under a more restrictive BP CFL condition (see the comparison in \Cref{fig:1744a}):
	\[
		\alpha^{\rm LF} \frac{\dt}{|K|} \le
		\mathcal{C}^{\tt ZXS}_{K,1} \quad \forall K \in \mathcal{T}_h, \quad
		\mathcal{C}^{\tt ZXS}_{K,1} := \frac{1}{9 \bar l_K},
	\]
    where $\bar l_K$ denotes the average edge length of the triangular cell $K$.
\end{remark}

\begin{remark}
Following our BP analysis in \Cref{thm:1008}, we can also use Chen and Shu's quadrature from \cite[Table C.1]{ChenShu2017} as a feasible convex decomposition for the $\mathbb{P}^1$ space. The resulting $\mathbb{P}^1$-based scheme for the 2D Euler system \eqref{eq:euler} is BP under the following CFL condition:
\begin{align}\label{eq:P1chenshu}
    \alpha^{\rm LF} \frac{\dt}{|K|} \le \mathcal{C}^{\tt CS}_{K,1} \quad \forall K \in \mathcal{T}_h, \quad \mathcal{C}^{\tt CS}_{K,1} := \frac{1}{6 l_K^{(1)}}.
\end{align}
Clearly, our optimal BP CFL condition \eqref{eq:1692} is less restrictive than \eqref{eq:P1chenshu}.
\end{remark}

\begin{corollary}\label{cor:1859}
    For hyperbolic systems of conservation laws \eqref{eq:euler} satisfying the weak LF property (\Cref{asp:1273}) (resp. scalar conservation laws \eqref{eq:scalar}), if the solution at time level $t^n$ satisfies \eqref{eq:1526} (resp. \eqref{eq:1275}), then the $\mathbb{P}^2$-based scheme \eqref{eq:1182} (resp. \eqref{eq:1033}) with the LF flux and viscosity coefficient $\alpha$ satisfying \eqref{eq:1250} (resp. $\alpha \ge \alpha^{\rm LF}$) is BP under the optimal BP CFL condition
    \begin{equation}\label{eq:1713}
	\alpha \frac{\dt}{|K|} \le \mathcal{C}^{\tt DCW}_{K,2} \quad \forall K \in \mathcal{T}_h, \quad
	\mathcal{C}^{\tt DCW}_{K,2} := \frac{2}{9\bar l_K + 3 \hat l_K},
	\end{equation}
    where $\bar l_K$ denotes the average edge length of the triangular cell $K$ and $\hat l_K$ is defined in \eqref{eq:1703}. This BP CFL condition \eqref{eq:1713} cannot be improved by selecting another convex decomposition.
\end{corollary}

\begin{remark}
    Based on \Cref{thm:1355}, one may choose $\alpha = \alpha^{\rm LF}$ for the 2D Euler system \eqref{eq:euler} in Corollaries \ref{cor:1837} and \ref{cor:1859}.
\end{remark}

\begin{remark}
	Based on the classical convex decomposition, Zhang, Xia, and Shu \cite{ZXSPP2012} proposed a $\mathbb{P}^2$-based BP scheme for the 2D Euler system \eqref{eq:euler} under a more restrictive BP CFL condition (see the comparison in \Cref{fig:1744b}):
	\[
	\alpha^{\rm LF} \frac{\dt}{|K|} \le 
	\mathcal{C}^{\tt ZXS}_{K,2} \quad \forall K \in \mathcal{T}_h, \quad
	\mathcal{C}^{\tt ZXS}_{K,2} := \frac{1}{27 \bar l_K}.
	\]
\end{remark}

\begin{remark}
Following our BP analysis in \Cref{thm:1008}, we can also use Chen and Shu's quadrature from \cite[Table C.1]{ChenShu2017} as a feasible convex decomposition for the $\mathbb{P}^2$ space. The resulting $\mathbb{P}^2$-based scheme for the 2D Euler system \eqref{eq:euler} is BP under the following CFL condition:
\begin{align}\label{eq:P2chenshu}
    \alpha^{\rm LF} \frac{\dt}{|K|} \le \mathcal{C}^{\tt CS}_{K,2} \quad \forall K \in \mathcal{T}_h, \quad \mathcal{C}^{\tt CS}_{K,2} := \frac{1}{6 l_K^{(1)}}.
\end{align}
Clearly, our optimal BP CFL condition \eqref{eq:1713} is milder than \eqref{eq:P2chenshu}.
\end{remark}

\begin{remark}
\Cref{tab:2019} provides a comparison of BP CFL numbers induced by three different feasible convex decompositions for $\mathbb{P}^1$ and $\mathbb{P}^2$ spaces on two specific triangular cells. As shown, the BP CFL number derived from the optimal convex decomposition is the largest among them. Indeed, our optimal convex decomposition consistently yields the highest BP CFL number for any triangular cell; see the comparison in \Cref{fig:1744}.

    \begin{table}[h]
    \centering
    \caption{Compare BP CFL numbers and total numbers of internal nodes corresponding to various feasible convex decompositions for $\mathbb{P}^1$ and $\mathbb{P}^2$ spaces on two example triangular cells.}\label{tab:2019}
    \renewcommand\arraystretch{1.7}
    \begin{tabular}{clrcrc} 
     \toprule[1.5pt]
     \multirow{3}{*}{example cell $K$}
     & \multirow{3}{*}{convex decomposition}
     & \multicolumn{2}{c}{$\mathbb{P}^1$ space}
     & \multicolumn{2}{c}{$\mathbb{P}^2$ space}
     \\
     \cmidrule(r){3-4} \cmidrule(l){5-6} 
     && BP CFL & internal
     & BP CFL & internal\\
     && number & nodes
     & number & nodes\\
     \midrule[1.5pt]
     \multirow{3}{*}{\includegraphics[trim = 20 0 20 0, clip, width = 0.18\textwidth]{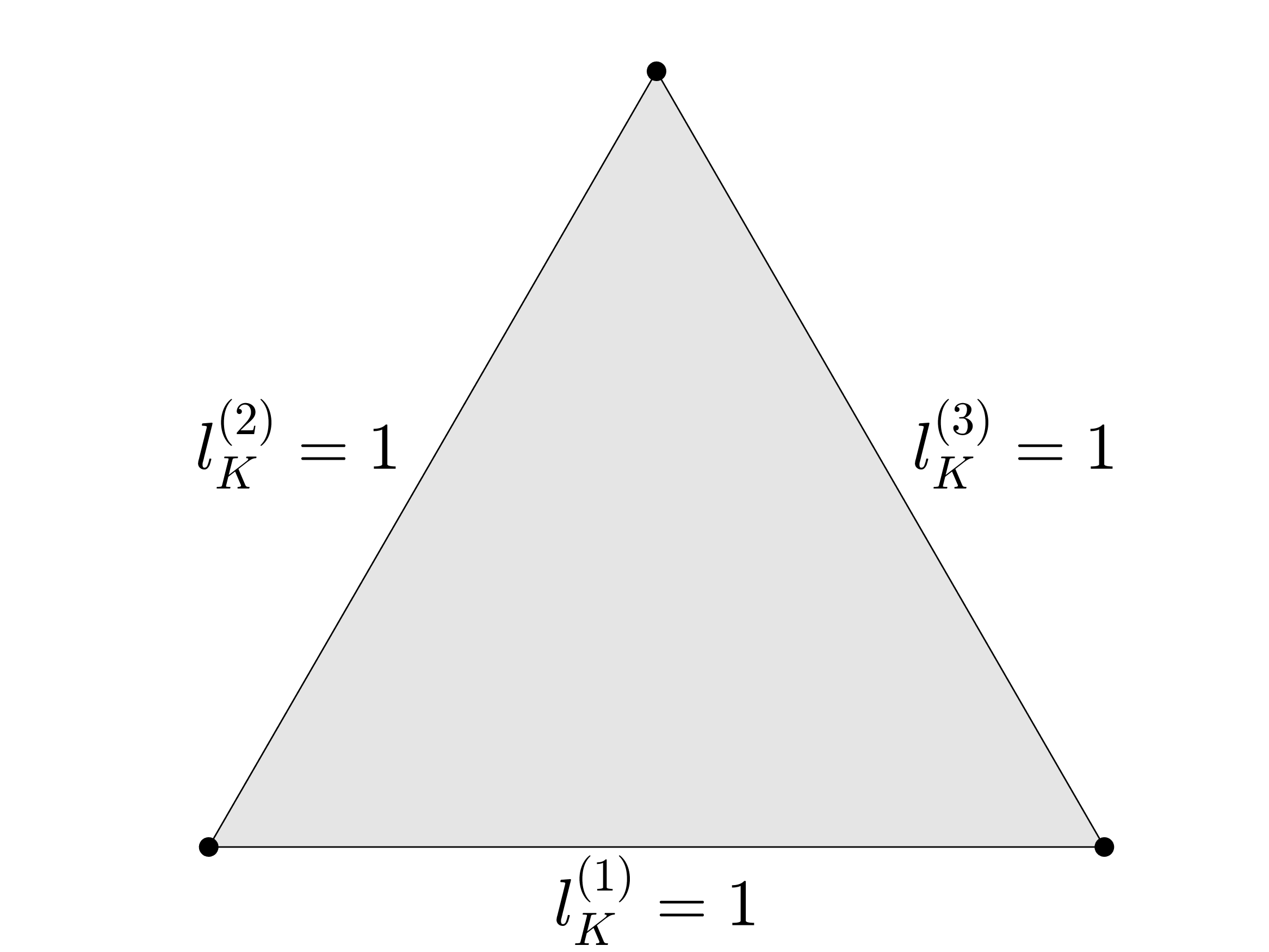}}
     & Optimal 
     & $\frac13 \approx 0.3333$ & 0 
     & $\frac16 \approx 0.1667$ & 1 \\ 
     & Zhang, Xia, Shu \cite{ZXSPP2012} & $\frac19 \approx 0.1111$ & 0 & $\frac{1}{27} \approx 0.0370$ & 9\\ 
     & Chen, Shu \cite{ChenShu2017} & $\frac16 \approx 0.1667$ & 6 & $\frac16 \approx 0.1667$ & 10 \\ 
    \midrule[1.5pt]
    \multirow{3}{*}{\includegraphics[trim = 10 0 20 0, clip, width = 0.18\textwidth]{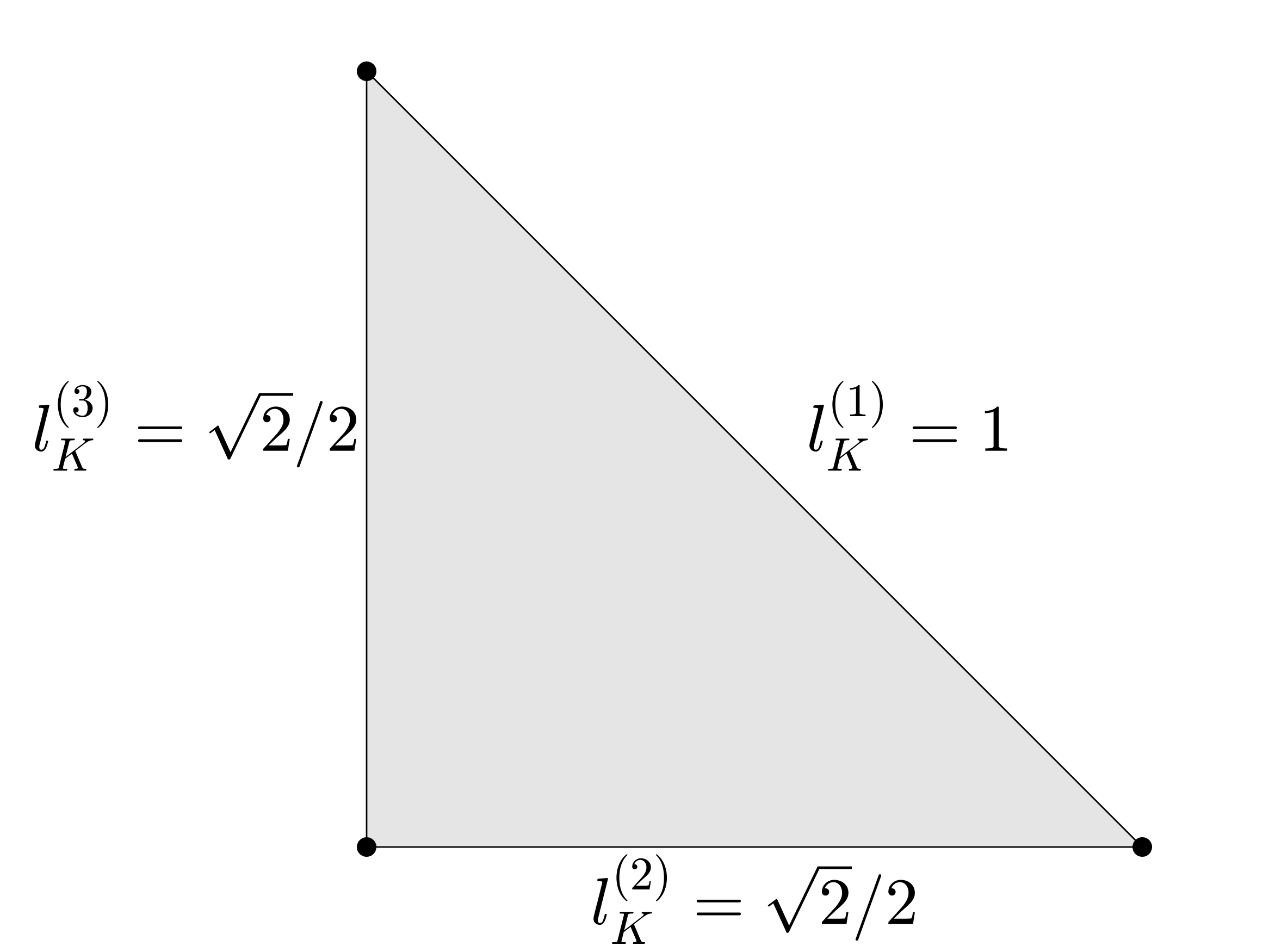}}
     & Optimal 
     & $\frac{2}{3(1+\sqrt{2}/2)} \approx 0.3905$ & 1 
     & $\frac{2}{3\sqrt{2}+\sqrt{15-6\sqrt{2}}+3} \approx 0.2042$ & 2 \\ 
     & Zhang, Xia, Shu \cite{ZXSPP2012} 
     & $\frac{1}{3(1+\sqrt{2})} \approx 0.1381$ & 0 
     & $\frac{1}{9(1+\sqrt{2})} \approx 0.0460$ & 9\\ 
     & Chen, Shu \cite{ChenShu2017} 
     & $\frac16\approx 0.1667$ & 6 
     & $\frac16\approx 0.1667$ & 10 \\ 
    \bottomrule[1.5pt]
    \end{tabular}
    \end{table}    
\end{remark}

\begin{figure}
	\centerline{
		\begin{subfigure}{\textwidth}
		\includegraphics[width=0.33\textwidth]{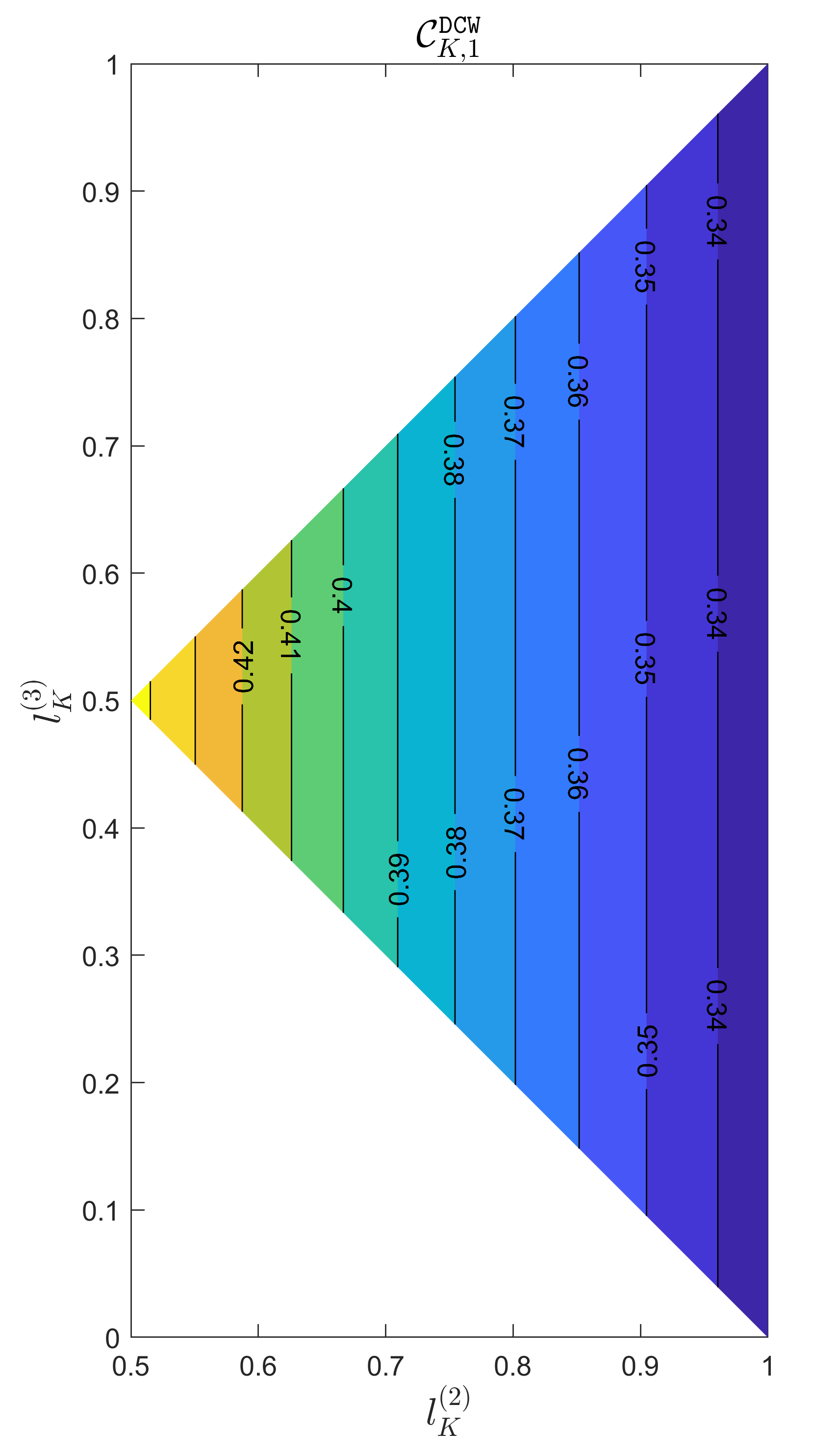}
		\includegraphics[width=0.33\textwidth]{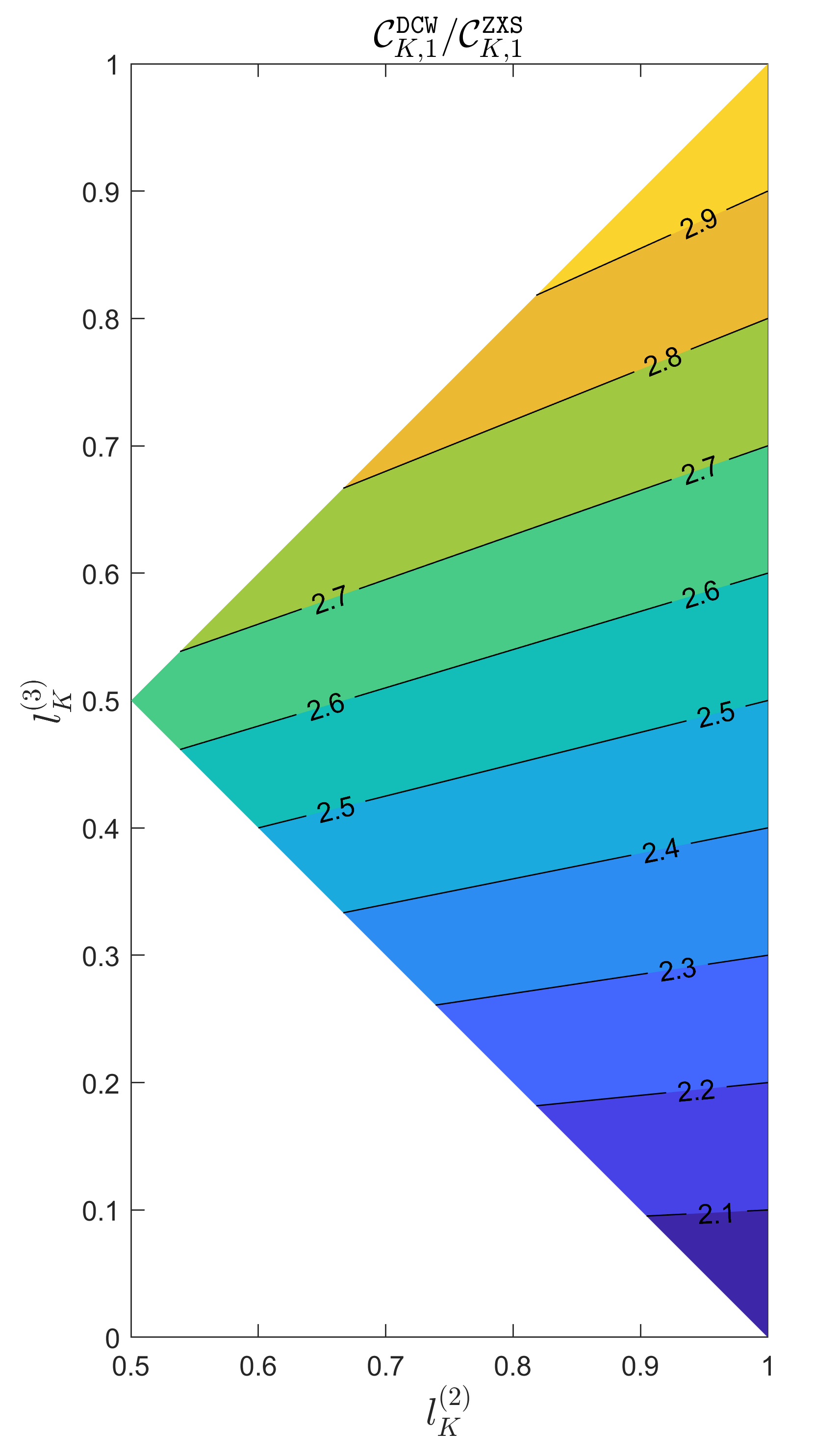}
            \includegraphics[width=0.33\textwidth]{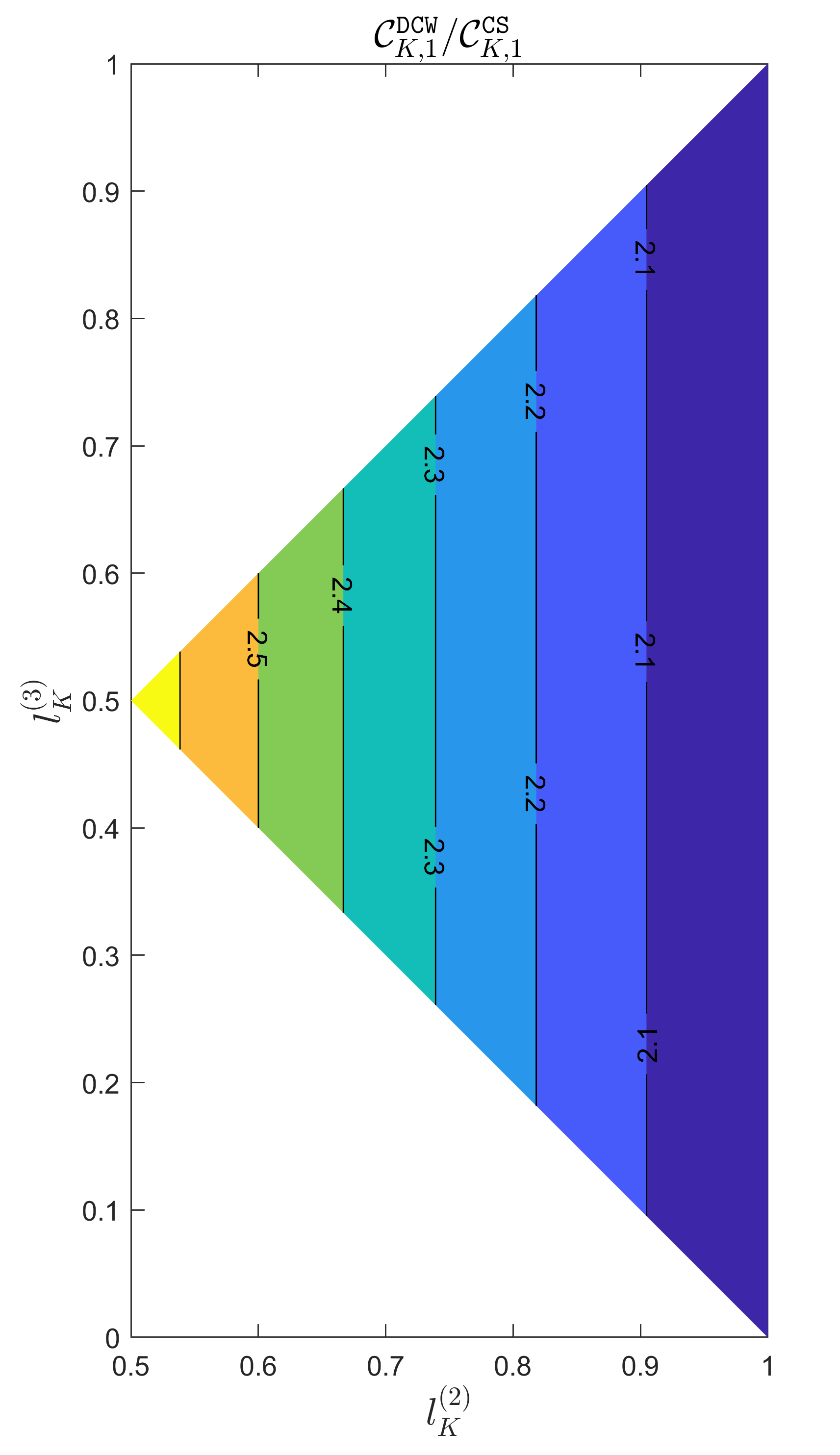}
		\caption{$\mathcal{C}^{\tt DCW}_{K,1}$, $\mathcal{C}^{\tt DCW}_{K,1}/\mathcal{C}^{\tt ZXS}_{K,1} \in [2,3]$, and $\mathcal{C}^{\tt DCW}_{K,1}/\mathcal{C}^{\tt CS}_{K,1} \in [2,2.6667]$.}
		\label{fig:1744a}
		\end{subfigure}
	}
	\centerline{
	\begin{subfigure}{\textwidth}
		\includegraphics[width=0.33\textwidth]{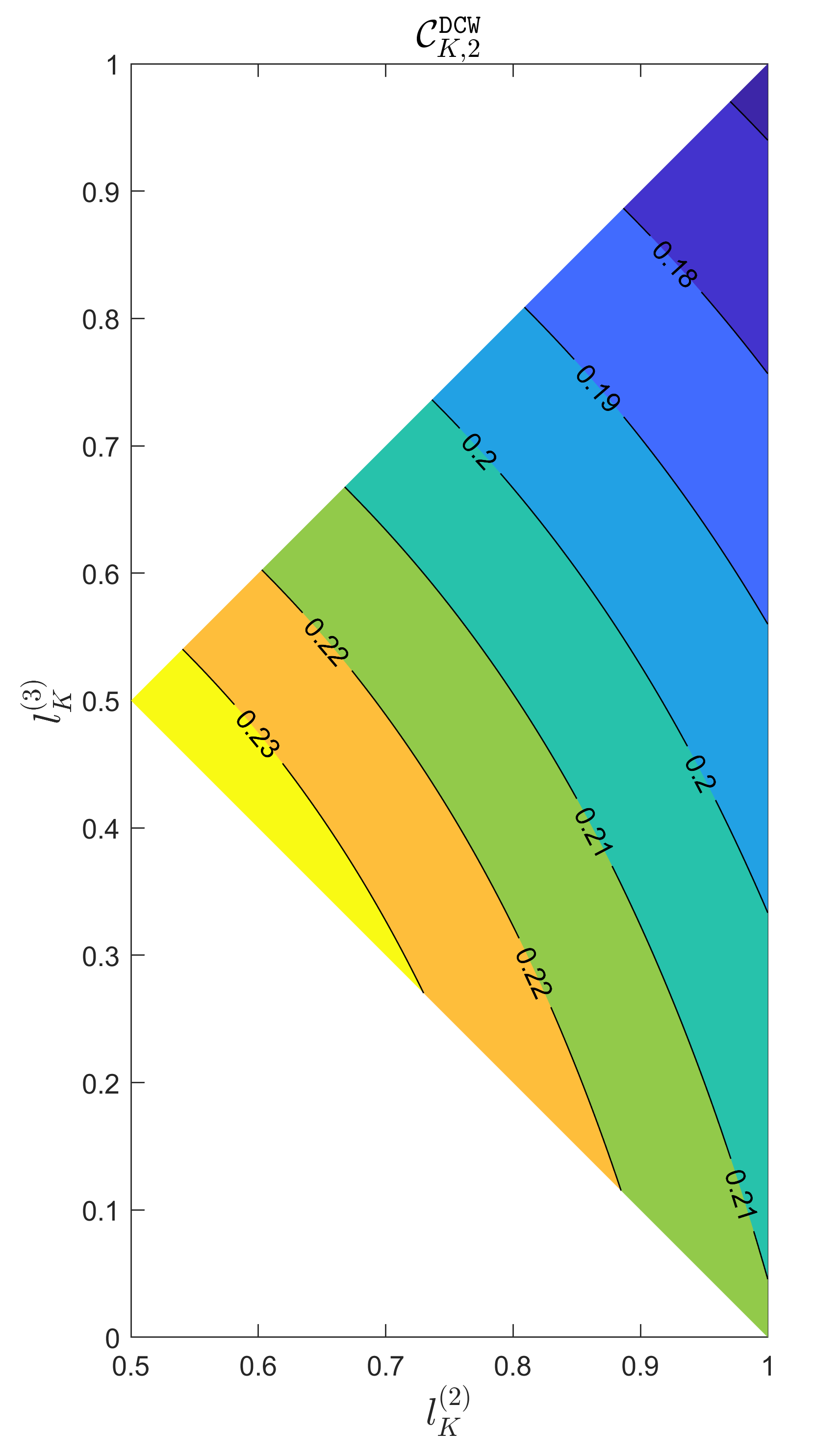}
		\includegraphics[width=0.33\textwidth]{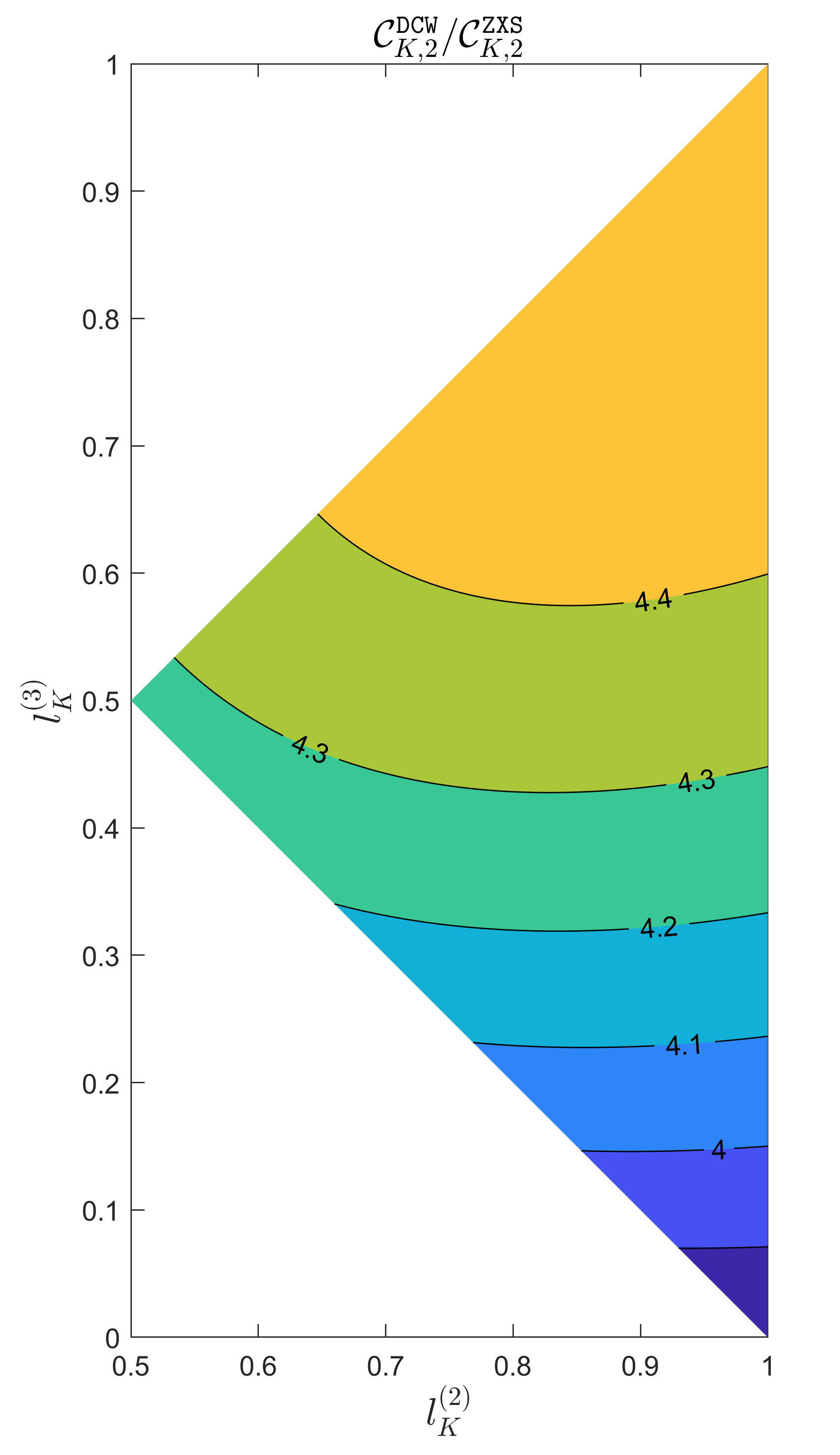}
            \includegraphics[width=0.33\textwidth]{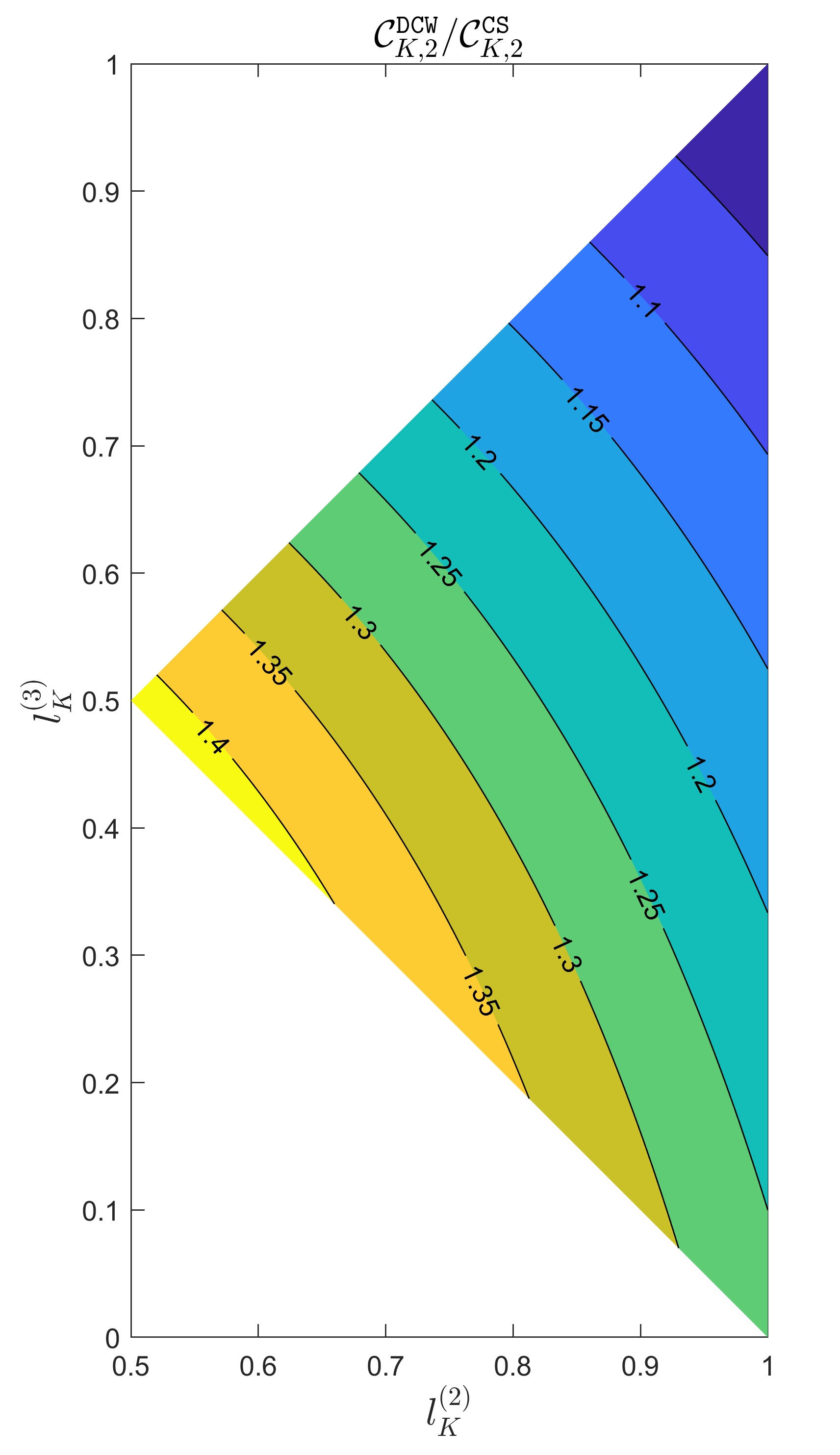}
		\caption{$\mathcal{C}^{\tt DCW}_{K,2}$, $\mathcal{C}^{\tt DCW}_{K,2}/\mathcal{C}^{\tt ZXS}_{K,2} \in [3.8038,4.5]$, $\mathcal{C}^{\tt DCW}_{K,2}/\mathcal{C}^{\tt CS}_{K,2} \in [1,1.4202]$.}
		\label{fig:1744b}
	\end{subfigure}
	}
	\caption{Comparison of BP CFL numbers $\mathcal{C}^{\tt DCW}_{K,k}$, $\mathcal{C}^{\tt ZXS}_{K,k}$, and $\mathcal{C}^{\tt CS}_{K,k}$ on general triangular cells $K$, $k=1,2$. Without loss of generality, assume $1 = l^{(1)}_K\ge l^{(2)}_K\ge l^{(3)}_K$ and $l^{(1)}_K\le l^{(2)}_K + l^{(3)}_K$. Please note that $\mathcal{C}^{\tt CS}_{K,1} = \mathcal{C}^{\tt CS}_{K,2} = 1/6$ in this case.}
	\label{fig:1744}
\end{figure}

\subsection{BP Limiter}\label{sec:BPlimiter}

To ensure the provable BP property via Theorem \ref{thm:1048} or \ref{thm:1008} and Corollary  \ref{cor:1837} or \ref{cor:1859}, it is crucial that the solution at each time level $t^n$ satisfies the conditions outlined in \eqref{eq:1275} or \eqref{eq:1526}. However, solutions produced by DG schemes do not typically satisfy these conditions inherently. To address this, a BP limiter can be employed to adjust any solution that violates these conditions, transforming it into one that satisfies them. By applying the BP limiter at the end of each RK stage, the conditions in \eqref{eq:1275} or \eqref{eq:1526} are ensured, thereby completing the design of a high-order BP scheme.

It is important to note that the implementation of the BP limiter for the optimal convex decomposition is straightforward and requires only minor modifications from the BP limiter described in \cite{zhang2011b} for the classical convex decomposition. To clearly illustrate both the similarities and differences between these limiters, we will use the 2D Euler system \eqref{eq:euler} as an example with ${\mathcal G}$ defined in \eqref{G:Euler},  
 presenting these limiters side-by-side in the following discussion.

On each triangular cell $K$, the simplified BP limiter modifies the polynomial solution $\bm u_K(\xbm) := \big(\rho_K(\xbm), ({m_1})_K(\xbm),$ $ ({m_2})_K(\xbm), E_K(\xbm)\big)^\top$ to a new polynomial $\hat{\bm u}_K(\xbm)$ in such a way that $\hat{\bm u}_K(\xbm)$ satisfies the desired BP conditions. This modification is performed in two steps:
\begin{itemize}
    \item[\textbf{Step 1}:] Enforce the positivity of the density by modifying the density $\rho_K(\xbm)$ as follows:
    \[
    \tilde{\rho}_K(\xbm) = \theta_1\Big(\rho_K(\xbm) - \bar{\rho}_K\Big) + \bar{\rho}_K, \qquad 
    \theta_1 = \min\left\{\frac{\bar{\rho}_K-\varepsilon_1}{\bar{\rho}_K-\rho_{\rm min}},1\right\},
    \]
    where $\varepsilon_1=\min\{\bar \rho_K,10^{-13}\}$. 
    If the classical convex decomposition \cite{ZXSPP2012} is employed, then 
    \[
        \rho_{\rm min} := 
        \begin{cases}
            \check\rho, & \text{ if } k=1,\\
            \min\left\{\check\rho, \rho_*\right\}, & \text{ if } k=2,\\
        \end{cases}
    \]
    where
    \[
        \check \rho
        =
        \underset{1\leq i\leq 3, 1\leq \nu \leq Q}{\min} \,
        \rho_K(x^{(i),\nu}_K, y^{(i),\nu}_K), 
    \qquad 
        \rho_*
        = 
        \frac
        {\bar{\rho}_K-\sum_{i=1}^{3}w_i \sum_{\nu=1}^{Q} \omega^{\tt G}_{\nu}\, \rho_K(x^{(i),\nu}_K, y^{(i),\nu}_K)}
        {1-\sum_{i=1}^{3} w_i}.
    \]
    If our optimal convex decomposition is employed, then 
    \begin{equation}\label{eq:1944}
        \rho_{\rm min} := 
        \begin{cases}
            \min\left\{\check\rho, \rho(x_1,y_1), \rho(x_2,y_2) \right\}, & \text{ if } k=1,\\
            \min\left\{\check\rho, \rho_*\right\}, & \text{ if } k=2,\\
        \end{cases}
    \end{equation}
    where ${\bm v}_1 = (x_1,y_1)$ and ${\bm v}_2 = (x_2,y_2)$ are the vertices of $K$ opposite the longest and the second longest edges, respectively.

    \item[\textbf{Step 2}:] Enforce the positivity of the specific internal energy $\mathcal{E}$. Let $\tilde{\bm u}_K(\xbm):=\big(\tilde{\rho}_K(\xbm), ({m_1})_K(\xbm),({m_2})_K(\xbm), E_K(\xbm)\big)^\top$ be the result from Step 1. Further modify $\tilde{\bm u}_K(\xbm)$ as follows:
    \[
    \hat{\bm u}_K(\xbm) = \theta_2\Big(\tilde{\bm u}_K(\xbm) - \bar{\bm u}_K\Big) + \bar{\bm u}_K, \qquad 
    \theta_2 = \min\left\{\frac{\mathcal{E}(\bar{\bm u}_K)-\varepsilon_2}{\mathcal{E}(\bar{\bm u}_K)-\mathcal{E}_{\rm min}},1\right\},
    \]
    where $\varepsilon_2=\min\{\mathcal{E}(\bar{\bm u}_K),10^{-13}\}$. 
    If the classical convex decomposition \cite{ZXSPP2012} is employed, then 
    \[
        \mathcal{E}_{\rm min} = 
        \begin{cases}
            \check{\mathcal{E}}, & \text{ if } k=1,\\
            \min\{\check{\mathcal{E}}, \mathcal{E}_*\}, & \text{ if } k=2,\\
        \end{cases}
    \]
    where
    \[
    \check{\mathcal{E}}
    = 
    \underset{1\leq i\leq 3, 1\leq \nu \leq Q}{\min}
    {\mathcal{E}}(\tilde{\bm u}_K(x^{(i),\nu}_K, y^{(i),\nu}_K)), 
    \qquad
    {\mathcal{E}}_*
    = 
    \frac{\mathcal{E}(\bar{\bm u}_K)-\sum_{i=1}^{3}w_i \sum_{\nu=1}^{Q}\omega^{\tt G}_{\nu} \, \mathcal{E}(\tilde{\bm u}_K(x^{(i),\nu}_K, y^{(i),\nu}_K))}{1-\sum_{i=1}^{3} w_i}.
    \]
    If our optimal convex decomposition is employed, then 
    \begin{equation}\label{eq:1981}
        \mathcal{E}_{\rm min} = 
        \begin{cases}
            \min\{\check{\mathcal{E}}, \mathcal{E}(\tilde{\bm u}_K(x_1,y_1)), \mathcal{E}(\tilde{\bm u}_K(x_2,y_2))\}, & \text{ if } k=1,\\
            \min\{\check{\mathcal{E}}, \mathcal{E}_*\}, & \text{ if } k=2,\\
        \end{cases}
    \end{equation}
    where ${\bm v}_1 = (x_1,y_1)$ and ${\bm v}_2 = (x_2,y_2)$ are defined as in Step 1.
\end{itemize}

\begin{remark}
    When employing the optimal convex decomposition for the $\mathbb{P}^1$ space, one might consider using $\rho_{\min} := \min\{\check\rho, \rho_*\}$ (resp. $\mathcal{E}_{\min} := \min\{\check{\mathcal{E}}, \mathcal{E}_*\}$), which aligns with the formula used for the $k=2$ case instead of the formulas in the first line of \eqref{eq:1944} (resp. \eqref{eq:1981}).
    However, this approach can be problematic because the expressions for $\rho_*$ (resp. $\mathcal{E}_*$) may become singular due to division by nearly zero, particularly when $k=1$ and $l^{(1)}_K = l^{(2)}_K = l^{(3)}_K$, which occurs when $K$ is an equilateral triangle.
    Such singularities can lead to large numerical errors and loss of the BP property when $K$ is close to being an equilateral triangle.
    To address this issue, it is noted that $\rho_*$ (resp. $\mathcal{E}_*$) can be equivalently expressed as $\rho_* = \rho_K(\xi_1,\eta_1)$ (resp. $\mathcal{E}_* = \mathcal{E}(\tilde{\bm u}_K(\xi_1,\eta_1))$), where $(\xi_1,\eta_1)$ is the internal node of the optimal convex decomposition \eqref{eq:1040}, which is always located on the edge $e_K^{(3)}$. Therefore, it holds that
    \[
        \rho_* = \rho_K(\xi_1,\eta_1) 
        \ge 
        \min \{\rho(x_1,y_1), \rho(x_2,y_2)\}
        \quad \text{and} \quad
        \mathcal{E}_* = \mathcal{E}(\tilde{\bm u}_K(\xi_1,\eta_1))
        \ge
        \min \{\mathcal{E}(\tilde{\bm u}_K(x_1,y_1)), \mathcal{E}(\tilde{\bm u}_K(x_2,y_2))\},
    \]
    leading to the novel formulas in the first line of \eqref{eq:1944} (resp. \eqref{eq:1981}), which are more robust for all possible shapes of the triangular cell $K$.
\end{remark}

\subsection{BP OEDG Schemes}

The BP techniques introduced above can be applied to the $\mathbb{P}^1$- and $\mathbb{P}^2$-based DG and finite volume methods, including the OEDG methods presented in \Cref{sec:scalar} and \Cref{sec:system}. For instance, using the third-order SSP RK method \eqref{RKTime}, the resulting fully discrete BP OEDG method for the hyperbolic system \eqref{eq:System} is given by
\begin{equation*}
	\begin{aligned}
		\bm{u}_h^{[1]} &= \bm{u}_{\sigma}^{n} + \Delta t \mathcal{L}(\bm{u}_{\sigma}^{n}), & 
		\hat{\bm{u}}_{\sigma}^{[1]} &= \mathcal{F}_{\tau}\bm{u}_h^{[1]}, &
        \bm{u}_{\sigma}^{[1]} &=
        \mathcal{B}_{\mathcal{G}} \, \hat{\bm{u}}_{\sigma}^{[1]},
        \\
		\bm{u}_h^{[2]} &= \frac{3}{4}\bm{u}_{\sigma}^{n} + \frac{1}{4} \left(\bm{u}_{\sigma}^{[1]} + \Delta t \mathcal{L}(\bm{u}_{\sigma}^{[1]}) \right), &
		\hat{\bm{u}}_{\sigma}^{[2]} &= \mathcal{F}_{\tau}\bm{u}_h^{[2]}, &
        \bm{u}_{\sigma}^{[2]} &=
        \mathcal{B}_{\mathcal{G}} \, \hat{\bm{u}}_{\sigma}^{[2]},\\
		\bm{u}_h^{n+1} &= \frac{1}{3}\bm{u}_{\sigma}^{n} + \frac{2}{3} \left(\bm{u}_{\sigma}^{[2]} + \Delta t \mathcal{L}(\bm{u}_{\sigma}^{[2]}) \right), & 
		\hat{\bm{u}}_{\sigma}^{n+1} &= \mathcal{F}_{\tau}\bm{u}_h^{n+1},&
        \bm{u}_{\sigma}^{n+1} &=
        \mathcal{B}_{\mathcal{G}} \, \hat{\bm{u}}_{\sigma}^{n+1},
	\end{aligned}
\end{equation*}
where $\mathcal{B}_{\mathcal{G}}$ denotes the simplified BP limiter introduced in \Cref{sec:BPlimiter}, and the timestep $\Delta t$ is determined by the conditions \eqref{eq:1692} for $\mathbb{P}^1$-based methods or \eqref{eq:1713} for $\mathbb{P}^2$-based methods with the optimal convex decomposition.

\section{Numerical Experiments}\label{sec:numexp}

In this section, we present a series of numerical experiments to validate the accuracy, robustness, and efficiency of the proposed OEDG methods on unstructured triangular meshes. We utilize the DG method with the proposed OE procedure. When the RIOE procedure $\mathcal{F}^{\rm RI}_{\tau}$ is applied to a rotationally invariant (RI) physical system, such as the Euler equations, we refer to the method as the \textbf{RI-OEDG} method. 
In all examples, we employ the LF flux \eqref{fluxLF} and the $n_{\rm s}$-stage $n_{\rm o}$th-order strong-stability-preserving Runge-Kutta (SSPRK) time discretization, denoted as {\tt SSP-RK}$(n_{\rm s},n_{\rm o})$ for convenience. 
Specifically, except for Examples \ref{Ex:LinSmooth} and \ref{Ex:BurgS} for accuracy tests, all other examples use the three-stage third-order SSPRK method, i.e., {\tt SSP-RK}$(3,3)$. 
The unstructured triangular meshes are generated using EasyMesh \cite{Niceno2002easymesh}, and we denote by $N$ the total number of cells and by $h$ the uniform edge length of the cells along the domain boundary.

In the following Examples \ref{Ex:LinSmooth}--\ref{Ex:EulerFPC}, the time step size $\Delta t$ is determined by
$$
    \Delta t
    =
    \frac{C_{\tt SSP}}{\alpha^{\rm LF}} 
    \underset{K \in \mathcal{T}_h}{\min} 
    \left\{
        \frac{|K|}{2k+1}
        \left(\sum_{i=1}^{3}l_K^{(i)}\right)^{-1}
    \right\},
$$
where $C_{\tt SSP}$ is the SSP coefficient of the employed {\tt SSP-RK} time discretization method. 
In Examples \ref{Ex:EulerFFS}--\ref{Ex:EulerSDiffWedge}, we focus on the 2D Euler system \eqref{eq:euler} with the solutions involving states close to vacuum, making the BP technique essential. For comparison, we use the $\mathbb{P}^k$-based RI-OEDG method ($k \in \{1, 2\}$) with either the classical convex decomposition on triangular cells from \cite{ZXSPP2012} or our optimal convex decomposition on triangular cells from \Cref{sec:OCD}. To ensure the BP property, the time step sizes for these two methods are determined by
$$
    \Delta t = \frac{C_{\tt SSP}}{\alpha^{\rm LF}} \underset{K \in \mathcal{T}_h}{\min} 
    \left\{\,
    \mathcal{C}^{\tt ZXS}_{K,k} \, |K|
    \,\right\},
$$
and
$$
    \Delta t = \frac{C_{\tt SSP}}{\alpha^{\rm LF}} \underset{K \in \mathcal{T}_h}{\min} 
    \left\{\,
    \mathcal{C}^{\tt DCW}_{K,k} \, |K|
    \,\right\},
$$
respectively. 
In these examples, we also compare the CPU times required by different numerical methods. The implementations are carried out in C++ using double precision and are executed on a Linux server equipped with an Intel(R) Xeon(R) Platinum 8458P CPU running at 3.3GHz, with 628GB of RAM.

\subsection{Linear Convection Equation}\label{Ex:Linear}
In Examples \ref{Ex:LinSmooth}--\ref{Ex:LinDis}, we consider the linear convection equation $u_t + u_x + u_y = 0$ over a square domain with periodic boundary conditions.

\begin{expl}[Smooth problem]\label{Ex:LinSmooth}
	This example considers a smooth initial condition $u(x,y,0) = \sin(2\pi(x+y))$ on the square domain $[0,1]^2$. The exact solution is given by $u(x,y,t) = \sin(2\pi(x+y-2t))$. The final time is set to $t = 0.1$. 
	We start by generating an unstructured mesh $\mathcal{T}_{h_0}$ with $N_0 = 44$ cells and a uniform edge length of $h_0 = 0.25$; see \Cref{fig:Ex_Linmesh_a}. We then refine this mesh $\mathcal{T}_{h_1}$ by halving the edge length and dividing each cell into four smaller cells, resulting in a finer mesh $\mathcal{T}_{h_1}$ with $N_1 = 4N_0 = 176$ cells; see \Cref{fig:Ex_Linmesh_b}, where the black lines are the cell edges in the mesh $\mathcal{T}_{h_0}$, and the  dashed blue lines are the cell edges newly added in the refined mesh $\mathcal{T}_{h_1}$. This refinement process is repeated to generate a sequence of meshes $\{\mathcal{T}_{h_i}\}_{i=1}^6$. 
	\Cref{tab:Ex_LinearS1} presents the $L^1$, $L^2$, and $L^\infty$ errors of the numerical solutions and the corresponding convergence rates obtained using the $\mathbb{P}^k$-based OEDG schemes for $k \in \{1, 2, 3, 4\}$. As shown, the expected $(k+1)$th-order accuracy is achieved.

	\begin{figure}[!htb]
		\centering
		\begin{subfigure}{0.32\textwidth}
			\includegraphics[width=\textwidth]{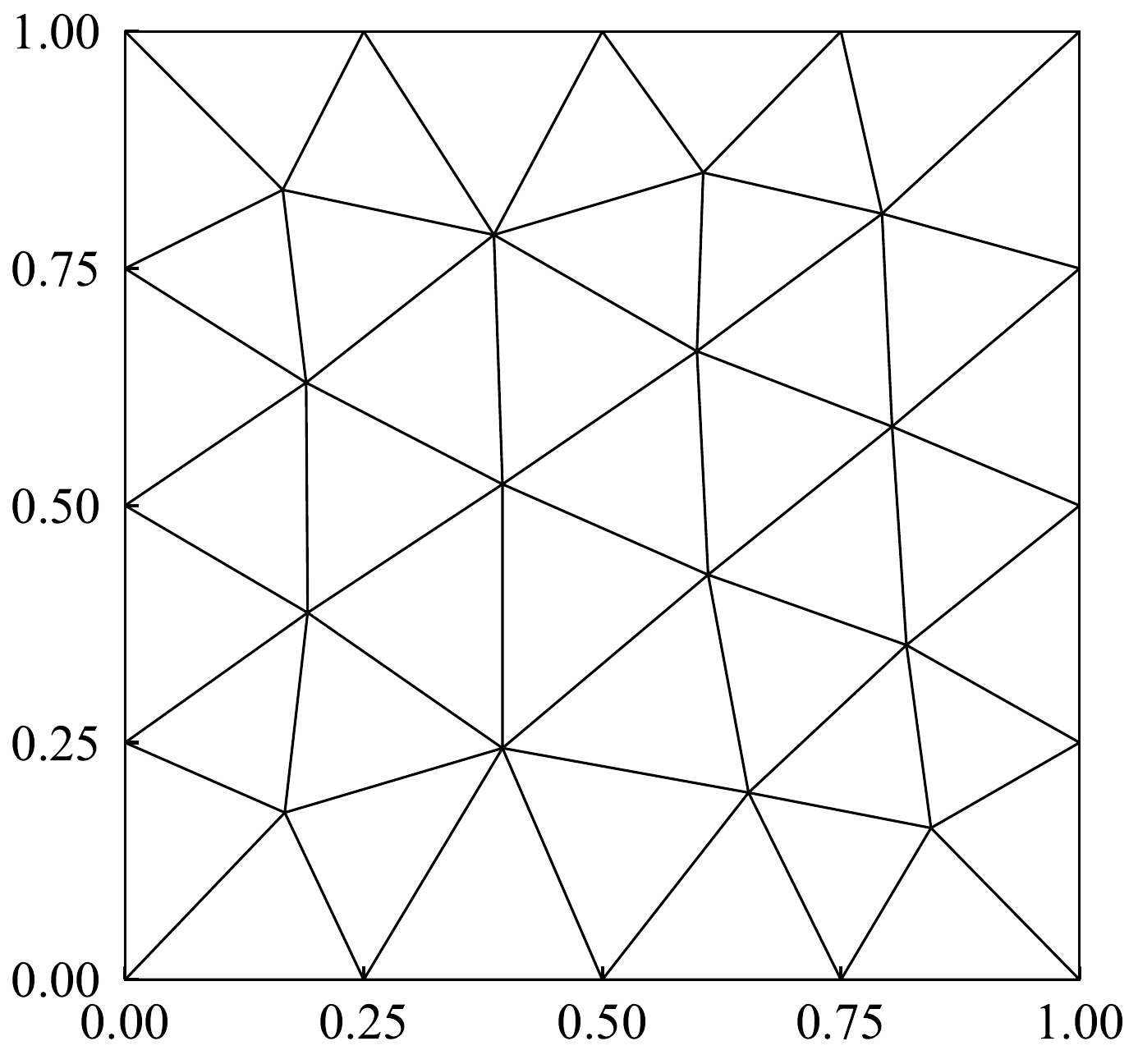}
			\caption{The mesh $\mathcal{T}_{h_0}$ with $N_0=44$ and $h_0=0.25$.}
			\label{fig:Ex_Linmesh_a}
		\end{subfigure}
		\qquad
		\begin{subfigure}{0.32\textwidth}
			\includegraphics[width=\textwidth]{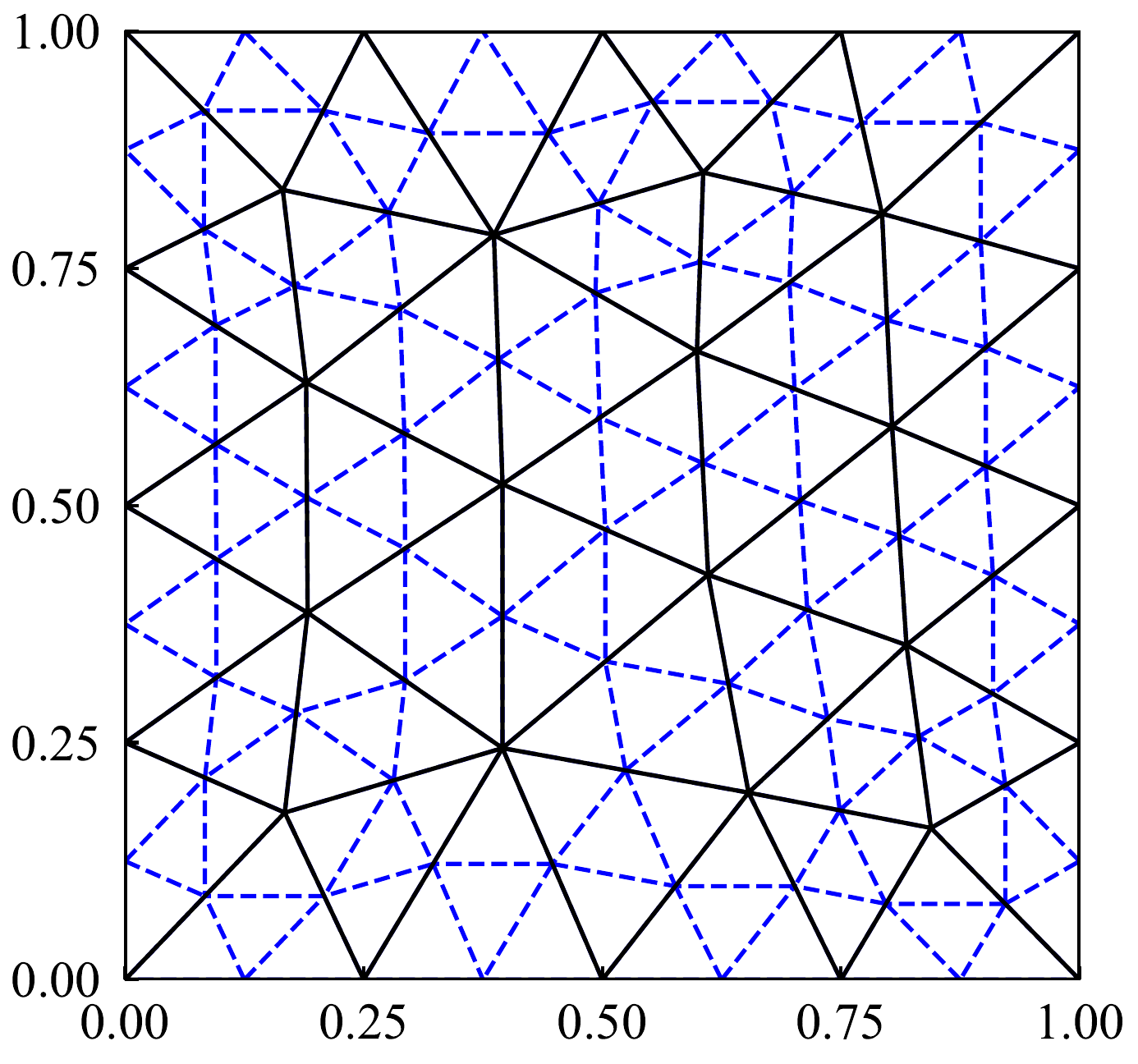}
			\caption{The mesh $\mathcal{T}_{h_1}$ with $N_1=176$ and $h_1=0.125$.}
			\label{fig:Ex_Linmesh_b}
		\end{subfigure}
		\caption{(\Cref{Ex:LinSmooth}) The meshes $\mathcal{T}_{h_0}$ and $\mathcal{T}_{h_1}$. 
		}
		\label{fig:Ex_Linmesh}
	\end{figure}
	
	\begin{table}[!htb] 
		\centering
		\caption{(\Cref{Ex:LinSmooth}) $L^1$, $L^2$, and $L^\infty$ errors and convergence orders, $\mathbb{P}^k$-based OEDG schemes, $k \in \{1,2,3,4\}$, $t=0.1$.
		}
		\label{tab:Ex_LinearS1}
		\setlength{\tabcolsep}{3mm}{
			\begin{tabular}{cccccccc}
				\toprule[1.5pt]
				& \multirow{2}{*}{$N$} & 
				\multicolumn{2}{c}{$L^1$ norm} & 
				\multicolumn{2}{c}{$L^2$ norm} & 
				\multicolumn{2}{c}{$L^\infty$ norm} \\
				\cmidrule(l){3-4} \cmidrule(l){5-6} \cmidrule(l){7-8} 
				& & error & order & error & order &  error & order\\
				
				\midrule[1.5pt]					
				\multirow{4}{*}{$\mathbb{P}^1$-based OEDG}
				&176   &8.18e{-}2&{-} &9.35e{-}2&{-} &2.53e{-}1&{-} \\%
				&704   &1.56e{-}2&2.39&1.91e{-}2&2.29&7.38e{-}2&1.78\\%
				\multirow{2}{*}{ {\tt SSP-RK}(2,2)}
				&2816  &2.79e{-}3&2.49&3.44e{-}3&2.47&1.68e{-}2&2.13\\%
				\multirow{2}{*}{ $C_{\tt SSP} = 1 $ } 
				&11264 &4.76e{-}4&2.55&6.01e{-}4&2.52&3.90e{-}3&2.11\\%
				&45056 &9.11e{-}5&2.39&1.21e{-}4&2.31&9.18e{-}4&2.09\\%
				&180224&2.01e{-}5&2.18&2.78e{-}5&2.12&1.96e{-}4&2.23\\%

				\midrule[1.0pt]	
				\multirow{4}{*}{$\mathbb{P}^2$-based OEDG}
				&176   &5.29e{-}2&{-} &6.25e{-}2&{-} &2.22e{-}1&{-} \\%
				&704   &2.90e{-}3&4.19&3.87e{-}3&4.01&1.89e{-}2&3.56\\%
				\multirow{2}{*}{ {\tt  SSP-RK}(3,3)} 
				&2816  &9.15e{-}5&4.99&1.22e{-}4&4.99&1.04e{-}3&4.18\\%
				\multirow{2}{*}{ $C_{\tt SSP} = 1 $ } 
				&11264 &5.70e{-}6&4.01&7.99e{-}6&3.93&1.06e{-}4&3.30\\%
				&45056 &4.85e{-}7&3.55&7.20e{-}7&3.47&8.26e{-}6&3.68\\%
				&180224&5.10e{-}8&3.25&7.79e{-}8&3.21&7.21e{-}7&3.52\\%

				\midrule[1.0pt]	
				\multirow{4}{*}{$\mathbb{P}^3$-based OEDG} 
				&176   &4.45e{-}2 &{-} &5.38e{-}2 &{-} &1.72e{-}1&{-} \\%
				&704   &3.20e{-}4 &7.12&4.96e{-}4 &6.76&3.85e{-}3&5.48\\%
				\multirow{2}{*}{ {\tt SSP-RK}(5,4)}
				&2816  &3.01e{-}6 &6.73&4.12e{-}6 &6.91&2.85e{-}5&7.08\\%
				\multirow{2}{*}{ $C_{\tt SSP} = 1.508$ } 
				&11264 &8.89e{-}8 &5.08&1.20e{-}7 &5.10&9.93e{-}7&4.84\\%
				&45056 &3.08e{-}9 &4.85&4.13e{-}9 &4.86&3.80e{-}8&4.71\\%
				&180224&1.23e{-}10&4.65&1.74e{-}10&4.57&1.69e{-}9&4.49\\%

				\midrule[1.0pt]	
				\multirow{3}{*}{$\mathbb{P}^4$-based OEDG} 
				&176   &6.34e{-}3 &{-}  &9.50e{-}3 &{-}  &6.24e{-}2 &{-}  \\%
				&704   &2.19e{-}6 &11.50&3.23e{-}6 &11.52&3.90e{-}5 &10.64\\%
				\multirow{1}{*}{ {\tt SSP-RK}(5,4)}
				&2816  &3.00e{-}8 &6.19 &4.15e{-}8 &6.28 &3.07e{-}7 &6.99 \\%
				\multirow{1}{*}{ $C_{\tt SSP} = 1.508$ } 
				&11264 &5.41e{-}10&5.80 &7.94e{-}10&5.71 &7.17e{-}9 &5.42 \\%
				\multirow{1}{*}{ $\dt= \frac{C_{\tt SSP}}{9}\Big(\underset{K}{\min}\frac{|K|}{3\bar{l}_K}\Big)^{\frac54}$ } 
				&45056 &9.93e{-}12&5.77 &1.55e{-}11&5.68 &1.73e{-}10&5.37 \\%
				&180224&2.94e{-}13&5.08 &3.93e{-}13&5.30 &4.11e{-}12&5.40 \\%

				\bottomrule[1.5pt]
			\end{tabular}
		}
	\end{table}
	
\end{expl}

\begin{expl}[Discontinuous problem]\label{Ex:LinDis}
	In this example, we consider a discontinuous initial condition given by
	\begin{equation*}
		u(x,y,0) =
		\begin{cases}
			1, & \sqrt{x^2+y^2} \leq \frac{1}{8}(3+3^{\sin{5\theta}}), \\
			0, & \text{otherwise},
		\end{cases}
		\quad \text{with} \quad
		\theta =
		\begin{cases}
			\arccos{\frac{x}{\sqrt{x^2+y^2}}}, & y \geq 0, \\
			2\pi - \arccos{\frac{x}{\sqrt{x^2+y^2}}}, & y < 0.
		\end{cases}
	\end{equation*}
	The computational domain $[-1,1]^2$ is discretized into $N = 92,474$ cells with $h = 0.01$. 
	\Cref{fig:Ex_Linear2D} shows the numerical solutions at $t = 1.8$ computed using the $\mathbb{P}^k$-based OEDG methods, for $k \in \{1, 2, 3, 4\}$. The results are comparable to those reported in \cite{lu2021oscillation}. 
    Additionally, \Cref{fig:Ex_LinearD2} shows the solutions along the line $y = -0.25$ obtained using the OEDG method and the DG method without the OE procedure. It can be observed that the discontinuities are well-resolved without spurious oscillations when the OE procedure is applied; however, nonphysical oscillations appear when the OE procedure is not used. Moreover, the resolution improves as the order of accuracy increases, as demonstrated by the close-up views near the shocks shown in \Cref{fig:Ex_LinearD3}.

	\begin{figure}[!htb]
		\centering		
		\begin{subfigure}{0.4\textwidth}
			\centering
			\includegraphics[width=\textwidth]{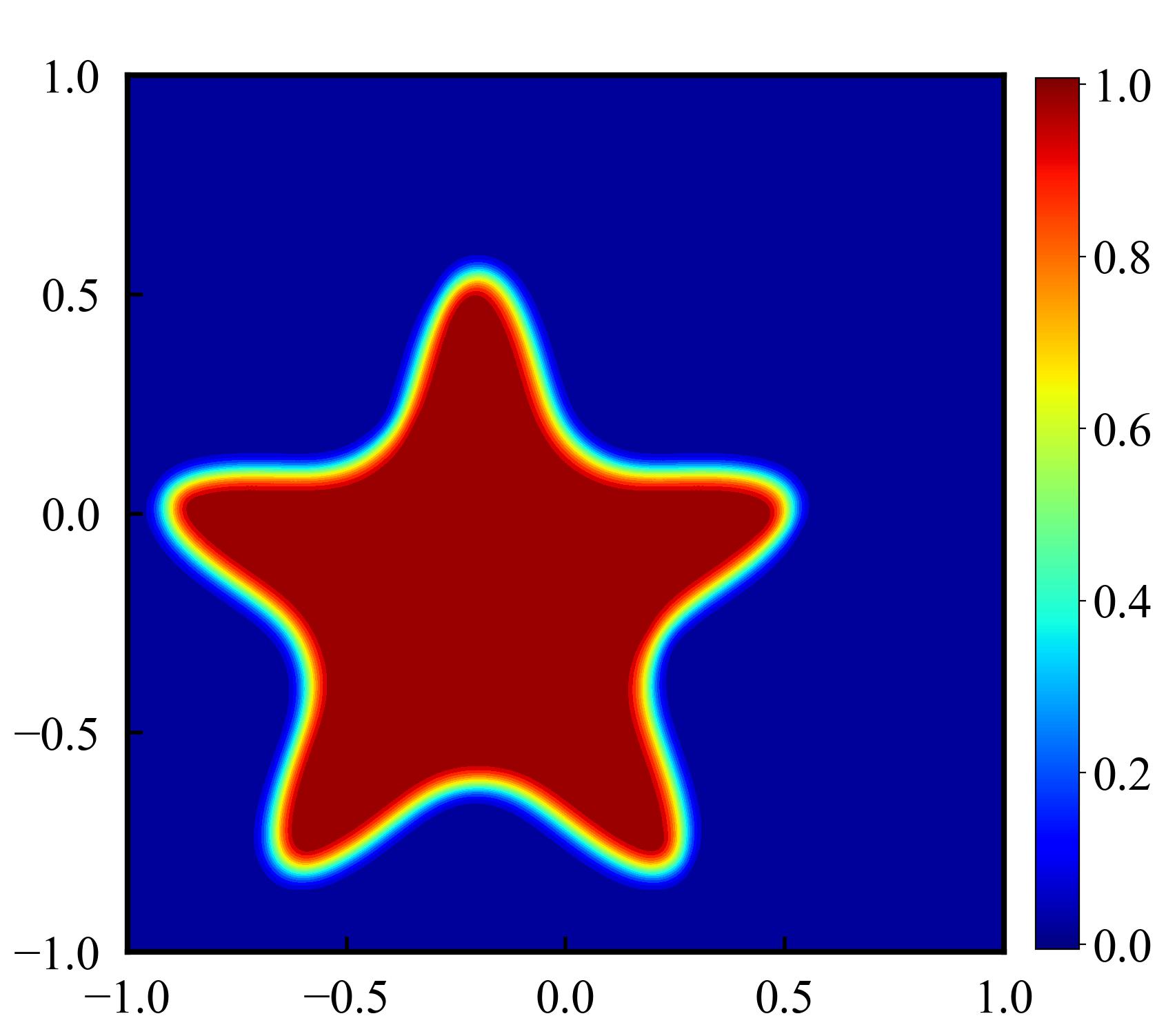}
			\caption{$\mathbb{P}^1$-based OEDG scheme.}
		\end{subfigure}
		\qquad
		\begin{subfigure}{0.4\textwidth}
			\centering
			\includegraphics[width=\textwidth]{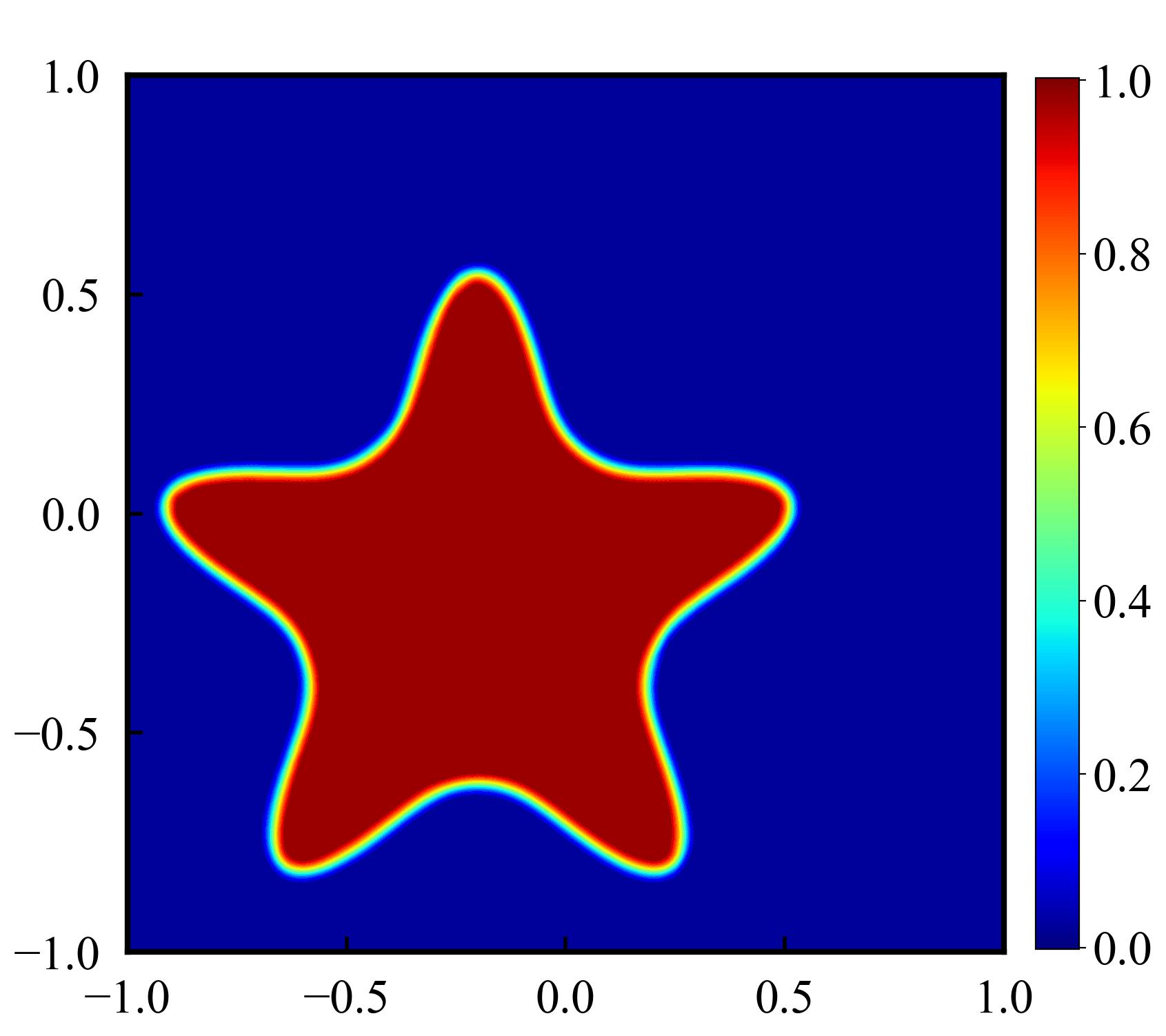}
			\caption{$\mathbb{P}^2$-based OEDG scheme.}
		\end{subfigure}
		
		\begin{subfigure}{0.4\textwidth}
			\centering
			\includegraphics[width=\textwidth]{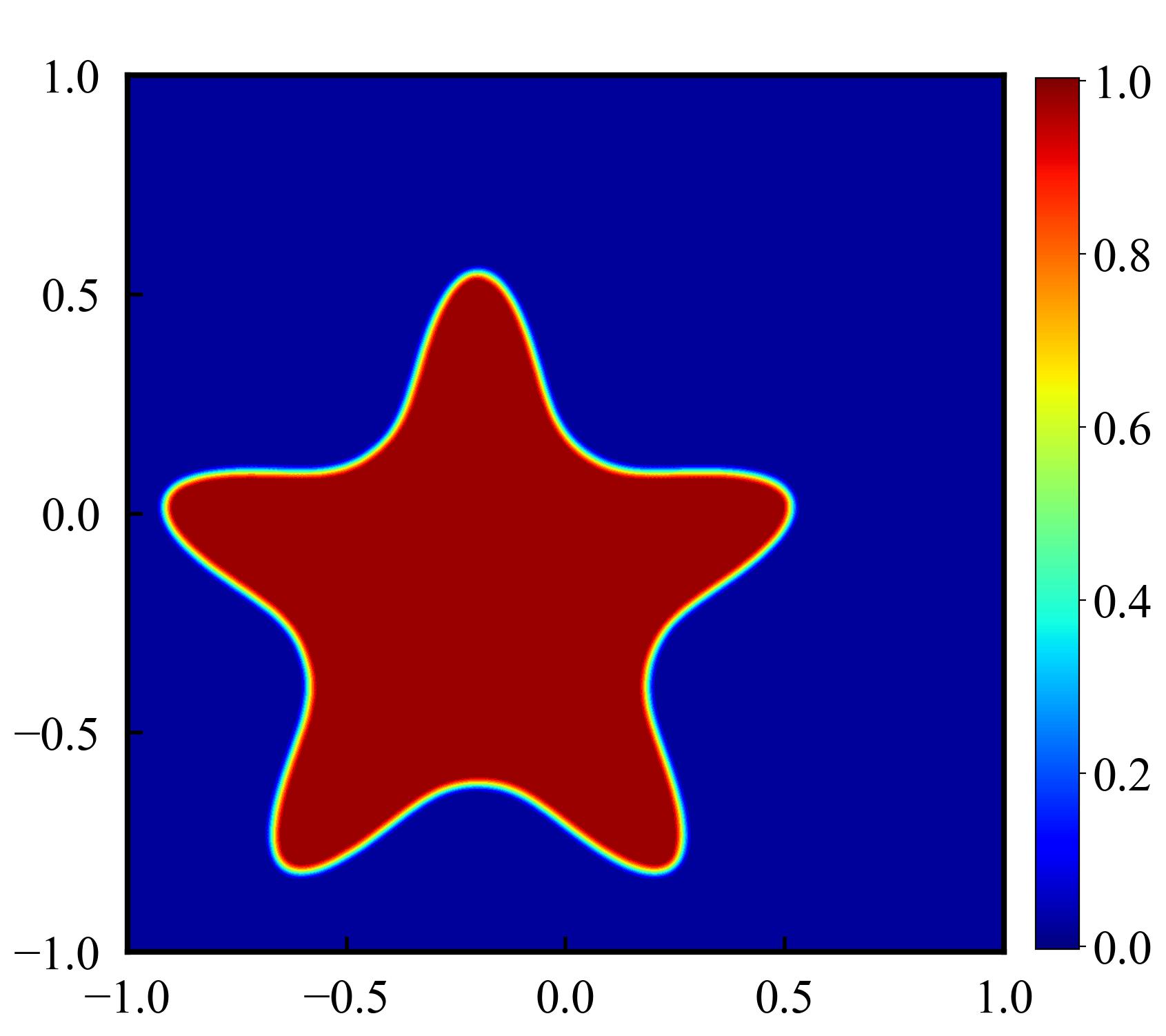}
			\caption{$\mathbb{P}^3$-based OEDG scheme.}
		\end{subfigure}
		\qquad
		\begin{subfigure}{0.4\textwidth}
			\centering
			\includegraphics[width=\textwidth]{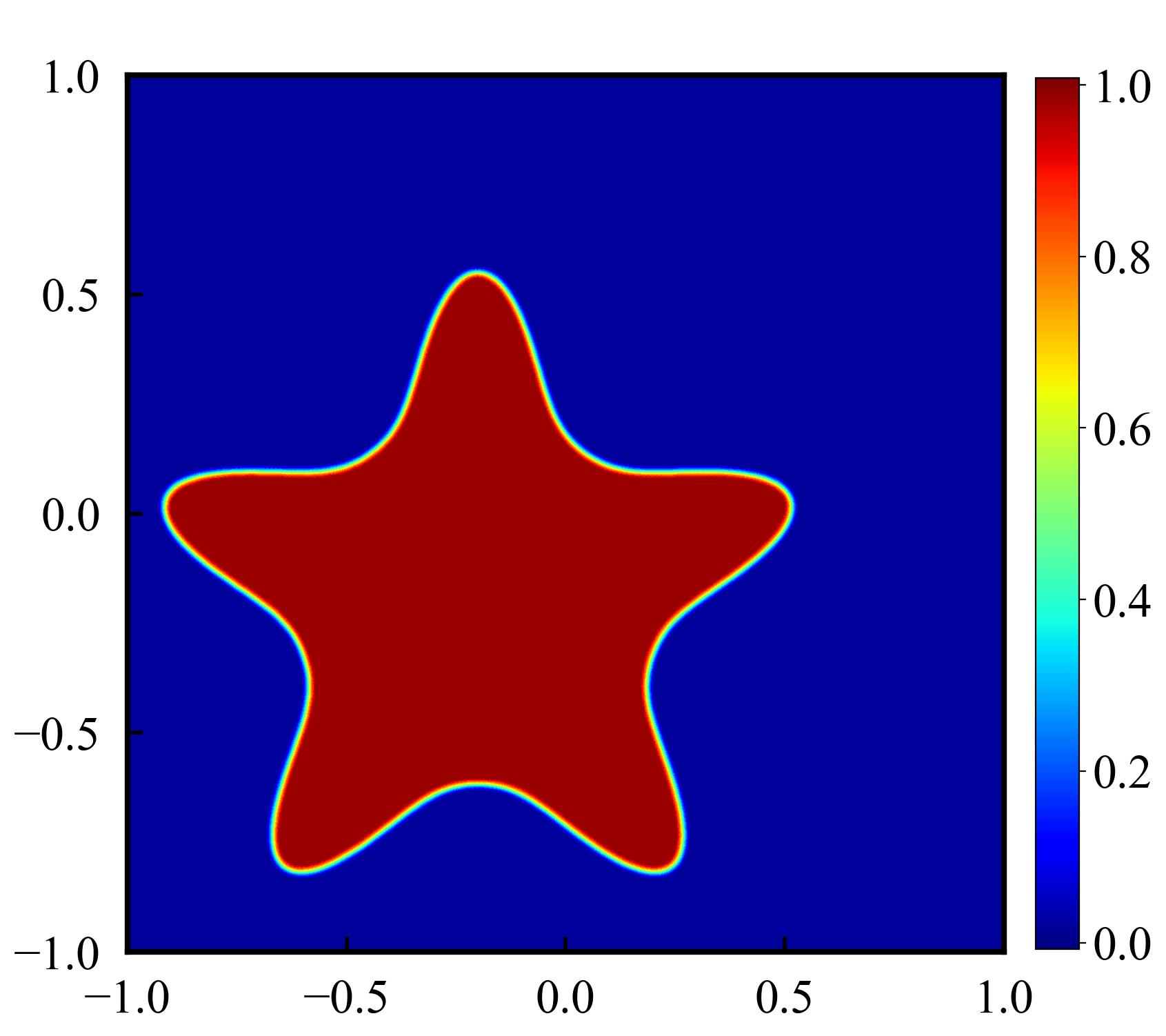}
			\caption{$\mathbb{P}^4$-based OEDG scheme.}
		\end{subfigure}
		
		\caption{(\Cref{Ex:LinDis}) Contours of the numerical solutions obtained by using the $\mathbb{P}^k$-based OEDG schemes, $k \in \{1,2,3,4\}$.
		}
		\label{fig:Ex_Linear2D}
	\end{figure} 
	
	\begin{figure}[!htb]
		\centering
		\centering
		\begin{subfigure}{0.45\textwidth}
			\centering
			\includegraphics[width=\textwidth]{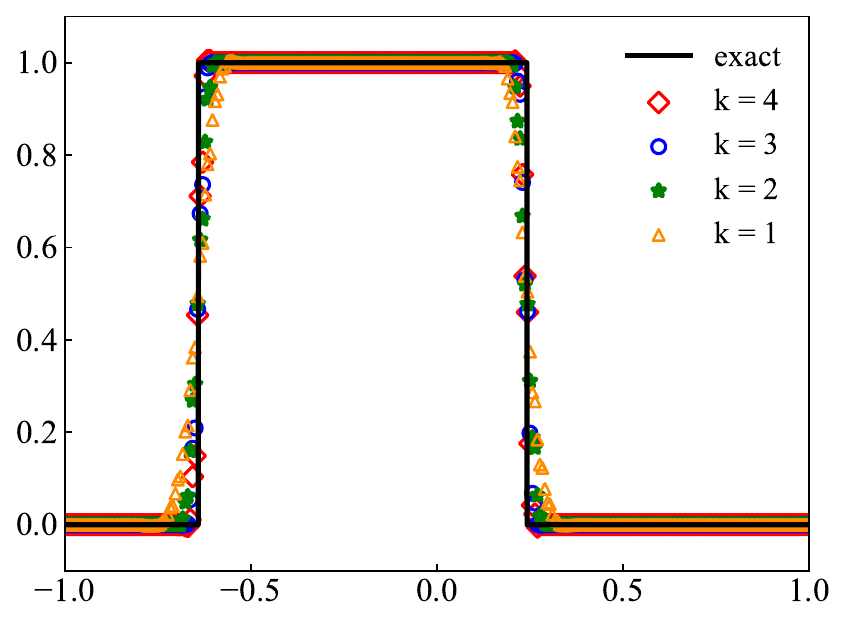}
		\end{subfigure}	
		\quad
		\begin{subfigure}{0.45\textwidth}
			\centering
			\includegraphics[width=\textwidth]{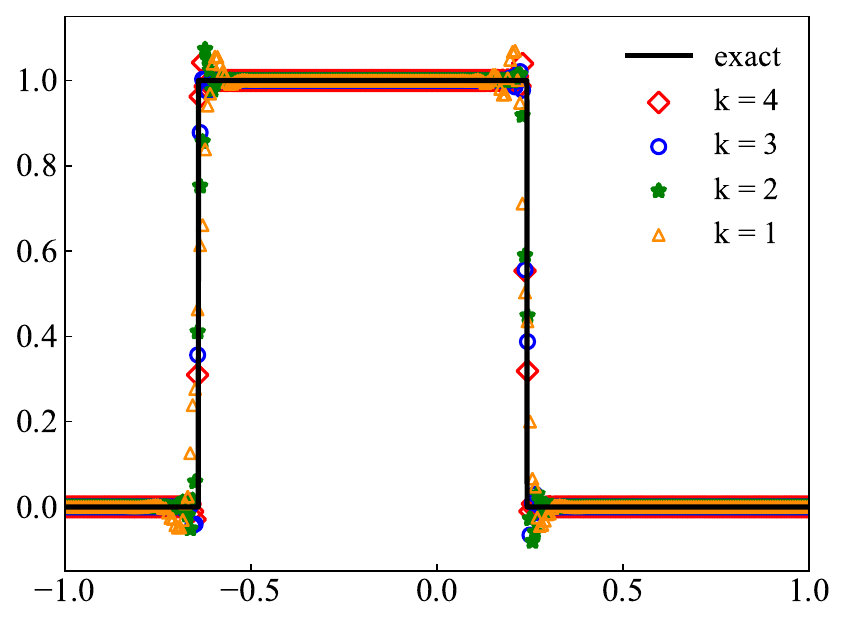}
		\end{subfigure}		
		\caption{(\Cref{Ex:LinDis}) Numerical solutions cut along $y=-0.25$. Left: the OEDG method; right: the DG method without the OE procedure.
		}
		\label{fig:Ex_LinearD2}
	\end{figure} 

        \begin{figure}[!htb]
		\centering
		\begin{subfigure}{0.45\textwidth}
			\centering
			\includegraphics[width=\textwidth]{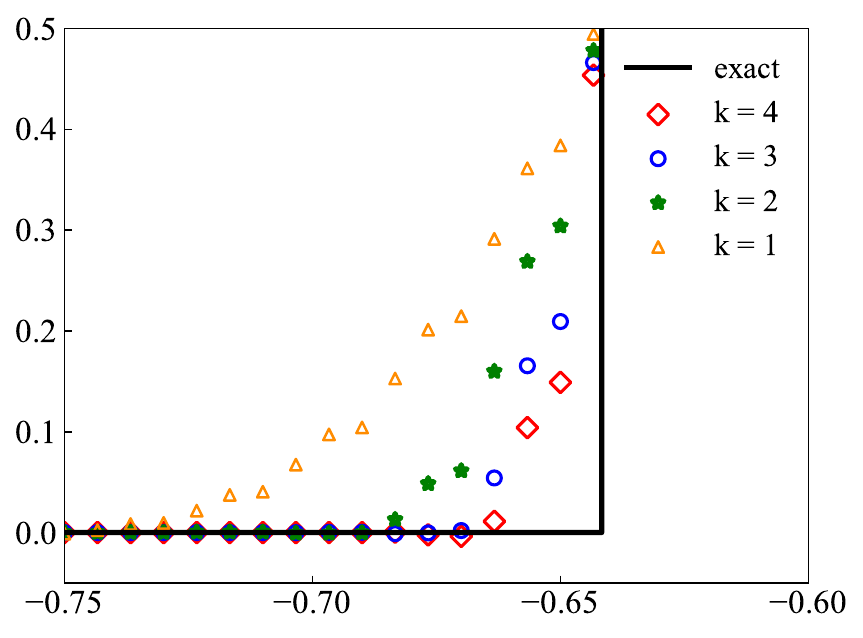}
		\end{subfigure}	
		\quad
		\begin{subfigure}{0.45\textwidth}
			\centering
			\includegraphics[width=\textwidth]{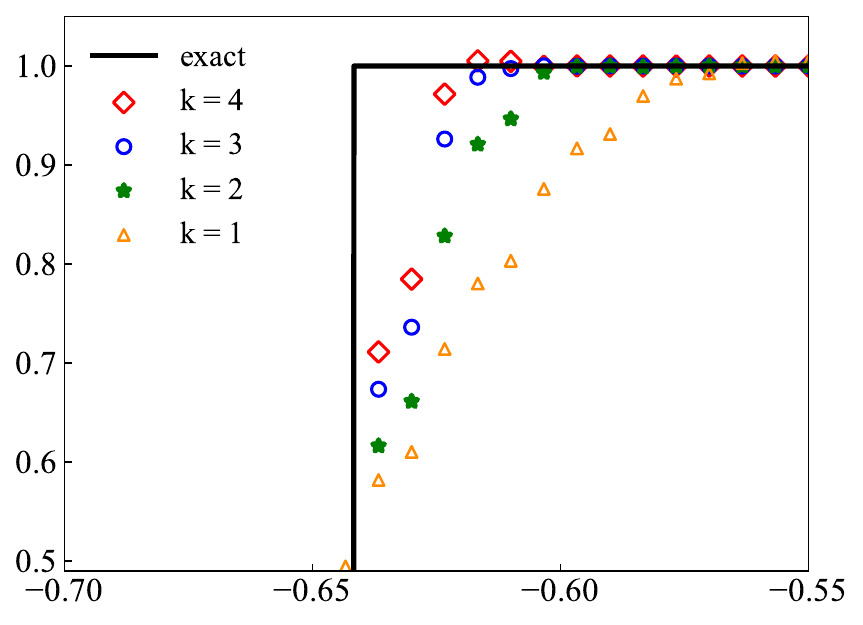}
		\end{subfigure}	
		
		\begin{subfigure}{0.45\textwidth}
			\centering
			\includegraphics[width=\textwidth]{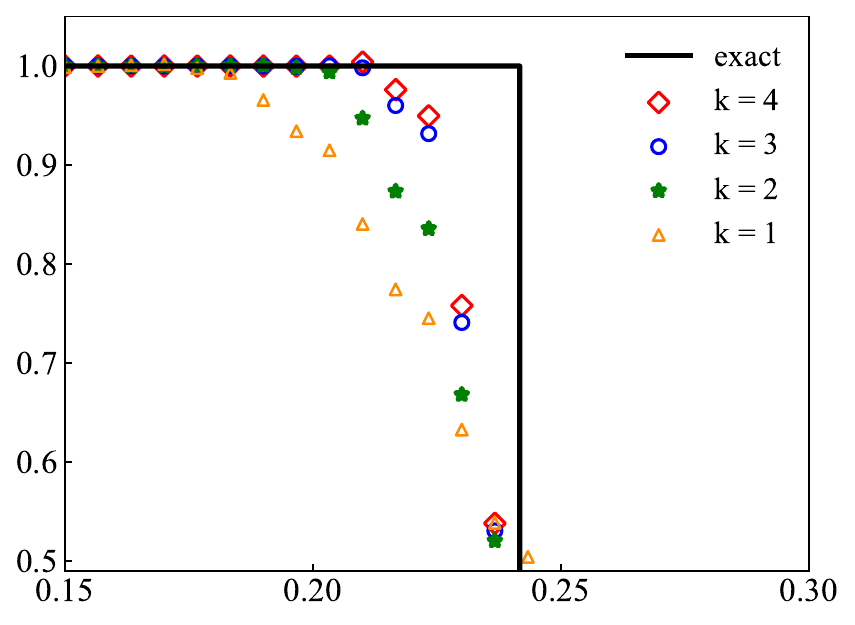}
		\end{subfigure}	
		\quad
		\begin{subfigure}{0.45\textwidth}
			\centering
			\includegraphics[width=\textwidth]{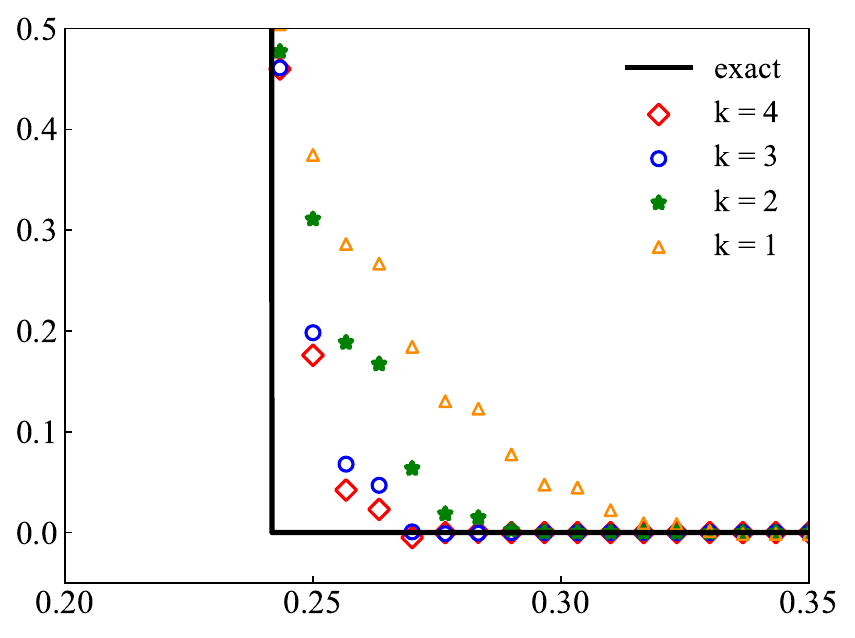}
		\end{subfigure}		
		\caption{(\Cref{Ex:LinDis}) Close-up view of the results in the left of \Cref{fig:Ex_LinearD2}.
		}
		\label{fig:Ex_LinearD3}
	\end{figure}
\end{expl}

\subsection{Burgers Equation}\label{Ex:Burg} 

In the following Examples \ref{Ex:BurgS}--\ref{Ex:BurgRM2}, we consider the Burgers equation:
\begin{equation}\label{eq:1955}
u_t + \left(\frac{u^2}{2}\right)_x + \left(\frac{u^2}{2}\right)_y = 0, \quad (x,y)\in[0,1]^2.
\end{equation}

\begin{expl}[Smooth initial condition]\label{Ex:BurgS}
	We first consider a smooth initial condition $u(x,y,0) = 0.5\sin(2\pi(x+y))$ with periodic boundary conditions. The exact solution remains smooth until a stationary shock wave forms at $t = \frac{1}{2\pi} \approx 0.159$.
	
	\Cref{tab:Ex_BurgS1} shows the $L^1$, $L^2$, and $L^\infty$ errors of the numerical solutions at $t = 0.05$ and the corresponding convergence orders obtained using the $\mathbb{P}^k$-based OEDG schemes for $k = 1, 2, 3, 4$. As seen, the expected $(k+1)$th-order convergence is achieved, confirming that the OE procedure maintains the accuracy of the original DG methods. 
	The meshes used in this example are the same as those used in \Cref{Ex:LinSmooth}. \Cref{fig:Ex_BurDisS} displays the numerical solutions at $t = 0.23$, obtained using our $\mathbb{P}^k$-based OEDG methods on the triangular mesh $\mathcal{T}_{h_4}$ with $N_4 = 11,264$ cells. It can be observed that the OEDG methods effectively capture the strong shock waves without oscillations, demonstrating the effectiveness of our OE procedure in suppressing nonphysical oscillations. 
	
	\begin{table}[!htb] 
		\centering
		\caption{(\Cref{Ex:BurgS}) $L^1$, $L^2$, and $L^\infty$ errors as well as the convergence orders at $t=0.05$ obtained by the $\mathbb{P}^k$-based OEDG schemes with $k \in \{1,2,3,4\}$. 
		}
		\label{tab:Ex_BurgS1}
		\setlength{\tabcolsep}{3mm}{
			\begin{tabular}{cccccccc}
				\toprule[1.5pt]
				& \multirow{2}{*}{$N$} & 
				\multicolumn{2}{c}{$L^1$ norm} & 
				\multicolumn{2}{c}{$L^2$ norm} & 
				\multicolumn{2}{c}{$L^\infty$ norm} \\
				\cmidrule(l){3-4} \cmidrule(l){5-6} \cmidrule(l){7-8} 
				& & error & order & error & order &  error & order\\
				
				\midrule[1.5pt]	
				\multirow{4}{*}{$\mathbb{P}^1$-based OEDG} 
				&176   &1.88e{-}2&{-} &2.48e{-}2&{-} &1.45e{-}1&{-} \\%
				&704   &3.77e{-}3&2.32&5.21e{-}3&2.25&3.85e{-}2&1.91\\%
				\multirow{2}{*}{ {\tt SSP-RK}(2,2)} 
				&2816  &7.57e{-}4&2.32&1.06e{-}3&2.30&6.16e{-}3&2.65\\%
				\multirow{2}{*}{ $C_{\tt SSP} = 1 $ } 
				&11264 &1.68e{-}4&2.17&2.48e{-}4&2.10&1.24e{-}3&2.32\\%
				&45056 &4.10e{-}5&2.04&6.21e{-}5&2.00&3.90e{-}4&1.66\\%
				&180224&1.02e{-}5&2.00&1.57e{-}5&1.98&1.06e{-}4&1.88\\%

				\midrule[1.0pt]	
				\multirow{4}{*}{$\mathbb{P}^2$-based OEDG} 	
				&176   &3.05e{-}3&{-} &5.13e{-}3&{-} &3.95e{-}2&{-} \\%
				&704   &3.22e{-}4&3.24&5.34e{-}4&3.26&4.94e{-}3&3.00\\%
				\multirow{2}{*}{ {\tt  SSP-RK}(3,3)} 
				&2816  &3.52e{-}5&3.19&5.97e{-}5&3.16&5.29e{-}4&3.22\\%
				\multirow{2}{*}{ $C_{\tt SSP} = 1 $ } 
				&11264 &3.95e{-}6&3.16&7.18e{-}6&3.06&6.13e{-}5&3.11\\%
				&45056 &4.82e{-}7&3.04&9.55e{-}7&2.91&1.01e{-}5&2.60\\%
				&180224&6.28e{-}8&2.94&1.36e{-}7&2.81&1.67e{-}6&2.60\\%

				\midrule[1.0pt]	
				\multirow{4}{*}{$\mathbb{P}^3$-based OEDG} 
				&176   &1.20e{-}3 &{-} &2.05e{-}3 &{-} &1.62e{-}2&{-} \\%
				&704   &3.39e{-}5 &5.15&6.67e{-}5 &4.94&7.73e{-}4&4.39\\%
				\multirow{2}{*}{ {\tt SSP-RK}(5,4)}
				&2816  &1.28e{-}6 &4.73&2.94e{-}6 &4.50&3.53e{-}5&4.45\\%
				\multirow{2}{*}{ $C_{\tt SSP} = 1.508$ } 
				&11264 &6.11e{-}8 &4.39&1.46e{-}7 &4.33&1.87e{-}6&4.24\\%
				&45056 &2.97e{-}9 &4.36&7.44e{-}9 &4.30&8.92e{-}8&4.39\\%
				&180224&1.58e{-}10&4.23&4.15e{-}10&4.16&6.81e{-}9&3.71\\%

				\midrule[1.0pt]	
				\multirow{3}{*}{$\mathbb{P}^4$-based OEDG} 	
				&176   &3.24e{-}4 &5.63&5.84e{-}4 &5.31&5.16e{-}3 &4.40\\%
				&704   &3.65e{-}6 &6.47&7.85e{-}6 &6.22&8.06e{-}5 &6.00\\%
				\multirow{1}{*}{ {\tt SSP-RK}(5,4)}
				&2816  &6.40e{-}8 &5.84&1.56e{-}7 &5.65&1.77e{-}6 &5.51\\%
				\multirow{1}{*}{ $C_{\tt SSP} = 1.508$ } 
				&11264 &1.42e{-}9 &5.50&3.68e{-}9 &5.40&4.60e{-}8 &5.27\\%
				\multirow{1}{*}{ {\small $\dt= \frac{C_{\tt SSP}}{9}\Big( \frac{\underset{K}{\min}\left\{|K|/(3\bar{l}_K)\right\}}{\alpha^{\rm LF}}\Big)^{\frac54}$ } }
				&45056 &3.74e{-}11&5.25&1.01e{-}10&5.20&1.30e{-}9 &5.15\\%
				&180224&1.18e{-}12&4.98&3.32e{-}12&4.92&4.84e{-}11&4.75\\%
				\bottomrule[1.5pt]
			\end{tabular}
		}
	\end{table}
	
	\begin{figure}[!htb]
		\centering		
		\begin{subfigure}{0.43\textwidth}
			\centering
			\includegraphics[width=\textwidth]{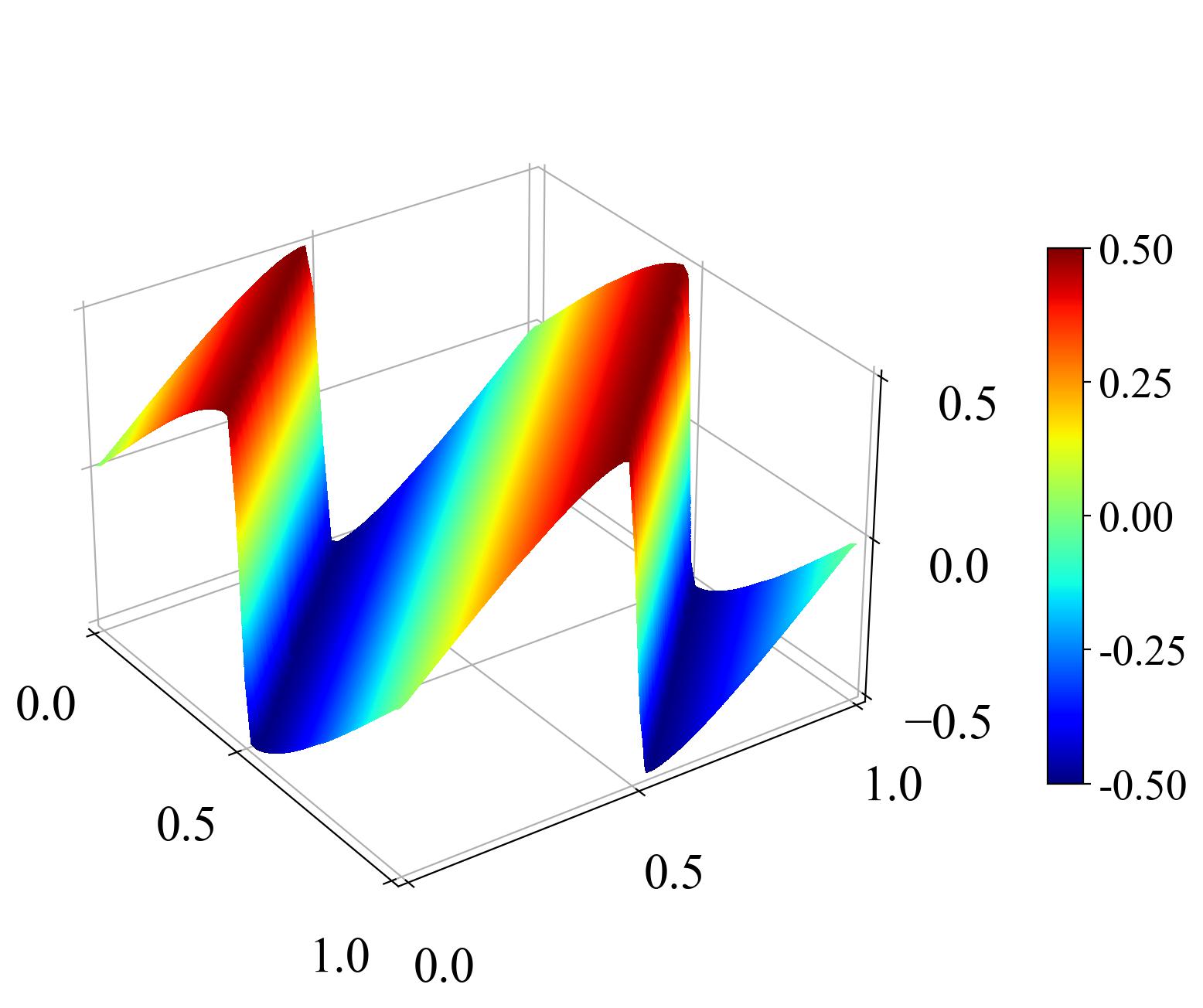}
            \caption{$\mathbb{P}^1$-based OEDG scheme.}
		\end{subfigure}
		\qquad
		\begin{subfigure}{0.43\textwidth}
			\centering
			\includegraphics[width=\textwidth]{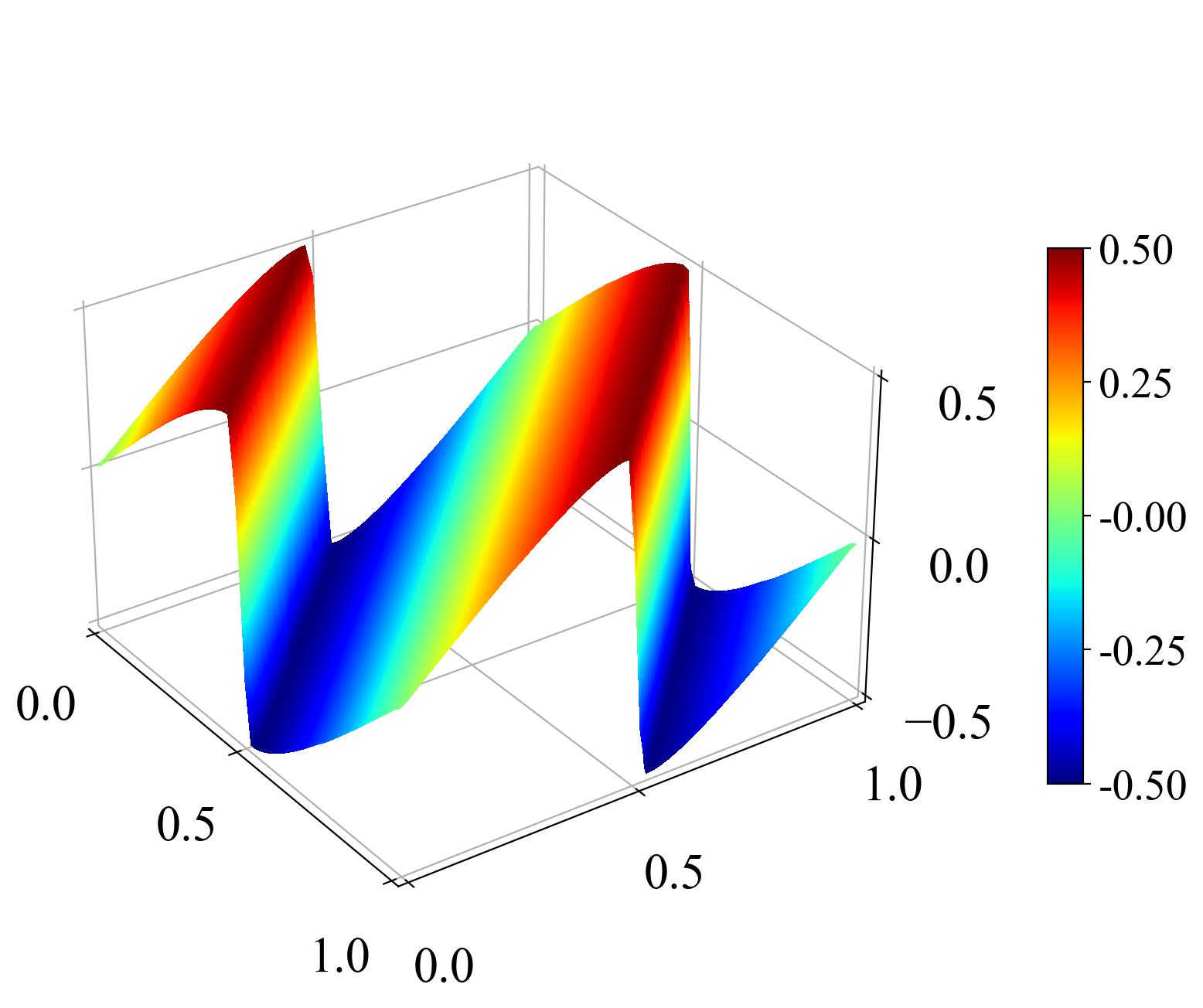}
            \caption{$\mathbb{P}^2$-based OEDG scheme.}
		\end{subfigure}
		
		\begin{subfigure}{0.43\textwidth}
			\centering
			\includegraphics[width=\textwidth]{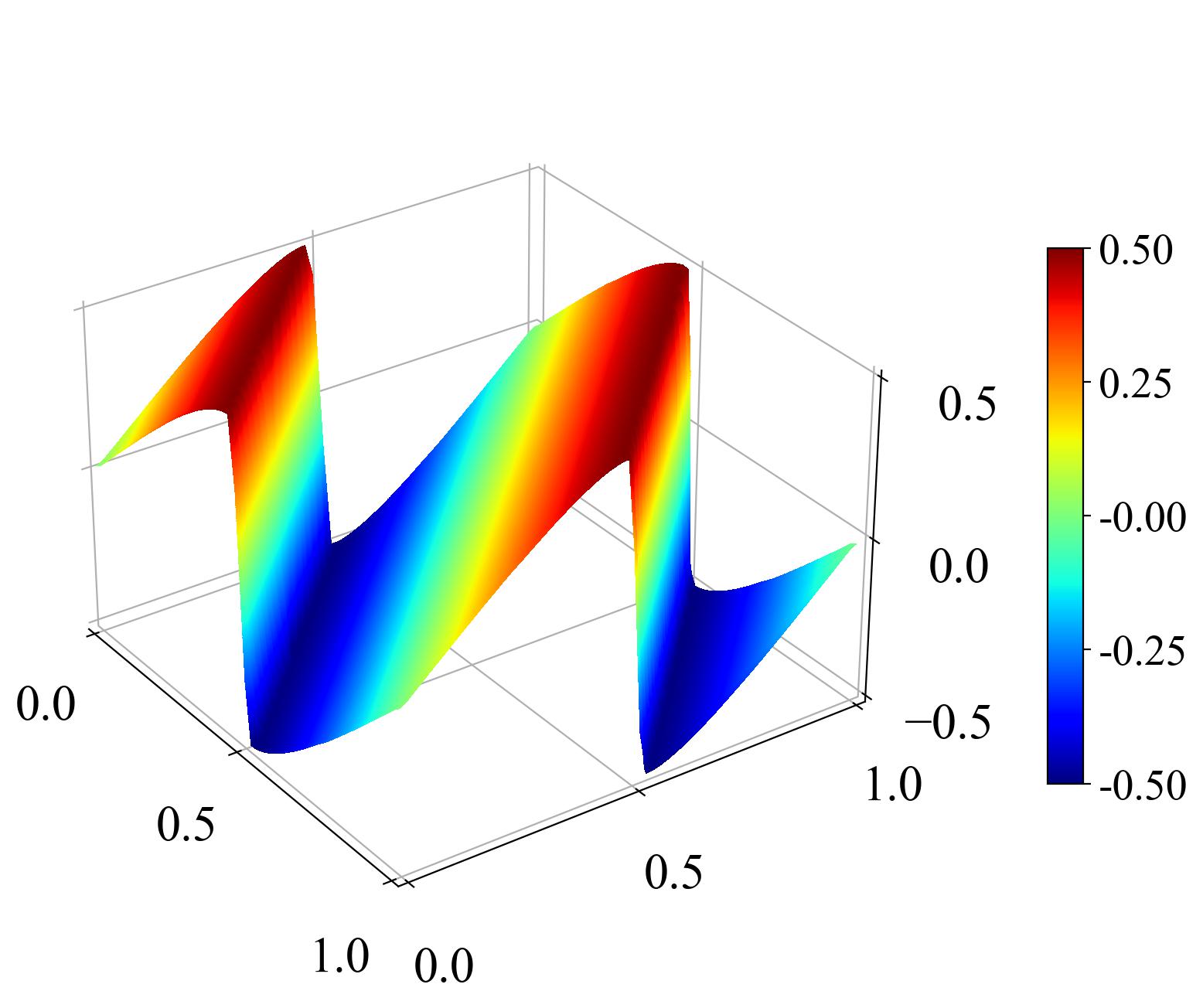}
            \caption{$\mathbb{P}^3$-based OEDG scheme.}
		\end{subfigure}
		\qquad
		\begin{subfigure}{0.43\textwidth}
			\centering
			\includegraphics[width=\textwidth]{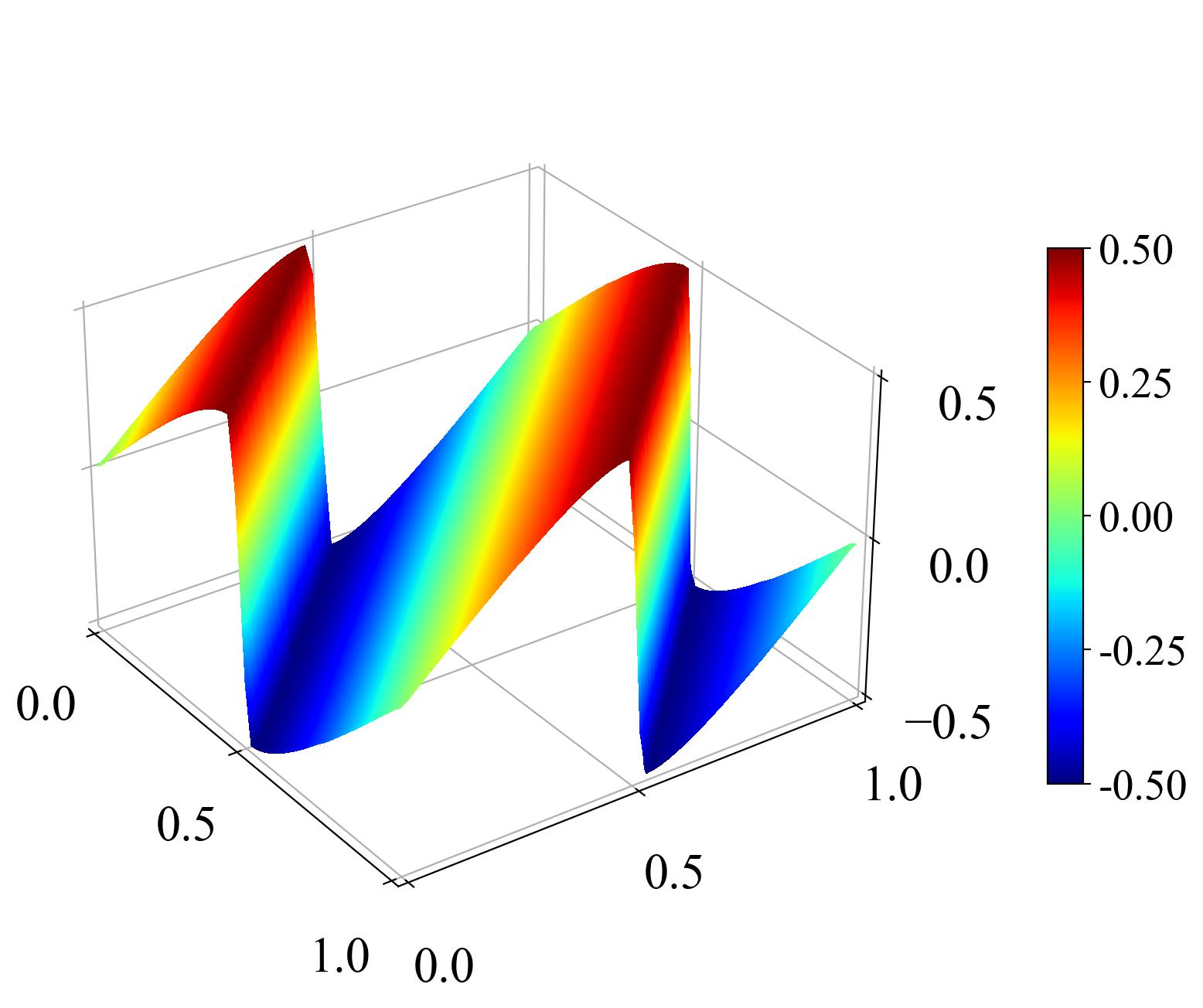}
			\caption{$\mathbb{P}^4$-based OEDG scheme.}
		\end{subfigure}
		
		\caption{(\Cref{Ex:BurgS}) Numerical solutions obtained by using the $\mathbb{P}^k$-based OEDG schemes with $k \in \{1,2,3,4\}$, at $t=0.23$.
		}
		\label{fig:Ex_BurDisS}
	\end{figure} 
	
\end{expl}

\begin{figure}[!htb]
	\centering
	\begin{subfigure}{0.32\textwidth}
		\centering
		\includegraphics[width=\textwidth]{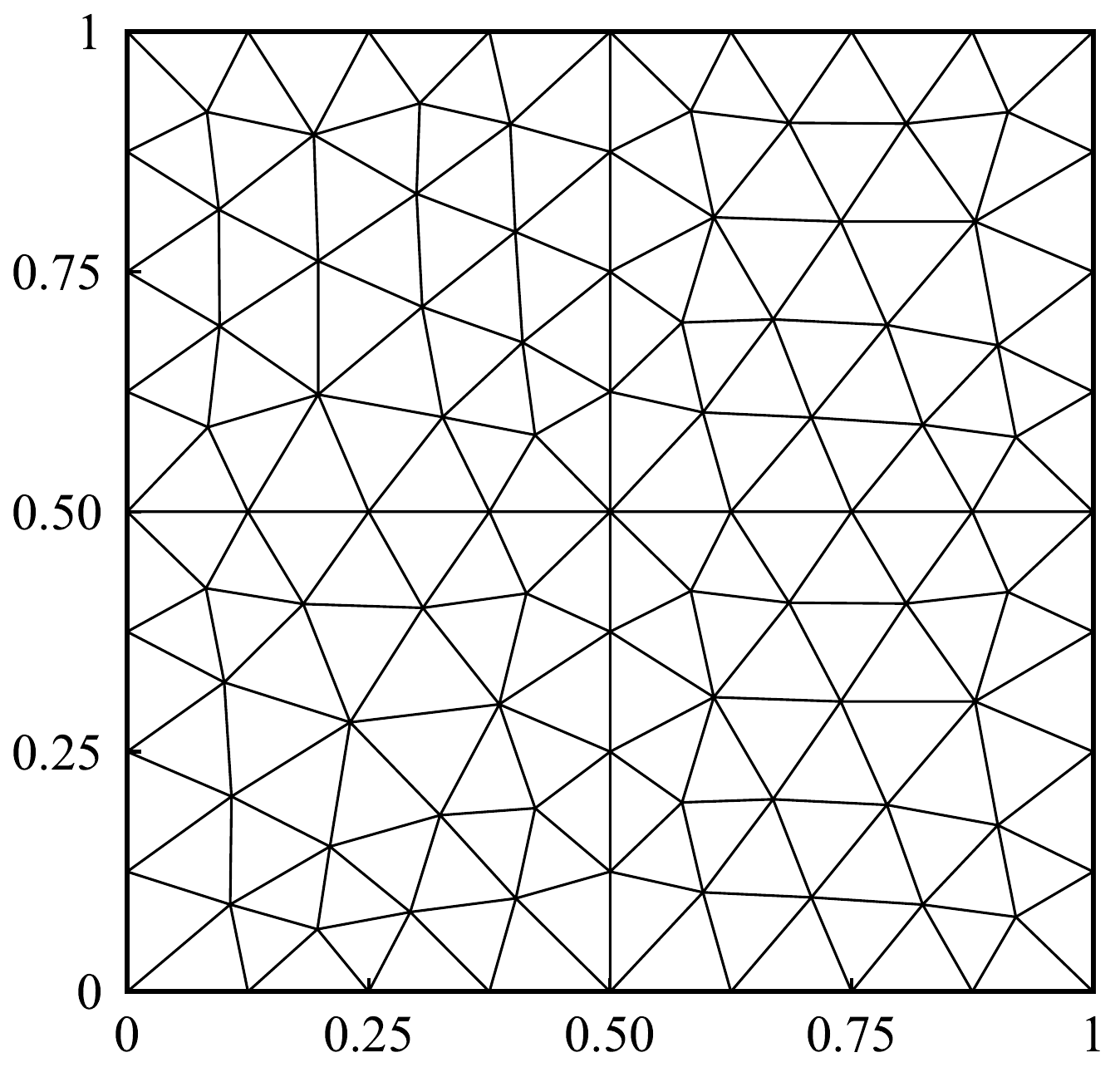}
		\caption{\Cref{Ex:BurgRM1}.} 
		\label{fig:Ex_BurgRM1mesh}
	\end{subfigure}
	\qquad\qquad
	\begin{subfigure}{0.32\textwidth}
		\centering
		\includegraphics[width=\textwidth]{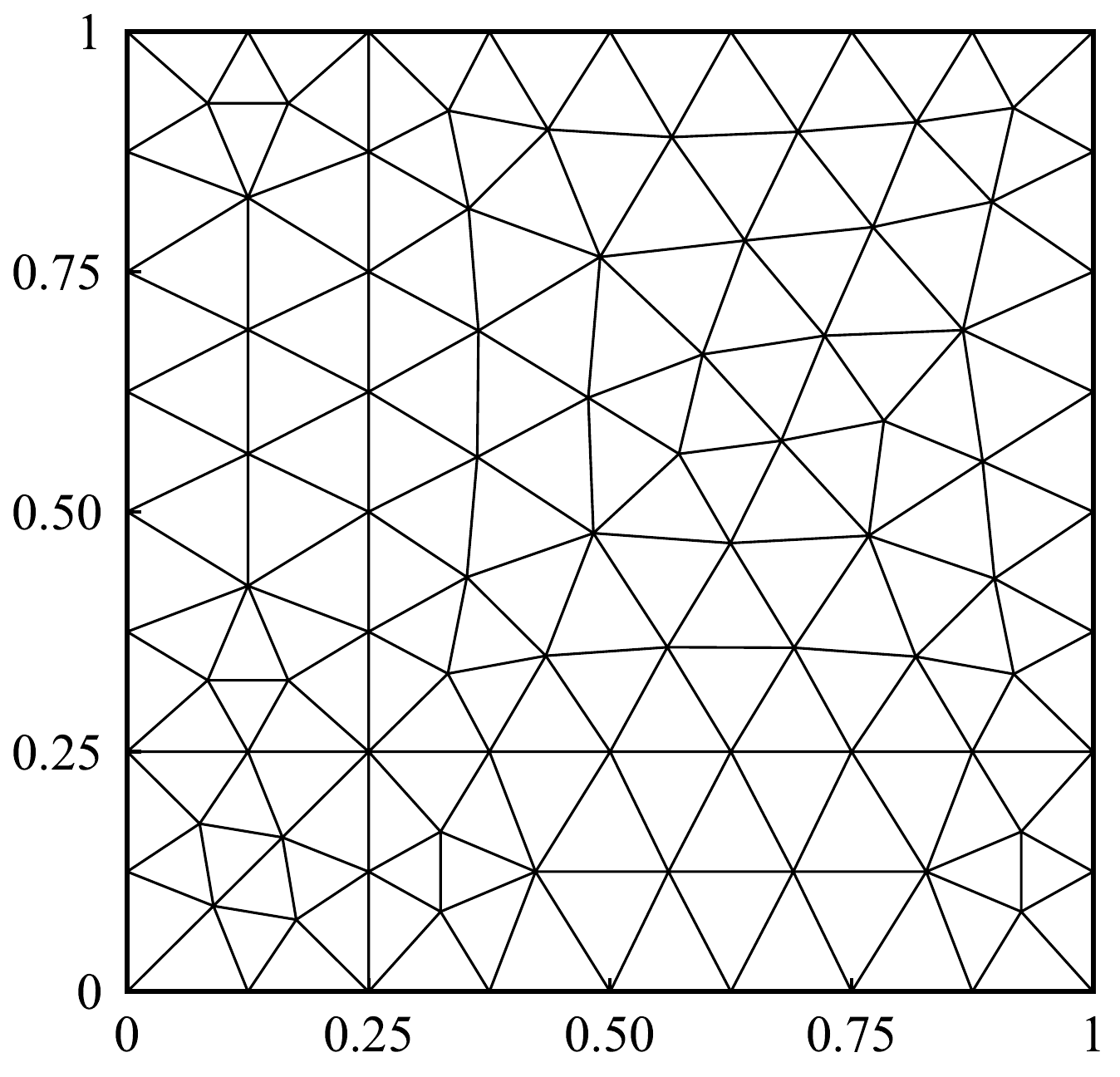}
		\caption{\Cref{Ex:BurgRM2}.} 
		\label{fig:Ex_BurgRM2mesh}
	\end{subfigure}
	\caption{The sample meshes with $h=0.125$.}
	\label{fig:Ex_BurgRM}
\end{figure}

\begin{expl}[Riemann problem \Rmnum{1}]\label{Ex:BurgRM1}
	We consider the ``oblique'' Riemann problem \cite{Christov2008JCP,CDWOCAD2023} for the Burgers equation \eqref{eq:1955}, subject to the following discontinuous initial condition:
	\begin{equation*}
		u(x,y,0) = 
		\begin{cases}
			-0.2, & x < 0.5, \, y \ge 0.5, \\
			-1,   & x \ge 0.5, \, y \ge 0.5, \\
			0.5,  & x < 0.5, \, y < 0.5, \\
			0.8,  & x \ge 0.5, \, y < 0.5.
		\end{cases}
	\end{equation*} 
	The computational domain $[0,1]^2$ is discretized using $N = 151,742$ triangular cells with $h = \frac{1}{256}$. A sample mesh with $h = 0.125$ and $N = 176$ cells is shown in \Cref{fig:Ex_BurgRM1mesh}. For any time $t \ge 0$, inflow conditions are imposed on all boundaries except for $\{x = 0, \, y \geq 0.5 + 0.15t\}$ and $\{x = 1, \, y \leq 0.5 - 0.1t\}$, where outflow conditions are applied. 
	This problem is simulated until $t = 0.5$ using the $\mathbb{P}^1$-, $\mathbb{P}^2$-, $\mathbb{P}^3$-, and $\mathbb{P}^4$-based OEDG schemes. \Cref{fig:Ex_BurgRM1} presents the numerical solutions, and \Cref{fig:Ex_BurgRM1Diag} shows the numerical solution along the line $y = 1 - x$. These results demonstrate the absence of nonphysical oscillations and are consistent with those reported in \cite{Christov2008JCP,CDWOCAD2023}.

	\begin{figure}[!htb]
		\centering
		\begin{subfigure}{0.43\textwidth}
			\centering
			\includegraphics[width=\textwidth]{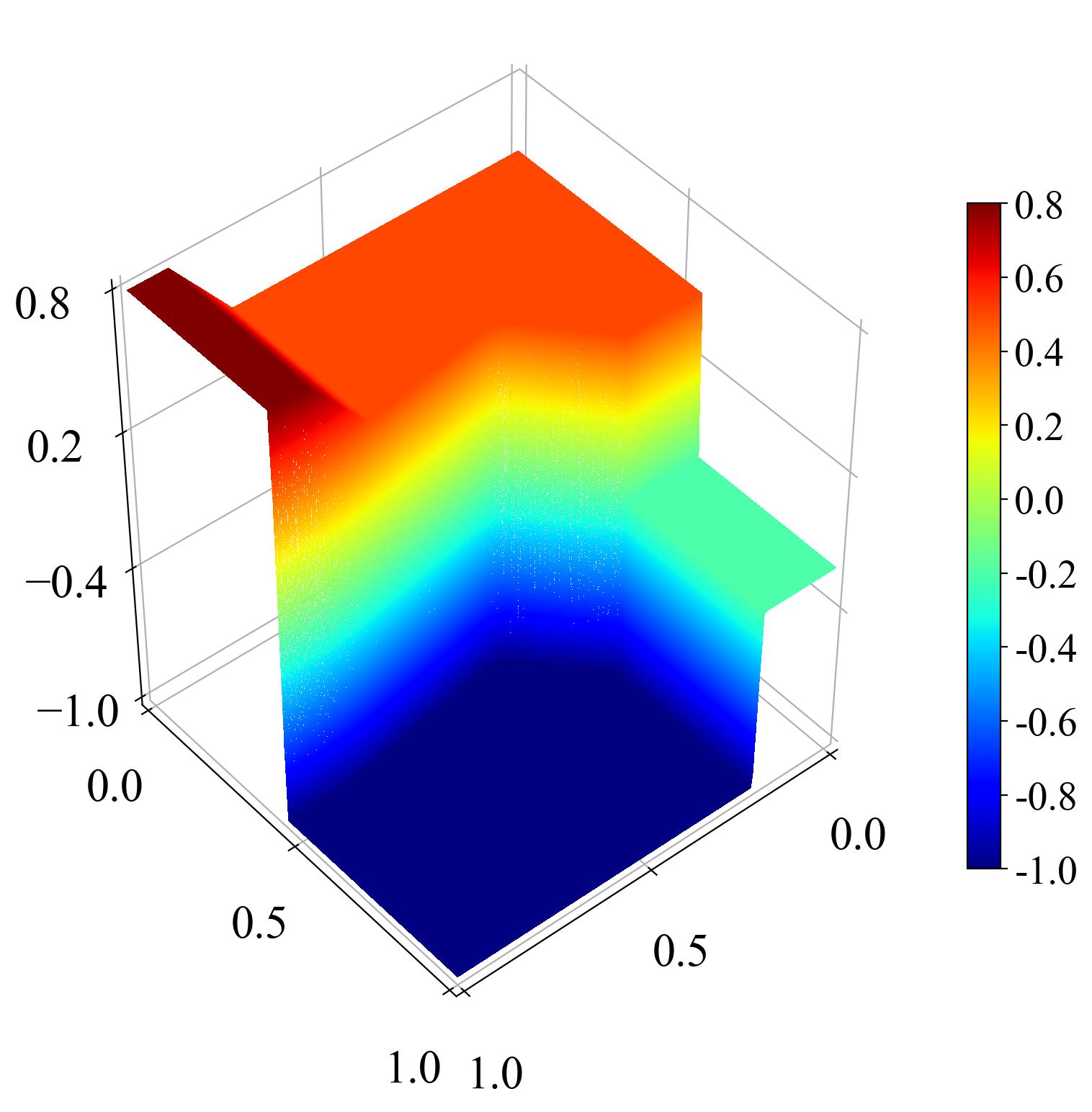}
			\caption{$\mathbb{P}^1$-based OEDG scheme.}
		\end{subfigure}
		\qquad
		\begin{subfigure}{0.43\textwidth}
			\centering
			\includegraphics[width=\textwidth]{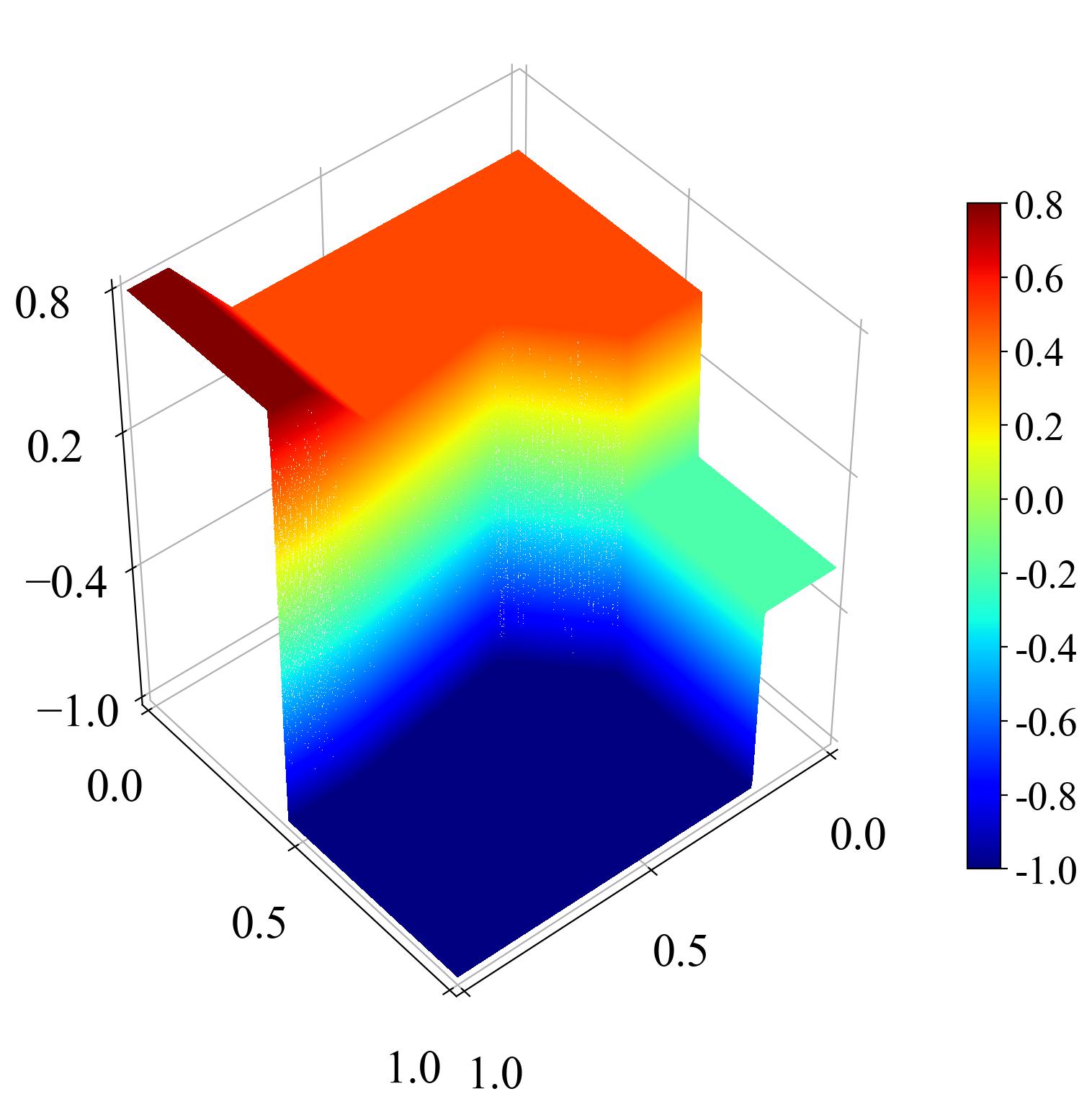}
			\caption{$\mathbb{P}^2$-based OEDG scheme.}
		\end{subfigure}
		
		\begin{subfigure}{0.43\textwidth}
			\centering
			\includegraphics[width=\textwidth]{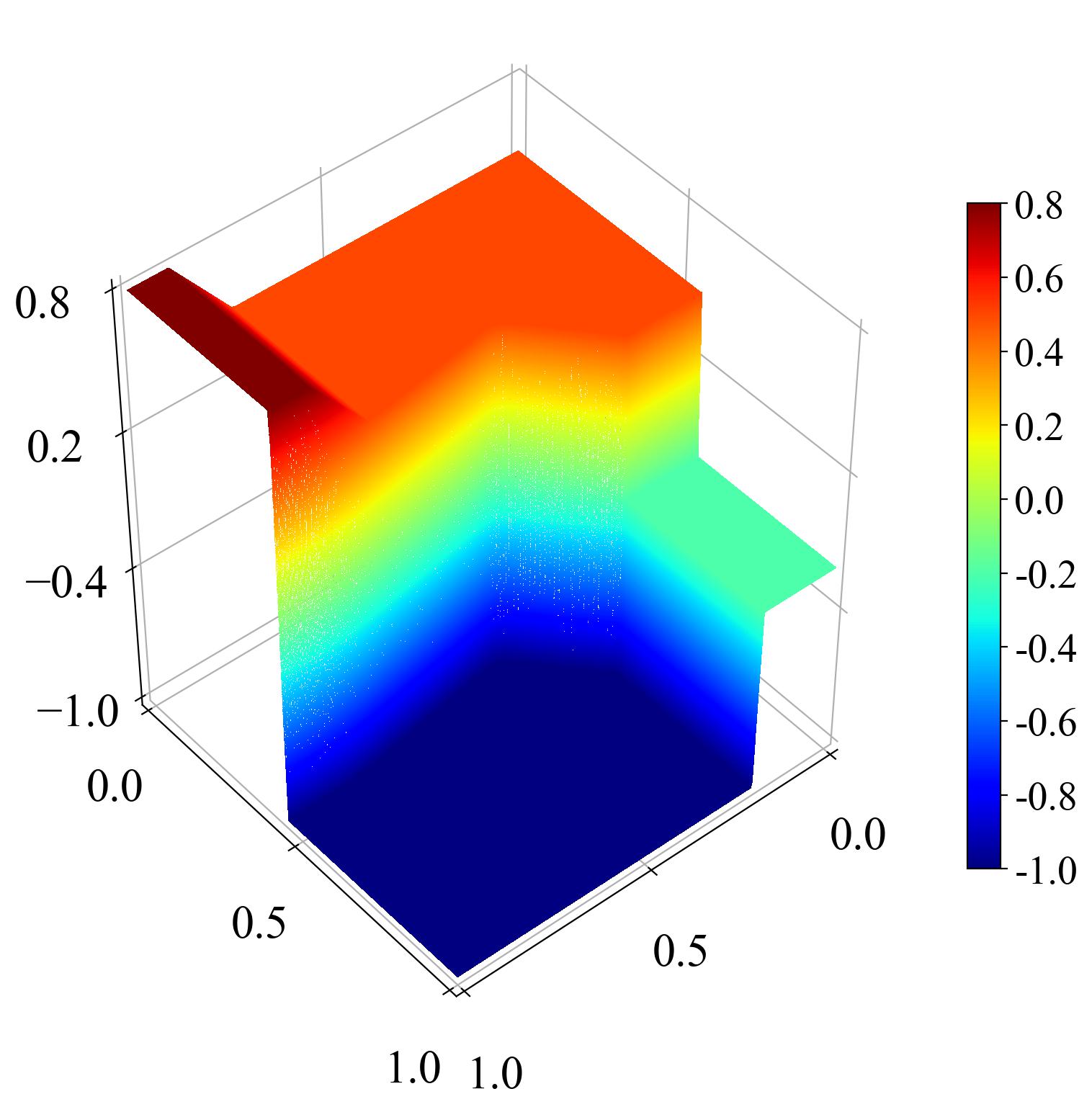}
			\caption{$\mathbb{P}^3$-based OEDG scheme.}
		\end{subfigure}
		\qquad
		\begin{subfigure}{0.43\textwidth}
			\centering
			\includegraphics[width=\textwidth]{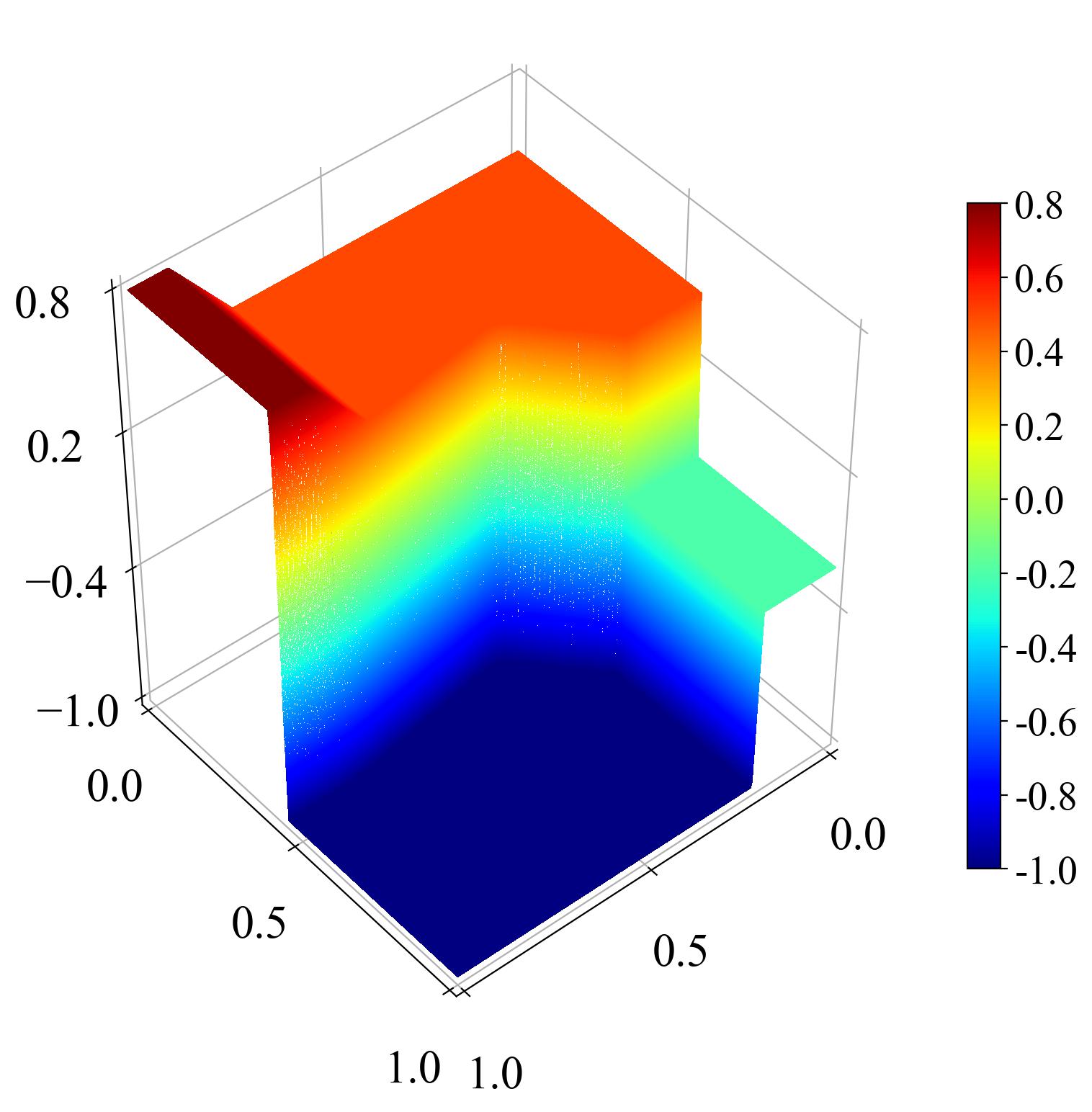}
			\caption{$\mathbb{P}^4$-based OEDG scheme.}
		\end{subfigure}
		
		\caption{(\Cref{Ex:BurgRM1}) Contours of numerical solutions at $t = 0.5$ obtained by using the $\mathbb{P}^k$-based OEDG method with $k \in \{1,2,3,4\}$.
		}
		\label{fig:Ex_BurgRM1}
	\end{figure}

	\begin{figure}[!htb]
		\centering
		\begin{subfigure}{0.45\textwidth}
			\includegraphics[width=\textwidth]{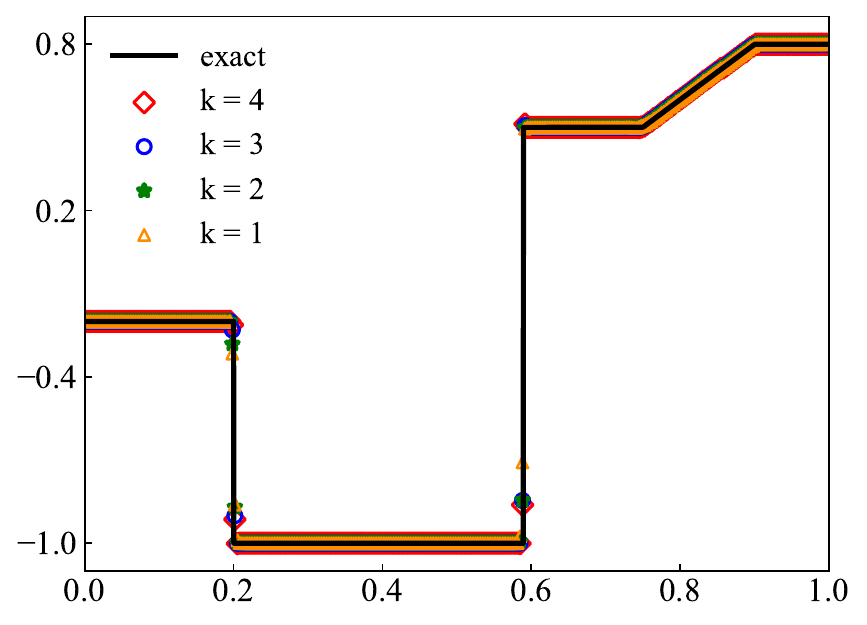}
		\end{subfigure}
		\caption{(\Cref{Ex:BurgRM1}) Numerical solutions cut along $y=1-x$.}
		\label{fig:Ex_BurgRM1Diag}
	\end{figure}
\end{expl}

\begin{expl}[Riemann problem \Rmnum{2}]\label{Ex:BurgRM2}
	We consider another Riemann problem for the Burgers equation \eqref{eq:1955}, as investigated in \cite{Christov2008JCP}, with the following discontinuous initial condition:
	\begin{equation*}
		u(x,y,0) = 
		\begin{cases}
			2, & x < 0.25, \, y < 0.25, \\
			3, & x \ge 0.25, \, y \ge 0.25, \\
			1, & \text{otherwise}.
		\end{cases}
	\end{equation*}
	The computational domain $[0,1]^2$ is discretized using $N = 151,756$ triangular cells; see the sample mesh in \Cref{fig:Ex_BurgRM2mesh}. Outflow conditions are applied at the right and upper boundaries, while inflow conditions are imposed at the left and bottom boundaries. 
	This problem describes the interaction of two shock waves and two rarefaction waves moving towards the center of the domain, eventually forming a cusp. \Cref{fig:Ex_BurgRM2} shows the numerical solutions at $t = \frac{1}{12}$ obtained using the $\mathbb{P}^1$-, $\mathbb{P}^2$-, $\mathbb{P}^3$-, and $\mathbb{P}^4$-based OEDG schemes. It can be observed that a cusp forms at the center, and the resolution of this cusp improves with increasing order of accuracy. The OEDG methods perform effectively, with no evidence of nonphysical oscillations in the computed results.

	\begin{figure}[!htb]
		\centering
		\begin{subfigure}{0.4\textwidth}
			\centering
			\includegraphics[width=\textwidth]{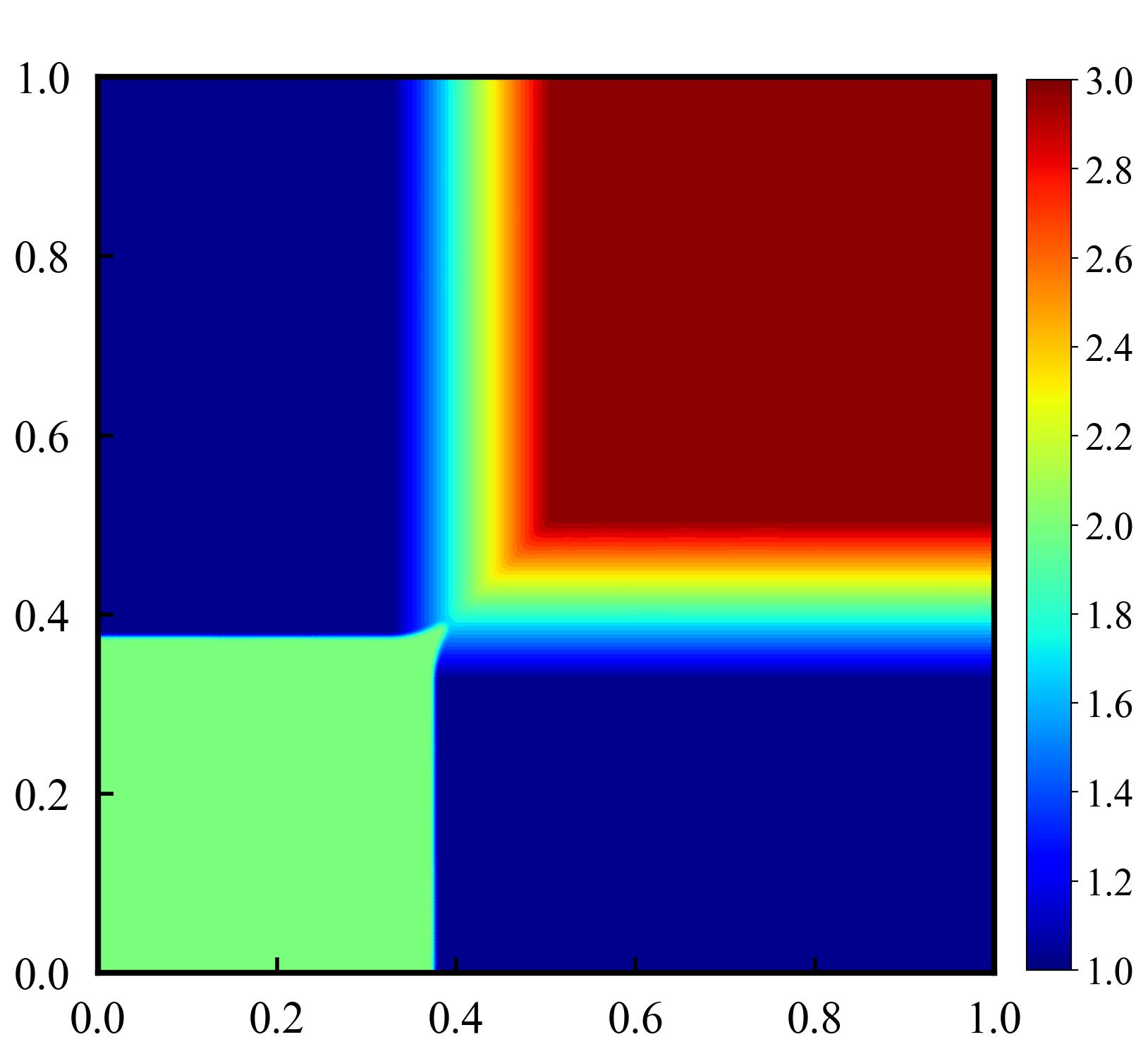}
            \caption{$\mathbb{P}^1$-based OEDG scheme.}
		\end{subfigure}
		\qquad
		\begin{subfigure}{0.4\textwidth}
			\centering
			\includegraphics[width=\textwidth]{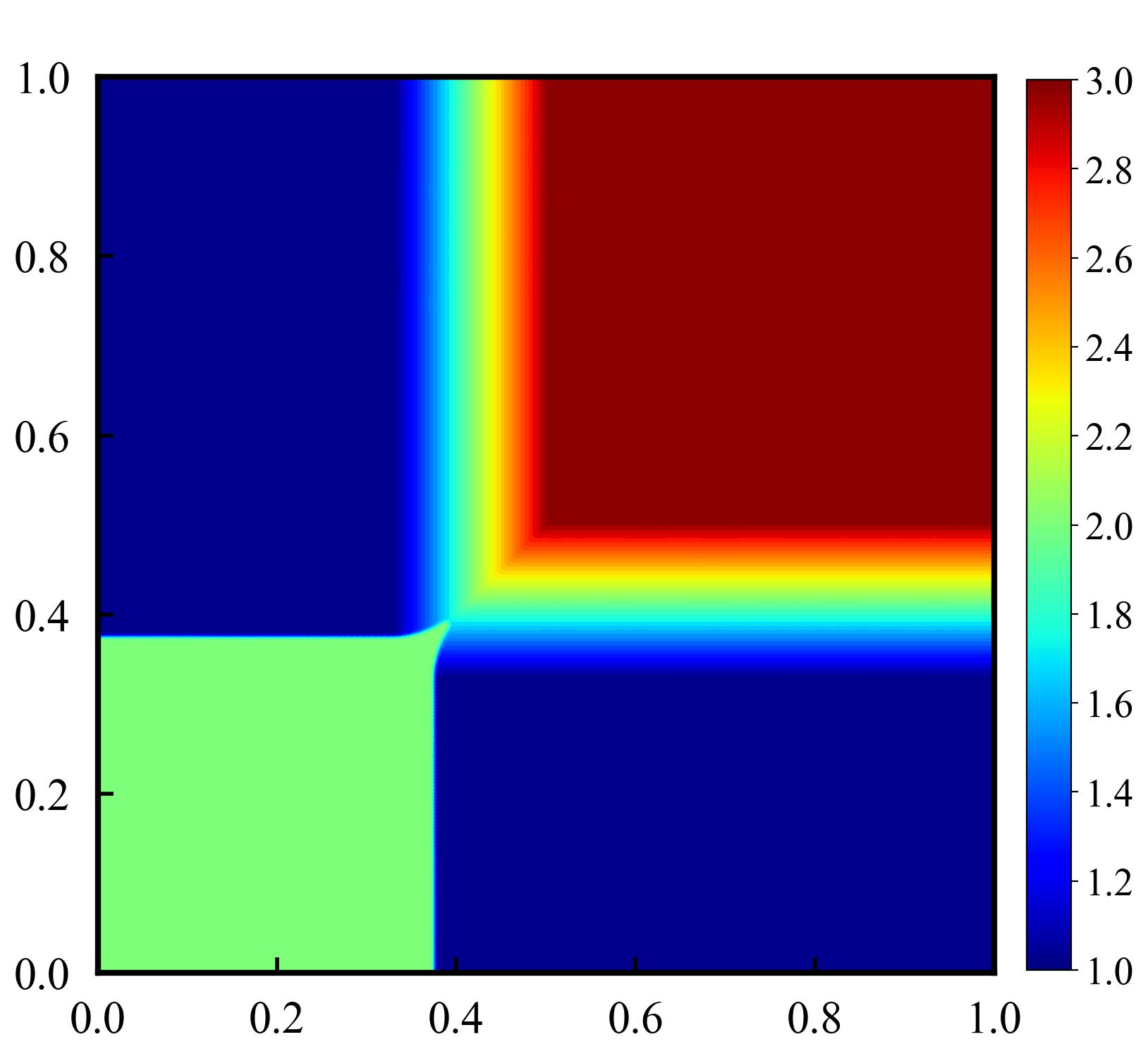}
            \caption{$\mathbb{P}^2$-based OEDG scheme.}
		\end{subfigure}
		
		\begin{subfigure}{0.4\textwidth}
			\centering
			\includegraphics[width=\textwidth]{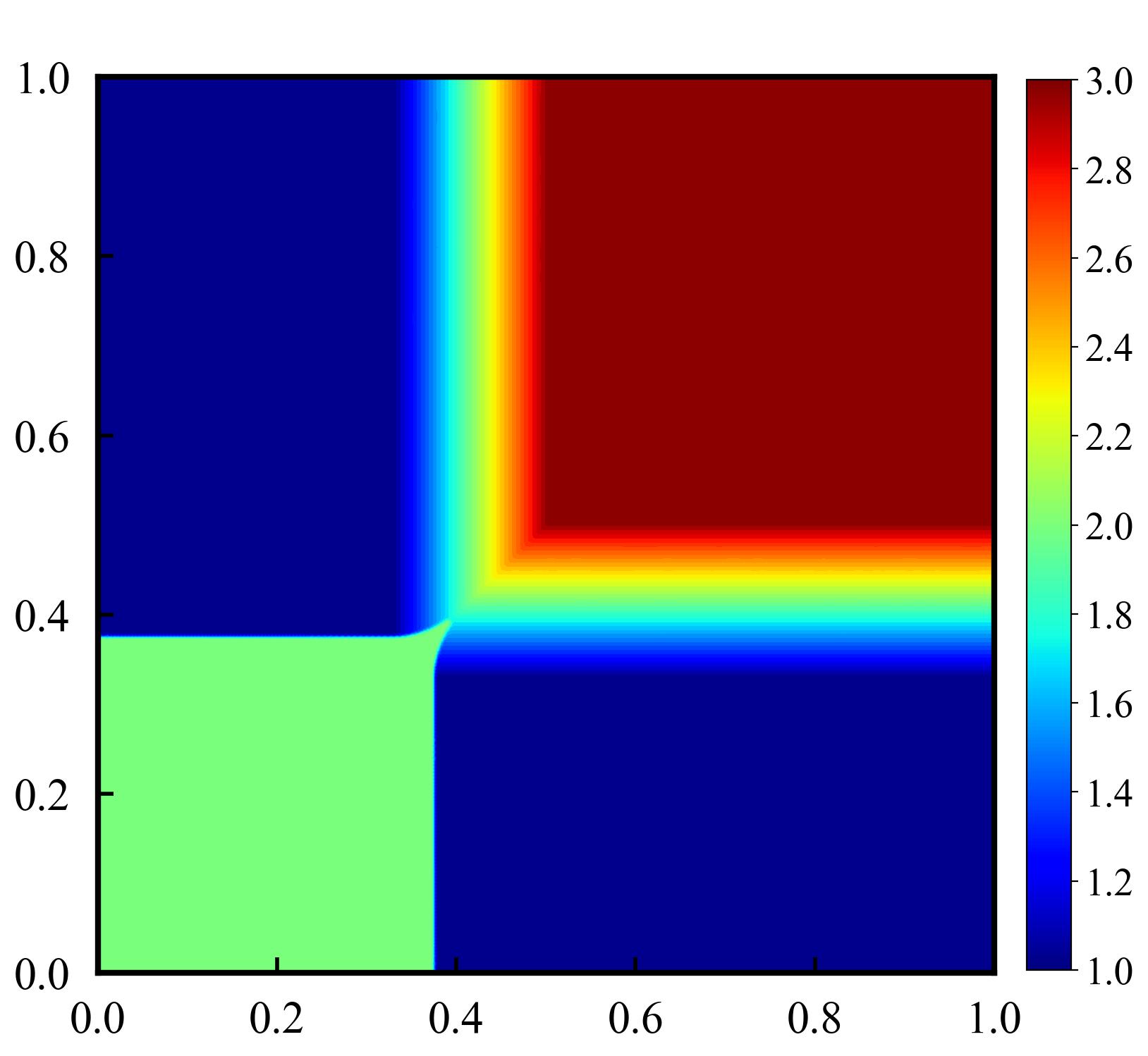}
            \caption{$\mathbb{P}^3$-based OEDG scheme.}
		\end{subfigure}
		\qquad
		\begin{subfigure}{0.4\textwidth}
			\centering
			\includegraphics[width=\textwidth]{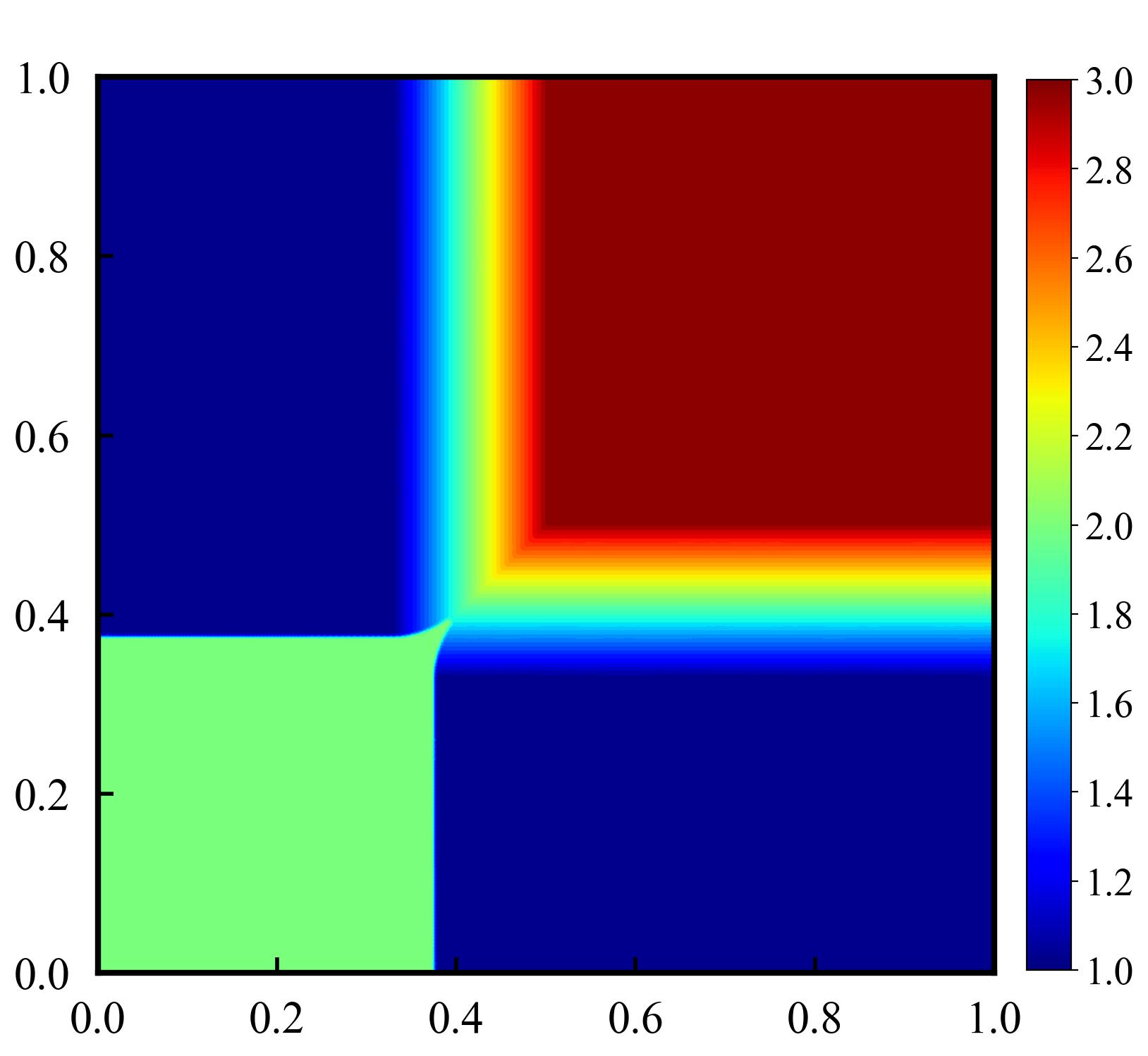}
            \caption{$\mathbb{P}^4$-based OEDG scheme.}
		\end{subfigure}
		
		\caption{(\Cref{Ex:BurgRM2}) Numerical solutions at $t = 1/12$ obtained by the $\mathbb{P}^k$-based OEDG method with $k \in \{1,2,3,4\}$.
		}
		\label{fig:Ex_BurgRM2}
	\end{figure}
\end{expl}

\subsection{Compressible Euler System}\label{Ex:Euler}

In the following Examples \ref{Ex:EulerImp}--\ref{Ex:EulerSDiffWedge}, we simulate several problems involving strong discontinuities by solving the compressible Euler system described in Example \ref{ex:euler}, with an adiabatic index $\gamma = 1.4$.

\begin{expl}[Implosion problem]\label{Ex:EulerImp}
	We consider the implosion problem investigated in \cite{Liska2003FD, Tchekhovskoy2007wham}. This problem involves the interaction between low-density, low-pressure gas inside a diamond-shaped region and high-density, high-pressure gas in the surrounding area. In the original settings \cite{Liska2003FD, Tchekhovskoy2007wham}, the computational domain is the square region $[-0.3,0.3]^2$, with the vertices of the diamond region located at $(\pm0.15,0)$ and $(0,\pm0.15)$. Due to geometric symmetry, computations can be restricted to the upper right quadrant of the domain. 
	Specifically, we define the computational domain as $\Omega = [0,0.3]^2$ and the interior region $D$ as the triangle with vertices at $(0,0)$, $(0.15,0)$, and $(0,0.15)$. A sample mesh of this domain is shown in green solid lines in \Cref{fig:Ex_EulImpMesh}, while a finer mesh (not shown) with $N = 23,290$ cells and $h = 0.003$ is used for computation. The initial condition is given by:
	\begin{equation*}
		(\rho,v_1,v_2,p)^\top ({\bm x},0) =
		\begin{cases}
			(0.125,0,0,0.14)^\top, & \xbm \in D = \{(x,y) \in \mathbb{R}^2 : x \ge 0, \, y \ge 0, \, x+y \leq 0.15\}, \\
			(1,0,0,1)^\top, & \xbm \in \Omega / D.
		\end{cases}
	\end{equation*}
	Reflective boundary conditions are applied to all four boundaries. 
	We use the $\mathbb{P}^k$-based RI-OEDG method with $k \in \{1,2,3,4\}$ to solve this problem up to $t = 2.5$. The pressure of the obtained numerical results is presented in the left column of \Cref{fig:Ex_EulImp}. These results match very well with those reported in \cite{Liska2003FD, Tchekhovskoy2007wham}, demonstrating the effectiveness of the RI-OEDG methods.

	Next, we denote the numerical solutions in the computational domain $\Omega$ as $\bm u_h = (\rho_h,v_{1,h},v_{2,h},p_h)^\top$. To test the rotational invariance, we rotate the domain clockwise by an angle of $\pi/4$. This rotation yields the following rotational matrix:
	\begin{equation*}
		M_{\bm \xi} = 
		\begin{pmatrix}
			\sqrt{2}/2 & \sqrt{2}/2 \\
			-\sqrt{2}/2 & \sqrt{2}/2
		\end{pmatrix}.
	\end{equation*}
	The new computational domain $\hat{\Omega}$, as shown in blue in \Cref{fig:Ex_EulImpMesh}, is defined as:
	$$\hat{\Omega} = \left\{(x,y) : 0 \leq x+y \leq 0.3\sqrt{2}, \, -0.3\sqrt{2} \leq y-x \leq 0\right\},$$ 
	with two vertices of the diamond relocated to $(0.15/\sqrt{2},\pm0.15/\sqrt{2})$. The interior region $D$ transforms into:
	$$\hat{D} = \left\{(x,y) : 0 \leq x \leq 0.15/\sqrt{2}, \, -x \leq y \leq x\right\}.$$
	The initial condition $\hat{\bm u}_0$ in $\hat{\Omega}$ is derived by rotating $\bm u_0$ by $\pi/4$. Notably, since the initial velocities in $\bm u_0$ are zero, $\hat{\bm u}_0$ remains similar to $\bm u_0$ except for the rotation of $D$ and $\Omega$ into $\hat{D}$ and $\hat{\Omega}$, respectively. The blue dashed lines in \Cref{fig:Ex_EulImpMesh} show the domain $\hat{\Omega}$ and the corresponding sample mesh $\hat{\mathcal{T}}_h$, which is derived by rotating the original sample mesh $\mathcal{T}_h$. 
	Using the RI-OEDG method with $\mathbb{P}^1$-, $\mathbb{P}^2$-, $\mathbb{P}^3$-, and $\mathbb{P}^4$-elements, we perform computations until $t = 2.5$. The results, shown in the right column of \Cref{fig:Ex_EulImp}, exhibit no spurious oscillations. The numerical solutions are denoted as $\hat{\bm u}_h = (\hat{\rho}_h,\hat{v}_{1,h},\hat{v}_{2,h},\hat{p}_h)^\top$. 
 
	To further demonstrate the rotational invariance of the RI-OEDG method, we quantitatively evaluate the numerical errors between $\hat{\bm u}_h$ and ${\bm u}_h$, termed RI errors, as follows:
	\begin{equation*}
		\bm \varepsilon := (\varepsilon_{\rho}, \varepsilon_{v_1}, \varepsilon_{v_2}, \varepsilon_p)^\top = 
		\Big(\norm{\bar{\rho}_h-\bar{\hat{\rho}}_h}_{L^\infty}, 
		\norm{\bar{v}_{1,h}-\bar{\hat{v}}_{1,h}^{T}}_{L^\infty}, 
		\norm{\bar{v}_{2,h}-\bar{\hat{v}}_{2,h}^{T}}_{L^\infty}, 
		\norm{\bar{p}_h-\bar{\hat{p}}_h}_{L^\infty}\Big)^\top,
	\end{equation*}
	where $(\bar{\hat{v}}_{1,h}^{T},\bar{\hat{v}}_{2,h}^{T})^\top = M^{-1}_{\bm \xi} (\bar{\hat{v}}_{1,h},\bar{\hat{v}}_{2,h})^\top$. From \Cref{fig:Ex_EulImpError}, we observe that the RI errors $\bm \varepsilon$ are maintained at about the level of round-off errors, demonstrating that the RI-OEDG method preserves the RI property.  
	For comparison, we also use the OEDG method with a component-wise OE procedure, referred to as the CW-OEDG method. The corresponding RI errors $\bm \varepsilon$ are significantly larger, as shown in \Cref{fig:Ex_EulImpError}, indicating that the CW-OEDG method is not RI.

	\begin{figure}[!htb]
		\centering
		\begin{subfigure}{0.35\textwidth}
			\centering
			\includegraphics[width=\textwidth]{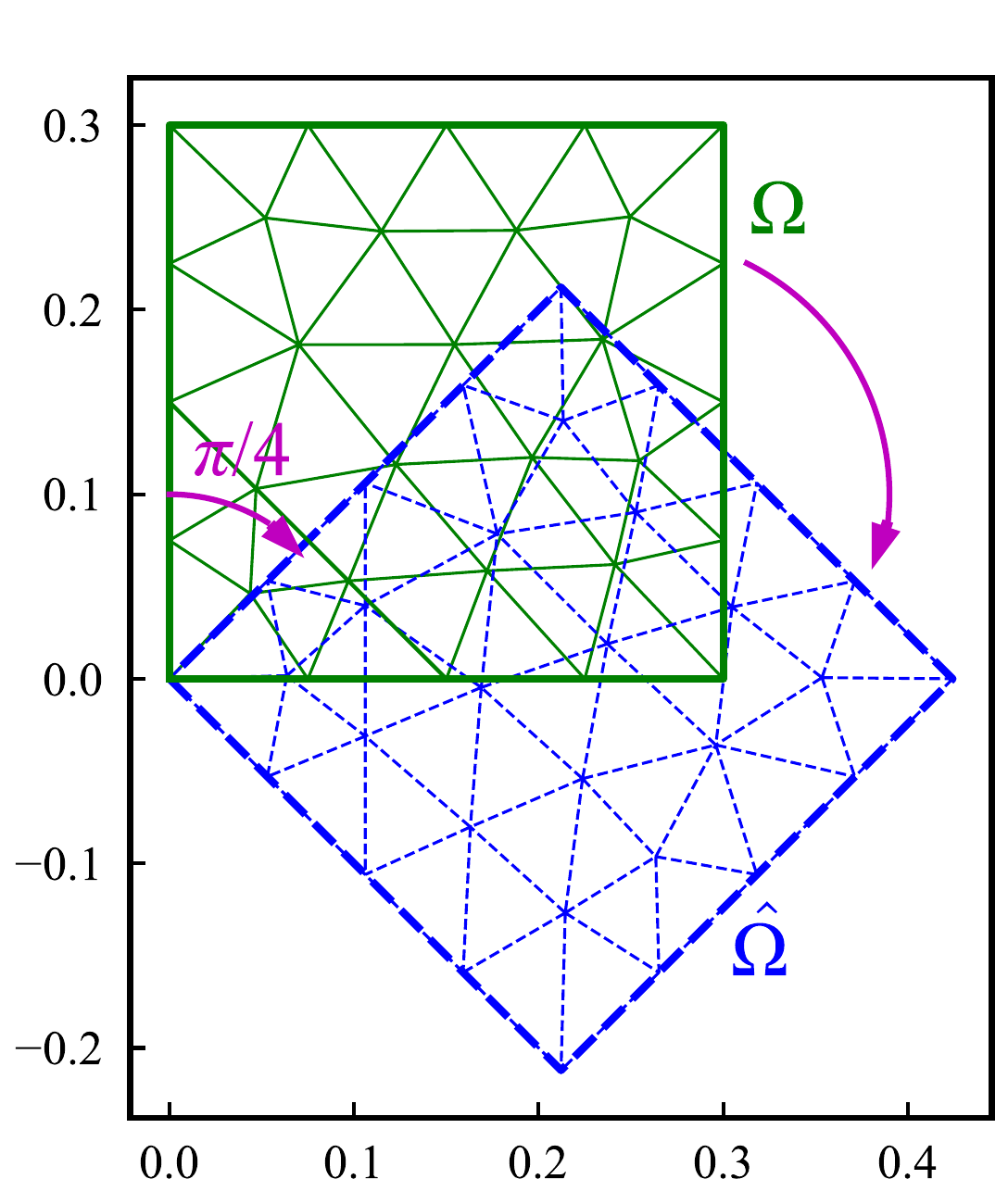}
		\end{subfigure}
		\caption{(\Cref{Ex:EulerImp}) The computational domain $\Omega$ with its corresponding sample mesh $\mathcal{T}_h$, and the rotated domain $\hat{\Omega}$ with its rotated sample mesh $\hat{\mathcal{T}}_h$, obtained after a rotation by $\pi/4$.
		}
		\label{fig:Ex_EulImpMesh}
	\end{figure}

	\begin{figure}[!htbp]
		\centering	
		\begin{subfigure}{0.36\textwidth}
			\centering
			\includegraphics[width=\textwidth]{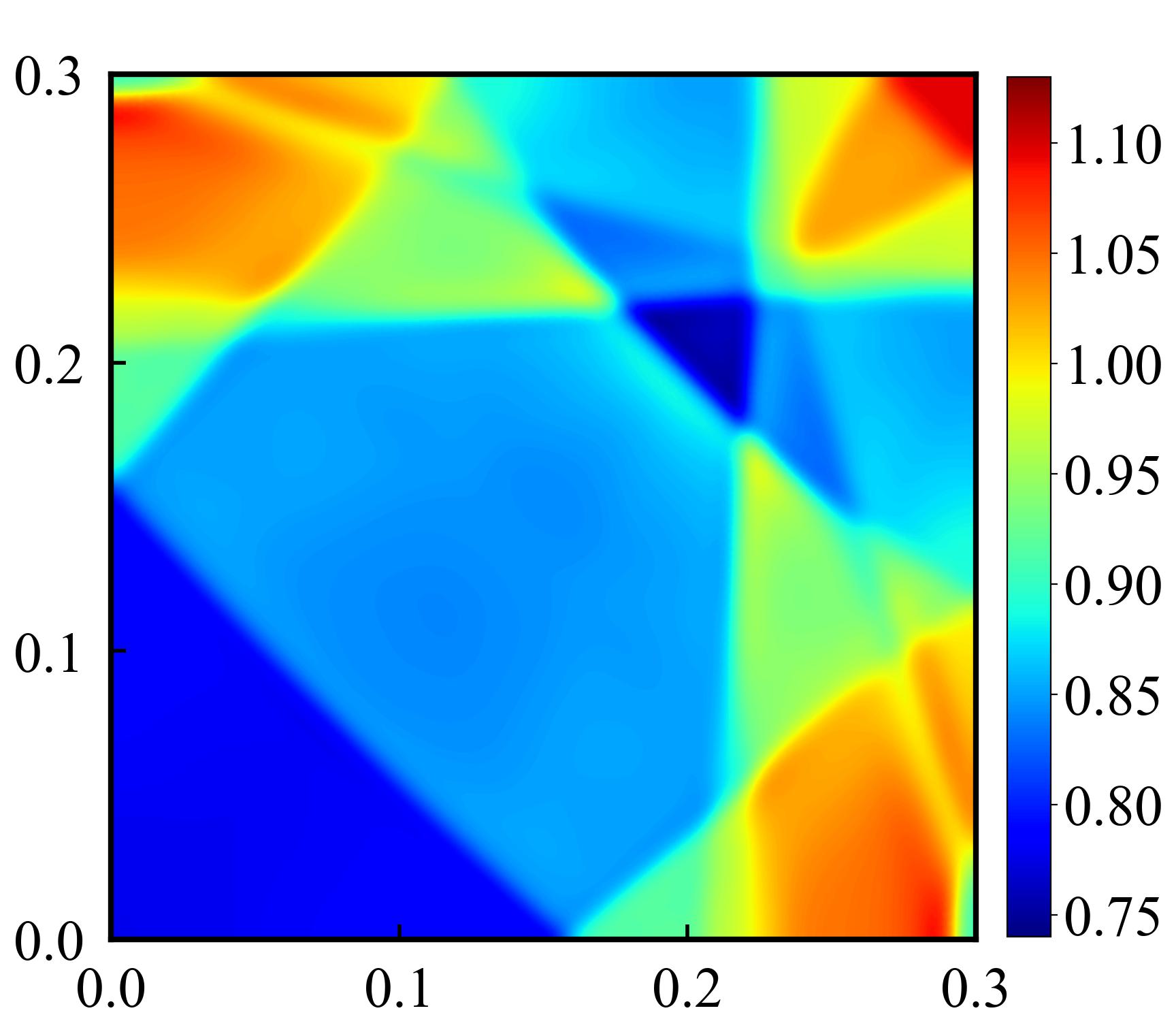}
		\end{subfigure}
		\qquad\qquad
		\begin{subfigure}{0.36\textwidth}
			\centering
			\includegraphics[width=\textwidth]{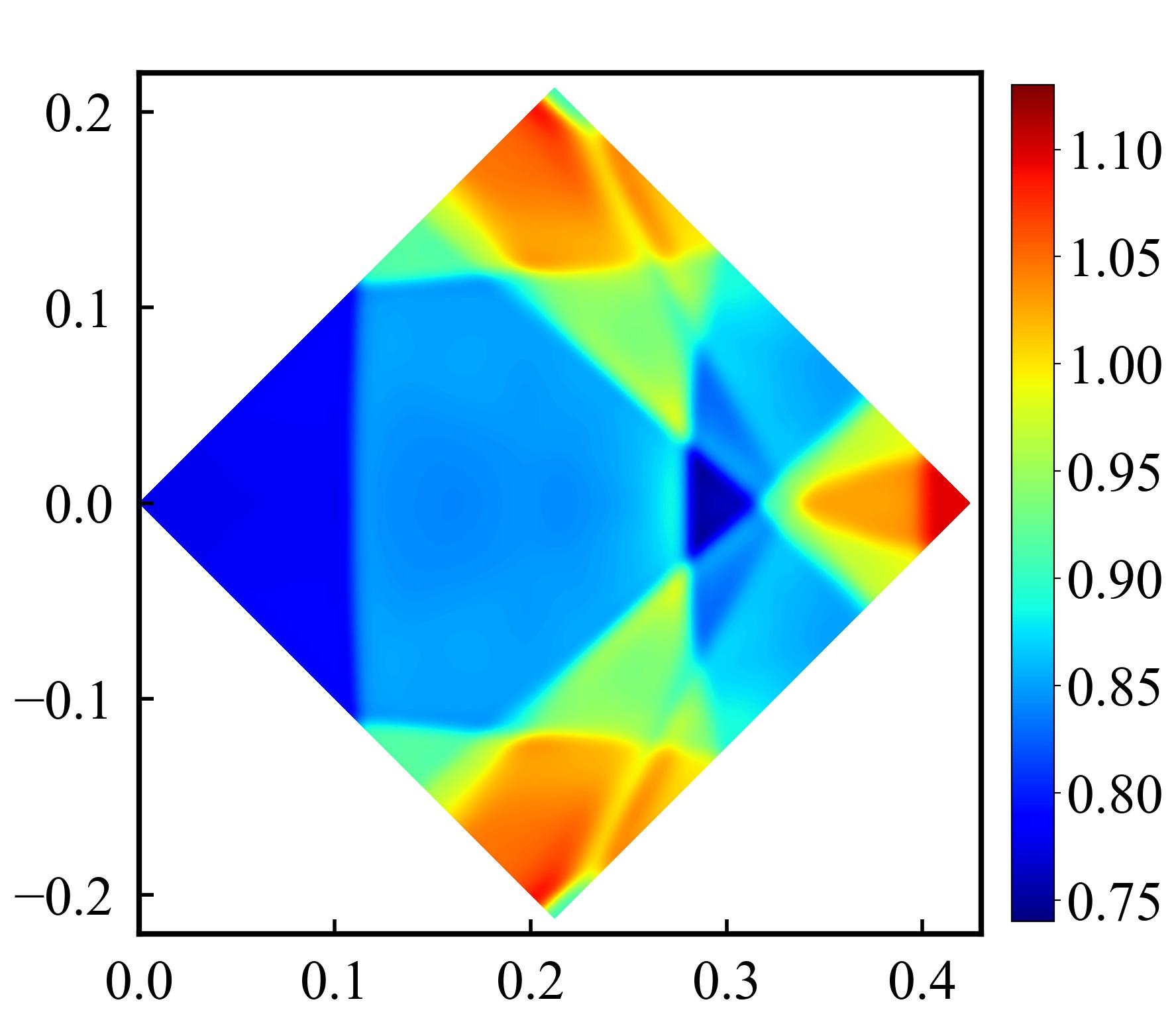}
		\end{subfigure}
		
		\begin{subfigure}{0.37\textwidth}
			\centering
			\includegraphics[width=\textwidth]{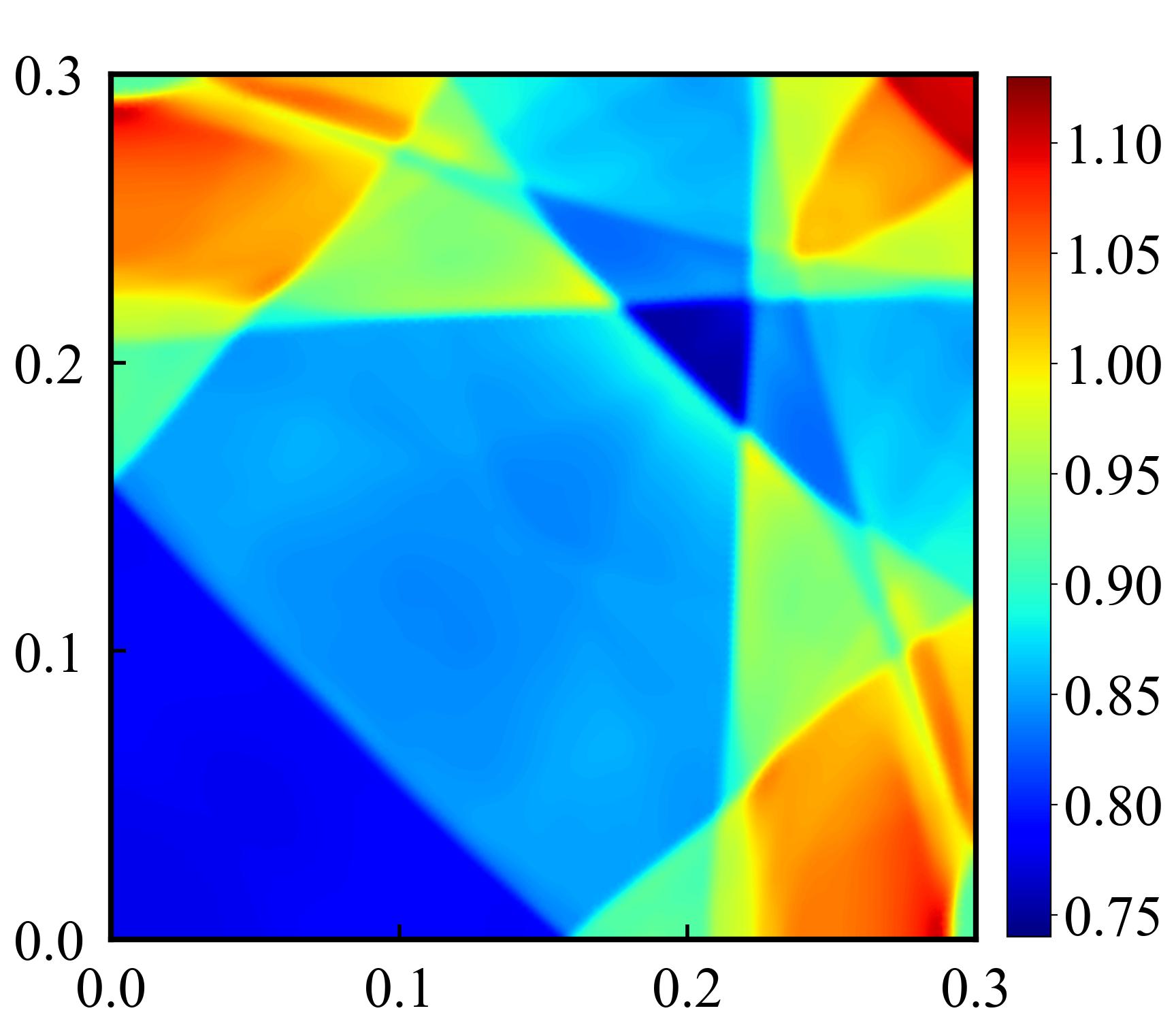}
		\end{subfigure}
		\qquad\qquad
		\begin{subfigure}{0.37\textwidth}
			\centering
			\includegraphics[width=\textwidth]{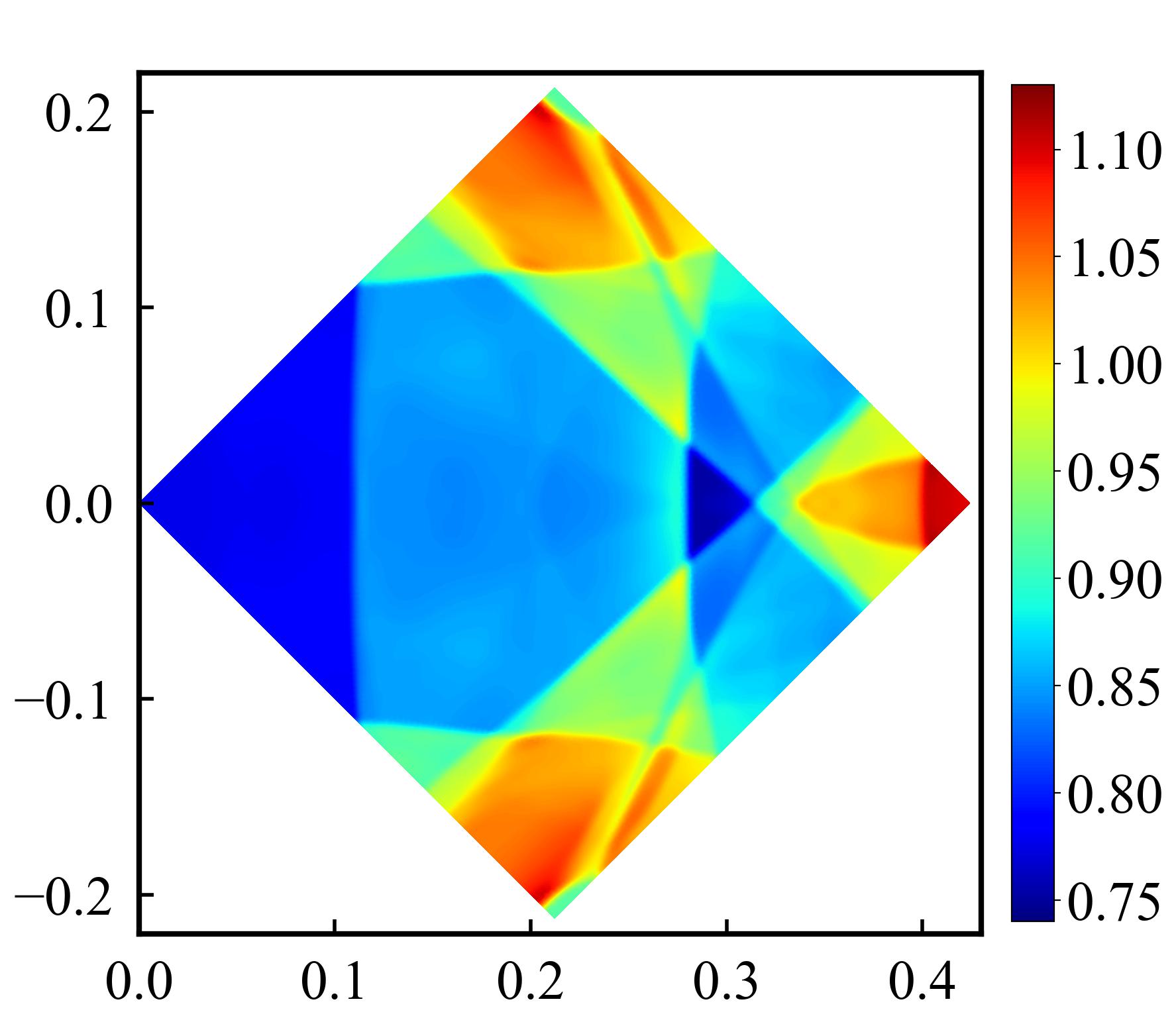}
		\end{subfigure}
		
		\begin{subfigure}{0.37\textwidth}
			\centering
			\includegraphics[width=\textwidth]{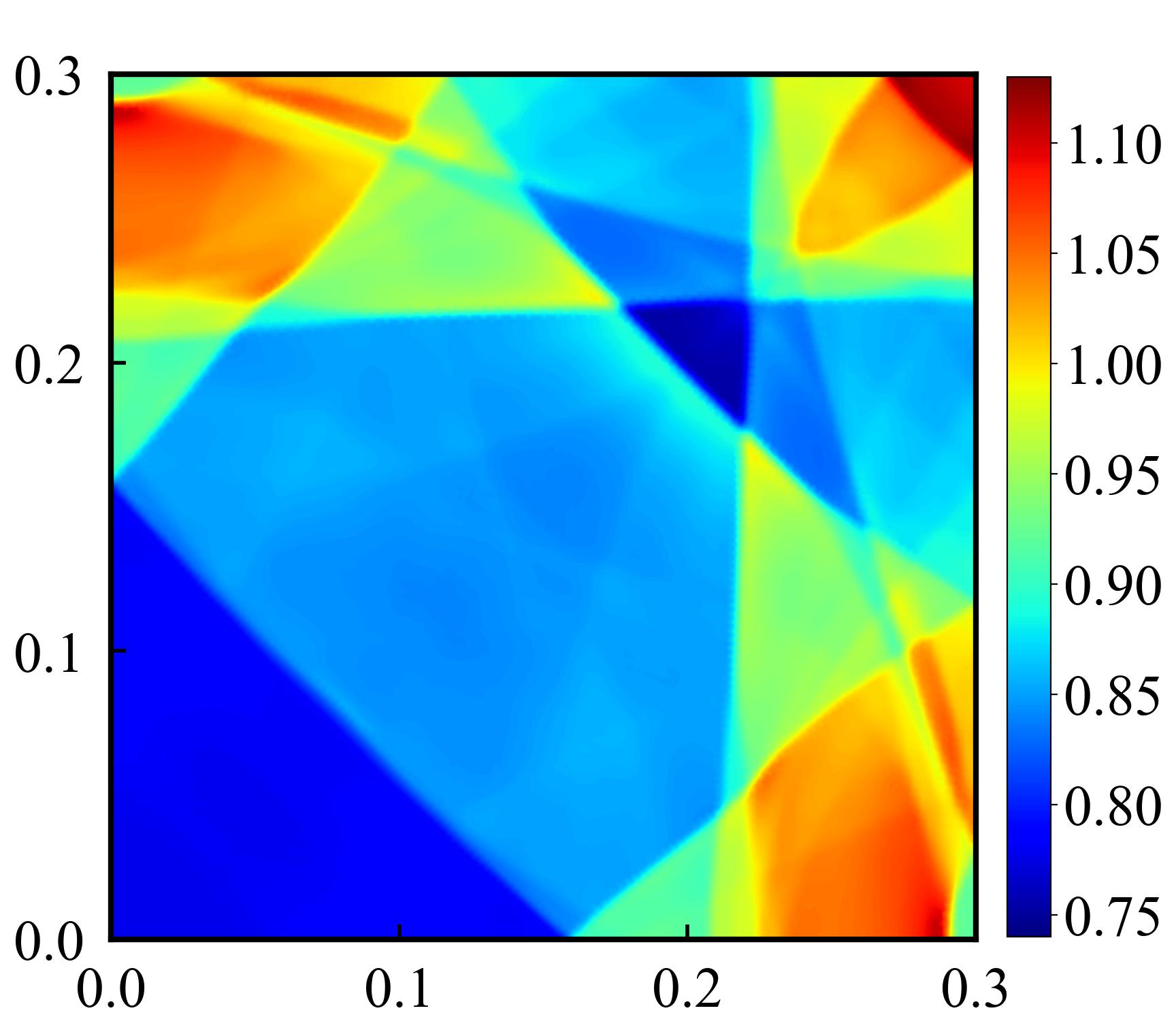}
		\end{subfigure}
		\qquad\qquad
		\begin{subfigure}{0.37\textwidth}
			\centering
			\includegraphics[width=\textwidth]{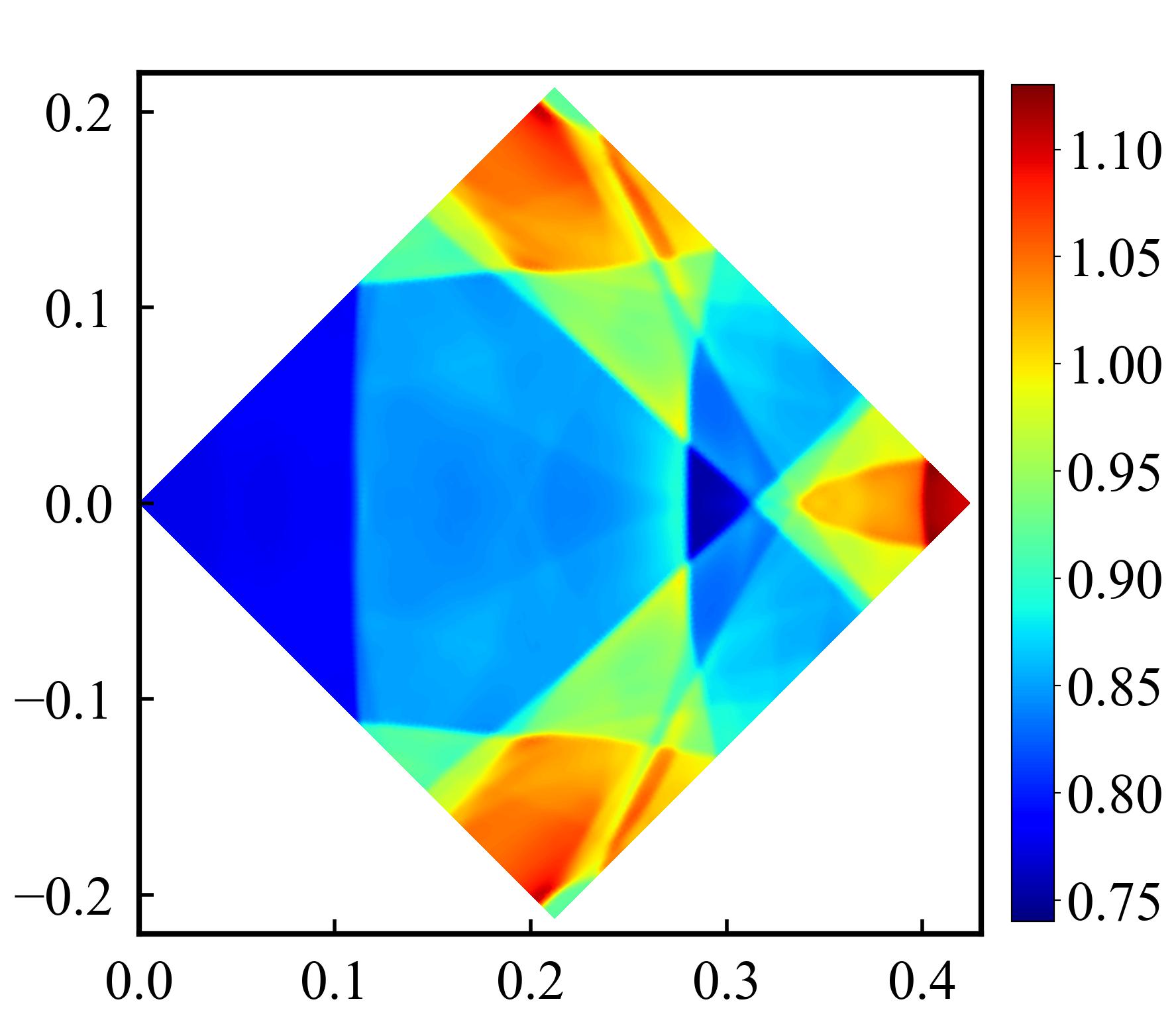}
		\end{subfigure}
	
		\begin{subfigure}{0.37\textwidth}
			\centering
			\includegraphics[width=\textwidth]{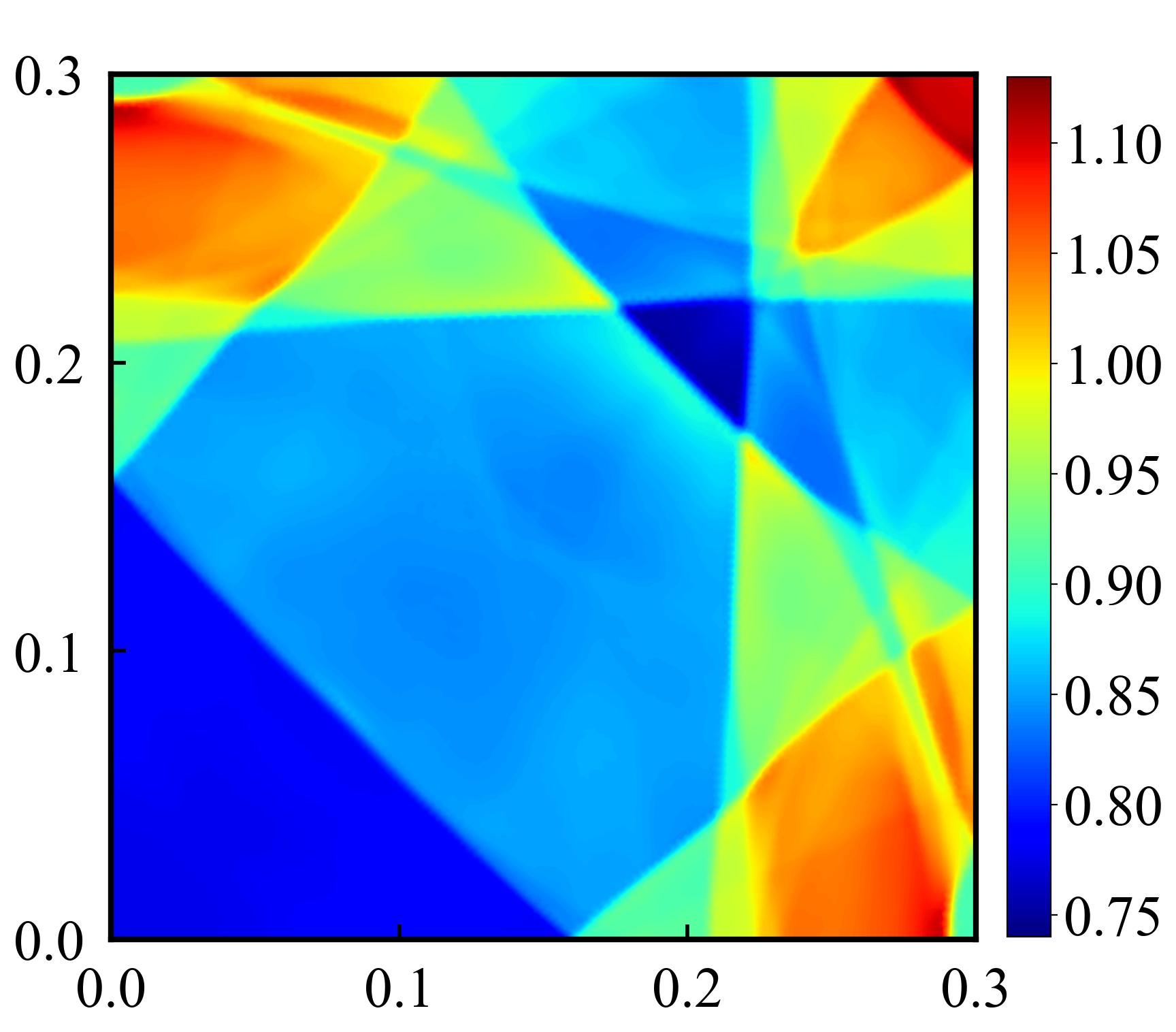}
		\end{subfigure}
		\qquad\qquad
		\begin{subfigure}{0.37\textwidth}
			\centering
			\includegraphics[width=\textwidth]{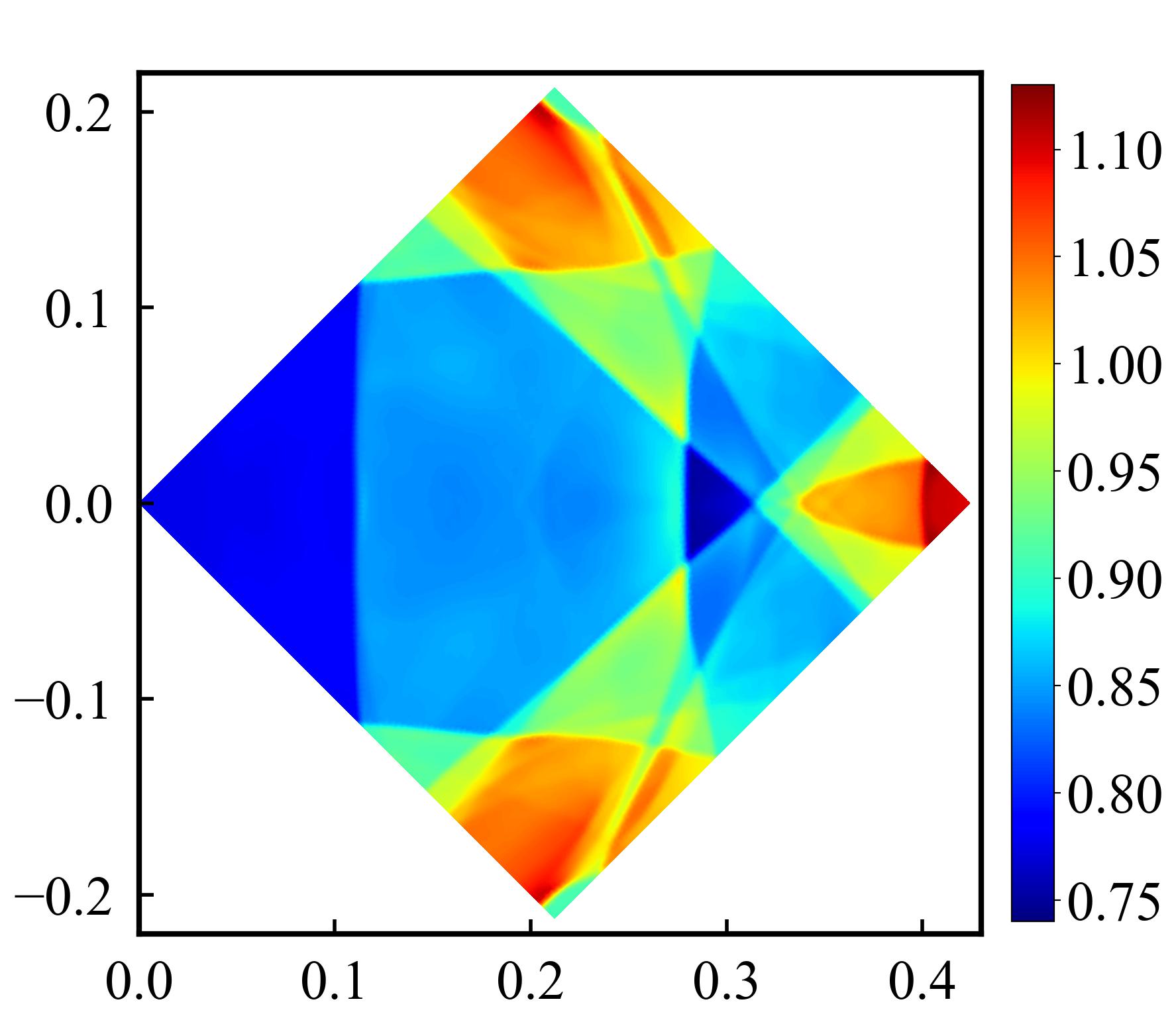}
		\end{subfigure}
		\caption{(\Cref{Ex:EulerImp}) Contours of pressure obtained by RI-OEDG methods. Left: $\Omega$, right: $\hat{\Omega}$. From top to bottom: $\mathbb{P}^1$, $\mathbb{P}^2$, $\mathbb{P}^3$, and $\mathbb{P}^4$. 
		}
		\label{fig:Ex_EulImp}
	\end{figure}

	\begin{figure}[!htb]
		\centering	
		\begin{subfigure}{0.4\textwidth}
			\centering
			\includegraphics[width=\textwidth]{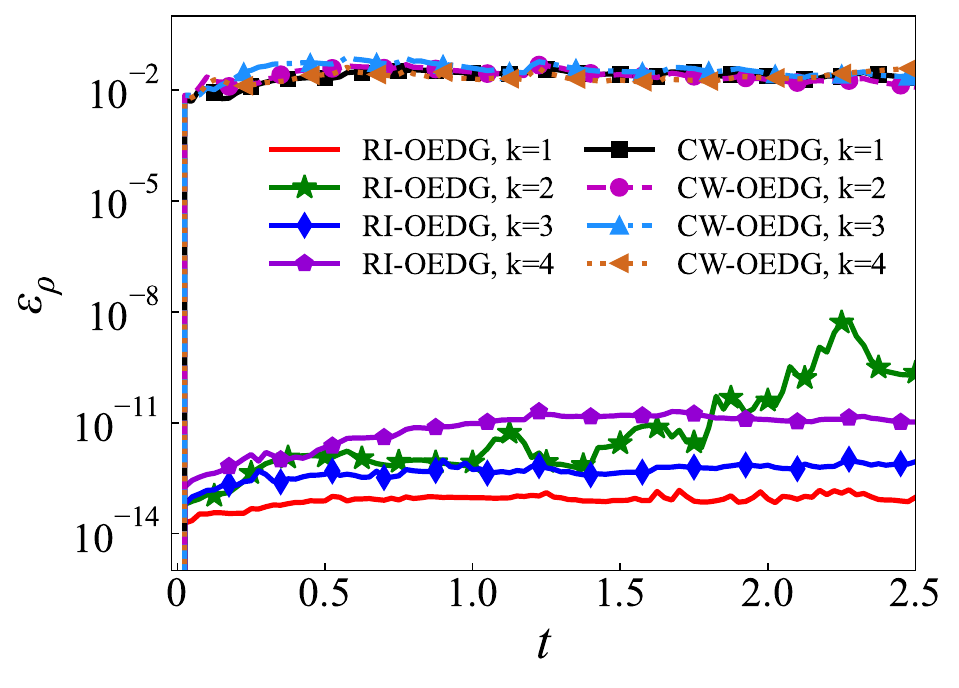}
		\end{subfigure}
		\qquad
		\begin{subfigure}{0.4\textwidth}
			\centering
			\includegraphics[width=\textwidth]{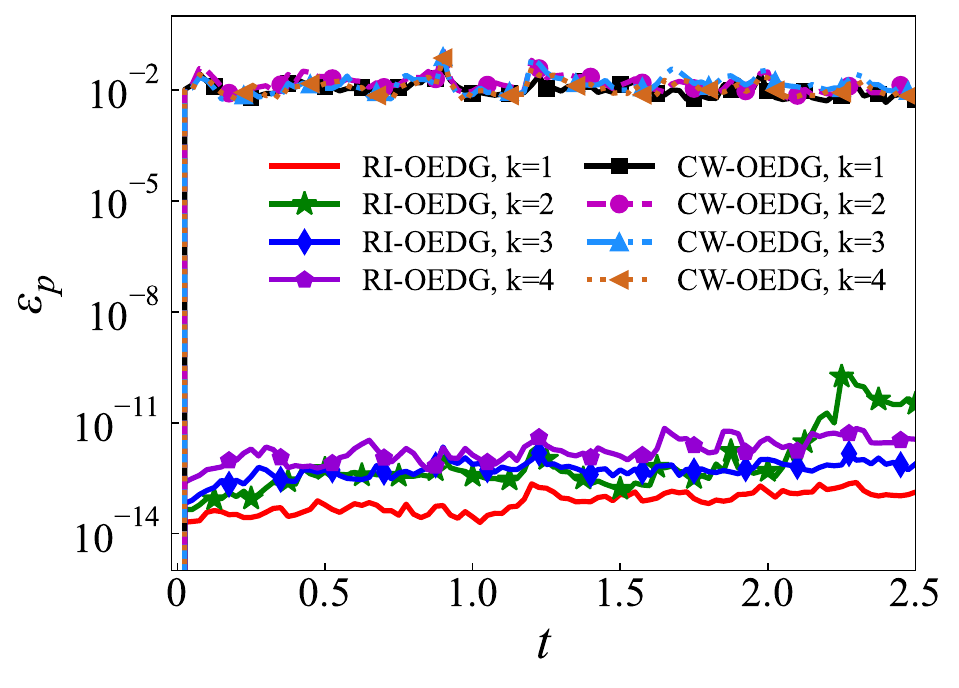}
		\end{subfigure}
		
		\begin{subfigure}{0.4\textwidth}
			\centering
			\includegraphics[width=\textwidth]{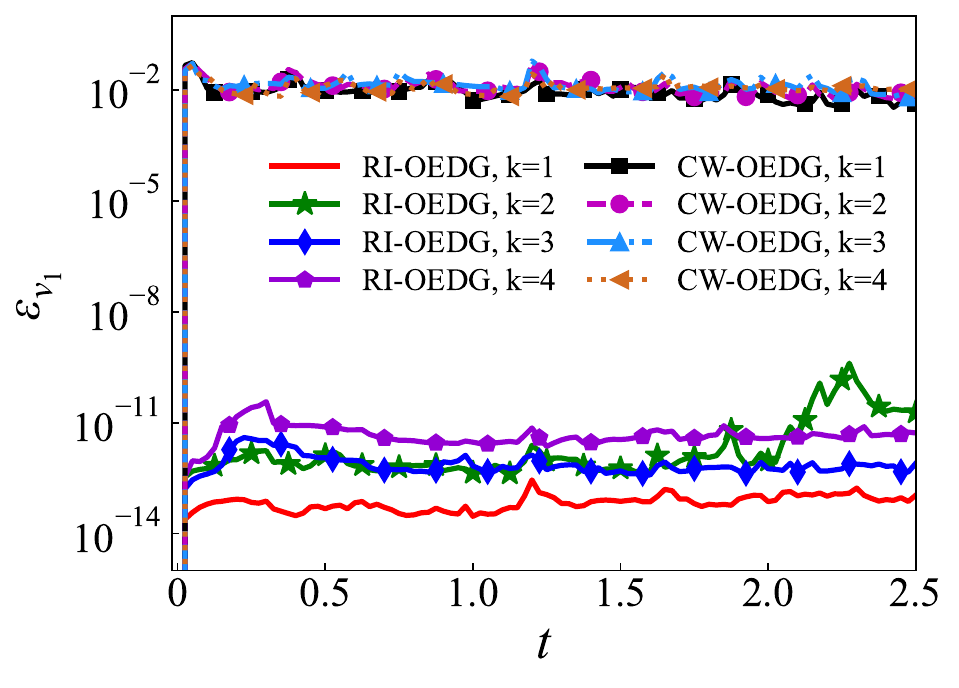}
		\end{subfigure}
		\qquad
		\begin{subfigure}{0.4\textwidth}
			\centering
			\includegraphics[width=\textwidth]{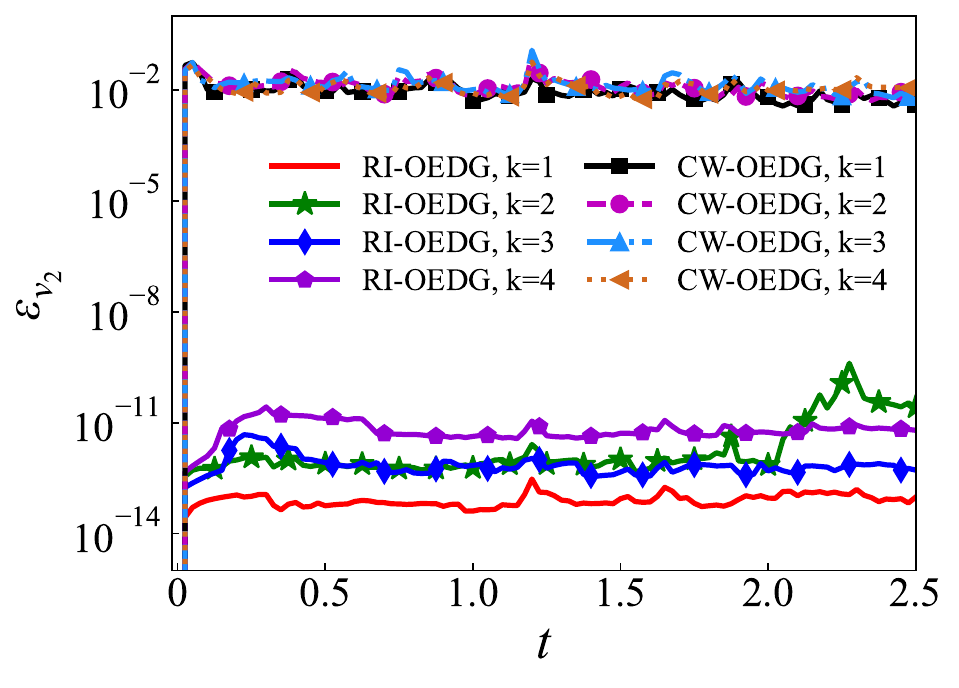}
		\end{subfigure}
		
		\caption{(\Cref{Ex:EulerImp}) Time evolution of the RI errors $\bm \varepsilon_{RI}$ computed every 0.025 time units. 
		}
		\label{fig:Ex_EulImpError}
	\end{figure}
	
\end{expl}

\begin{expl}[Flow past a circular cylinder]\label{Ex:EulerFPC}
	This benchmark problem, described in \cite{Ding2020ILW}, involves a rightward parallel flow with a Mach number of 3 passing a circular cylinder within the domain $[-3,0]\times[-6,6]$. The circular cylinder, with a unit radius, is centered at the origin on the $x$-$y$ plane. The initial conditions are $\rho = 1.4$, $v_1 = 3$, $v_2 = 0$, and $p = 1$. Due to symmetry, only the upper half of the irregular domain $\Omega$ is simulated, which is discretized into 15,973 triangular cells with $h = 0.05$. A sample mesh with $h = 0.2$ is shown in \Cref{fig:Ex_EulFPCMesh}. Inflow boundary condition is applied along the left boundary $\{x=-3, 0\leq y\leq 6\}$, reflective boundary conditions are imposed on the bottom boundary $\{-3 \leq x \leq -1, y=0\}$ and the curved surface of the cylinder $\{-1 \leq x \leq 0, \sqrt{x^2+y^2} = 1\}$, while outflow conditions are enforced on all other boundaries. 
	\Cref{fig:Ex_EulFPC_P1pre}-\Cref{fig:Ex_EulFPC_P2pre} show contour plots of the pressure at $t=40$, by which time the numerical solutions have reached a steady state, obtained using the $\mathbb{P}^1$- and $\mathbb{P}^2$-based RI-OEDG methods. The bow shocks are well-captured, and our results are comparable to those reported in \cite{Ding2020ILW}.

	\begin{figure}[!htb]
		\centering
		\begin{subfigure}{0.2\textwidth}
			\raggedright
			\includegraphics[width=\textwidth]{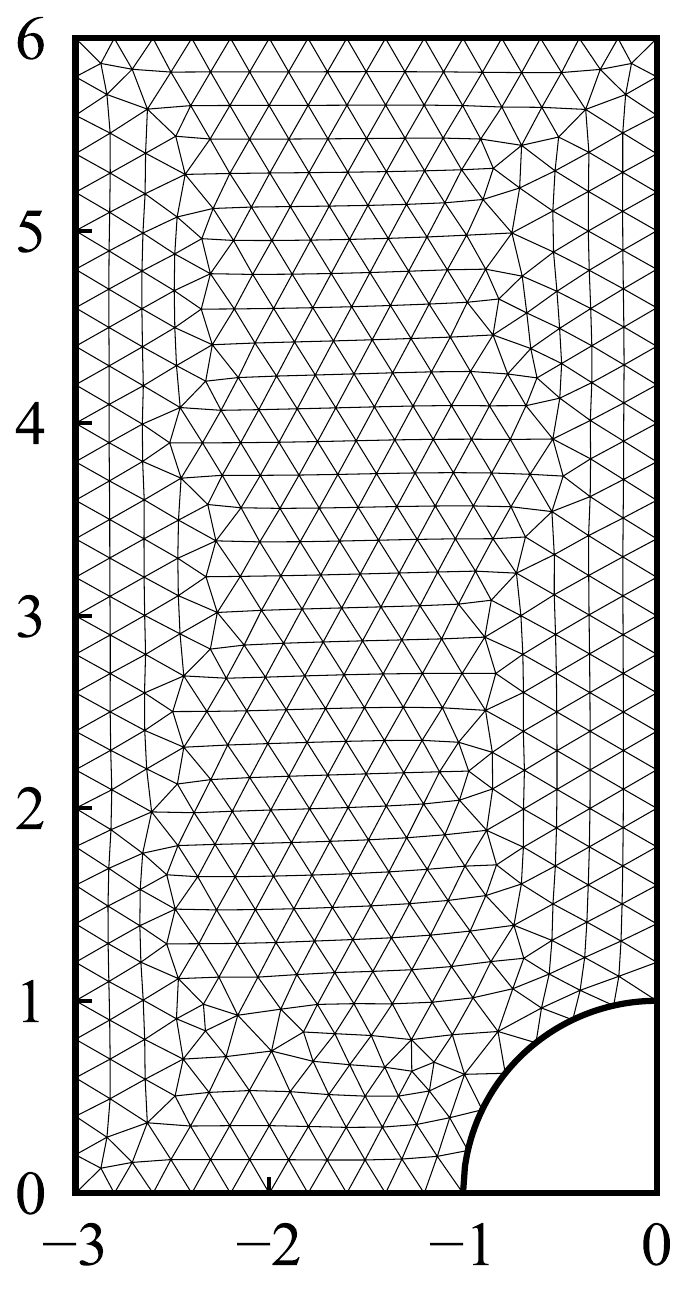}
			\caption{Sample mesh with $h=0.2$.} 
			\label{fig:Ex_EulFPCMesh}
		\end{subfigure}
		\qquad\qquad\qquad
		\begin{subfigure}{0.2\textwidth}
			\includegraphics[width=\textwidth]{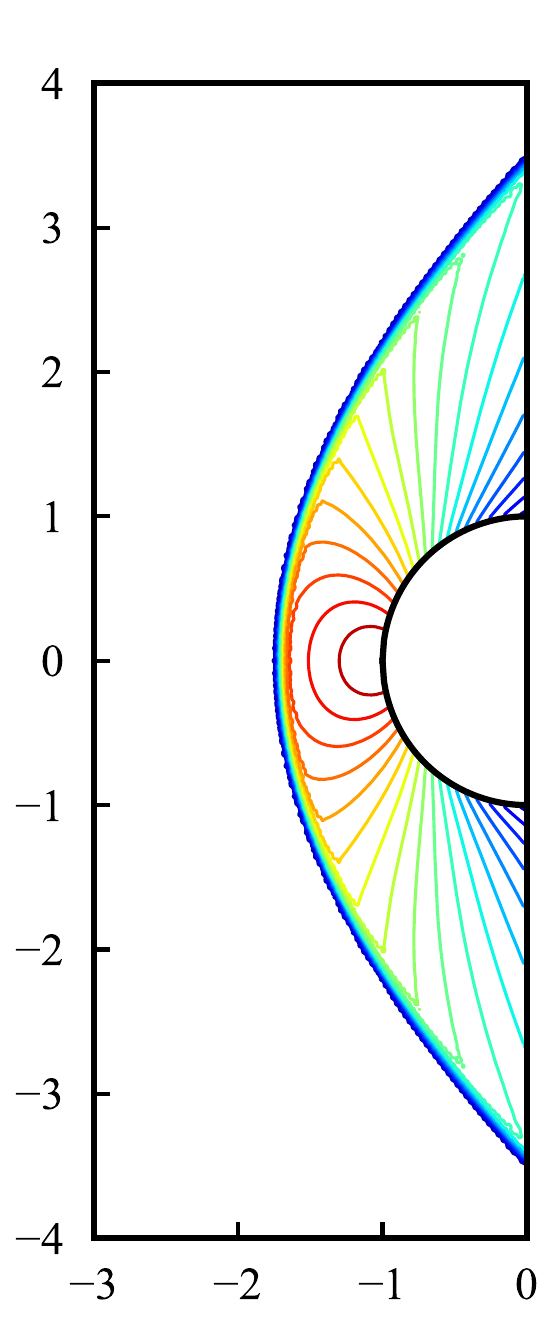}
			\caption{$\mathbb{P}^1$-based RI-OEDG method.}
			\label{fig:Ex_EulFPC_P1pre}
		\end{subfigure}
		\qquad
		\begin{subfigure}{0.2\textwidth}
			\centering
			\includegraphics[width=\textwidth]{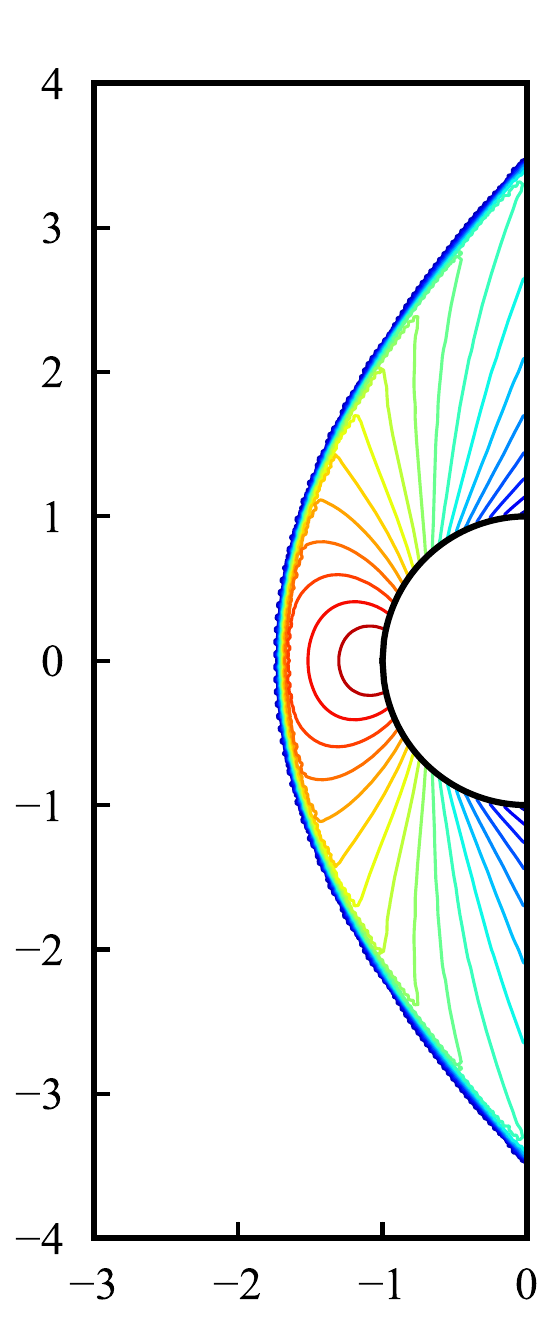}
			\caption{$\mathbb{P}^2$-based RI-OEDG method.}
			\label{fig:Ex_EulFPC_P2pre}
		\end{subfigure}
		
		\caption{(\Cref{Ex:EulerFPC}) Sample mesh and contour plots of pressure $p$ with twenty equally spaced contour lines from 0.9 to 12.1.
		}
		\label{fig:Ex_EulFPC}
	\end{figure}
\end{expl}

\begin{expl}[Flow past a forward facing step]\label{Ex:EulerFFS}
	This problem simulates a rightward uniform flow with a Mach number of 3 entering a wind tunnel, following the setup detailed in \cite{rkdg5}. The tunnel has a width of 1 and a length of 3, with a step of height 0.2 positioned at $x = 0.6$, as shown in \Cref{fig:Ex_EulFFSmesh}. For the simulations, we use a triangular mesh consisting of 73,862 cells. The mesh is designed with a size of $h=\frac{1}{80}$ away from the step corner and $h=\frac{1}{320}$ near the corner, specifically within the rectangle $[0.5, 0.7] \times [0.1, 0.3]$. \Cref{fig:Ex_EulFFSmesh} shows a sample mesh with a larger mesh size for illustration purposes. 
	The initial conditions are set as $\rho = 1.4$ for density, $(v_1,v_2)^\top=(3,0)^\top$ for velocity, and $p = 1$ for pressure. Inflow conditions are applied at the left boundary, and outflow conditions at the right boundary, with reflective conditions on all other boundaries of the tunnel. To ensure stability during simulations, we apply the BP limiter (see \cref{sec:BPlimiter}) to maintain the positivity of both density and pressure.  The RI-OEDG code would quickly break down due to nonphysical numerical solutions if we turn off the BP limiter.
 
	\Cref{fig:Ex_EulFFS} presents the density contours $\rho$ at $t=4$, computed using the $\mathbb{P}^1$- and $\mathbb{P}^2$-based BP RI-OEDG schemes with either the optimal or the classical convex decomposition. The results show good agreement with those reported in \cite{rkdg5}, demonstrating the  effectiveness of our BP schemes with optimal convex decomposition. Notably, the CPU times using the optimal convex decomposition are significantly lower than those with the classical convex decomposition, as detailed in \Cref{tab:Ex_EulFFS}. 
	To quantify the efficiency, let $\overline{\dt}$ denote the average time step size, computed as $\overline{\dt} = \frac{1}{N_t}\sum_{n=1}^{N_t} \dt^n$ over $N_t$ time steps. The ratios of average time step sizes $\overline{\dt}^{\tt CDW}/\overline{\dt}^{\tt ZXS}$ are 2.851852 for $k=1$ and 4.489796 for $k=2$, respectively. These ratios align with the BP CFL numbers shown in \Cref{fig:1744}, further demonstrating the higher efficiency of our BP DG method with the optimal convex decomposition compared to the classical convex decomposition.

	\begin{figure}[!htb]
		\centering
		\includegraphics[width=0.8\textwidth]{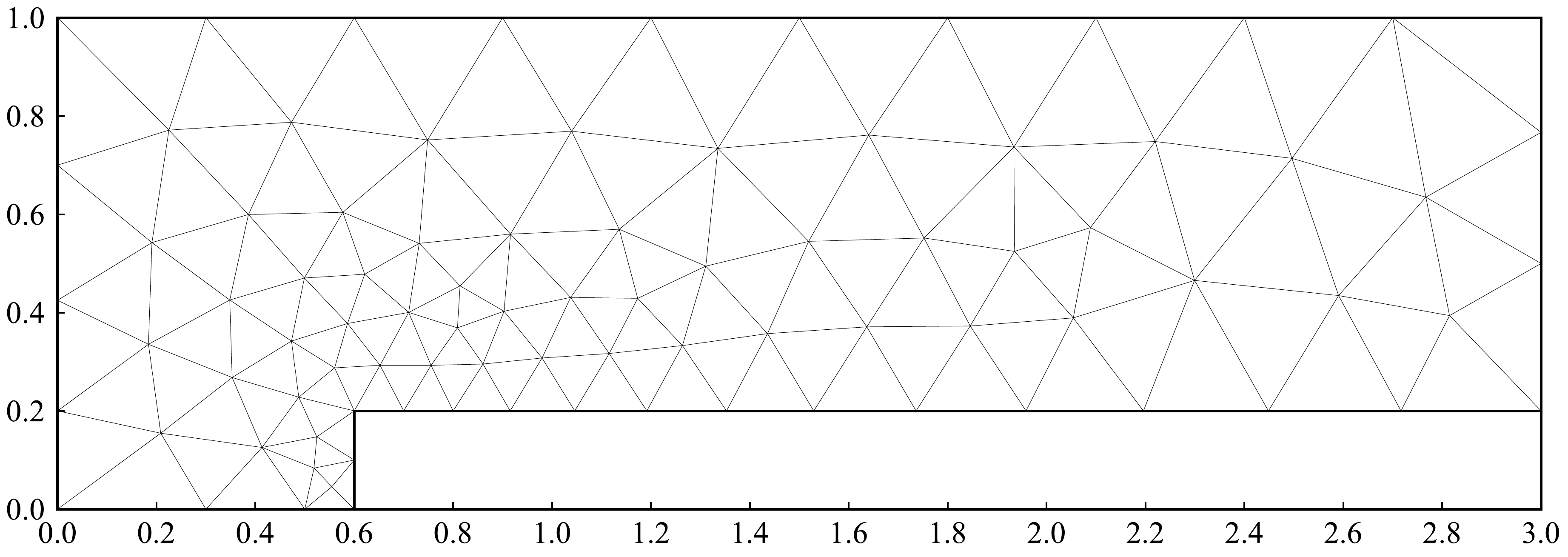}
		\caption{Sample mesh with $h=0.2$ away from the corner and $h=0.05$ near the corner.}
		\label{fig:Ex_EulFFSmesh}
	\end{figure}
	
	\begin{figure}[!htb]
		\centering
		\begin{subfigure}{0.48\textwidth}
			\centering
			\includegraphics[width=\textwidth]{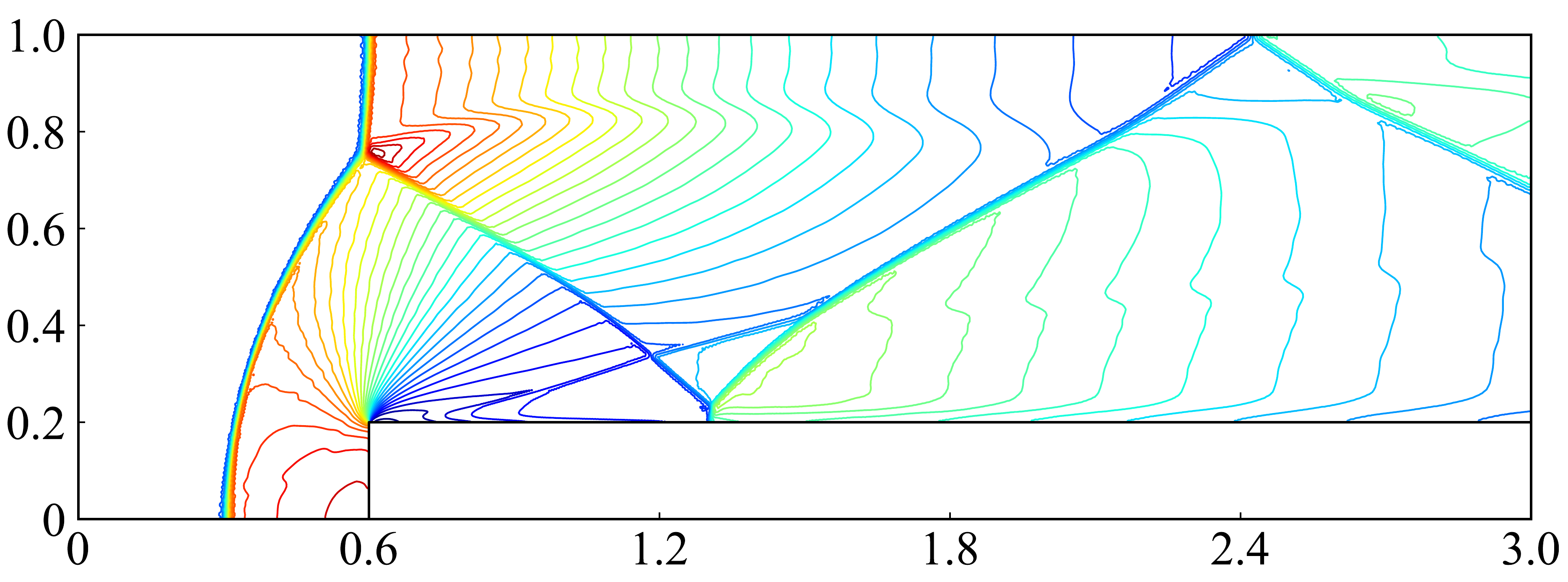}
		\end{subfigure}
		\hfill
		\begin{subfigure}{0.48\textwidth}
			\centering
			\includegraphics[width=\textwidth]{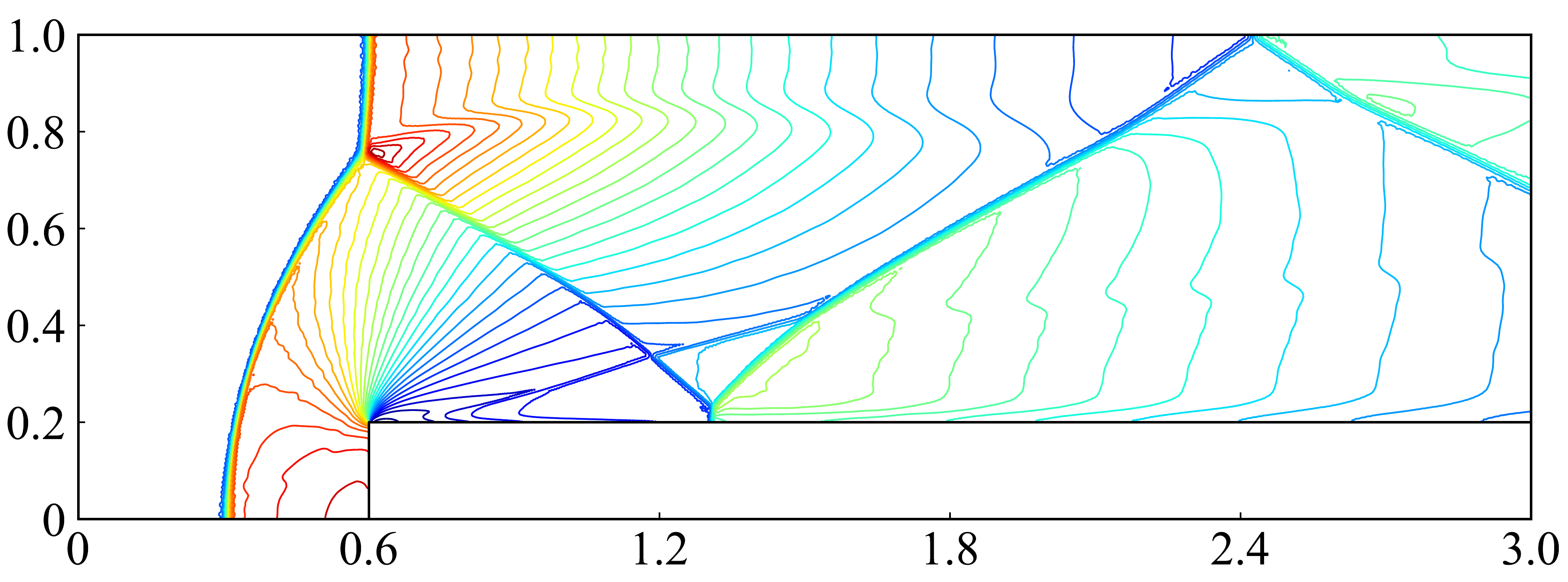}
		\end{subfigure}
		
		\begin{subfigure}{0.48\textwidth}
			\centering
			\includegraphics[width=\textwidth]{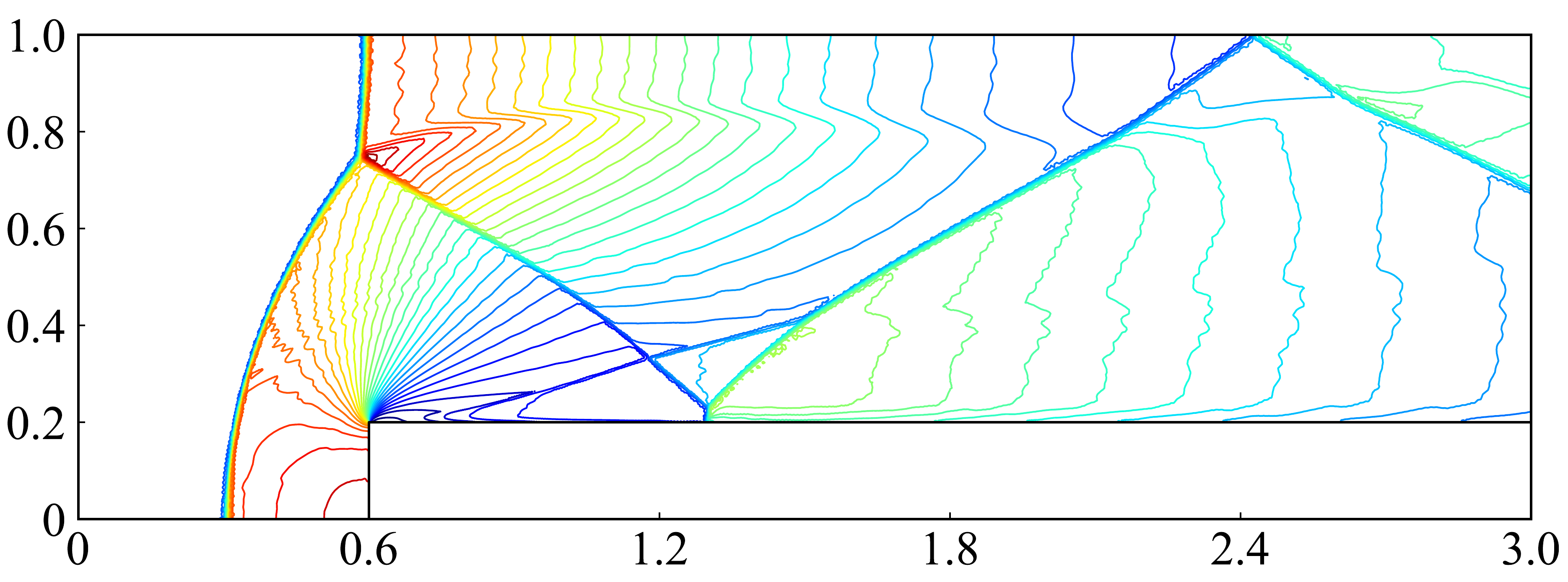}
		\end{subfigure}
		\hfill
		\begin{subfigure}{0.48\textwidth}
			\centering
			\includegraphics[width=\textwidth]{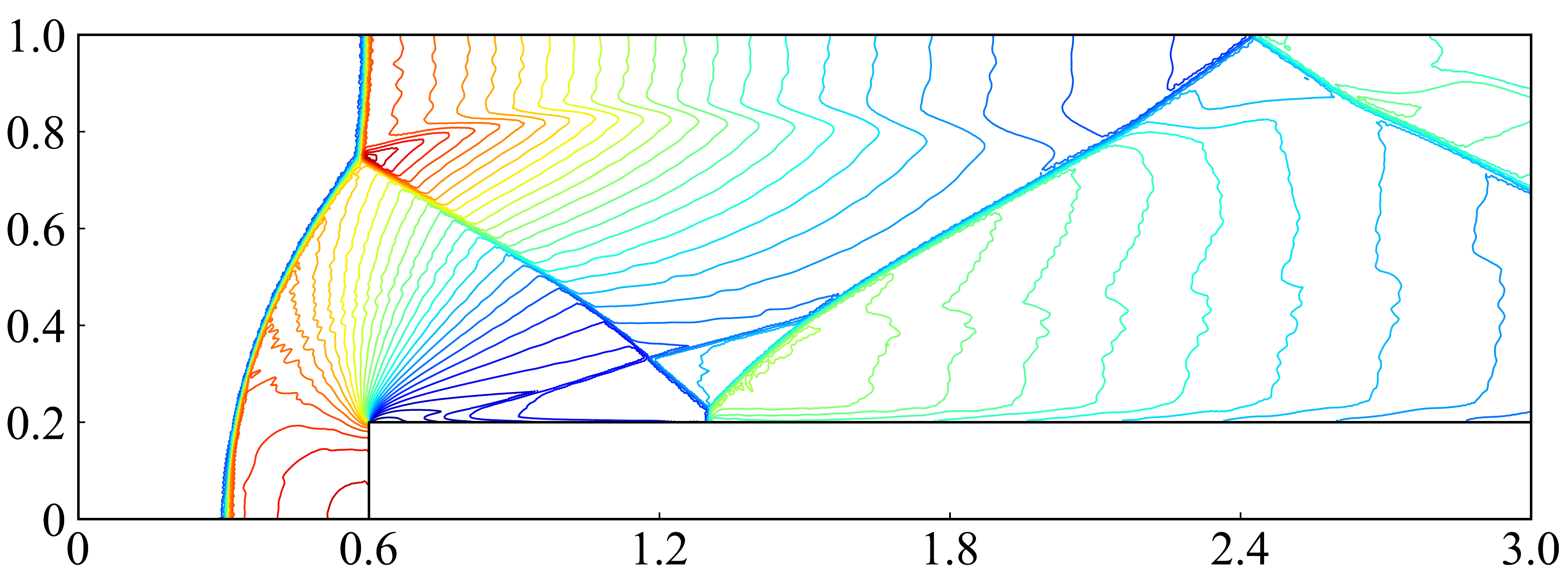}
		\end{subfigure}
		
		\caption{(\Cref{Ex:EulerFFS}) Contours plots of density obtained by $\mathbb{P}^1$-based (top) and $\mathbb{P}^2$-based (bottom) BP RI-OEDG methods with thirty equally spaced plots from 0.2 to 6.36.
		Left: optimal convex decomposition; right: classical convex decomposition.
		}
		\label{fig:Ex_EulFFS}
	\end{figure}

	\begin{table}[!htbp] 
		\centering
		\caption{(\Cref{Ex:EulerFFS}) CPU time (hours) computed by using the $\mathbb{P}^1$- and $\mathbb{P}^2$-based BP OEDG methods. 
		}
		\label{tab:Ex_EulFFS}
		\setlength{\tabcolsep}{4mm}{
			\begin{tabular}{ccc}
				\toprule[1.5pt]
				&
				\multicolumn{1}{c}{optimal convex decomposition} &
				\multicolumn{1}{c}{classical convex decomposition}  \\
				
				\midrule[1.5pt]		
				{$\mathbb{P}^1$} & 95.46  & 250.07  \\
				{$\mathbb{P}^2$} & 319.97 & 1274.51 \\
				
				\bottomrule[1.5pt]
			\end{tabular}
		}
	\end{table}
\end{expl}

\begin{figure}[!htb]
	\centering
	\begin{subfigure}{0.4\textwidth}
		\centering
		\includegraphics[width=\textwidth]{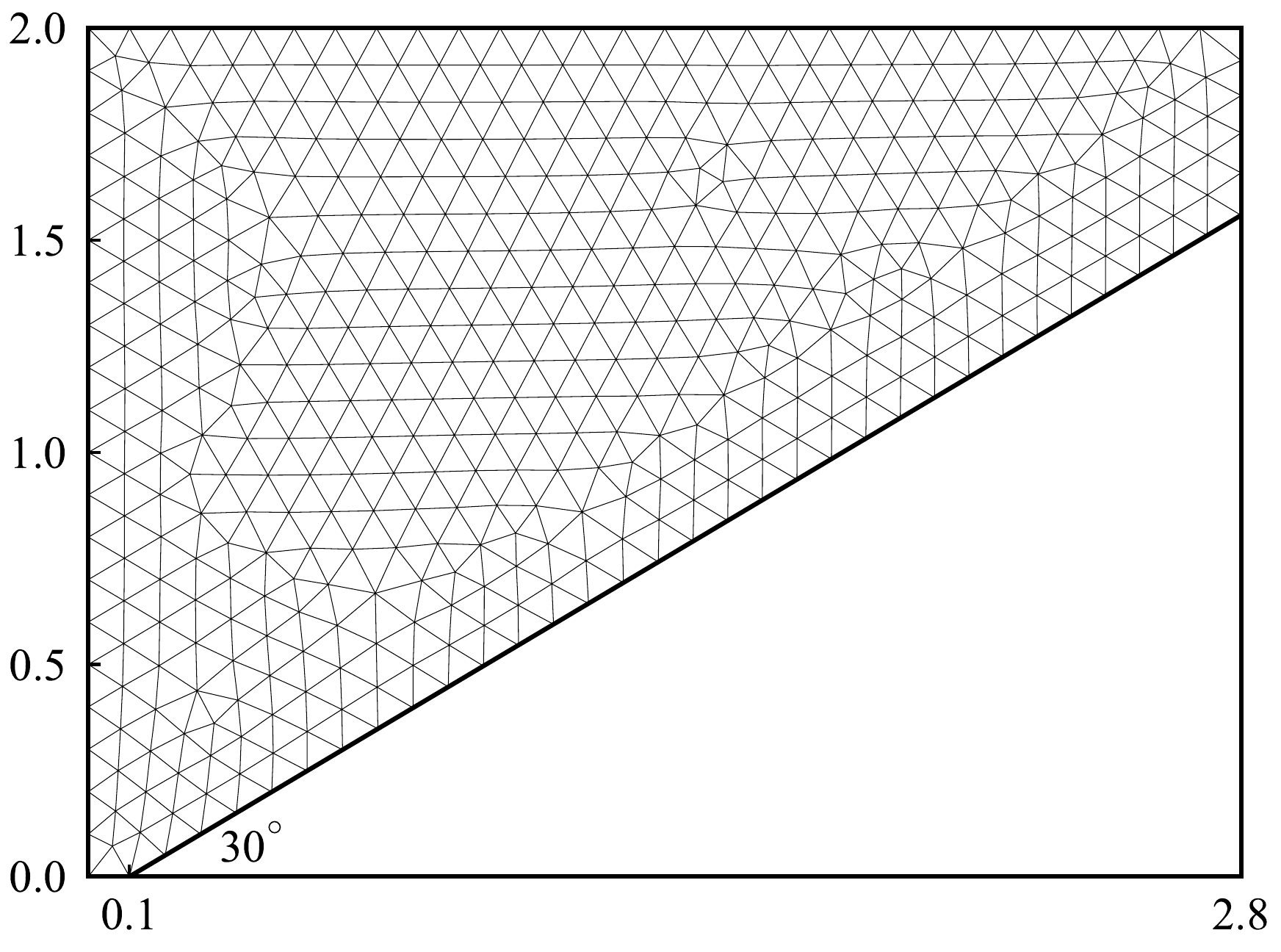}
		\caption{\Cref{Ex:EulerDMR}.}
		\label{fig:Ex_EulDMRmesh}
	\end{subfigure}
	\qquad\qquad
	\begin{subfigure}{0.34\textwidth}
		\centering
		\includegraphics[width=\textwidth]{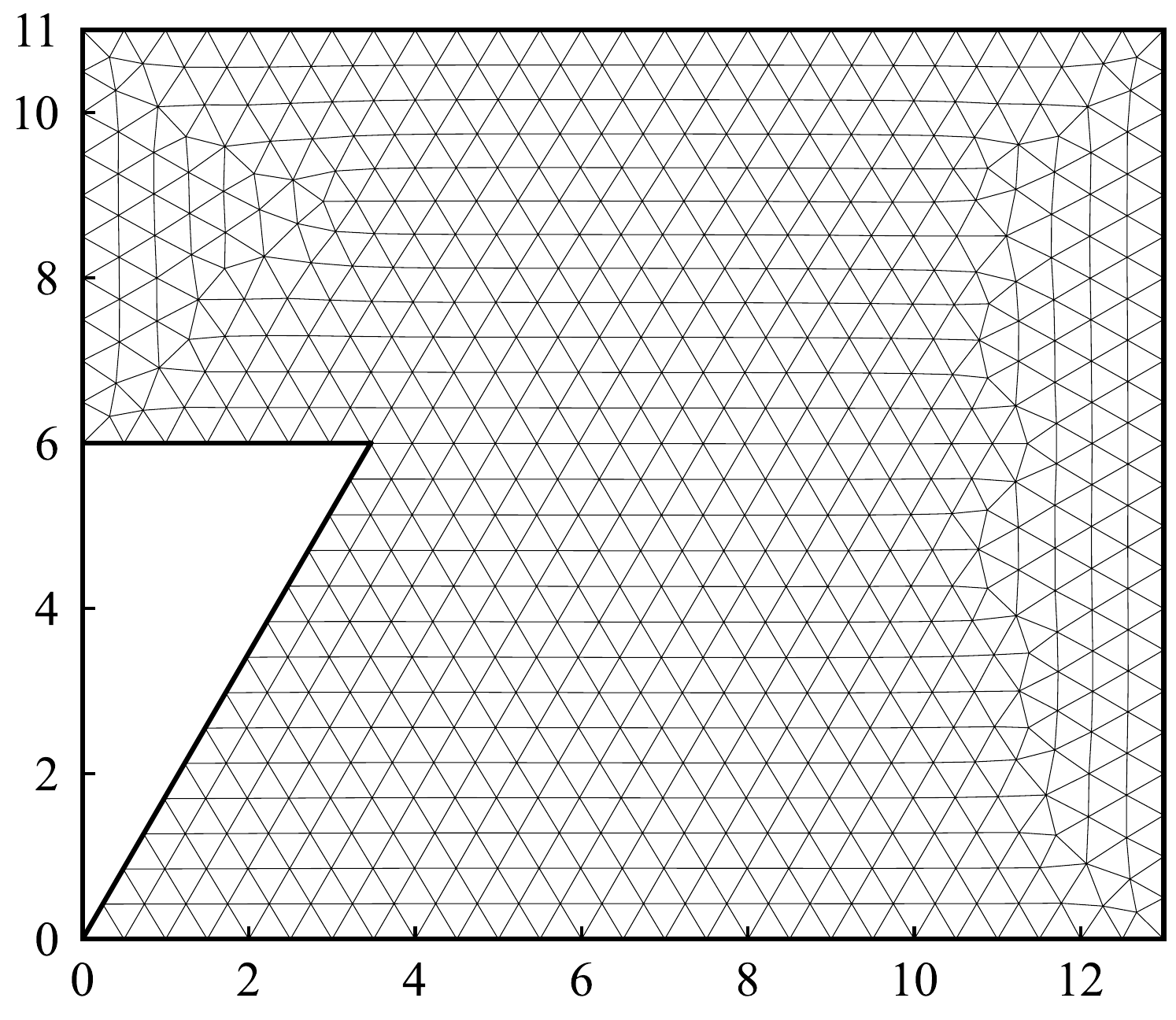}
		\caption{\Cref{Ex:EulerSDiffWedge}.}
		\label{fig:Ex_EulSDWmesh}
	\end{subfigure}
	\caption{Illustration of domain and sample mesh. 
	}
\end{figure}

\begin{expl}[Double Mach reflection]\label{Ex:EulerDMR}
	The double Mach reflection problem, originally introduced in \cite{Woodward1984}, is widely used to test the robustness of numerical schemes. In this study, we simulate the problem in a computational domain $\Omega$, modeled as a tube containing a $30^{\circ}$ wedge. The domain $\Omega$ is discretized using unstructured triangular cells with a mesh size of $h=\frac{1}{320}$, resulting in a total of 828,840 cells; a sample mesh with $h=0.1$ is shown in \Cref{fig:Ex_EulDMRmesh}. 
	Initially, a rightward shock with a Mach number of 10 is located at $x=0.1$. The initial conditions are specified as:
	\begin{equation*}
		(\rho, v_1, v_2, p) = 
		\begin{cases}
			(8, 8.25, 0, 116.5), & \text{if } x \leq 0.1, \\
			(1.4, 0, 0, 1), & \text{otherwise}.
		\end{cases}
	\end{equation*}
	Inflow conditions are applied on the left boundary ($x=0$), outflow conditions on the right boundary ($x=2.8$), and the exact solution is imposed on $\{0 \leq x \leq 0.1, y=0\}$ and the top boundary ($y=2$). Reflective conditions are imposed on the remaining boundaries. Due to the presence of low-density and low-pressure regions, the BP limiter described in \Cref{sec:BPlimiter} is necessary to ensure the positivity of density and pressure. The RI-OEDG code would fail at the second RK stage in the first time step if the BP limiter is turned off.

	\Cref{fig:Ex_EulDMR} displays the density contours at $t=0.2$, computed using the $\mathbb{P}^1$- and $\mathbb{P}^2$-based BP RI-OEDG methods with either the optimal or the classical convex decomposition. Additionally, \Cref{fig:Ex_EulDMRZoom} provides zoomed-in views around the double Mach stem, showing that the resolution improves as the polynomial degree $k$ increases for the same mesh, as expected. 
	Our results demonstrate that the RI-OEDG methods effectively capture complex flow structures without spurious oscillations near discontinuities, even at strong shock interactions. The CPU times for the methods based on the optimal convex decomposition are significantly lower than those based on the classical convex decomposition, as shown in \Cref{tab:Ex_EulDMR}. The computed ratios $\overline{\dt}^{\tt CDW}/\overline{\dt}^{\tt ZXS}$ are 2.869570 for $k=1$ and 4.486689 for $k=2$, indicating that using the optimal convex decomposition greatly enhances the computational efficiency of BP DG schemes.

	\begin{figure}[!htb]
		\centering
		\begin{subfigure}{0.43\textwidth}
			\centering
			\includegraphics[width=\textwidth]{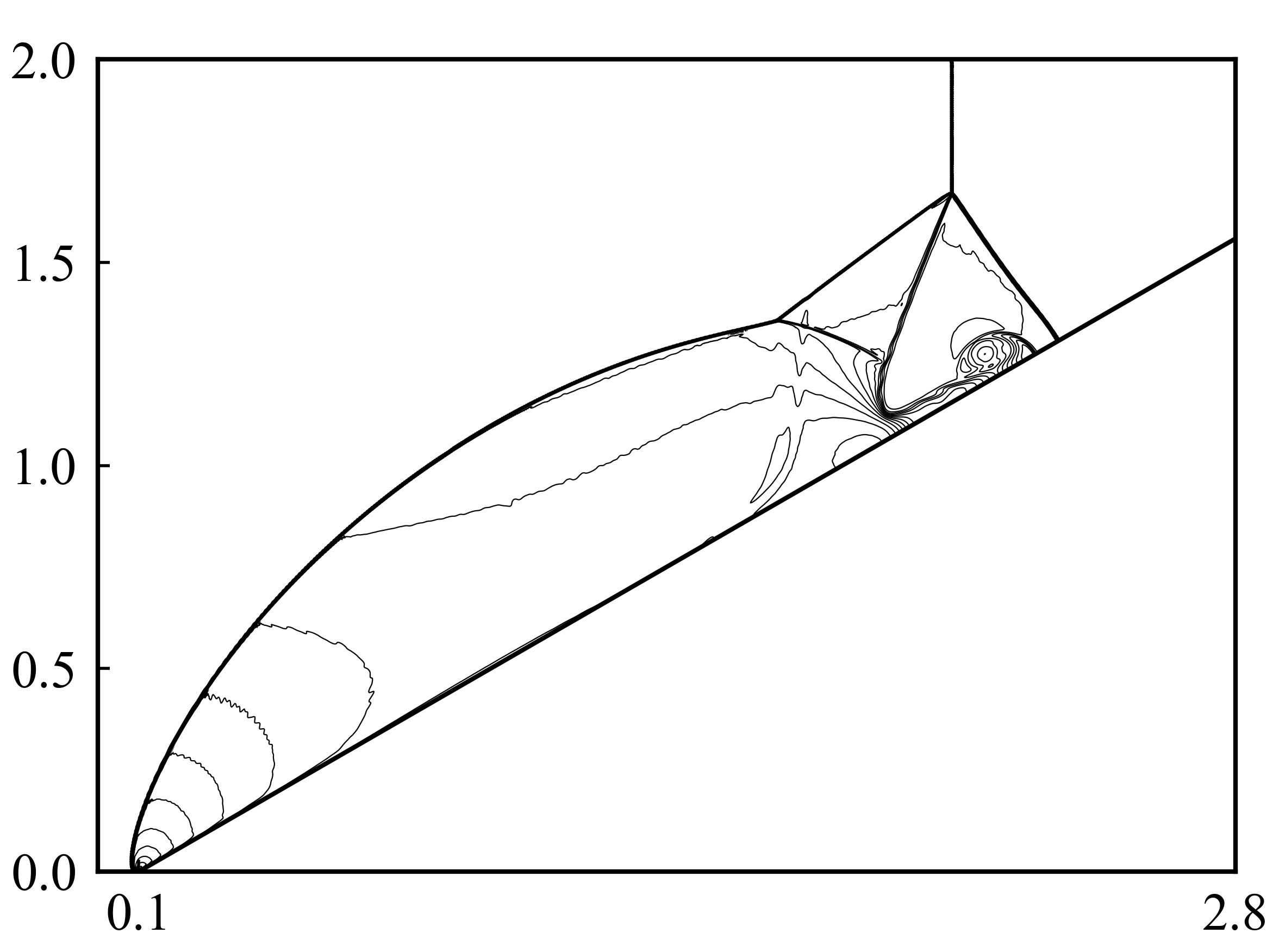}
		\end{subfigure}
		\qquad
		\begin{subfigure}{0.43\textwidth}
			\centering
			\includegraphics[width=\textwidth]{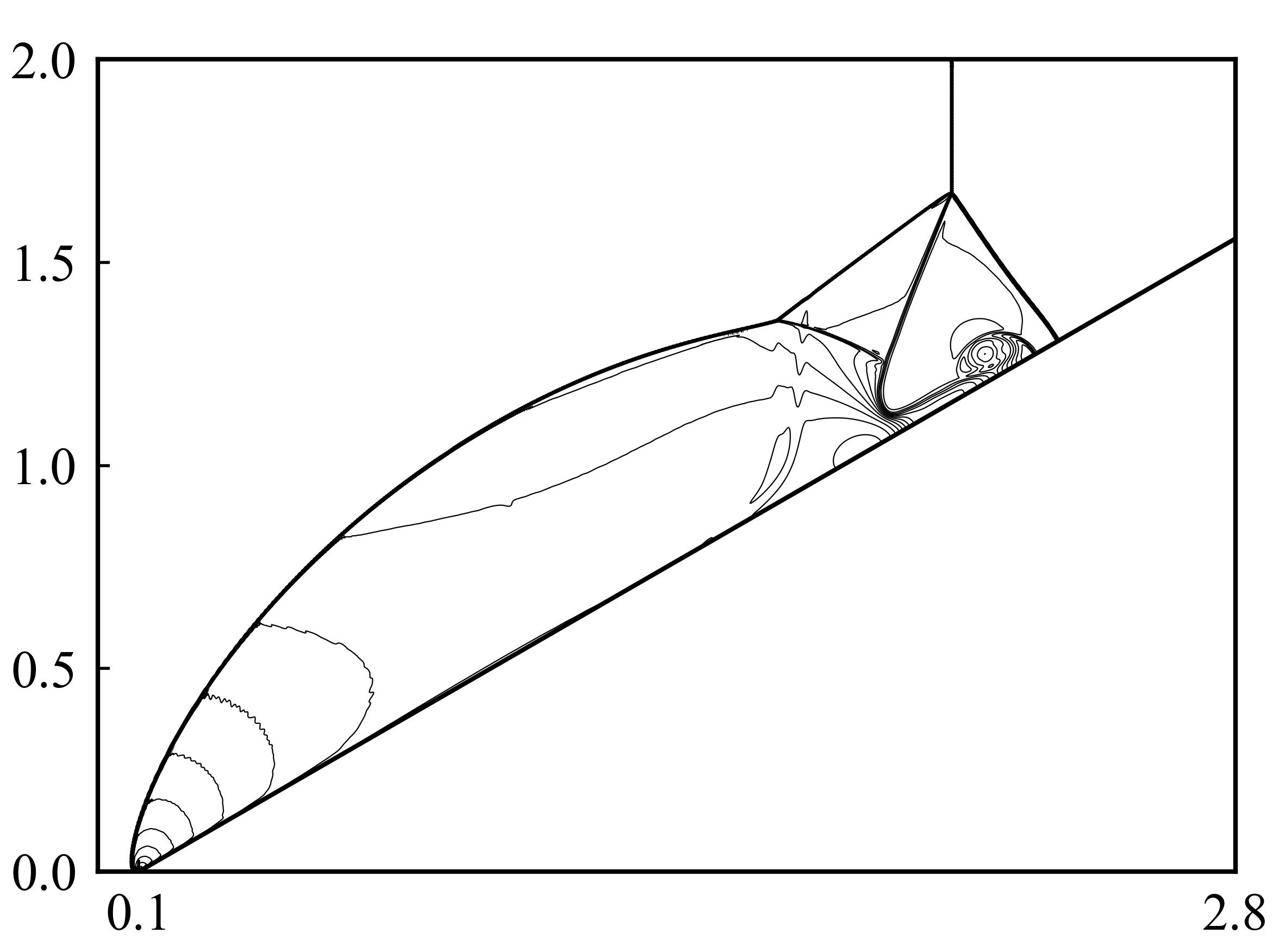}
		\end{subfigure}
	
		\begin{subfigure}{0.43\textwidth}
			\centering
			\includegraphics[width=\textwidth]{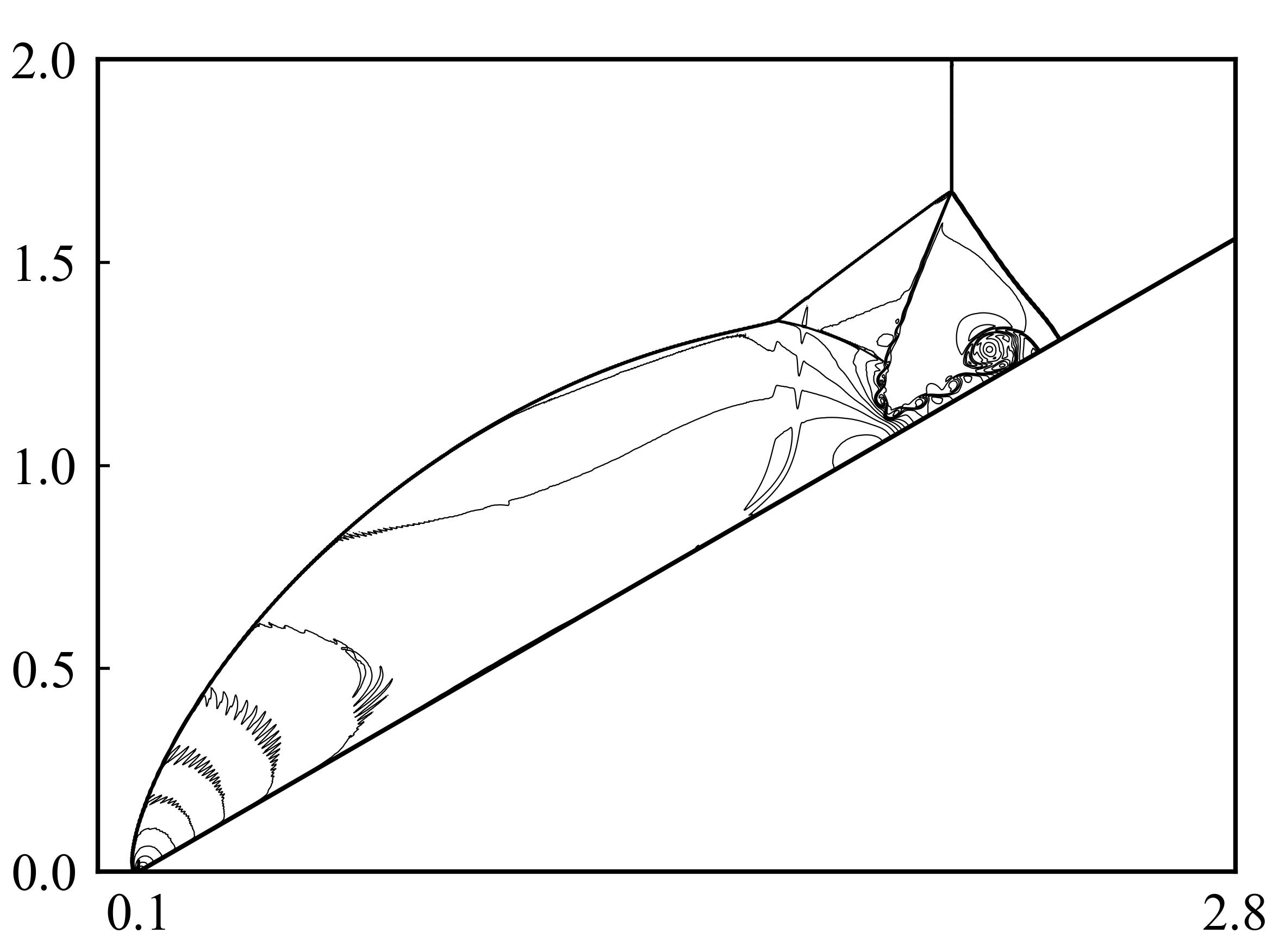}
		\end{subfigure}
		\qquad
		\begin{subfigure}{0.43\textwidth}
			\centering
			\includegraphics[width=\textwidth]{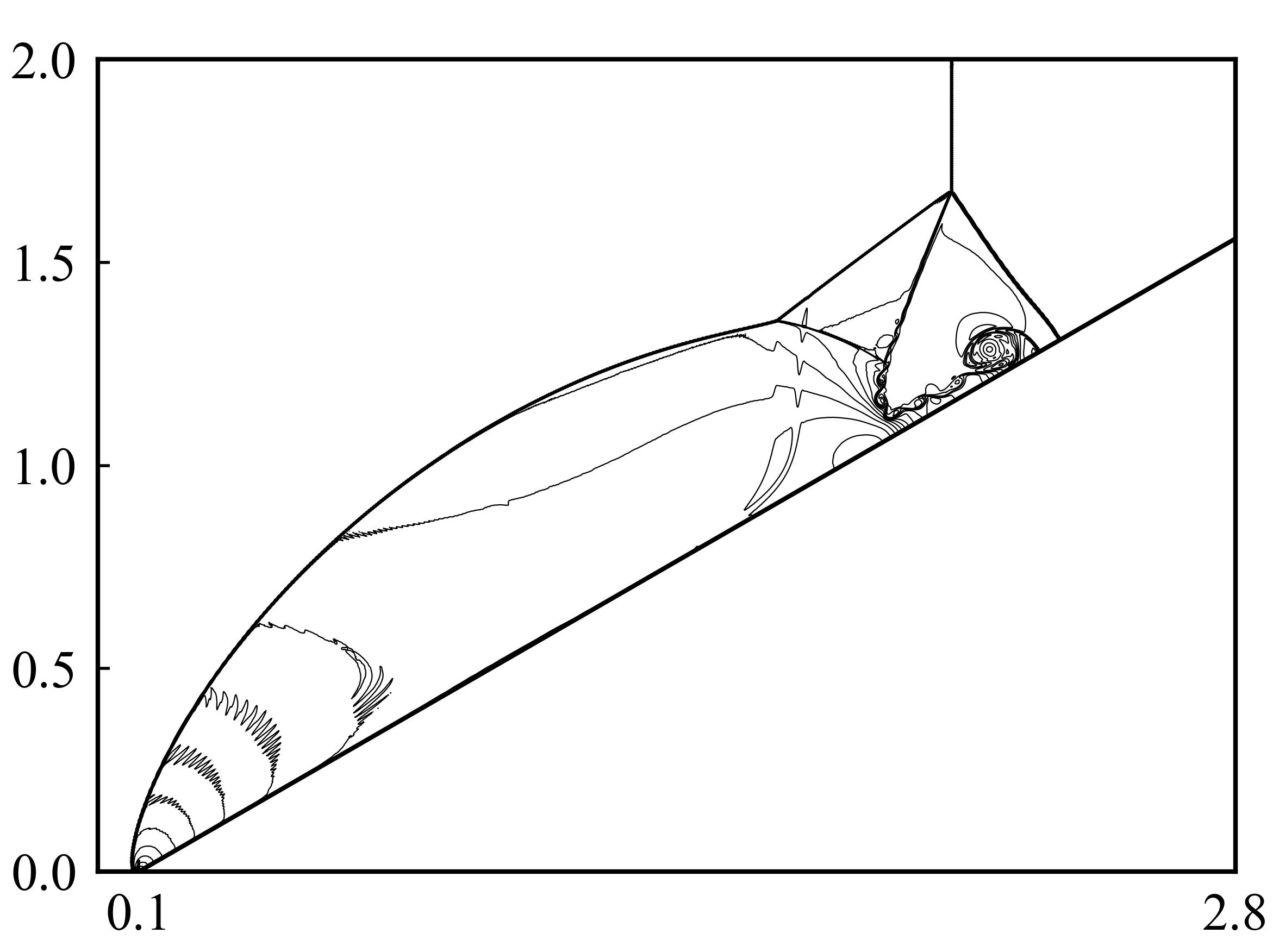}
		\end{subfigure}
		
		\caption{(\Cref{Ex:EulerDMR}) Contours of density obtained by the $\mathbb{P}^1$-based (top) and $\mathbb{P}^2$-based (bottom) BP RI-OEDG schemes. Thirty equally spaced contour lines between 1.5 and 22.7.
		Left: optimal convex decomposition; right: classical convex decomposition.
		}
		\label{fig:Ex_EulDMR}
	\end{figure}

	\begin{figure}[!htb]
		\centering
		\begin{subfigure}{0.43\textwidth}
			\centering
			\includegraphics[width=\textwidth]{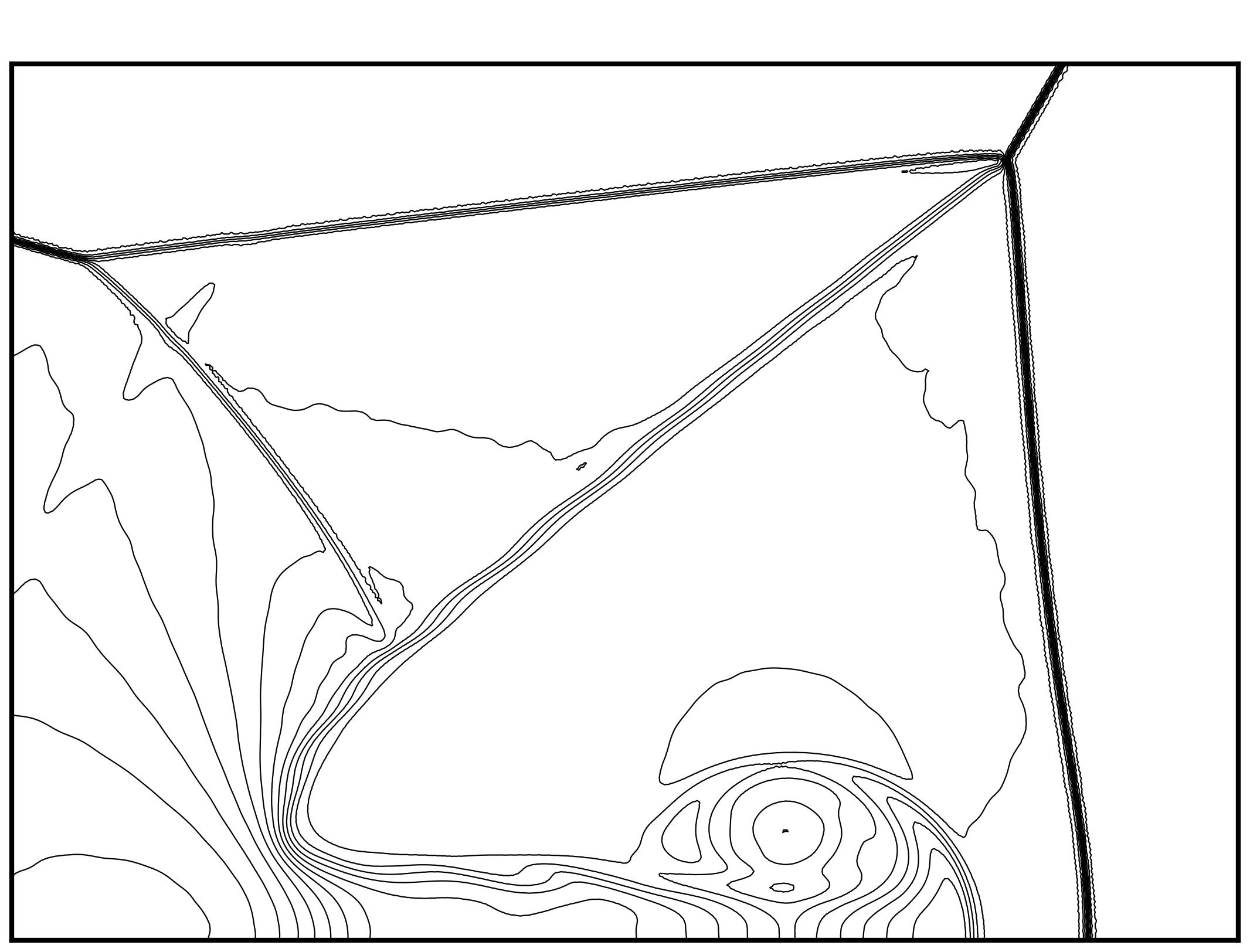}
		\end{subfigure}
		\qquad
		\begin{subfigure}{0.43\textwidth}
			\centering
			\includegraphics[width=\textwidth]{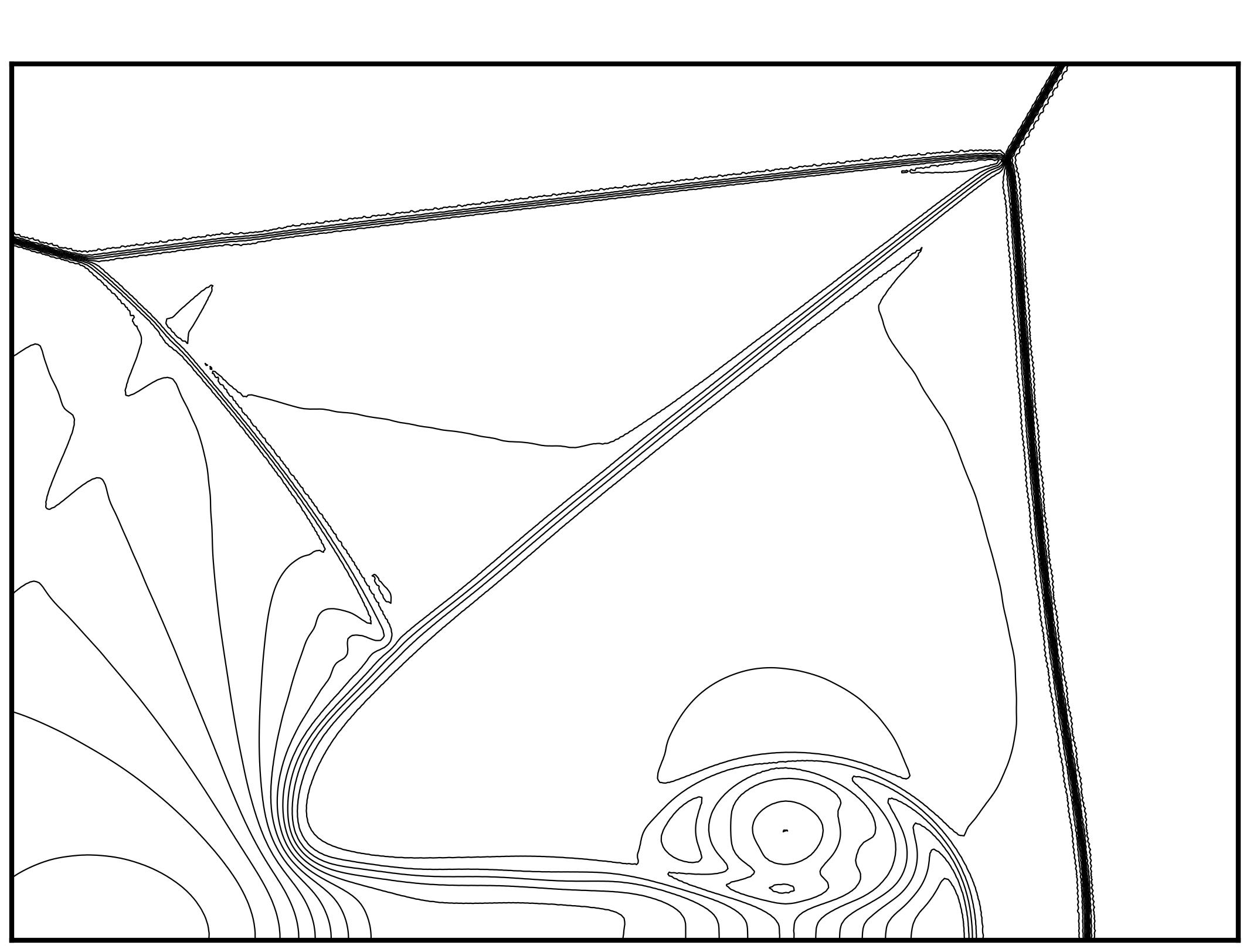}
		\end{subfigure}
		
		\begin{subfigure}{0.43\textwidth}
			\centering
			\includegraphics[width=\textwidth]{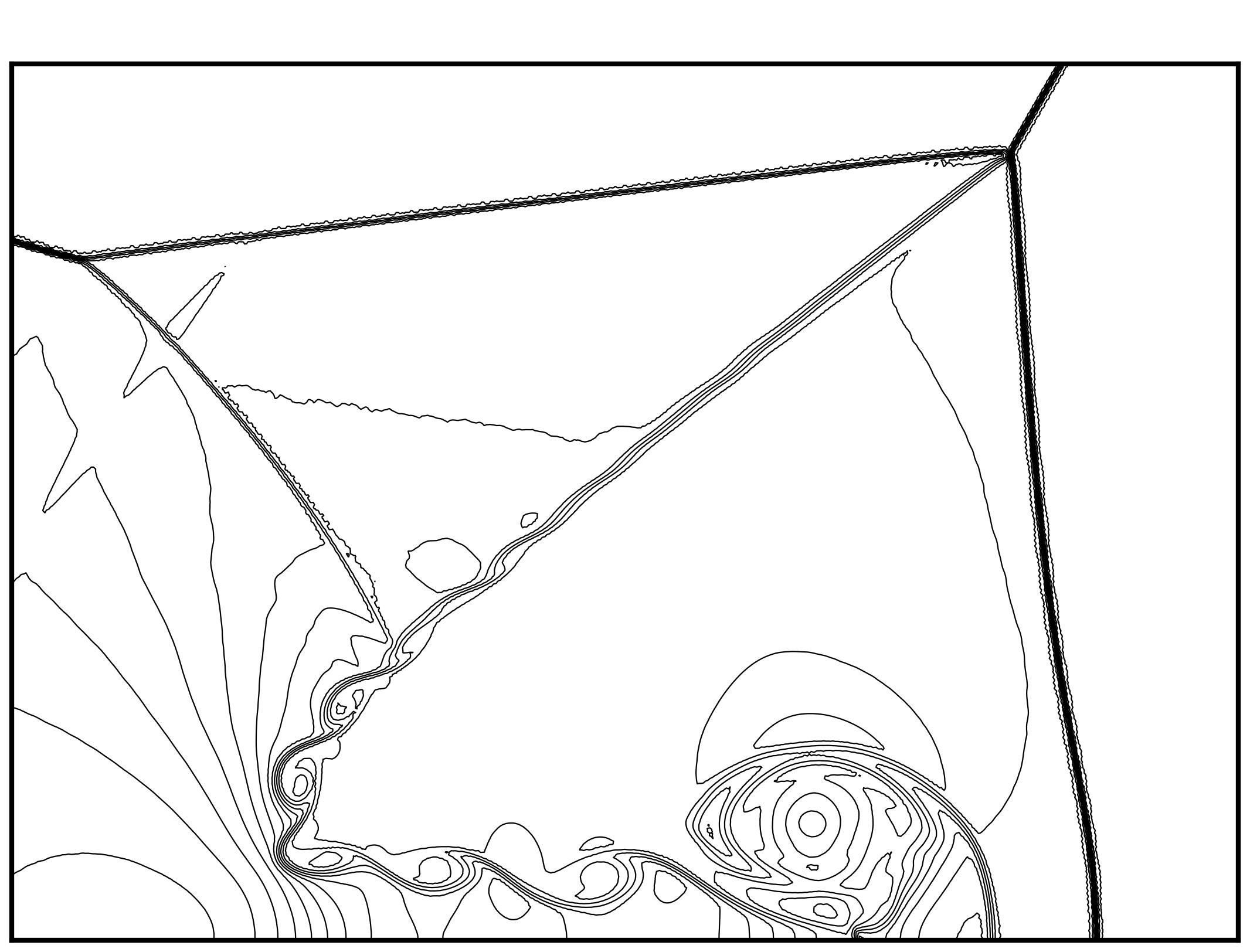}
		\end{subfigure}
		\qquad
		\begin{subfigure}{0.43\textwidth}
			\centering
			\includegraphics[width=\textwidth]{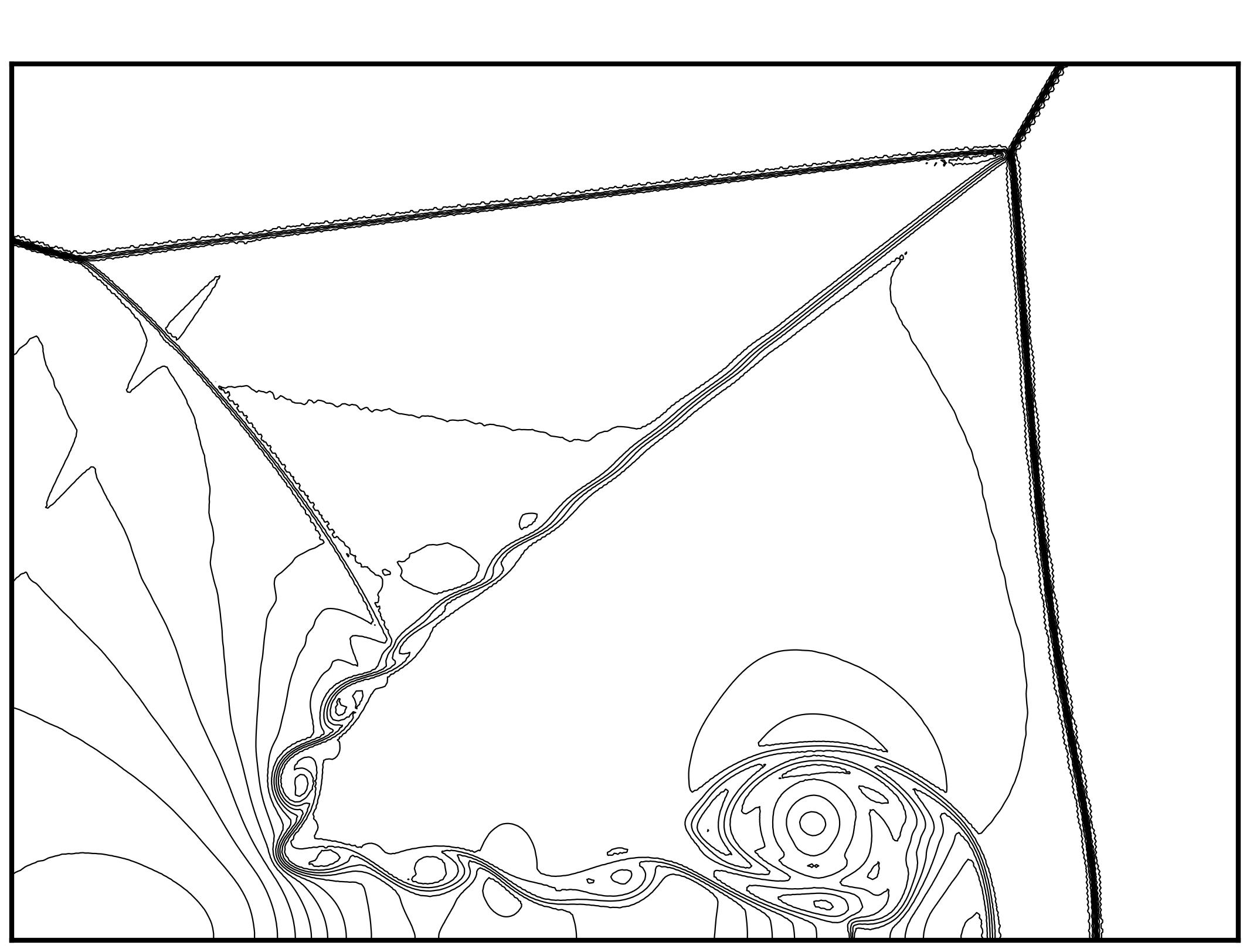}
		\end{subfigure}
		
		\caption{(\Cref{Ex:EulerDMR}) Zoom-in images around the double Mach stem of \Cref{fig:Ex_EulDMR}.
			Top: $\mathbb{P}^1$-element; bottom: $\mathbb{P}^2$-element. 
			Left: optimal convex decomposition; right: classical convex decomposition.
		}
		\label{fig:Ex_EulDMRZoom}
	\end{figure}

	\begin{table}[!htbp] 
		\centering
		\caption{(\Cref{Ex:EulerDMR}) CPU time (hours) computed by using the $\mathbb{P}^1$- and $\mathbb{P}^2$-based BP OEDG methods. 
		}
		\label{tab:Ex_EulDMR}
		\setlength{\tabcolsep}{4mm}{
			\begin{tabular}{ccc}
				\toprule[1.5pt]
				&
				\multicolumn{1}{c}{optimal convex decomposition} &
				\multicolumn{1}{c}{classical convex decomposition}  \\
				
				\midrule[1.5pt]		
				{$\mathbb{P}^1$} & 285.13 & 741.32  \\
				{$\mathbb{P}^2$} & 895.55 & 3195.07 \\
				
				\bottomrule[1.5pt]
			\end{tabular}
		}
	\end{table}
	
\end{expl}

\begin{expl}[Shock diffraction]\label{Ex:EulerSDiffWedge}
	This example simulates the diffraction of a shock wave at a convex corner with a $120^{\circ}$ angle, using the setup detailed in \cite{ZXSPP2012}. The computational domain $\Omega$, bounded by line segments connecting the points (0,0), (13,0), (13,11), (0,11), (0,6), and $(2\sqrt{3},6)$, is discretized into 123,253 triangular cells with a mesh size of $h=0.05$. \Cref{fig:Ex_EulSDWmesh} illustrates the domain and a sample mesh with $h=0.5$. 
	Initially, a right-going shock, parallel to the $x$-axis, is positioned at $x=3.4$ for $6 \leq y \leq 11$. The post-shock state is defined as:
	$$
	(\rho, v_1, v_2, p) = (8, 8.25, 0, 116.5),
	$$
	while the pre-shock state is:
	$$
	(\rho, v_1, v_2, p) = (1.4, 0, 0, 1).
	$$
	Inflow condition is applied on the left boundary $\{x=0, 6 \leq y \leq 11\}$, and outflow conditions are used on the right boundary $\{x=13, 0 \leq y \leq 11\}$ and the bottom boundary $\{0 \leq x \leq 13, y=0\}$. The exact solution is enforced on the top boundary $\{0 \leq x \leq 13, y=11\}$, while reflective conditions are applied to the remaining boundaries. 
	Due to the presence of strong shock waves, the BP limiter is essential to prevent nonphysical solutions and ensure numerical stability. Without the BP limiter, the OEDG codes would yield nonphysical solutions, leading to simulation failure. 
	\Cref{fig:Ex_EulSDW1} shows the pressure contours at $t=0.9$ obtained using the $\mathbb{P}^1$- and $\mathbb{P}^2$-based RI-OEDG methods with either the optimal or the classical convex decomposition. These results demonstrate that all BP OEDG schemes accurately capture the complex wave patterns, highlighting the effectiveness of the RIOE procedure in suppressing nonphysical oscillations near discontinuities. 
	Notably, the results of the BP OEDG method using the optimal convex decomposition are comparable to those using the classical convex decomposition. However, as shown in \Cref{tab:Ex_EulSDW}, the CPU times for the optimal convex decomposition approach are significantly lower. The computed ratios $\overline{\dt}^{\tt CDW}/\overline{\dt}^{\tt ZXS}$ are 2.868430 for $k=1$ and 4.488572 for $k=2$, indicating that the use of optimal convex decomposition greatly enhances the computational efficiency of BP schemes.

	\begin{figure}[!htb]
		\centering
		\begin{subfigure}{0.43\textwidth}
			\centering
			\includegraphics[width=\textwidth]{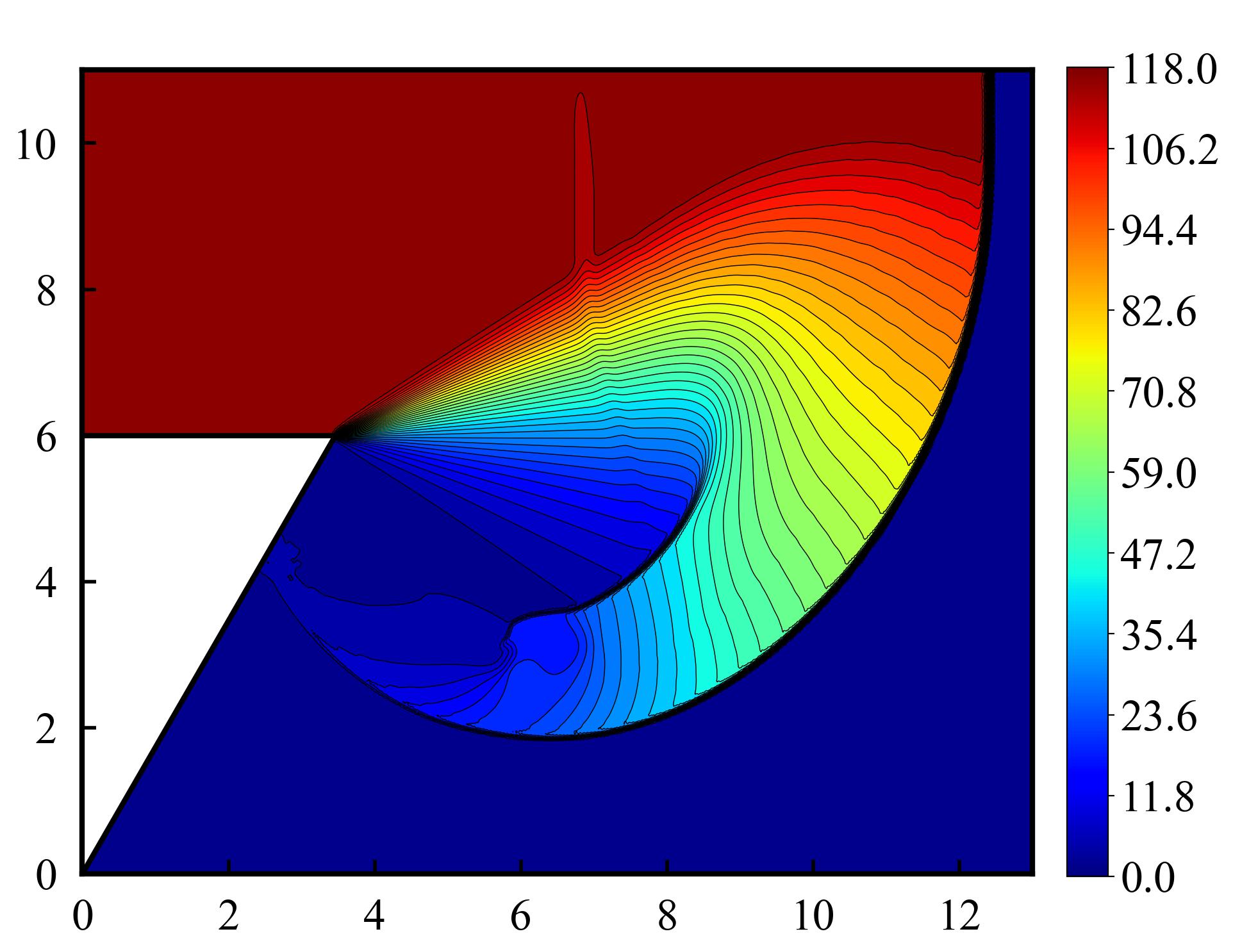}
		\end{subfigure}
		\qquad
		\begin{subfigure}{0.43\textwidth}
			\centering
			\includegraphics[width=\textwidth]{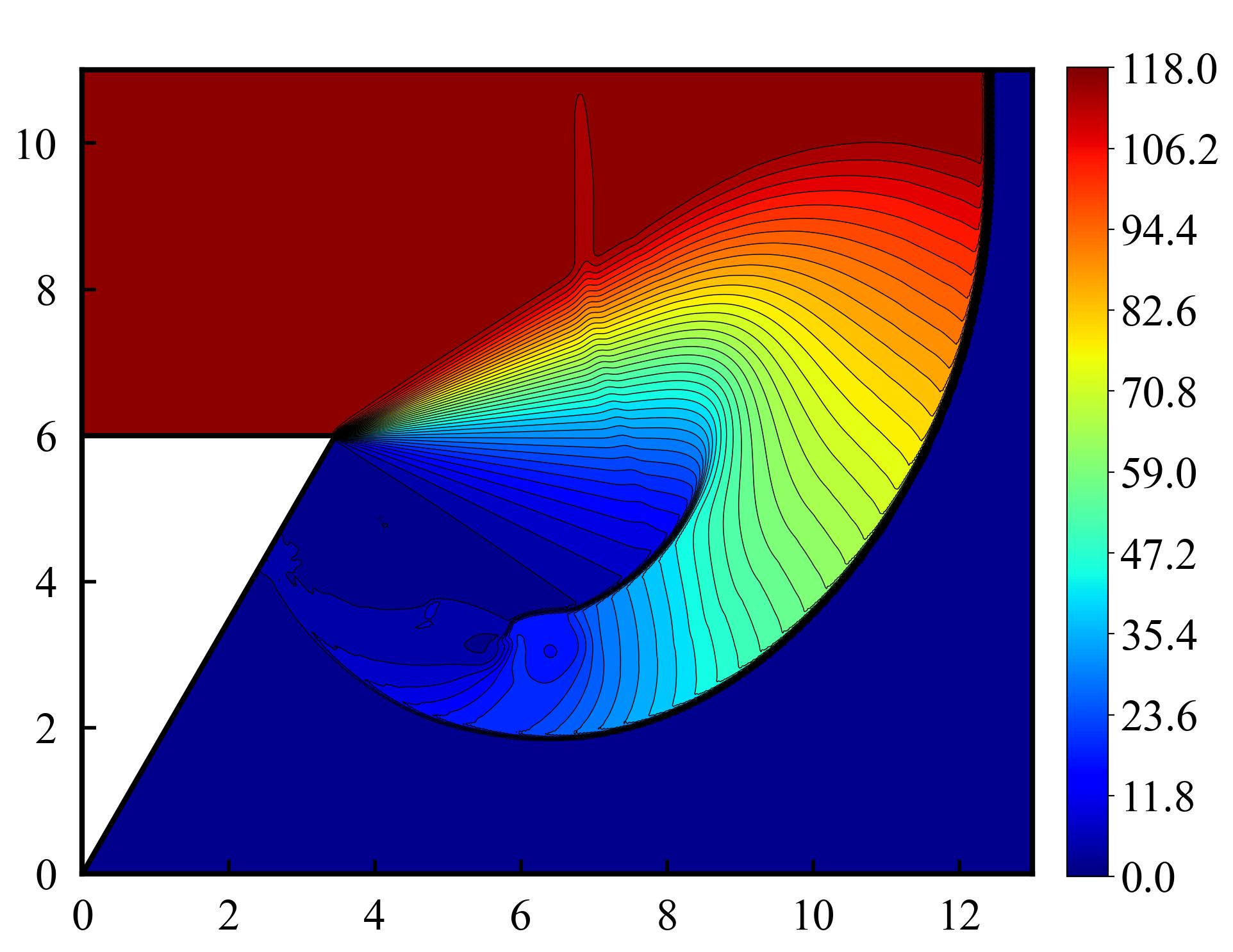}
		\end{subfigure}
		
		\begin{subfigure}{0.43\textwidth}
			\centering
			\includegraphics[width=\textwidth]{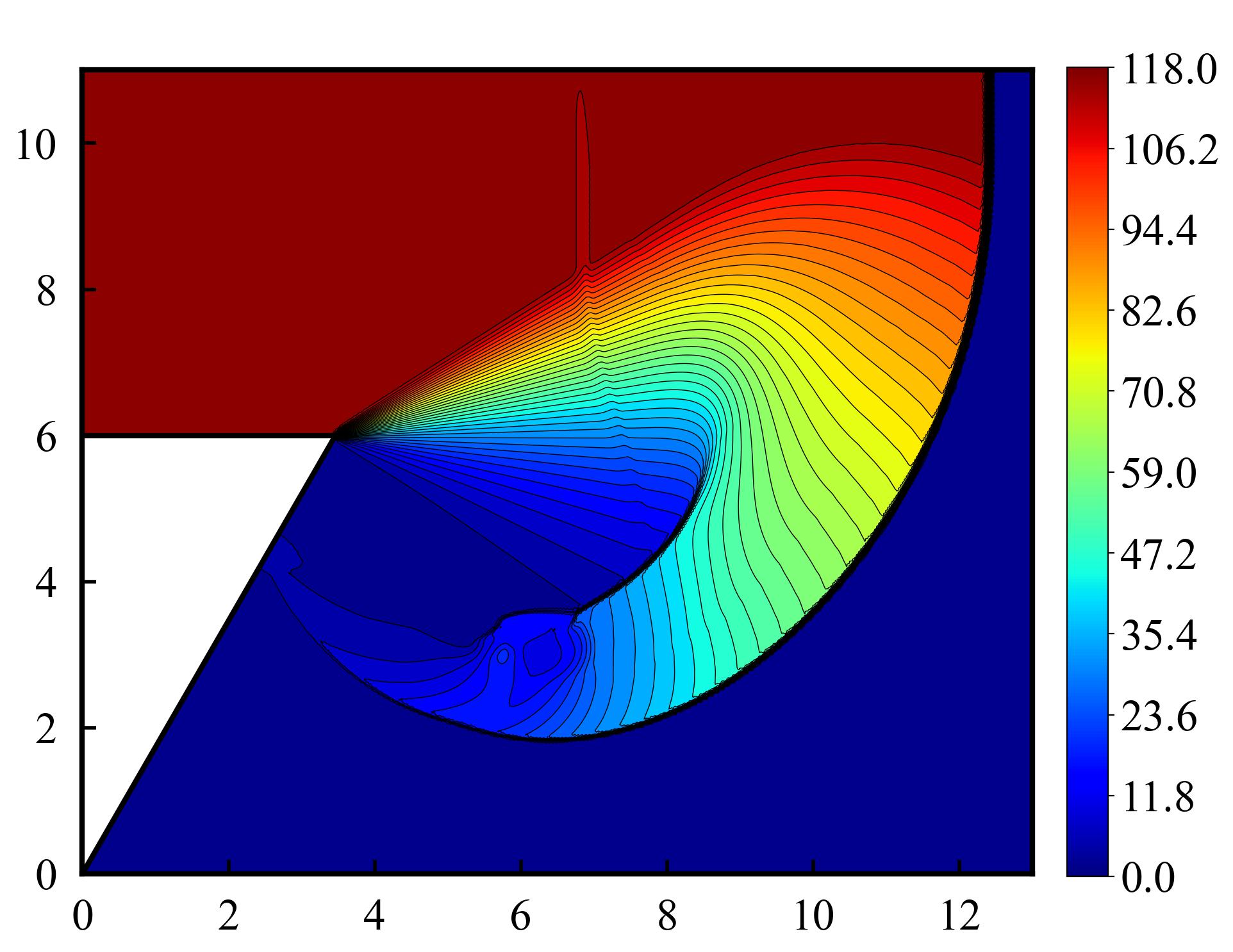}
		\end{subfigure}
		\qquad
		\begin{subfigure}{0.43\textwidth}
			\centering
			\includegraphics[width=\textwidth]{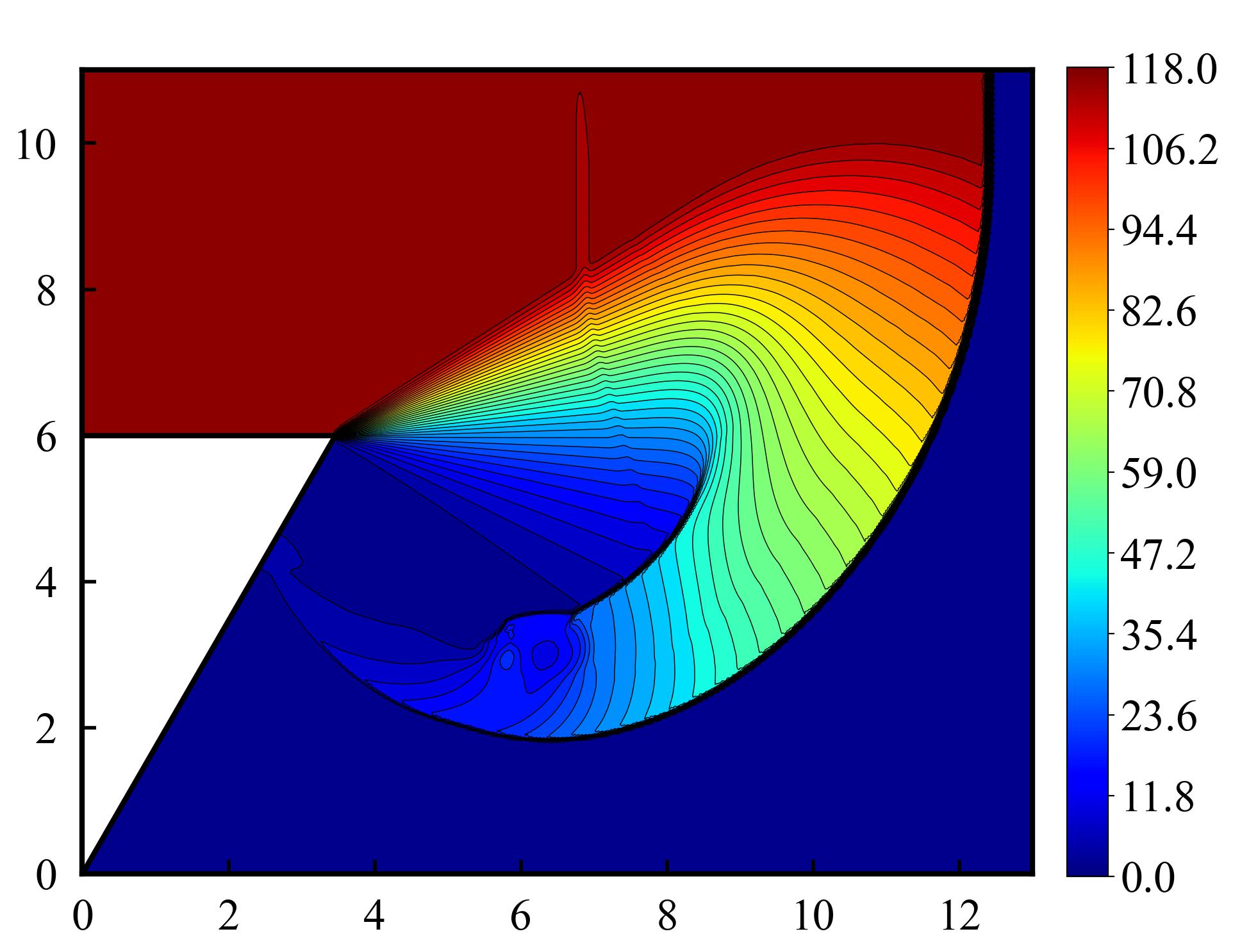}
		\end{subfigure}
		
		\caption{(\Cref{Ex:EulerSDiffWedge}) Contours of pressure obtained by the $\mathbb{P}^1$-based (top) and $\mathbb{P}^2$-based (bottom) BP RI-OEDG schemes. Forty levels from 0 to 118. Left: optimal convex decomposition; right: classical convex decomposition.
		}
		\label{fig:Ex_EulSDW1}
	\end{figure}

	\begin{table}[!htbp] 
		\centering
		\caption{(\Cref{Ex:EulerSDiffWedge}) CPU time (hours) obtained by using the $\mathbb{P}^1$- and $\mathbb{P}^2$-based BP OEDG methods. 
		}
		\label{tab:Ex_EulSDW}
		\setlength{\tabcolsep}{4mm}{
			\begin{tabular}{ccc}
				\toprule[1.5pt]
				&
				\multicolumn{1}{c}{optimal convex decomposition} &
				\multicolumn{1}{c}{classical convex decomposition}  \\
				
				\midrule[1.5pt]		
				{$\mathbb{P}^1$} & 7.28 & 21.70 \\
				{$\mathbb{P}^2$} & 26.75& 120.82 \\
				
				\bottomrule[1.5pt]
			\end{tabular}
		}
	\end{table}
\end{expl}

\section{Conclusions}\label{sec:conclu}

In this paper, we have presented two major advancements to enhance the robustness, reliability, and efficiency of discontinuous Galerkin (DG) methods for hyperbolic conservation laws on unstructured meshes while retaining the accuracy and compactness of DG methods. 
The first advancement involves the development of oscillation-eliminating (OE) techniques specifically tailored for DG schemes on unstructured meshes. Building on prior research, these methods have been shown to maintain essential properties such as conservation and scale/evolution invariance while introducing a novel mechanism to ensure rotational invariance. This new feature is crucial for maintaining the consistency of the numerical solutions of rotation-invariant (RI) physical equations across different coordinate systems. 

The second key contribution of this work is the design of an optimal convex decomposition for constructing high-order bound-preserving (BP) DG schemes on unstructured meshes. This novel decomposition leads to the mildest BP CFL condition, which is pivotal for the stability and efficiency of the schemes. Although the optimal decomposition problem has been thoroughly explored for rectangular meshes, its extension to triangular meshes was previously an open question. This research fills that gap by successfully constructing the optimal decomposition for $\mathbb{P}^1$ and $\mathbb{P}^2$ polynomial spaces on triangular elements, thereby significantly boosting the BP CFL numbers and, consequently, the computational efficiency. 
Furthermore, the proposed RIOE procedure and optimal decomposition approach are easy to integrate into existing DG codes. Their implementation requires only minor adjustments and relies on data from neighboring cells, ensuring that the inherent advantages of DG methods—such as compactness and efficiency in parallel computing—are preserved. 

Extensive numerical experiments have validated the proposed methodologies, demonstrating their effectiveness in maintaining accuracy, suppressing spurious oscillations, and preserving bounds across a variety of challenging test cases. These results underscore the effectiveness of our techniques in reliably handling complex geometries, thereby advancing the applicability of DG methods to a broader class of problems in computational fluid dynamics and beyond. 
This work may also open new avenues for further research and development in robust DG schemes for other related hyperbolic or convection-dominated equations. Future research could explore the extension and unified theory of optimal convex decomposition to even higher-order DG spaces on unstructured meshes.

\appendix
\section{$N$-point Quadrature Rule on the Triangular Element $K$}\label{Appx:Triangle}

For the $N$-point quadrature rule on a triangular element $K$, the quadrature nodes are given by
\begin{equation*}
	x_K^{\mu} = \sum_{i=1}^{3} b_i^{\mu} x_i, \qquad
	y_K^{\mu} = \sum_{i=1}^{3} b_i^{\mu} y_i, \qquad \text{for} \quad 1 \leq \mu \leq N.
\end{equation*}
Here, $(b_1^{\mu}, b_2^{\mu}, b_3^{\mu})$ represents the barycentric coordinates of the quadrature point $(x_K^{\mu}, y_K^{\mu})$ with respect to the triangle $K$, and these coordinates satisfy the following conditions:
$$
\sum_{i=1}^{3} b_i^{\mu} = 1, \qquad 0 < b_i^{\mu} < 1.
$$
The selection of $(b_1^{\mu}, b_2^{\mu}, b_3^{\mu})$ is critical and must be handled carefully to ensure the accuracy of the quadrature rule. In our simulations, we use a quadrature rule with $2k$th-order algebraic precision for $(k+1)$th-order OEDG methods. 
Tables \ref{tab:2th}, \ref{tab:4th}, and \ref{tab:6th} provide the barycentric coordinates $(b_1^{\mu}, b_2^{\mu}, b_3^{\mu})$ and the corresponding weights $\omega_{\mu}$ for quadrature rules with second, fourth, and sixth-order algebraic precision, respectively. For more details on triangle quadrature rules, refer to \cite{Zhang2009Quadrature}.

\begin{table}[!htb] 
	\centering
	\caption{Barycentric coordinates and weights for $3$-point quadrature rule with second-order algebraic precision.}
	\label{tab:2th}
	\renewcommand\arraystretch{1.2}
	\setlength{\tabcolsep}{6mm}{
		\begin{tabular}{ccc}
			\toprule[1.5pt]
			$\mu$ & $(b_1^{\mu},~ b_2^{\mu}, ~ b_3^{\mu})$ & $\omega_{\mu}$ \\
			
			\midrule[1.5pt]	
			1 & $(\frac{1}{6},\frac{1}{6},\frac{2}{3}) $ &  $\frac{1}{3}$  \\
			2 & $(\frac{1}{6},\frac{2}{3},\frac{1}{6})$  &  $\frac{1}{3}$  \\
			3 & $(\frac{2}{3},\frac{1}{6},\frac{1}{6})$  &  $\frac{1}{3}$  \\
			\bottomrule[1.5pt]
		\end{tabular}
	}
\end{table}

\begin{table}[!htb] 
	\centering
	\caption{Barycentric coordinates and weights for $6$-point quadrature rule with fourth-order algebraic precision.}
	\label{tab:4th}
	\setlength{\tabcolsep}{3mm}{
		\begin{tabular}{ccc}
			\toprule[1.5pt]
			$\mu$ & $(b_1^{\mu}, ~ b_2^{\mu}, ~ b_3^{\mu})$ & $\omega_{\mu}$ \\
			
			\midrule[1.5pt]	
			1 & (0.445948490915965, ~ 0.445948490915965, ~ 0.108103018168070) & 0.223381589678010\\
			2 & (0.445948490915965, ~ 0.108103018168070, ~ 0.445948490915965) & 0.223381589678010\\
			3 & (0.108103018168070, ~ 0.445948490915965, ~ 0.445948490915965) & 0.223381589678010\\
			4 & (0.091576213509771, ~ 0.091576213509771, ~ 0.816847572980458) & 0.109951743655322\\
			5 & (0.091576213509771, ~ 0.816847572980458, ~ 0.091576213509771) & 0.109951743655322\\
			6 & (0.816847572980458, ~ 0.091576213509771, ~ 0.091576213509771) & 0.109951743655322\\
			\bottomrule[1.5pt]
		\end{tabular}
	}
\end{table}

\begin{table}[!htb] 
	\centering
	\caption{Barycentric coordinates and weights for $12$-point quadrature rule with sixth-order algebraic precision.}
	\label{tab:6th}
	\setlength{\tabcolsep}{3mm}{
		\begin{tabular}{ccc}
			\toprule[1.5pt]
			$\mu$ & $(b_1^{\mu}, ~ b_2^{\mu}, ~ b_3^{\mu})$ & $\omega_{\mu}$ \\
			
			\midrule[1.5pt]	
			1 & (0.063089014491502, ~ 0.063089014491502, ~ 0.873821971016996) &0.050844906370207\\
			2 & (0.063089014491502, ~ 0.873821971016996, ~ 0.063089014491502) &0.050844906370207\\
			3 & (0.873821971016996, ~ 0.063089014491502, ~ 0.063089014491502) &0.050844906370207\\
			4 & (0.249286745170910, ~ 0.249286745170910, ~ 0.501426509658180) &0.116786275726379\\
			5 & (0.249286745170910, ~ 0.501426509658180, ~ 0.249286745170910) &0.116786275726379\\
			6 & (0.501426509658180, ~ 0.249286745170910, ~ 0.249286745170910) &0.116786275726379\\
			7 & (0.053145049844817, ~ 0.636502499121399, ~ 0.310352451033784) &0.082851075618374\\
			8 & (0.636502499121399, ~ 0.053145049844817, ~ 0.310352451033784) &0.082851075618374\\
			9 & (0.053145049844817, ~ 0.310352451033784, ~ 0.636502499121399) &0.082851075618374\\
			10& (0.636502499121399, ~ 0.310352451033784, ~ 0.053145049844817) &0.082851075618374\\
			11& (0.310352451033784, ~ 0.053145049844817, ~ 0.636502499121399) &0.082851075618374\\
			12& (0.310352451033784, ~ 0.636502499121399, ~ 0.053145049844817) &0.082851075618374\\
			\bottomrule[1.5pt]
		\end{tabular}
	}
\end{table}

\section{$Q$-Point Gauss Quadrature Rule on the Edge $e_K^{(i)}$}\label{Appx:Edge}
Let $\{x^{\nu}\}_{\nu=1}^{Q}$ represent the $Q$-point quadrature nodes in the interval $[-1,1]$. Assume that $\bm x_c := (x_c, y_c)^\top$ and $\bm x_d := (x_d, y_d)^\top$ are the coordinates of the  endpoints ${\bm v}_c$ and  ${\bm v}_d$ of the edge $e_K^{(i)}$, respectively. The coordinates of the quadrature points on the edge $e_K^{(i)}$ are then given by
\begin{equation}\label{eq:AppxEdge}
	x^{(i),\nu}_K = x_c + (x_d - x_c) x^{\nu}, \qquad
	y^{(i),\nu}_K = y_c + (y_d - y_c) x^{\nu}.
\end{equation}

\bibliographystyle{model1-num-names}%

\bibliography{refs}

\end{document}